\documentclass[11pt]{amsart}%

\usepackage{amsmath, amssymb, amsthm, amsfonts, bbm, dsfont, stmaryrd}
\usepackage{graphicx}

\usepackage[bbgreekl]{mathbbol}

\input xy
\xyoption{all}

\usepackage{hyperref}

\usepackage{float}
\restylefloat{table}

\usepackage{caption}
\usepackage{subcaption}

\numberwithin{equation}{section}

\theoremstyle{plain}
\newtheorem{theorem}[subsubsection]{Theorem}
\newtheorem{lemma}[subsubsection]{Lemma}
\newtheorem{prop}[subsubsection]{Proposition}
\newtheorem{cor}[subsubsection]{Corollary}
\newtheorem{conj}[subsubsection]{Conjecture}
\newtheorem*{claim}{Claim}

\newtheorem{aprop}[subsection]{Proposition}

\theoremstyle{definition}
\newtheorem{defn}[subsubsection]{Definition}
\newtheorem{cons}[subsubsection]{Construction}
\newtheorem{remark}[subsubsection]{Remark}
\newtheorem{exam}[subsubsection]{Example}

\def\AA{\mathbb{A}}
\def\BB{\mathbb{B}}
\def\CC{\mathbb{C}}
\def\DD{\mathbb{D}}

\def\GG{\mathbb{G}}
\def\HH{\mathbb{H}}

\def\JJ{\mathbb{J}}

\def\PP{\mathbb{P}}
\def\QQ{\mathbb{Q}}
\def\RR{\mathbb{R}}

\def\TT{\mathbb{T}}

\def\WW{\mathbb{W}}

\def\ZZ{\mathbb{Z}}

\def\cc{\mathbbm{c}}
\def\ee{\mathbbm{e}}
\def\ff{\mathbbm{f}}
\def\gg{\mathbbm{g}}
\def\hh{\mathbbm{h}}
\def\rr{\mathbbm{r}}
\def\ss{\mathbbm{s}}
\def\tt{\mathbbm{t}}
\def\bchi{\mathbbm{\chi}}
\def\kkappa{\mathbbm{\kappa}}
\def\PPhi{\bbphi}

\def\calA{\mathcal{A}}
\def\calB{\mathcal{B}}
\def\calC{\mathcal{C}}

\def\calE{\mathcal{E}}

\def\calH{\mathcal{H}}

\def\calK{\mathcal{K}}
\def\calL{\mathcal{L}}
\def\calM{\mathcal{M}}
\def\calN{\mathcal{N}}
\def\calO{\mathcal{O}}
\def\calP{\mathcal{P}}
\def\calQ{\mathcal{Q}}

\def\calV{\mathcal{V}}

\def\bG{\mathbf{G}}

\def\bI{\mathbf{I}}

\def\bP{\mathbf{P}}
\def\bQ{\mathbf{Q}}
\def\bR{\mathbf{R}}

\def\bT{\mathbf{T}}

\newcommand{\fra}{\mathfrak{a}}
\newcommand{\frg}{\mathfrak{g}}
\newcommand{\frp}{\mathfrak{p}}

\newcommand{\frb}{\mathfrak{b}}
\newcommand{\frc}{\mathfrak{c}}
\newcommand{\frl}{\mathfrak{l}}
\newcommand{\frs}{\mathfrak{s}}

\newcommand{\frA}{\mathfrak{A}}
\newcommand{\frF}{\mathfrak{F}}

\newcommand\frL{\mathfrak{L}}
\newcommand\frM{\mathfrak{M}}

\newcommand{\tilL}{\widetilde{L}}
\newcommand{\tilW}{\widetilde{W}}
\newcommand{\tilS}{\widetilde{S}}

\newcommand{\tilw}{\widetilde{w}}
\newcommand{\tila}{\widetilde{a}}

\newcommand{\tilV}{\widetilde{V}}
\newcommand{\tilJ}{\widetilde{J}}

\newcommand{\tilQ}{\widetilde{Q}}
\newcommand{\tilG}{\widetilde{G}}

\newcommand{\tq}{\widetilde{q}}
\newcommand{\tSp}{\widetilde{\textup{Sp}}}

\newcommand\act{\textup{act}}
\newcommand\aff{\textup{aff}}
\newcommand\Ann{\textup{Ann}}

\newcommand{\can}{\textup{can}}
\newcommand{\codim}{\textup{codim}}
\newcommand{\coker}{\textup{coker}}
\newcommand\ev{\textup{ev}}

\newcommand\Gal{\textup{Gal}}

\newcommand{\Gr}{\textup{Gr}}
\newcommand{\gr}{\textup{gr}}
\newcommand\id{\textup{id}}
\newcommand\Image{\textup{Im}}
\newcommand{\Ind}{\textup{Ind}}

\newcommand\Irr{\textup{Irr}}

\newcommand\lcm{\textup{lcm}}
\newcommand\leng{\textup{leng}}
\newcommand\Lie{\textup{Lie}\ }

\newcommand\Out{\textup{Out}}

\newcommand{\Pic}{\textup{Pic}}
\newcommand{\prim}{\textup{prim}}
\newcommand\rank{\textup{rank}\ }
\newcommand{\red}{\textup{red}}
\newcommand{\reg}{\textup{reg}}
\newcommand{\Reg}{\textup{Reg}}

\newcommand{\Res}{\textup{Res}}

\newcommand\rs{\textup{rs}}

\newcommand\Span{\textup{Span}}
\newcommand\Spec{\textup{Spec}\ }

\newcommand\st{\textup{st}}

\newcommand\Sym{\textup{Sym}}
\newcommand{\tors}{\textup{tors}}

\newcommand{\Tr}{\textup{Tr}}

\newcommand\triv{\textup{triv}}

\newcommand{\val}{\textup{val}}

\newcommand\Aut{\textup{Aut}}
\newcommand\Hom{\textup{Hom}}
\newcommand\End{\textup{End}}

\newcommand{\Ext}{\textup{Ext}}

\newcommand\GL{\textup{GL}}

\newcommand\PGL{\textup{PGL}}

\newcommand\SL{\textup{SL}}
\def\sl{\mathfrak{sl}}
\newcommand\SU{\textup{SU}}

\newcommand\SO{\textup{SO}}

\newcommand\Sp{\textup{Sp}}

\newcommand{\ad}{\textup{ad}}
\newcommand{\Ad}{\textup{Ad}}
\def\sc{\textup{sc}}

\newcommand\xch{\mathbb{X}^*}
\newcommand\xcoch{\mathbb{X}_*}

\newcommand\incl{\hookrightarrow}
\newcommand\surj{\twoheadrightarrow}
\newcommand\bij{\leftrightarrow}
\newcommand{\isom}{\stackrel{\sim}{\to}}
\newcommand{\quash}[1]{}
\newcommand{\leftexp}[2]{{\vphantom{#2}}^{#1}{#2}}

\newcommand{\twtimes}[1]{\stackrel{#1}{\times}}

\newcommand{\jiao}[1]{\langle{#1}\rangle}
\newcommand{\nth}{^\textup{th}}
\newcommand{\pt}{\textup{pt}}

\newcommand{\KM}{\textup{KM}}  % Kac-Moody
\newcommand{\wt}{\widetilde}
\newcommand{\wh}{\widehat}
\newcommand{\ep}{\epsilon}
\newcommand{\vep}{\varepsilon}

\newcommand\mat[4]{\left(\begin{array}{cc} #1 & #2 \\ #3 & #4 \end{array}\right)}  % 2-by-2 matrix

\newcommand{\SUodd}{\leftexp{2}{A_{2n}}}

\newcommand{\Cha}{Cherednik algebra }
\newcommand{\Chas}{Cherednik algebras }
\newcommand{\Hgr}{\mathfrak{H}^{\textup{gr}}}
\newcommand{\barHgr}{\overline{\mathfrak{H}}^{\textup{gr}}}
\newcommand{\rat}{\textup{rat}}
\newcommand{\Hrat}{\mathfrak{H}^{\textup{rat}}}
\newcommand{\sHrat}{\mathfrak{H}^{\textup{rat},\square}}

\newcommand{\pH}{\leftexp{p}{\textup{H}}}
\newcommand{\ptau}{\leftexp{p}{\tau}}
\newcommand{\upH}{\textup{H}}
\newcommand{\cohog}[2]{\upH^{#1}({#2})}     % plain group
\newcommand{\homog}[2]{\upH_{#1}({#2})}
\newcommand{\cohoc}[2]{\upH_{c}^{#1}({#2})}     % compact support cohomology
\newcommand{\eqnu}[2]{\upH^{#1}_{\Gnu}({#2})}

\newcommand{\eqnust}[2]{\upH^{#1}_{\Gnu}({#2})_{\st}}

\newcommand{\spcoh}[2]{\upH^{#1}_{\ep=1}({#2})}  % specialized equiv cohomology

\newcommand{\Bun}{\textup{Bun}}
\newcommand{\Fl}{\textup{Fl}}
\newcommand{\fl}{f\ell}

\newcommand{\MHit}{\calM^{\textup{Hit}}}
\newcommand{\fHit}{f^{\textup{Hit}}}

\newcommand{\Aa}{\calA^{\textup{ell}}}
\newcommand{\fa}{f^{\textup{ell}}}
\newcommand{\fda}{f^{\vee,\textup{ell}}}
\newcommand{\h}{\heartsuit}
\newcommand{\Ah}{\calA^{\heartsuit}}

\newcommand{\dGG}{\GG^{\vee}}

\newcommand{\fst}{(f^{\textup{ell}}_{U,!}\QQ)_{\st}}
\newcommand{\fdst}{(f^{\vee,\textup{ell}}_{U,*}\QQ)_{\st}}
\newcommand{\faUQl}{f^{\textup{ell}}_{U,*}\QQ}
\newcommand{\fUQl}{f_{U,*}\QQ}

\newcommand{\Gm}{\GG_m}

\newcommand{\Gnu}{\GG_m(\nu)}

\newcommand{\rot}{\textup{rot}}
\newcommand{\dil}{\textup{dil}}
\newcommand{\cen}{\textup{cen}}
\newcommand{\Grot}{\Gm^{\rot}}
\newcommand{\Grote}{\Gm^{\rot,[e]}}
\newcommand{\hGrot}{\widehat{\GG}_m^{\rot}}
\newcommand{\hGm}{\widehat{\GG}_m}
\newcommand{\Gdil}{\Gm^{\dil}}
\newcommand{\Gcen}{\Gm^{\textup{cen}}}
\newcommand{\Gtwo}{\Grote\times\Gdil}

\newcommand{\Wa}{W_{\aff}}

\renewcommand\L{\Lambda}
\renewcommand\l{\lambda}
\renewcommand\c{\circ}

\newcommand{\rhh}{\dot{\hh}}

\newcommand{\cG}{\mathcal{G}}

\newcommand{\hZZ}{\widehat{\mathbb{Z}}}

\newcommand{\hO}{\widehat{\calO}}
\newcommand{\hK}{\widehat{K}}
\newcommand{\Hess}{\textup{Hess}}
\newcommand{\tHess}{\widetilde{\textup{Hess}}}

\newcommand{\barnu}{\overline{\nu}}
\newcommand{\barzeta}{\overline{\zeta}}
\newcommand{\barx}{\overline{x}}
\newcommand{\barbeta}{\overline{\beta}}
\newcommand{\tbP}{\widetilde{\bP}}
\newcommand{\tpi}{\widetilde{\pi}}
\newcommand{\tX}{\widetilde{X}}

\begin{document}
%opening
\title{Geometric representations of graded and rational Cherednik algebras}
\author{Alexei Oblomkov}
\email{oblomkov@math.umass.edu}
\address{Department of Mathematics, University of Massachusetts at Amherst}
\author{Zhiwei Yun}
\email{zwyun@stanford.edu}
\address{Department of Mathematics, Stanford University, 450 Serra Mall, Bldg 380, Stanford, CA 94305}
\date{}
\subjclass[2010]{Primary ; Secondary }

\begin{abstract} We provide geometric constructions of modules over the graded Cherednik algebra $\Hgr_{\nu}$ and the rational Cherednik algebra $\Hrat_{\nu}$ attached to a simple algebraic group $\GG$ together with a pinned automorphism $\theta$. These modules are realized on the cohomology of affine Springer fibers (of finite type) that admit $\CC^*$-actions. In the rational Cherednik algebra case, the standard grading on these modules is derived from the perverse filtration
on the cohomology of affine Springer fibers coming from its global analog: Hitchin fibers. When $\theta$ is trivial, we show that our construction gives the irreducible finite-dimensional spherical modules $\frL_{\nu}(\triv)$ of $\Hgr_{\nu}$ and of $\Hrat_{\nu}$. We give a formula for the dimension of $\frL_{\nu}(\triv)$ and give a geometric interpretation of its Frobenius algebra structure. The rank two cases are studied in further details.

\end{abstract}

\maketitle

\tableofcontents

\section{Introduction}

\subsection{\Chas and its representations}
Double affine Hecke algebras were introduced by Cherednik to prove Macdonald conjectures \cite{Ch}. 
Etingof and Ginzburg \cite{EG} introduced the graded \Cha $\Hgr_{\nu}$ and rational \Cha $\Hrat_{\nu}$ which are 
the degenerations of the corresponding double affine Hecke algebras. Here $\nu\in\CC$ is the central charge. We will recall the definitions of these algebras in \S\ref{sss:intro Ch} and in \S\ref{s:grad alg def}-\ref{s:rat alg def} with more details.

The papers \cite{DO} and \cite{BEG} initiated the study of
the representation theory  of these algebras. In \cite{DO},  the category $\calO$ for $\Hrat_{\nu}$-modules was introduced. It consists of $\Hrat_{\nu}$-modules with locally nilpotent action of the subalgebra $\Sym(\fra)$. An irreducible representation $\tau$ of the finite Weyl group $W$ naturally gives rise to an $\Hrat_{\nu}$-modules $\frM_{\nu}(\tau)$ via induction from the subalgebra $\Sym(\fra^{*})\rtimes W$. It is shown in \cite{DO} that the simple quotients $\frL_{\nu}(\tau)$ of $\frM_{\nu}(\tau)$ exhaust all simple modules in the category $\calO$.
 
%$$\mathbbm{t}$$

The study of representation theory of $\Hrat_{\nu}$ continued in \cite{BEG} where the classification of all finite-dimensional simple modules in type $A$ was accomplished.  Besides type $A$ the classification of finite-dimensional representations remains a challenge. In type $B$ it is known how many finite-dimensional $\Hrat_{\nu}$-modules there are (see \cite{SV}) and more generally a combinatorial  formula for the characters of the simple modules in the category $\calO$ is proved in \cite{Lo}, \cite{SRVV} and \cite{W} for the classical root systems. With an exception of $H_3$ (see \cite{B}), outside of the classical types neither classification of the finite-dimensional representations nor character formula for the simple modules are known. However, it is completely understood when the polynomial representation $\frM_{\nu}(\triv)$ of a Cherednik algebra has a finite dimensional quotient (see \cite{EtPoly} for an algebraic solution and \cite{VV} for a geometric one): this happens if and only if $\frL_{\nu}(\triv)$ is finite-dimensional, if and only if $\nu\in\QQ_{>0}$ and its denominator is a {\em regular elliptic number} of $W$ (see Definition \ref{def:slope}).

In this paper we do not attempt a classification of irreducible finite-dimensional modules but rather provide a geometric construction of some families of finite-dimensional representations of $\Hgr_\nu$ and $\Hrat_\nu$. In all the examples we checked, our construction seems to give all finite-dimensional irreducible modules of $\Hrat_{\nu}$. In particular, we obtain a geometric construction of the finite-dimensional simple module $\frL_{\nu}(\triv)$ and a geometric interpretation of the grading and the Frobenius structure on $\frL_{\nu}(\triv)$. Our construction also allows us to derive a formula for the dimension of $\frL_{\nu}(\triv)$. 

%, which is expected to have an interpretation via the representation theory of $p$-adic groups as explained in \cite[page 443]{VV}.
% thus providing a solution to one of the problems posed in \cite{VV} \footnote{which one?}.

We remark that the geometric construction for simple modules of $\Hgr_{\nu}$ was first systematically carried out by Varagnolo and Vasserot in \cite{VV} using the equivariant $K$-theory of homogeneous affine Springer fibers. They then constructed simple modules of $\Hrat_{\nu}$ by a purely algebraic degeneration process. Our methods are related but different from those of \cite{VV}. We realize $\Hgr_{\nu}$-modules on the equivariant cohomology of homogeneous affine Springer fibers, which is technically simpler. Our realization of $\Hrat_{\nu}$-modules is also geometric: we construct a filtration (called the perverse filtration, see \S\ref{ss:perv fil}) on the cohomology of homogeneous affine Springer fibers coming from the global geometry of Hitchin fibration and construct an action of $\Hrat_{\nu}$ on the associated graded of the perverse filtration. Therefore the geometry of Hitchin fibration is a new ingredient in our construction compared to the approach in \cite{VV}. Our construction works thanks to the deep geometric results about Hitchin fibration proved by Ng\^o \cite{NgoFL} along the way of proving the fundamental lemma, and the global Springer theory developed by one of the authors (Z.Y.) in his thesis \cite{GS} and \cite{GSLD}.

\subsection{Main results}\label{ss:main results} All varieties are over $\CC$ in this paper. When talking about (equivariant) cohomology of varieties, we mean the singular cohomology of their underlying complex analytic spaces with $\QQ$-coefficients. 

To simplify notation, in this introduction we let $\GG$ be an almost simple, connected and simply-connected reductive group over $\CC$, and let $G=\GG\otimes_{\CC}F$ where $F=\CC((t))$ is the field of formal Laurent series in one variable $t$. In the main body of the paper we shall work with quasi-split groups over $F$ and drop the simple-connectedness condition. Let $\TT\subset\GG$ be a fixed maximal torus with Lie algebra $\tt$, Weyl group $\WW$ and root system $\Phi$. Let $\fra=\xcoch(\TT)\otimes\QQ$ and $\fra^{*}=\xch(\TT)\otimes\QQ$. 

We shall introduce one by one the algebraic and geometric objects that are involved in the statement of our main results.

\subsubsection{Homogeneous elements} Let $\frg$ be the Lie algebra of $G$ and let $\frc=\frg\sslash G$ be the GIT quotient, which is isomorphic to an affine space $\AA^{r}$ ($r$ is the rank of $\GG$). There is a canonical weighted action of $\Gm$ on $\frc$ induced from the dilation action on $\frg$. An element $a\in\frc(F)^{\rs}$ is homogeneous of slope $\nu=d/m$ (in lowest terms) if
\begin{equation*}
s^{d}\cdot a(t)=a(s^{m}t), \text{ for any } s\in\CC^{\times}.
\end{equation*}
Here we write coordinates of $a$ as formal Laurent series in $t$, and $s^{d}\cdot(-)$ is the weighted action of $s^{d}$. The slopes $\nu$ above are not arbitrary: their denominators $m$ are exactly the regular numbers of the Weyl group $\WW$ (i.e., orders of regular elements in $\WW$ in the sense of Springer \cite{Spr}). A rational number $\nu$ is called an elliptic slope if its denominator is an elliptic regular number of $\WW$, see Definition \ref{def:slope}. Let $\frc(F)^{\rs}_{\nu}$ be the set of homogeneous elements of slope $\nu$. This is an open subset of an affine space over $\CC$.

\subsubsection{Homogeneous affine Springer fibers}For any $\gamma\in\frg(F)^{\rs}$ one can define a closed subvariety $\Sp_{\gamma}$ of the affine flag variety $\Fl$ of $G$ called the {\em affine Springer fiber} (\cite{KL}). It classifies Iwahori subalgebras of $\frg$ containing $\gamma$. 

For $a\in\frc(F)^{\rs}_{\nu}$, we let $\Sp_{a}$ denote $\Sp_{\gamma}$ where $\gamma=\kappa(a)$, and $\kappa$ is the Kostant section $\kappa: \frc(F)\incl\frg(F)$. Because of the homogeneity of $a$, there is an action of a one-dimensional torus $\Gm(\nu)$ on the affine Springer fiber $\Sp_{a}$. When $\nu$ is elliptic, $\Sp_{a}$ is a projective scheme.

\subsubsection{Notation on $\Gnu$-equivariant cohomology}\label{sss:intro ep} For a variety $Y$ equipped with an action of the one-dimensional torus $\Gnu$, its equivariant cohomology $\eqnu{*}{Y}=\eqnu{*}{Y,\QQ}$ is a graded module over the graded polynomial ring $\eqnu{*}{\pt}\cong\QQ[\ep]$, where $\ep\in\eqnu{2}{\pt}\cong\xch(\Gnu)\otimes_{\ZZ}\QQ$ corresponds to $m$ times the canonical generator of $\xch(\Gnu)$.

We shall also consider localized and specialized equivariant cohomology. By localized equivariant cohomology we mean $\eqnu{*}{Y}[\ep^{-1}]:=\eqnu{*}{Y}\otimes_{\QQ[\ep]}\QQ[\ep,\ep^{-1}]$. When $Y$ is of finite type, this is a free $\QQ[\ep,\ep^{-1}]$-module of finite rank. The specialized equivariant cohomology is
\begin{equation*}
\spcoh{}{Y}:=\eqnu{*}{Y}/(\ep-1).
\end{equation*}
This is a vector space over $\QQ$ whose dimension is the same as the $\QQ[\ep,\ep^{-1}]$-rank of $\eqnu{*}{Y}[\ep^{-1}]$. There is a {\em cohomological filtration} on $\spcoh{}{Y}$ given by $\spcoh{\leq i}{Y}=\Image(\eqnu{\leq i}{Y}\to\spcoh{}{Y})$. %There is a canonical map of graded algebras $\oplus_{i\geq0}\Gr^{i}\spcoh{}{Y}\to\oplus_{i\in\ZZ}\cohog{i}{Y}$, which is an isomorphism if $Y$ is $\Gnu$-equivariantly formal.

Similarly, we define the localized and specialized equivariant cohomology with compact support $\upH^{*}_{c,\Gnu}(Y)[\ep^{-1}]$ and $\upH_{c,\ep=1}(Y)$.

\subsubsection{Symmetry on the cohomology of homogeneous affine Springer fibers} Let $a\in\frc(F)^{\rs}_{\nu}$. There is a diagonalizable group $S_{a}$ (a subgroup of the centralizer of $\gamma=\kappa(a)$ in $G$) acting on $\Sp_{a}$. 

When $\nu$ is elliptic, there is also an action of a braid group $B_{a}=\pi_{1}(\frc(F)^{\rs}_{\nu}, a)$ on $\spcoh{}{\Sp_{a}}$ coming from the monodromy of the family of affine Springer fibers over $\frc(F)^{\rs}_{\nu}$. Together there is an action of $S_{a}\rtimes B_{a}$ on $\spcoh{}{\Sp_{a}}$.

\subsubsection{Cherednik algebras}\label{sss:intro Ch}
We shall consider two versions of Cherednik algebras,  the graded (aka trigonometric) version $\Hgr_{\nu}$ and the rational version $\Hrat_{\nu}$. Both of them are free modules over $\QQ[\ep]$. For simplicity we only consider their specializations $\Hgr_{\nu,\ep=1}$ and $\Hrat_{\nu,\ep=1}$ (specializing $\ep$ to 1) in this introduction. The algebra $\Hgr_{\nu,\ep=1}$, as a vector space, can be written as a tensor product of subalgebras
\begin{equation*}
\Hgr_{\nu}=\Sym(\fra^{*})\otimes\QQ[\Wa].
\end{equation*}
where $\Wa=\xcoch(\TT)\rtimes \WW$ is the affine Weyl group (note that we assumed $\GG$ was simply-connected). The most essential commutation relation in $\Hgr_{\nu,\ep=1}$ is given by
\begin{equation*}
s_i\cdot \xi-\leftexp{s_i}{\xi}\cdot s_i=-\nu\jiao{\xi,\alpha_i^\vee}.
\end{equation*}
for each affine simple reflection $s_i\in\Wa$ (corresponding to the simple affine coroot $\alpha_{i}$) and $\xi\in\fra^{*}$. Here $\fra$ should be thought of as part of the ``Kac-Moody torus'' $\fra_{\KM}$, and the action of $\Wa$ on $\fra^{*}$ should be understood in this way.

The graded algebra $\Hrat_{\nu}$ is also a tensor product of subalgebras
\begin{equation*}
\Hrat_{\nu,\ep=1}=\Sym(\fra^{*})\otimes\Sym(\fra)\otimes\QQ[\WW]
\end{equation*}
with the most essential commutation relation given by
\begin{equation*}
[\eta, \xi]=\jiao{\xi,\eta}-\dfrac{\nu}{2}\left(\sum_{\alpha\in\Phi}\jiao{\xi,\alpha^\vee}\jiao{\alpha,\eta}r_{\alpha}\right), \forall \xi\in\fra^{*}, \eta\in\fra.
\end{equation*}

For the graded \Cha we prove
\begin{theorem} \label{thm:Hgr} Let $\nu>0$ be a slope and $a\in\frc(F)^{\rs}_{\nu}$. 
\begin{enumerate}
\item There is an action of $\Hgr_{\nu,\ep=1}$ on the compactly supported equivariant cohomology $\upH_{c,\ep=1}(\Sp_{a})$. 
\item If $\nu$ is elliptic, then $S_{a}\rtimes B_{a}$ acts on $\spcoh{}{\Sp_{a}}$ and the action commutes with the action of $\Hgr_{\nu,\ep=1}$. Moreover,  the image of the restriction map $\spcoh{}{\Fl}\to \spcoh{}{\Sp_{a}}$ is the invariant part $\spcoh{}{\Sp_{a}}^{S_{a}\rtimes B_{a}}$, realizing $\spcoh{}{\Sp_{a}}^{S_{a}\rtimes B_{a}}$ as a quotient of the polynomial representation of $\Hgr_{\nu,\ep=1}$.
\item If $\nu$ is elliptic, then $\spcoh{}{\Sp_{a}}^{S_{a}\rtimes B_{a}}$ is an irreducible $\Hgr_{\nu,\ep=1}$-module. 
\end{enumerate}
\end{theorem}
In the main body of the paper we prove (1) and (2) also  for quasi-split groups $G$. Part (1) above appears as Corollary \ref{c:local Hgr} to Theorem \ref{th:localHgr}; part (2) is proved in Theorem \ref{thm:surj}(2) and Lemma \ref{l:pol rep}; part (3) appears as Corollary \ref{c:Hgr irr}, which is a consequence of the next theorem.

For the rational \Cha we prove
\begin{theorem} \label{thm:L(triv)} Let $\nu>0$ be an elliptic slope and $a\in\frc(F)^{\rs}_{\nu}$. Then on $\eqnu{*}{\Sp_{a}}^{S_{a}}$ there is a geometrically defined filtration $P_{\leq i}\eqnu{*}{\Sp_{a}}^{S_{a}}$ (called the {\em perverse filtration}, coming from the global geometry of the Hitchin fibration), stable under the action of $B_{a}$,  such that
\begin{enumerate}
\item There is a graded action of $\Hrat_{\nu,\ep=1}$ on $\Gr^{P}_{*}\spcoh{}{\Sp_{a}}^{S_{a}}$  commuting with the action of $B_{a}$.
\item The $\Hrat_{\nu,\ep=1}$-module $\Gr^{P}_{*}\spcoh{}{\Sp_{a}}^{S_{a}\rtimes B_{a}}$ is isomorphic to the finite-dimensional irreducible spherical module $\frL_{\nu}(\triv)$. Moreover, $\Gr^{P}_{*}\spcoh{}{\Sp_{a}}^{S_{a}\rtimes B_{a}}$  carries a Frobenius algebra structure induced from the cup product.
\end{enumerate}
\end{theorem}
In the main body of the paper, we also prove a weaker version of the above result when $G$ is quasi-split but non-split, and conjecture that the above result should hold in this generality, see Proposition \ref{p:Hrat GrC}, Theorem \ref{L(triv)} and Conjecture \ref{conj qsplit}. The conjecture in the quasi-split case is supported by examples in \S\ref{s:examples}. The perverse filtration is constructed in \S\ref{ss:perv fil}; part (1) above is proved in \S\ref{ss:Hrat local}; part (2) is proved in \S\ref{Frobenius}.

We also give a formula for the dimension of the irreducible spherical module $\frL_{\nu}(\triv)$. Let $W_\nu\subset \Wa$ be the finite subgroup of the affine Weyl group that stabilizes the point $\nu\rho^\vee$ in the standard apartment $\fra_{\RR}$ of $G$. Taking the linear part of $W_{\nu}$ identifies it with a  subgroup of the finite Weyl group $\WW$ of $\GG$ which acts
on $\fra^{*}$. Let $\mathcal{H}_{W_\nu}=\Sym(\fra^*)/\Sym(\fra^*)^{W_{\nu}}_+$ where $\Sym(\fra^*)^{W_{\nu}}_+$ is the ideal generated by the homogeneous $W_\nu$-invariant elements of positive degree.
To an element $\tilw\in\Wa$ we attach the element $\lambda^{\tilw}_{\nu}\in \mathcal{H}_{W_\nu}$ defined as
\begin{equation}
\lambda^{\tilw}_{\nu}:=\prod_{\substack{\alpha(\nu\rho^\vee)=\nu\\ \tilw^{-1}(\alpha)<0}}\bar{\alpha}
\end{equation}
where $\bar{\alpha}\in\fra^{*}$ is finite part of the affine root $\alpha$. Let $\Ann(\lambda^{\tilw}_{\nu})\subset\calH_{W_{\nu}}$ be the ideal annihilating $\l^{\tilw}_{\nu}$. 

\begin{theorem} Let $\nu>0$ be an elliptic slope. Then the dimension of the spherical irreducible module of $\Hrat_{\nu,\ep=1}$ is given by
\begin{equation}\label{intro dim}
\dim \frL_\nu(\triv)=\sum_{\tilw\in W_{\nu}\backslash\Wa} \dim \mathcal{H}_{W_\nu}/\Ann(\lambda^{\tilw}_{\nu}).
\end{equation}
\end{theorem} 
This theorem follows from the combination of Theorem \ref{thm:L(triv)}(2) and Theorem \ref{thm:dim}. In practice, the calculation of $\dim\frL_{\nu}(\triv)$ can be made more effective. If $\nu=d/m$ in lowest terms, we have $\frL_{d/m}(\triv)=d^{r}\frL_{1/m}(\triv)$  where $r$ is the rank of $\GG$. To calculate $\frL_{1/m}(\triv)$, the right side of \eqref{intro dim} only involves very few $\tilw$ lying in certain bounded {\em clans} (see \S\ref{sss:clans}) which are easy to determine. In \S\ref{s:examples}, we compute $\dim\frL_{\nu}(\triv)$ by hand for low rank groups, and use computer to generate tables of dimensions for exceptional groups (see \S\ref{ss:tables}).

\subsection{Organization of the paper}
The paper consists of three parts. 

The Algebra part provides basic setup of the groups and algebras involved in the main results. The materials in \S\ref{s:gp} and \S\ref{s:H} are standard in the situation of split $G$, but we set things up to treat the quasi-split cases uniformly. The key notion of homogeneous elements in a loop algebra is introduced in \S\ref{ss:homo local}, and we relate it to more familiar notions in Lie theory such as graded Lie algebras, regular elements in Weyl groups and Moy-Prasad filtration. 

The Geometry part studies the geometric and homological properties of homogeneous affine Springer fibers (\S\ref{s:Spr}) and homogeneous Hitchin fibers (\S\ref{s:Hit}). The main new result in \S\ref{s:Spr} is a formula for the dimension of the monodromy-invariant part of the cohomology of homogeneous affine Springer fibers (Theorem \ref{thm:dim}). Although our main results can be stated without referring to Hitchin fibers, the proof of  Theorem \ref{thm:L(triv)} uses the relation between affine Springer fibers and Hitchin fibers, and uses the global geometry of Hitchin fibration in an essential way. In \S\ref{s:Hit} we prove basic geometric properties of Hitchin moduli spaces over a {\em weighted projective line}, which is needed in order to establish a clean comparison result with homogeneous affine Springer fibers (Proposition \ref{p:prod}).

The Representations part then connects the Algebra part and the Geometry part together. In \S\ref{s:Hgr}, we construct representations of the graded \Cha $\Hgr_{\nu}$ on the equivariant cohomology of homogeneous affine Springer fibers and Hitchin fibers. The construction is a straightforward modification of what has been done in \cite{GS}. In \S\ref{s:Hrat}, we construct representations of the rational \Cha $\Hrat_{\nu}$ on the associated graded pieces of the equivariant cohomology of homogeneous affine Springer fibers and Hitchin fibers. The key step in the construction is to define a filtration on these cohomology groups such that when passing to the associated graded, the action of $\Hgr_{\nu}$ induces an action of $\Hrat_{\nu}$. We propose two filtrations, an algebraic one (called the {\em Chern filtration}, see \S\ref{ss:pol Hrat}) which works for quasi-split $G$ but is only defined on the quotient of the polynomial representation of $\Hrat$, and a geometric one (called the {\em perverse filtration}, see \S\ref{ss:perv fil}) using deep geometric input from Hitchin fibration (such as the support theorem proved by Ng\^{o} \cite{NgoFL}), which currently only works for split groups $G$. These filtrations coincide on spaces where they are both defined, and the final construction of the $\Hrat_{\nu}$-module structure and the proof of its properties (such as irreducibility) uses both of them.  We also prove a duality theorem (Corollary \ref{c:duality}) for the $\Hrat_{\nu}$-modules constructed from homogeneous affine Springer fibers for Langlands dual groups $\GG$ and $\dGG$.

In the final section \S\ref{s:examples}, we compute all examples where the $F$-rank of $G$ is at most two. The relevant Hessenberg varieties appear to have close relationship with classical projective geometry such as pencils of quadrics. The fact that the dimensions of their cohomology add up to the correct number (a sum of dimensions of irreducible $\Hrat_{\nu}$-modules people computed earlier) in each single example is a miracle to us.  Using computer we also give tables of dimensions of $\frL_{\nu}(\triv)$ for exceptional groups, and make conjectures about the dimension for certain classical groups.

\part{Algebra}

\section{Group-theoretic preliminaries}\label{s:gp}

In this section, we collect the basic definitions and properties of quasi-split reductive groups over $F=\CC((t))$. All results here are either well-known or variants of well-known ones. The reader is invited to come back to this section for clarification of notation.

\subsection{The group $\GG$} Let $\GG$ be an almost simple reductive group over $\CC$. We fix a pinning $\dagger=(\TT,\BB,\cdots)$ of $\GG$, where $\TT$ is a maximal torus of $\GG$ and $\BB$ is a Borel subgroup containing $\TT$. This determines a based root system $\PPhi\subset\xch(\TT)$ and a based coroot system $\PPhi^{\vee}\subset\xcoch(\TT)$. Let $\Aut^\dagger(\GG)$ be the group of pinned automorphisms of $\GG$, which is naturally isomorphic to $\Out(\GG)$. The Lie algebra of $\GG$ is denoted by $\gg$. For a $\CC$-algebra $R$, we sometimes denote $\gg\otimes_{k}R$ by $\gg(R)$.

Fix $\theta:\mu_e\incl\Out(\GG)\cong\Aut^\dagger(\GG)$ an injective homomorphism. We have $e=1$ or $e=2$ (type $A, D$ or $E_{6}$) or $e=3$ (type $D_{4}$).

Let $\HH$ be the neutral component of $\GG^{\mu_{e},\circ}$. Let $\AA$ be the neutral component of $\TT^{\mu_{e},\circ}$. Then $\AA$ is a maximal torus of $\HH$ and $\xcoch(\AA)=\xcoch(\TT)^{\mu_{e}}$.

Notation: we use boldfaced (or blackboard) letters to denote the absolute data attached to the group $\GG$ and the usual letters for the relative data attached to either $\HH$ or $(\GG,\theta)$. For example,
\begin{itemize}
\item $\rr=\rank\GG=\dim\TT$ while $r=\rank\HH=\dim\AA$;
\item The absolute root system $\PPhi=\Phi(\GG,\TT)$; the relative root system $\Phi=\Phi(\GG,\AA)$ (which is not necessarily reduced). 
\item The absolute Weyl group  $\WW=N_{\GG}(\TT)/\TT$ and the relative one $W=N_{\HH}(\AA)/\AA=\WW^{\mu_{e}}$;
\end{itemize}
We also introduce the group $\WW'=\WW\rtimes\mu_e$ where $\mu_e$ acts on $\WW$ via $\theta$.

\subsection{The group $G$ over $F$}
\subsubsection{The field $F$} Let $F=\CC((t))$, the field of formal Laurent series with coefficients in $\CC$. Let $\calO_{F}=\CC[[t]]$ be the valuation ring in $F$. For each integer $n\geq1$, let $F_n$ be the extension $\CC((t^{1/n}))$ of $F$. Then $F_\infty=\cup_{n\geq1}F_n$ is an algebraic closure of $F$, with $\Gal(F_\infty/F)$ is identified with the projective limit $\hZZ(1)=\varprojlim_{n}\mu_n$. 

Let $\nu\in\QQ$. The Galois action of $\hZZ(1)$ on $t^{\nu}\in F_{\infty}$ gives a character which we denote by $\zeta\mapsto\zeta^{\nu}$. This character only depends on the class of $\nu$ in $\QQ/\ZZ$. Concretely, if $\nu=a/b$ in lowest terms with $b>0$, then the corresponding character is $\hZZ(1)\xrightarrow{\textup{natural}}\mu_{b}\xrightarrow{[a]}\mu_{b}$.

\subsubsection{The quasi-split group $G$}\label{sss:gp} The homomorphism $\theta$ gives a descent datum of the constant group $\GG\otimes_{\CC}F_e$ from $F_e$ to $F$. We denote the resulting group scheme over $F$ by $G$. 
Explicitly, for any $F$-algebra $R$, $G(R)=\{g\in\GG(R)\otimes_{F}F_{e}:\zeta(g)=\theta(\zeta)(g)\textup{ for all }\zeta\in\mu_{e}\}$, where $\zeta(g)$ means the action on $F_{e}$ via the Galois action. In other words, $G=(\Res^{F_{e}}_{F}(\GG\otimes_{\CC}F_{e}))^{\mu_{e}}$, with $\mu_{e}$ acting simultaneously on $F_{e}$ through Galois action and on $\GG$ through $\theta$ {\em composed with the inversion on $\Out(\GG)$}. The torus $A:=\AA\otimes_{\CC}F$ of $G$ is a maximal split torus of $G$. 
The torus $\TT$ gives rise to a maximal torus $T=(\Res^{F_{e}}_{F}(\TT\otimes_{\CC}F_{e}))^{\mu_{e}}$ of $G$. 

\subsubsection{Integral models} We need an integral model of $G$ over $\calO_{F}$. First consider $\bG_{1}=(\Res^{\calO_{F_{e}}}_{\calO_{F}}(\GG\otimes_{\CC}\calO_{F_{e}}))^{\mu_{e}}$ (the action of $\mu_{e}$ is the same as in the case of $G$ above). This is a group scheme over $\Spec\calO_{F}$ whose generic fiber is isomorphic to $G$. The special fiber of $\bG_{1}$ has reductive quotient equal to $\GG^{\mu_{e}}$, which may or may not be connected. Let $\bG\subset \bG_{1}$ be the fiberwise neutral component, which is a smooth group scheme over $\calO_{F}$ with generic fiber $G$. In fact, $\bG$ is a {\em special parahoric subgroup} of $G$. We can similarly define an $\calO_{F}$-model $\bT$ for $T$. Namely, we first define $\bT_{1}=(\Res^{\calO_{F_{e}}}_{\calO_{F}}(\TT\otimes_{\CC}\calO_{F_{e}}))^{\mu_{e}}$, then take $\bT$ to be the fiberwise neutral component of $\bT_{1}$.

The Lie algebra of $\bG$ is denoted $\frg$. We have $\frg(\calO_{F})=\{X(t^{1/e})\in\gg\otimes F_{e}: \theta(\zeta)X(t^{1/e})=X(\zeta t^{1/e})\textup{ for all }\zeta\in\mu_{e}\}$. This is a Lie algebra over $\calO_{F}$. For an $\calO_{F}$-algebra $R$, we use $\frg(R)$ to mean $\frg\otimes_{\calO_{F}}R$. %  We have an explicit description of $\frg$. Consider the action of $\mu_e$ on $\gg(\calO_{F_{e}})$ induced from the multiplication on the variable $t^{1/e}$  ({\bf A: you probably mean multiplication of $t^{1/e}$ by some root of 1?}) and by pinned automorphisms on $\gg$ via $\theta$. Then $\frg=\gg(\calO_{F_{e}})^{\mu_e}$.

\subsection{Invariant quotient}\label{ss:inv} Let $\cc=\tt\sslash\WW=\gg\sslash\GG$ be the invariant quotient of $\gg$. There is a $\GG$-invariant morphism $\bchi:\gg\to\cc$. As an affine scheme over $\CC$, $\cc$ is isomorphic to an affine space with coordinate functions given by the fundamental invariants polynomials $f_1,\cdots, f_{\rr}$ on $\gg$. 

Define $\frc=(\Res^{\calO_{F_{e}}}_{\calO_{F}}(\cc\otimes_{\CC}\calO_{F_{e}}))^{\mu_{e}}$, where $\mu_{e}$ acts on both $\cc$ (via $\theta$) and on $\calO_{F_{e}}$ (via the inverse of the usual Galois action on $F_{e}$). We have a morphism $\chi: \frg\to\frc$ over $\calO_{F}$, which is induced from $\bchi\otimes\calO_{F_{e}}:\gg\otimes_{\CC}\calO_{F_{e}}\to\cc\otimes_{\CC}\calO_{F_{e}}$. 

It is convenient to choose the fundamental invariants $f_{1},\cdots, f_{\rr}$ so that each of them is an eigenvector under the pinned action of $\mu_{e}$. Let $\vep_{i}\in\{0,1,\cdots,e-1\}$ be the unique number such that $\mu_{e}$ acts on $f_{i}$ via the $\vep_{i}\nth$ power of the tautological character of $\mu_{e}$. Using the fundamental invariants, we may write
\begin{equation}\label{frc}
\frc(\calO_{F})=\bigoplus_{i=1}^{\rr}t^{\vep_{i}/e}\calO_{F}.
\end{equation}

We use $\frc_{F}$ to denote the generic fiber of $\frc$, which is the same as $(\Res^{F_{e}}_{F}(\cc\otimes_{\CC}F_{e}))^{\mu_{e}}$. Let $\cc^\rs\subset\cc$ be the complement of the discriminant divisor. Since $\cc^{\rs}$ is invariant under $\Out(\GG)$, we can define $\frc^{\rs}_{F}:=(\Res^{F_{e}}_{F}(\cc^{\rs}\otimes_{\CC}F_{e}))^{\mu_{e}}$, which is an open subscheme of $\frc_{F}$. We denote the $F$-points of $\frc^{\rs}_{F}$ by $\frc(F)^{\rs}$.
% ({\bf A: I think there is some ambiguity here, you probably mean below $\frc^{\rs}(F)=\frc^{\rs}(\calO_{F})$?})

\subsubsection{Ellipticity}\label{sss:ell}
The composition $\tt\otimes F_{e}\to \cc\otimes F_{e}\to \frc_{F}$ is a branched $\WW'$-cover that is  \'etale over $\frc^{\rs}_{F}$. Let $a\in\frc(F)^{\rs}$, viewed as a morphism $\Spec F\to\frc^{\rs}_{F}$, then it induces a homomorphism $\Pi_{a}: \Gal(F_{\infty}/F)\cong\hZZ(1)\to \WW'$ (up to conjugacy). A point $a\in\frc(F)^{\rs}$ is called {\em elliptic} if $\tt^{\Pi_a(\hZZ(1))}=0$. Equivalently, if we fix a topological generator $\zeta\in\hZZ(1)$, then $a$ is elliptic if $\tt^{\Pi_a(\zeta)}=0$. 

\subsubsection{Kostant section}
Let $\ss\subset \gg$ be the Kostant section defined using the pinning of $\GG$. Recall if $(\ee,2\rho^{\vee}, \ff)$ is the principal $\mathfrak{sl}_{2}$-triple of $\gg$, then $\ss=\ee+\gg^{\ff}$. Since $\ee,\ff$ are invariant under pinned automorphisms of $\GG$, the isomorphism $\bchi|_{\ss}:\ss\to\cc$ is  $\mu_{e}$-equivariant. Let $\frs=(\Res^{\calO_{F_{e}}  }_{\calO_{F}}(\ss\otimes\calO_{F_{e}}))^{\mu_{e}}\subset\frg$. Then the characteristic map $\chi$ restricts to an isomorphism $\chi|_{\frs}:\frs\isom\frc$. The inverse of $\chi|_{\frs}$ is denoted by
\begin{equation*}
\kappa:\frc\isom\frs\subset\frg,
\end{equation*}
and is called the Kostant section for $\frg$.

\subsection{Regular centralizers}\label{ss:J}
We define regular elements in $\frg$ as the open subset $\frg^{\reg}=\frg\cap\gg^{\reg}(\calO_{F_{e}})$. This defines an open subscheme $\frg^{\reg}\subset\frg$ over $\calO_{F}$.

Let $I$ over $\frg$ be the universal centralizer group scheme. Consider the group scheme $I|_{\frs}$ over $\frs$ and view it as a group scheme over $\frc$ via $\chi|_{\frs}$. We denote this group scheme over $\frc$ by $J$, and call it the {\em regular centralizer group scheme of $\frg$}. 

\begin{lemma}\label{l:J} There is a canonical morphism $\iota:\chi^{*}J\to I$ which is the identity when restricted to $\frg^{\reg}$.
\end{lemma}
\begin{proof} Let $J'$ and $I'$ be the regular centralizer and the universal centralizer group schemes over $\gg\otimes_{\CC}\calO_{F_{e}}$. Let $\chi':\gg\otimes\calO_{F_{e}}\to\frc\otimes\calO_{F_{e}}$ be the invariant quotient map. Then it is shown in \cite{NgoHit} that there is a canonical morphism $\iota':\chi'^{*}J'\to I'$ which restricts to the identity on $\gg^{\reg}\otimes_{\CC}\calO_{F_{e}}$.

There are obvious maps $I\incl I_{1}:=\Res^{\gg\otimes\calO_{F_{e}}}_{\frg}(I')^{\mu_{e}}$ and $J\incl J_{1}:=\Res^{\cc\otimes\calO_{F_{e}}}_{\frc}(J')^{\mu_{e}}$. The morphism $\iota'$  induces $\iota_{1}:\chi^{*}J_{1}\to I_{1}$, which is the identity on $\frg^{\reg}$. Consider the composition
\begin{equation*}
\beta_{J}:\chi^{*}J_{1}\to I_{1}\xrightarrow{\beta_{I}}\bG_{1}.
\end{equation*}
By definition, $I=\beta_{I}^{-1}(\bG)\subset I_{1}$; $J=\beta_{J}^{-1}(\bG)\subset J_{1}$. Therefore $\iota_1$ restricts to a morphism $\iota:\chi^{*}J\to I$ which is the identity on $\frg^{\reg}$.
\end{proof}

There is an alternative description of $J$ in the style of \cite{DG}. Consider the group scheme $J^{\sharp}:=(\Res^{\tt\otimes\calO_{F_{e}}}_{\frc}((\TT\times\tt)\otimes\calO_{F_{e}}))^{\WW'}$. Then $J$ is an open subgroup of $J^{\sharp}$ given by removing components over the discriminant locus of $\frc$ as well as over the special fiber of $\frc$. We omit the details here.

We need a strengthening of Lemma \ref{l:J}. For each Iwahori $\bI$ (viewed as a group scheme over $\calO_{F}$) of $G(F)$, we can define the corresponding universal centralizer $I_{\bI}$ over $\Lie\bI$. It consists of pairs $(g,\gamma)\in\bI\times\Lie\bI$
such that $\Ad(g)\gamma=\gamma$. There is a natural inclusion $I_{\bI}\incl I|_{\Lie\bI}$. Let $\chi_{\bI}:\Lie\bI\incl\frg\to\frc$ be the restriction of $\chi$ to $\Lie\bI$. 

\begin{lemma}\label{l:J Iw} There is a canonical morphism $\iota_{\bI}:\chi^{*}_{\bI}J\to I_{\bI}$ which is such that the composition $\chi^{*}_{\bI}J\to I_{\bI}\incl I|_{\Lie\bI}$ is the restriction of $\iota$ to $\Lie\bI$.
\end{lemma}
\begin{proof}
The argument for Lemma \ref{l:J} applies here with the following modifications. We only need to produce a canonical map $\iota'_{\bI'}:\chi'^{*}_{\bI'}J'\to I'_{\bI'}$. This follows from \cite[Lemma 2.3.1]{GS}.
\end{proof}

\subsection{The affine root system}

\subsubsection{The loop group and parahoric subgroups}\label{sss:loop group} The group $G(F)$ can  be viewed as a the $\CC$-points of a group ind-scheme $LG$ over $\CC$. For any $\CC$-algebra $R$, $LG(R)$ is defined to be $G(R((t)))$. We shall abuse the notation and write $G(F)$ for $LG$ if it is clear from the context that $G(F)$ denotes a group ind-scheme over $\CC$.  Parahoric subgroups $\bP$ of $G(F)$ are the $\calO_{F}$-points of certain smooth models of $G$ over $\calO_{F}$ (the Bruhat-Tits group schemes). Later we will also considered $\bP$ as group schemes (of infinite type) over $\CC$, whose $R$-points are $\bP(R[[t]])$.

\subsubsection{The Kac-Moody group} Attached to $G$ there is an affine Kac-Moody group which is a group ind-scheme over  $\CC$. First there is a central extension $1\to \Gcen\to G^{\cen}\to G(F)\to 1$, where $\Gcen$ is a one-dimensional torus. The central extension $G^{\cen}$ is constructed as in \cite[6.2.2]{GS}: a $\CC$-point of $\Gcen$ is a pair $(g,\tau)$ where $g\in G(F)$ and $\tau$ is a trivialization of the determinant line $\det(\Ad(g)\frg:\frg)$. The affine Kac-Moody group is the semidirect product $G_{\KM}=G^{\cen}\rtimes \Grote$,  where $\Grote$-acts as usual on $G(F)$ and as the identity on $\Gcen$. 

We have the torus $\AA_{\KM}:=\Gcen\times\AA\times\Grote\subset G_{\KM}$ (recall $\AA$ is the neutral component of $\TT^{\mu_{e}}$). Let $\fra:=\xcoch(\AA)\otimes_{\ZZ}\QQ$ and $\fra_{\KM}:=\xcoch(\AA_{\KM})\otimes_{\ZZ}\QQ$, and let $\fra^{*}$ and $\fra^{*}_{\KM}$ be the dual vector spaces.

Let $\delta/e$ and $\L_{\can}$ be the generators of $\xch(\Grote)$ and $\xch(\Gcen)$ respectively. Dually, let $e\partial$ and $K_{\can}$ be the generators of $\xcoch(\Grote)$ and $\xcoch(\Gcen)$ respectively. The imaginary roots of $G_{\KM}$ are $\frac{1}{e}\ZZ\delta-\{0\}$. \footnote{This is different from Kac's convention in \cite{Kac}, where he defines $\delta$ to be the generator of positive imaginary roots.}

\subsubsection{Affine simple roots}\label{sss:aff simple roots}  $\Delta=\{\alpha_{1},\cdots,\alpha_{r}\}$ are the simple roots of the reductive group $\HH$ with respect to $\AA$, which are in bijection with the $\mu_{e}$-orbits on the simple roots of $\GG$ with respect to $\TT$. Let $\beta\in\Phi$ be the highest weight of the action of $\AA$ on $\gg_{1}$ (the eigenspace of $\gg$ on which $\mu_{e}$ acts via the tautological character). When $e=1$, $\beta$ is the highest root of $\GG$. When $e>1$ and $G$ is not of type $\SUodd$,  $\beta$ is the dominant short root in the relative root system $\Phi$; when $G$ is of type $\SUodd$, $\beta$ is the longest dominant root in $\Phi$, which is twice a shortest root. We introduce the Dynkin labeling $\{a_{0}=1, a_{1}\cdots,a_{r}\}$ such that
\begin{equation*}
\beta=\sum_{i=1}^{r}a_{i}\alpha_{i}.
\end{equation*}
The {\em $\theta$-twisted Coxeter number} attached to $(\GG,\theta)$ is
\begin{equation*}
h_{\theta}=e\sum_{i=0}^{r}a_{i}.
\end{equation*}
Let
\begin{equation*}
\alpha_{0}=\delta/e-\beta.
\end{equation*}
Then $\Delta_{\aff}=\{\alpha_{0},\alpha_{1},\cdots,\alpha_{r}\}$ is the set of affine simple roots of the Kac-Moody group $G_{\KM}$. The set of affine roots of $G_{\KM}$ will be denoted by $\Phi_{\aff}$, which is a subset of $\xch(\Grote)\oplus\xch(\AA)$.
\footnote{When $G$ is of type $\SUodd$, our convention here is different from that in \cite{Kac}. Kac takes $\alpha_{0}$ to be the shortest node in the Dynkin diagram of type $\SUodd$ while we take $\alpha_{0}$ to be the longest one. In Kac's convention $a_{0}=2$ and $a_{0}^{\vee}=1$, while in our convention $a_{0}=1$ and $a_{0}^{\vee}=2$.}

\subsubsection{Affine simple coroots} Let $\beta^{\vee}\in\Phi^{\vee}$ be the coroot attached to $\beta$. Let $a_{0}^{\vee}=1$ if $\GG$ is not of type $\SUodd$ and $a_{0}^{\vee}=2$ otherwise. We introduce the dual labeling $\{a^{\vee}_{0}, a^{\vee}_{1}\cdots,a^{\vee}_{r}\}$ such that
\begin{equation}\label{dual label}
a_{0}^{\vee}\beta^{\vee}=\sum_{i=1}^{r}a^{\vee}_{i}\alpha_{i}^{\vee}.
\end{equation}
The {\em $\theta$-twisted dual Coxeter number} is defined as
\begin{equation}\label{dual Cox}
h^{\vee}_{\theta}=\sum_{i=0}^{r}a^{\vee}_{i}.
\end{equation}
Note that $h^{\vee}_{\theta}$ is always equal to the dual Coxeter number of $\GG$. Let 
\begin{equation}\label{define alpha0}
\alpha^{\vee}_{0}=\frac{2h^{\vee}_{\theta}}{a_{0}^{\vee}}K_{\can}-\beta^{\vee}.
\end{equation}
Then $\Delta^{\vee}_{\aff}=\{\alpha^{\vee}_{0}, \alpha^{\vee}_{1},\cdots,\}$ is the set of affine simple coroots of $G_{\KM}$. The set of affine coroots of $G_{\KM}$ will be denoted by $\Phi^{\vee}_{\aff}$, which is a subset of $\xcoch(\Gcen)\oplus\xcoch(\AA)\subset\fra_{\KM}$.

\subsubsection{The affine Weyl group} First we have the extended affine Weyl group $\tilW$ defined as follows. Recall the maximal torus $T\subset G$ is given an $\calO_{F}$-structure $\bT$. Then $\tilW:=N_{G}(T)/\bT$. Since $W=N_{G}(T)/T$, we have a canonical short exact sequence
\begin{equation}\label{tilW}
1\to \xcoch(\TT)_{\mu_{e}}\to \tilW\to W\to 1
\end{equation}
in which the first term is canonically isomorphic to $T(F)/\bT(\calO_{F})$. 

Let $\frA$ be the apartment in the building of $G$ attached to $T$. This is a torsor under $\fra\otimes_{\QQ}{\RR}$. The group $\tilW/\xcoch(\TT)_{\mu_{e},\tors}$ acts on $\frA$, and $\xcoch(\TT)_{\mu_{e}}/\xcoch(\TT)_{\mu_{e},\tors}$ acts as translations. This defines an embedding
\begin{equation*}
\iota: \xcoch(\TT)_{\mu_{e}}/\xcoch(\TT)_{\mu_{e},\tors}\incl\fra.
\end{equation*}

Each real affine root of $G_{\KM}$ is viewed as an affine function on $\frA$. The zero sets of the real affine roots gives a stratification of $\frA$ into facets. Each facet $\frF$ gives rise to a parahoric $\bP\subset 
G(F)$ containing $\bT$ and vice versa. The special parahoric $\bG$ determines a facet $\frF_{\bG}$ which is a point. We often use this point to identify $\frA$ with $\fra\otimes_{\QQ}{\RR}$. Since $W$ can also be identified with $N_{\bG}(\bT)/\bT$, we may identify $W$ with a reflection subgroup of $\tilW$ fixing the vertex $\frF_{\bG}$. This way we can write $\tilW=\xcoch(T)_{\mu_{e}}\rtimes W$.

The complement of the affine root hyperplanes is a disjoint union of alcoves, which are in bijection with Iwahori subgroups containing $\bT$. The choice of the standard Iwahori $\bI$ determines the standard alcove $\frF_{\bI}$ adjacent to $\frF_{\bG}$. The standard parahorics corresponds bijectively to the facets in the closure of $\frF_{\bI}$. The affine roots whose zero set is tangent to $\frF_{\bI}$ and whose gradient is pointing towards $\frF_{\bI}$ are called affine simple roots.

The affine Weyl group $\Wa$ attached to $G$ is the subgroup of $\tilW$ generated by reflections across the affine root hyperplanes. Equivalently, $\Wa$ is generated by the simple reflections corresponding to elements in $\Delta_{\aff}$. We have a similar short exact sequence
\begin{equation*}
1\to \L\to \Wa\to W\to 1.
\end{equation*}
for some sublattice $\L\subset\xcoch(\TT)_{\mu_{e}}$. Let $\Omega:=\xcoch(\TT)_{\mu_{e}}/\L$. When $e=1$, $\Omega$ is the fundamental group of $\GG$, and $\L$ is the coroot lattice of $\GG$. In general, we have an exact sequence
\begin{equation*}
1\to \Wa\to \tilW\to \Omega\to 1.
\end{equation*} 
Let $\Omega_{\bI}$ be the stabilizer of $\frF_{\bI}$ under $\tilW$, then $\Omega_{\bI}\to\Omega$ is an isomorphism via the projection $\tilW\surj\Omega$. Therefore we may write $\tilW=\Wa\rtimes\Omega_{\bI}$. We may identify $\Omega_{\bI}$ with $N_{G}(\bI)/\bI$.

\subsubsection{The invariant symmetric bilinear form} Let
\begin{eqnarray*}
\BB(\cdot,\cdot): \xcoch(\TT)\times\xcoch(\TT)\to\ZZ\\
(x,y)\mapsto\sum_{\alpha\in\PPhi}\jiao{\alpha,x}\jiao{\alpha,y}.
\end{eqnarray*}
be the Killing form on $\xcoch(\TT)$ (we are summing over all roots of $\GG$). Restricting to $\xcoch(\AA)$ we get a $W$-invariant symmetric bilinear form on $\xcoch(\AA)$ and hence on $\fra$. As an element of $\Sym^{2}(\fra)^{W}$, we have
\begin{equation*}
\BB=\sum_{\alpha\in\PPhi}\alpha\otimes\alpha.
\end{equation*}
We define a symmetric bilinear form $B_{\KM}$ on $\fra_{\KM}$ extending $\BB$:
\begin{equation}\label{Killing KM}
B_{\KM}=\BB+\delta\otimes\L_{\can}+\L_{\can}\otimes\delta.
\end{equation}
In other words, we have
\begin{eqnarray*}
B_{\KM}(\partial,\fra)=0; B_{\KM}(K_{\can},\fra)=0; B_{\KM}(\partial,\partial)=B_{\KM}(K_{\can},K_{\can})=0; \\
B_{\KM}(\partial,K_{\can})=B_{\KM}(K_{\can},\partial)=1; B_{\KM}|_{\fra}=\BB.
\end{eqnarray*}

\begin{lemma} The symmetric bilinear form $B_{\KM}$ on $\fra_{\KM}$ is $\tilW$-invariant.
\end{lemma}
\begin{proof} It is well-known that a $\tilW$-invariant extension of $\BB$ exists and takes the form $\BB+c(\delta\otimes\L_{\can}+\L_{\can}\otimes\delta)$ for some constant $c$. The $s_{0}$-invariance forces
\begin{equation*}
c=\frac{a_{0}^{\vee}\BB(\beta^{\vee},\beta^{\vee})}{4eh^{\vee}_{\theta}}.
\end{equation*}
We only need to show that $c=1$.
When $e=1$, $\beta$ the highest root of $\GG$ and we have the well-known formula $\BB(\beta^{\vee},\beta^{\vee})=4h^{\vee}$, hence $c=1$. When $e>1$ and $G$ is not of type $\SUodd$, $\beta^{\vee}$ is the sum of $e$ mutually orthogonal coroots $\gamma^{\vee}_{i}$ for $i=0,\cdots,e-1$ (which form an orbit under the action of $\mu_{e}$). Each $\BB(\gamma^{\vee}_{i},\gamma^{\vee}_{i})=4h^{\vee}_{\theta}$ (since $\GG$ is simply-laced). Hence $\BB(\beta^{\vee},\beta^{\vee})=4e h^{\vee}_{\theta}$ and $c=1$. Finally, when $G$ is of type $\SUodd$, $\beta^{\vee}$ is a coroot of $\GG$, hence $\BB(\beta^{\vee},\beta^{\vee})=4h^{\vee}_{\theta}$ and again $c=1$.
\end{proof}

\begin{lemma}\label{l:tr action} Suppose $\tilw\in\ker(\tilW\to W)$ with $\iota(\tilw)=\l\in\fra$. Then the action of $\tilw$ on $\fra^{*}_{\KM}$ is given by
\begin{eqnarray}
\label{tr delta}\leftexp{\tilw}{\delta}&=&\delta;\\
\label{tr xi}\leftexp{\tilw}{\xi}&=&\xi+\jiao{\xi,\l}\delta, \textup{ for }\xi\in\fra^{*}; \\
\label{tr L}\leftexp{\tilw}{\L_{\can}}&=&\L_{\can}-\l^{*}-\frac{1}{2}\BB(\l,\l)\delta.
\end{eqnarray}
where $\l^{*}\in\fra^{*}$ is defined by $\jiao{\l^{*},y}=\BB(\l,y)$. Dually, the action of $\tilw$ on $\fra_{\KM}$ is given by
\begin{eqnarray*}
\leftexp{\tilw}{K_{\can}}&=&K_{\can};\\
\leftexp{\tilw}{\eta}&=&\eta+\BB(\l,\eta)K_{\can}, \textup{ for }\eta\in\fra; \\
\leftexp{\tilw}{\partial}&=&\partial-\l-\frac{1}{2}\BB(\l,\l)K_{\can}.
\end{eqnarray*}
\end{lemma}
\begin{proof}
We only give the proof of the first three equalities and the last three are obtained by duality. Equation \eqref{tr delta} is clear since $\delta$ is invariant under all $s_{i}$ and $\Omega$. For $\xi\in\fra^{*}$, we have $\leftexp{\tilw}{\xi}\in\xi+\QQ\delta$. We can define a pairing
\begin{eqnarray*}
f:\xcoch(\TT)_{\mu_{e}}\times\fra^{*}&\to&\QQ\\
(\tilw,\xi)&\mapsto&(\leftexp{\tilw}{\xi}-\xi)/\delta.
\end{eqnarray*}
It is easy to check that $f$ is bilinear and $W$-invariant ($W$ acts diagonally on $\xcoch(\TT)_{\mu_{e}}\times\fra^{*}$). Therefore $f=c\jiao{\xi,\iota(\tilw)}$ for some constant $c\in\QQ$. Taking the special element $\tilw=r_{\beta}s_{0}$ ($\beta$ is used to define $\alpha_{0}$, see \S\ref{sss:aff simple roots}). One easily calculates that $\iota(\tilw)=\frac{1}{e}\beta^{\vee}$ and $\leftexp{\tilw}{\xi}=\xi+\frac{1}{e}\jiao{\xi,\beta^{\vee}}\delta$. Therefore for this $\tilw$ we have $f(\tilw,\xi)=\jiao{\xi,\iota(\tilw)}$ and the constant $c=1$. This proves \eqref{tr xi}.

Finally, for $\iota(\tilw)=\l$, $\leftexp{\tilw}{\L_{\can}}$ takes the form $\L_{\can}+\varphi(\l)+q(\l)\delta$ for a linear function $\varphi: \xcoch(\TT)_{\mu_{e}}\to\fra^{*}$ and a quadratic function $q:\xcoch(\TT)_{\mu_{e}}\to\QQ$. 
On the other hand, $B_{\KM}$ is invariant under $\tilW$. Using proven formulas \eqref{tr delta} and \eqref{tr xi} one can calculate $\leftexp{\tilw}{B_{\KM}}$. Comparing with $B_{\KM}$ we conclude that $\varphi(\l)=-\sum_{\alpha\in\PPhi}\jiao{\alpha,\l}\alpha=-\l^{*}$ and $q(\l)=-\frac{1}{2}\BB(\l,\l)$. This proves \eqref{tr L}. The dual statements are immediate corollaries of what we already proved.
\end{proof}

\section{Homogeneous elements in the loop Lie algebra}
In this section, we shall systematically study homogeneous elements (or rather conjugacy classes) in the loop Lie algebra $\frg(F)$. Homogeneous conjugacy classes are those which are stable under a one-dimensional torus that is a mixture of  ``loop rotation'' and dilation. The main result is a classification theorem for homogeneous conjugacy classes in terms of any one of the three well-studied objects in representation theory: regular elements in a Weyl group; periodic gradings on $\gg$ and Moy-Prasad filtrations on $\frg(F)$. Out treatment here is strongly influenced by the work of Gross, Levy, Reeder and Yu (see \cite{GLRY} and \cite{RY}).

\subsection{Definition and basic properties}\label{ss:homo local}

\subsubsection{Two tori acting on $\cc(F_{\infty})$}
The one dimensional torus $\Gm^{\rot,[n]}$ acts on $F_n$ by scaling $t^{1/n}$ (the notation $\Gm^{\rot,[n]}$ is to emphasize the dependence on $n$; when $n=1$ we write $\Gm^{\rot,[1]}$ as $\Grot$).  Let $\hGrot=\varprojlim\Gm^{\rot,[n]}$ where  the transition maps are given by $[\ell]:\Gm^{\rot,[\ell n]}\to\Gm^{\rot,[n]}$. This can be viewed as the  universal cover of $\Grot$. We have a natural isomorphism $\xch(\hGrot)=\QQ$, with $\xch(\Gm^{\rot,[n]})$ identified with $\frac{1}{n}\ZZ\subset\QQ$ . The actions of $\Gm^{\rot,[n]}$ on $F_n$ passes to the limit to give an action of $\hGrot$ on $F_\infty$, and on $\gg(F_\infty)$. We have an exact sequence
\begin{equation*}
1\to\hZZ(1)\to\hGrot\to\Grot\to1.
\end{equation*}

On the other hand, the one-dimensional torus $\Gdil$ acts on $\gg(F_\infty)$ by dilation: $\Gdil\ni\l:X\mapsto \l X$ for $X\in \gg(F_\infty)$. The dilation action induces the dilation action of $\Gdil$ on $\cc(F_{\infty})$ with weights $d_1,\cdots, d_{\rr}$. 

Thus we get an action of $\hGrot\times\Gdil$ on both $\gg(F_\infty)$ and $\cc(F_{\infty})$ such that $\chi:\gg(F_\infty)\to\cc(F_\infty)$ is $\hGrot\times\Gdil$-equivariant. At finite level, we have an action of $\Grote\times\Gdil$ on $\gg(F_e)$ and $\cc(F_{e})$ (but not an action of $\Grot\times\Gdil$ on $\frg(F)$ or $\frc(F)$ if $e>1$). 

For a rational number $\nu\in\QQ$, we define a subtorus $\hGm(\nu)\subset\hGrot\times\Gdil$ as follows. The character group of $\hGm(\nu)$ is identified with the quotient of $\xch(\hGrot\times\Gdil)=\QQ\oplus\ZZ$ defined using the exact sequence
\begin{equation*}
0\to\ZZ\xrightarrow{(\nu,1)}\QQ\oplus\ZZ\to\xch(\hGm(\nu))\to0.
\end{equation*}
In particular, the image of $\hGm(\nu)$ in $\Gm^{\rot,[n]}\times\Gdil$ is the one-dimensional subtorus whose cocharacter lattice is a lattice in $\ZZ^{2}$ of slope $-n\nu$.

\begin{defn}\label{def:homo} Let $\nu\in\QQ$. 
\begin{enumerate}
\item An element $a\in \cc^\rs(F_\infty)$ is called {\em homogeneous of slope $\nu$} if $a$ is fixed under $\hGm(\nu)$. Let $\frc(F)^{\rs}_{\nu}$ denote the set of all homogeneous elements of slope $\nu$. 
\item A regular semisimple element $\gamma\in \gg(F_\infty)$ is called {\em homogeneous of slope $\nu$} if $\chi(\gamma)\in\cc^{\rs}(F_\infty)$ is.
\end{enumerate}
\end{defn}

We denote the Galois action of $\hZZ(1)$ on $\cc(F_{\infty})$ (without the twisting by $\theta$) by $\zeta: a\mapsto\zeta\cdot_{\Gal}a$. The dilation action of $s\in \Gm$ on $\frc$ will be denoted by $a\mapsto s\cdot_{\dil}a$.

\begin{lemma}\label{l:hom in Fs} Let $\nu\in\QQ$. 
\begin{enumerate}
\item If $a\in \cc^\rs(F_\infty)$ is homogeneous of slope $\nu$, then for any $\zeta\in\hZZ(1)$ we have
\begin{equation}\label{Gal vs dil}
\zeta\cdot_{\Gal}a=\zeta^{\nu}\cdot_{\dil}a.
\end{equation}
\item An element $a\in\cc^\rs(F_\infty)$ is homogeneous of slope $\nu$ if and only if $f_{i}(a)=c_it^{\nu d_i}$ for some $c_i\in \CC$, $i=1,\cdots, \rr$.
\item\label{Xtnu} A regular semisimple element $\gamma\in\gg(F_\infty)$ is homogeneous of slope $\nu$ if and only if it is $\GG(F_\infty)$-conjugate to an element of the form $Xt^\nu$ where $X\in\tt^\rs$.
\end{enumerate} 
\end{lemma}
\begin{proof}
(1) and (2) are direct calculation. 

\eqref{Xtnu} The ``if'' direction is clear. Now suppose a regular semisimple element $\gamma\in\gg(F_\infty)$ is homogeneous of slope $\nu$. Then $\chi(\gamma)$ is as described in Part (1) and there exists $X\in\tt^\rs(\CC)$ such that $\chi(Xt^\nu)=a$. Both $Xt^\nu$ and $\gamma$ are regular semisimple elements in $\gg(F_\infty)$ with the same invariants, they are $\GG(F_\infty)$-conjugate to each other because $F_{\infty}$ is algebraically closed.
\end{proof}

\subsection{Homogeneous elements, principal gradings and regular elements in the Weyl group} In this section we shall give two ways to classify homogeneous elements in $\frc(F)^{\rs}$: one using principal gradings on the Lie algebra $\gg$ and the other using regular homomorphisms into the Weyl group $\WW'$.

\subsubsection{Periodic gradings on $\gg$}\label{sss:periodic} A periodic grading  on the Lie algebra $\gg$ is a homomorphism
\begin{equation*}
\Psi:\hZZ(1)\to\Aut(\gg)
\end{equation*}
that factors through some finite quotient. The group of characters of $\hZZ(1)$ is $\QQ/\ZZ$. We may decompose $\gg$ according to characters of $\hZZ(1)$
\begin{equation}\label{per grading}
\gg=\bigoplus_{\xi\in\QQ/\ZZ}\gg_{\xi}.
\end{equation}
The {\em order} of such a grading is the minimal positive integer $m$ such that $\Psi$ factors through $\mu_m$. The grading $\Psi$ is {\em adapted to $\theta$} if the composition $\mu_{m}\to\Aut(\gg)\to\Out(\gg)$ is the same as $\mu_{m}\xrightarrow{[m/e]}\mu_{e}\xrightarrow{\theta}\Out(\gg)$. Periodic gradings have been studied in depth by the Vinberg school.

\begin{defn}\label{def:principal} Let $\xi\in\QQ/\ZZ$. The {\em principal grading of slope $\xi$} is the following periodic grading adapted to $\theta$:
\begin{eqnarray}\label{Psi action}
\Psi_{\xi}:\hZZ(1)&\to&\GG^{\ad}\rtimes\mu_{e}\subset\Aut(\gg)\\
\notag\zeta&\mapsto&\zeta^{\xi\rho^{\vee}}\theta(\barzeta).
\end{eqnarray}
Here $\barzeta\in\mu_{e}$ is the image of $\zeta$ in $\mu_{e}$; $\zeta^{\xi\rho^{\vee}}$ is the image of $\zeta$ under the composition $\hZZ(1)\xrightarrow{\xi}\CC^{\times}\xrightarrow{\rho^{\vee}} \TT^{\ad}(\CC)$ (we have viewed $\QQ/\ZZ$ as the dual of $\hZZ(1)$).
\end{defn}

\subsubsection{Regular homomorphisms into $\WW'$} Reformulating the original definition of Springer slightly, we introduce the notion of regular homomorphisms into the group $\WW'$.

Let $\Pi:\hZZ(1)\to \WW'$ be a homomorphism over $\mu_e$, i.e., the composition $\hZZ(1)\xrightarrow{\Pi}\WW'\to\mu_{e}$ is the tautological projection. Composing with the reflection action of $\WW'$ on $\tt$, we get an action of  $\hZZ(1)$ on $\tt$. We thus get a decomposition of $\tt$ into eigenspaces according to characters of $\hZZ(1)$
\begin{equation}\label{t eigen}
\tt=\bigoplus_{\xi\in\QQ/\ZZ}\tt_{\xi}.
\end{equation}
The {\em order} of such a $\Pi$ is the minimal positive integer $m$ such that $\Pi$ factors through $\mu_m$. 

\begin{defn}[Springer]\label{reg hom} Let $\Pi:\hZZ(1)\to \WW'$ be a homomorphism over $\mu_e$. Then $\Pi$ is called {\em $\theta$-regular} if for some $\xi\in\QQ/\ZZ$, $\tt_\xi$ contains a regular element. The eigenvalues $\xi\in\QQ/\ZZ$ such that $\tt_{\xi}\cap\tt^{\rs}\neq\varnothing$ are called the {\em regular eigenvalues of $\Pi$}.
\end{defn}

Let $\sigma$ be a generator of $\mu_{e}$. Let $\Pi:\hZZ(1)\to \WW'$ be a $\theta$-regular homomorphism. Let $\zeta$ be a topological generator of $\hZZ(1)$, with $\Pi(\zeta)=w\sigma\in \WW\sigma\subset \WW'$. Then $w\sigma$ is a regular element in $\WW\sigma$ in the sense of Springer \cite[paragraph after Theorem 6.4]{Spr}.

\begin{theorem}\label{th:homog vs reg} Let $\nu\in\QQ$. We denote its image in $\QQ/\ZZ$ by $\barnu$. Then there are natural bijections between the following sets
\begin{equation}\label{three sets}
\frc(F)^{\rs}_{\nu}\bij \gg_{\barnu}^{\rs}/\GG^{\Psi_{\barnu},\circ} \bij \Reg(\WW')_{\barnu}/\WW.
\end{equation}
where
\begin{itemize}
\item $\frc(F)^{\rs}_{\nu}$ is the set of homogeneous elements of slope $\nu$ (see Definition \ref{def:homo});
\item $\gg_{\barnu}$ is the $\barnu$-piece of the principal grading $\Psi_{\barnu}$ of slope $\barnu$ (see Definition \ref{def:principal}) and $\GG^{\Psi_{\barnu},\circ}$ the neutral component of its centralizer in $\GG$; $\gg_{\barnu}^{\rs}=\gg_{\barnu}\cap\gg^{\rs}$. 
\item $\Reg(\WW')_{\barnu}$ is the set of pairs $(\Pi,X)$ where $\Pi:\hZZ(1)\to\WW'$ is a $\theta$-regular homomorphism (over $\mu_{e}$) and $X\in\tt_{\barnu}^{\rs}=\tt_{\barnu}\cap\tt^{\rs}$ (here $\tt_{\barnu}$ is the $\barnu$-piece of $\tt$ under the action of $\hZZ(1)$ via $\Pi$).The Weyl group $\WW$ acts on $\Reg(\WW')_{\barnu}$ by simultaneous conjugation.
\end{itemize}
Moreover, under the above bijections, the order of the principal grading $\Psi_{\barnu}$ in (2) and the order of the homomorphism $\Pi$ in (3) are both equal to $\lcm(m_{1},e)$, where $m_{1}$ is the denominator of $\nu$ in lowest terms. 
\end{theorem}

\begin{proof} We shall define maps 
\begin{equation*}
\xymatrix{ & \frc(F)^{\rs}_{\nu}\ar[dl]^{\phi_{1}} \\
\gg_{\barnu}^{\rs}/\GG^{\Psi_{\barnu},\circ}\ar[rr]^{\phi_{2}} & & \Reg(\WW')_{\barnu}/\WW\ar[ul]^{\phi_{3}} }
\end{equation*}
Then check their cyclic composition give identity maps.

The map $\phi_{1}$. To $a\in\frc(F)^{\rs}_{\nu}$ we will associate an element $Y\in\gg_{\barnu}$. Let $\gamma=\kappa(a)\in\frs(F)$. We write $\gamma$ as a formal Laurent series $\gamma(t^{1/e})$ in $t^{1/e}$ with coefficients in $\gg$.  Recall $s\cdot_{\dil} a$  denotes the  dilation action on $\cc$, which has positive weights $(d_{1},\cdots, d_{\rr})$. The Kostant section $\kkappa:\cc\isom\ss\subset\gg$ satisfies
\begin{equation}\label{Kos Gm}
\kkappa(s\cdot_{\dil} a)=s\Ad(s^{-\rho^{\vee}})\kkappa(a), \hspace{1cm} s\in \Gm, a\in \cc.
\end{equation}
We also denote the Galois action of $\zeta\in\hZZ(1)$ on $\gg(F_{\infty})$ (without $\theta$-twisting) by $\zeta\cdot_{\Gal}(-)$. Using \eqref{Gal vs dil}, we have
\begin{equation}\label{gamma md}
\zeta\cdot_{\Gal}\gamma=\kappa(\zeta\cdot_{\Gal}a)=\kappa(\zeta^{\barnu}\cdot_{\dil}a)=\zeta^{\barnu}\Ad(\zeta^{-\barnu\rho^{\vee}})\kappa(a)=\zeta^{\barnu}\Ad(\zeta^{-\barnu\rho^{\vee}})\gamma.
\end{equation}
Since $\gamma\in\frg$, we have $\zeta\cdot_{\Gal}\gamma=\theta(\barzeta)(\gamma)$, therefore
\begin{equation}\label{gamma eigen}
\Ad(\zeta^{\barnu\rho^{\vee}})\theta(\barzeta)\gamma=\zeta^{\barnu}\gamma.
\end{equation}
In other words, $\gamma\in\gg_{\barnu}\otimes F_{e}$. Homogeneity of $a$ implies that $\gamma$ lies in $\gg\otimes \CC[t^{1/e},t^{-1/e}]$, so that it makes sense to specialize $t$ to $1$. We assign the element $Y=\gamma(1)\in\gg_{\barnu}$ to $a$.

The map $\phi_{2}$. We define a more general map
\begin{equation}\label{more general grading}
\mbox{$\{(\Psi,Y):\Psi$ is a periodic grading adapted to $\theta$, $Y\in\gg_{\barnu}^{\rs}\}/\GG^{\ad}\to\Reg(\WW')_{\barnu}/\WW$.}
\end{equation}
Here the notation $\gg_{\barnu}$ means the $\barnu$-eigenspace with respect to the periodic grading $\Psi$.
Suppose we are given a pair $(\Psi,Y)$. Since $Y$ is an eigenvector under the action of $\hZZ(1)$ on $\gg$ via $\Psi_{\barnu}$,  the Cartan subalgebra $\gg_{Y}$ (centralizer of $Y$ in $\gg$) is normalized by $\hZZ(1)$. Choose $g\in\GG^{\ad}$ such that $\Ad(g)\gg_{Y}=\tt$,  we get a homomorphism
\begin{equation*}
\Pi_{g}:\mu_{m}\to N_{\GG^{\ad}\rtimes\mu_{e}}(\gg_{Y})\cong N_{\GG^{\ad}\rtimes\mu_{e}}(\tt)\surj\WW'.
\end{equation*}
Since $Y\in\gg_{\barnu}^{\rs}$, the element $X_{g}=\Ad(g)Y\in\tt$ then belongs to $\tt^{\rs}_{\barnu}$. Changing the choice of $g$ amounts to changing the pair $(\Pi_{g},X_{g})$ by $\WW$-conjugacy. If we change $(\Psi,Y)$ by $\GG^{\ad}$-conjugacy, the resulting $(\Pi,X)$ also changes by $\WW$-conjugacy. The map $\phi_{2}$ is obtained by applying this construction to the pair $(\Psi_{\barnu}, Y\in\gg_{\barnu}^{\rs})$.

The map $\phi_{3}$. Let $\Pi:\hZZ(1)\to \WW'$ be $\theta$-regular and let  $X\in\tt^{\rs}_{\barnu}$. We define $a=\chi(Xt^\nu)$ which is homogeneous of slope $\nu$ but {\em a priori} only an element in $\cc^{\rs}(F_{\infty})$. We need to check that $a\in\frc(F)^{\rs}$, i.e., for any $\zeta\in \hZZ(1)=\Gal(F_{\infty}/F)$,  $\zeta\cdot_{\Gal}(Xt^{\nu})$ is $\WW$-conjugate to $\theta(\barzeta)(X)t^{\nu}$ (recall $\barzeta$ is the image of $\zeta$ in $\mu_{e}$). Since $X\in\tt_{\barnu}$, we have $\Pi(\zeta)(X)=\zeta^{\barnu}X$, therefore
\begin{equation*}
\zeta\cdot_{\Gal}(Xt^{\nu})=\zeta^{\nu}Xt^{\nu}=\Pi(\zeta)(X)t^{\nu}=\Pi(\zeta\barzeta^{-1})\theta(\barzeta)(X)t^{\nu}.
\end{equation*}
Since the homomorphism $\Pi$ is over $\mu_{e}$, $\Pi(\zeta)\barzeta^{-1}\in\WW$, hence the right side above is $\WW$-conjugate to $\theta(\barzeta)(X)t^{\nu}$. This shows $a\in\frc(F)^{\rs}$.

The composition $\phi_{3}\circ\phi_{2}\circ\phi_{1}$ is the identity. This amounts to the fact that $\chi(\kappa(a(1))t^{\nu})=a$. By Lemma \ref{l:hom in Fs}, we have $a=(c_{i}t^{\nu d_{i}})_{i=1}^{\rr}$ in terms of fundamental invariants. Then $f_{i}(\kappa(a(1))t^{\nu})=t^{\nu d_{i}}f_{i}(\kappa(a(1)))=c_{i}t^{\nu d_{i}}$, as desired.

The composition $\phi_{1}\circ\phi_{3}\circ\phi_{2}$ is the identity. Start from $(\Psi_{\barnu},Y\in\gg_{\barnu}^{\rs})$. Note that $\ee\in\gg_{\barnu}$. It is shown in \cite[Theorem 3.5(ii)]{Pany} that there is an analog of Kostant section for the action of $\GG^{\Psi_{\barnu},\circ}$ on $\gg_{\barnu}$. The Kostant section in this situation is given by $\ee+\gg_{\barnu}\cap\gg^{\ff}$. Therefore, up to conjugation by $\GG^{\Psi_{\barnu},\circ}$, we may assume that $Y\in\ee+\gg_{\barnu}\cap\gg^{\ff}$. After applying the maps $\phi_{3}\circ \phi_{2}$, the resulting homogeneous element $a=\chi(Yt^{\nu})$. Applying $\phi_{1}$ again, we get the same principal grading $\Psi_{\barnu}$ and the element $\kappa(a(1))$ lying in the Kostant section. Therefore $\kappa(a(1))=Y$ because $Y$ already lies in the Kostant section.

The composition $\phi_{2}\circ\phi_{1}\circ\phi_{3}$ is the identity. We may describe the composition $\phi_{2}\circ\phi_{1}$ more directly. We take $\Pi=\Pi_{a}$. Also by Lemma \ref{l:hom in Fs}\eqref{Xtnu}, $a$ is conjugate to $Xt^{\nu}$ for some $X\in\tt^{\rs}$ well-defined up to $\WW$-conjugation. Since $a=\chi(Xt^{\nu})$ lies in $\frc(F)$, it is invariant under the Galois action of $\hZZ(1)$. The argument in the construction of the map $\phi_{3}$ then shows that $X\in\tt_{\barnu}$. This defines the map $a\mapsto(\Pi,X)$ which is inverse to the map $\phi_{3}$.

Now we have checked that the three sets in \eqref{three sets} are in natural bijection to each other.

Finally we calculate the order $m(\Psi)$ of $\Psi$. We show the following divisibility relations
\begin{equation*}
m(\Psi)\mid\lcm(m_{1},e), m(\Pi)\mid m(\Psi), \lcm(m_{1},e)\mid m(\Pi)
\end{equation*}
which then imply that $m(\Pi)=m(\Psi)=\lcm(m_{1},e)$.

The first relation $m(\Psi)\mid\lcm(m_{1},e)$ follows from \eqref{Psi action}, where we see that the inner part $\zeta\mapsto\zeta^{-\nu\rho^{\vee}}\in\GG^{\ad}$ has order divisible by $m_{1}$ and the outer part has order $e$ and they commute with each other.

The second relation $m(\Pi)\mid m(\Psi)$ follows directly from the construction of the map $\phi_{2}$.  

The last relation $\lcm(m_{1},e)\mid m(\Pi)$: $e\mid m(\Pi)$ because $\Pi$ is over $\mu_{e}$; $m_{1}\mid m(\Pi)$ because $\Pi$ has eigenvectors with eigenvalue $\barnu$. We have finished the proof of the theorem.
\end{proof}

\begin{defn}\label{def:slope} 
\begin{enumerate} 
\item (extending \cite[p.174]{Spr}) A natural number $m_{1}$ is a {\em $\theta$-regular number} (resp. {\em elliptic $\theta$-regular number}) for $\WW'$ if there is a $\theta$-regular (resp. elliptic $\theta$-regular) homomorphism $\Pi:\hZZ(1)\to \WW'$ which has a regular eigenvector of order $m_{1}$.
\item A rational number $\nu$ is a {\em $\theta$-admissible slope} (resp. {\em elliptic $\theta$-admissible slope}) if the denominator of $\nu$ (in lowest terms) is a $\theta$-regular number (resp. elliptic $\theta$-regular number).
\end{enumerate}  
When $\theta$ is clear from the context, we simply say ``(elliptic) regular numbers'' and ``(elliptic) slopes''.
\end{defn}

\begin{exam}
(1) The twisted Coxeter number $h_{\theta}$ of $(\GG,\theta)$  (see also \S\ref{sss:aff simple roots}) is an elliptic $\theta$-regular number because the twisted Coxeter conjugacy class is $\theta$-regular and elliptic of order $h_{\theta}$. 

(2) Let $e=1$ or $2$ be the order of  the opposition $\sigma\in\Out(\GG)$ ($\sigma$ acts as $-w_0$ under the reflection representation). Let $\theta:\mu_e\to\Out(\GG)$ denote the unique homomorphism sending $-1$ to $\sigma$ if $e=2$. Then 2 is an elliptic $\theta$-regular number because the element $\hZZ(1)\surj\mu_2\incl\WW'$ (sending $-1\in\mu_2$ to $w_{0}\sigma\in \WW'$) is elliptic and $\theta$-regular of order 2.
\end{exam}

As a consequence of Theorem \ref{th:homog vs reg}, we get the following property for $\theta$-regular homomorphisms. The statement (1) below was proved by Springer in \cite[Theorem 6.4(iv)]{Spr} by different methods, and the other statements are implicitly checked in \cite{Spr}.

\begin{cor}[Springer]\label{c:order det w} Recall that we assumed $\GG$ to be almost simple. \footnote{Statements (2) and (3) in this Corollary fail if $\GG$ is not almost simple and $e>1$: a regular homomorphism can have regular eigenvalues of different orders.}
\begin{enumerate}
\item If two $\theta$-regular homomorphisms into $\WW'$ over $\mu_{e}$ have regular eigenvalues of the same order, then they are $\WW$-conjugate to each other.
\item Let $\Pi:\hZZ(1)\to \WW'$ be a $\theta$-regular homomorphism. Then all regular eigenvalues of $\Pi$ have the same order.
\item Sending a $\theta$-regular homomorphism to the order of its regular eigenvalues (well-defined by (2) above) gives a bijection
\begin{equation}\label{RN to RH}
\mbox{\{$\theta$-regular homomorphisms into $\WW'\}/\WW\bij\{\theta$-regular numbers\}. }
\end{equation}
\end{enumerate}
\end{cor}
\begin{proof}
(1) Let $\Pi_{1},\Pi_{2}:\hZZ(1)\to\WW'$ be $\theta$-regular homomorphisms over $\mu_{e}$, with regular eigenvalues $\xi_{1}$ and $\xi_{2}$ respectively. Assume $\xi_{1}$ and $\xi_{2}$ have the same order (as elements in $\QQ/\ZZ$), and we would like to show that $\Pi_{1}$ and $\Pi_{2}$ are conjugate under $\WW$. Since the reflection representation of $\WW'$ on $\tt$ is defined over $\QQ$, all eigenspaces of $\tt$ under $\Pi_{1}$ whose eigenvalues have the same order are permuted by $\Aut(\CC/\QQ)$. Therefore we may assume $\xi_{1}=\xi_{2}$ and denote it by $\xi$.  Let $X_{i}$ be a regular eigenvector for $\Pi_{i}$ with eigenvalue $\xi$, $i=1,2$.

Applying the bijection $\gg^{\rs}_{\xi}/C_{\GG^{\ad}}(\Psi_{\xi})\bij\Reg(\WW')_{\xi}/\WW$ in Theorem \ref{th:homog vs reg}, $(\Pi_{i},X_{i})$ corresponds to $Y_{i}\in\gg_{\xi}^{\rs}$ (up to $C_{\GG^{\ad}}(\Psi_{\xi})$) for $i=1,2$. Applying the map $\phi_{2}$ in the proof of Theorem \ref{th:homog vs reg}, the construction $Y\mapsto\Pi$ is independent of the choice of $Y\in\gg_{\xi}^{\rs}$ (because $\Pi$ varies locally constantly with $Y$ and $\gg_{\xi}^{\rs}$ is connected), therefore $\Pi_{1}$  (which corresponds to $Y_{1}$) and $\Pi_{2}$ (which corresponds to $Y_{2}$) are $\WW$-conjugate. 

(2) Since $\GG$ is almost simple, $e=1,2$ or $3$. Suppose $\Pi$ has two regular eigenvalues of orders $m_{1}<m_{2}$. By Theorem \ref{th:homog vs reg}, we have $\lcm(m_{1},e)=\lcm(m_{2},e)$ (both equal to the order of $\Pi$). The only possibility is when $e>1$, $\gcd(m_{1},e)=1$ and $m_{2}=em_{1}$. Let $\Pi'$ be the composition $\hZZ(1)\xrightarrow{m_{1}}\hZZ(1)\xrightarrow{\Pi}\WW'$, then $\Pi'$ is $\theta$-regular of order $e$, with two different regular eigenvalues one of which is $1$. Note that $\theta: \mu_{e}\incl\WW'$ itself is a $\theta$-regular homomorphism of order $e$ with regular eigenvalue $1$. By (1) above, $\Pi'$ is $\WW$-conjugate to $\theta$. However, we shall check case-by-case that $\theta$ does not admit regular eigenvalues other than 1, hence getting a contradiction. When $e=2$, $\theta(-1)$ acts as $-w_{0}$ on $\tt$, where $w_{0}\in\WW$ is the longest element. Therefore the nontrivial eigenspace of $\theta$ on $\tt$ is $\tt^{w_{0}}$, hence does not contain regular elements. When $e=3$, $\GG$ is of type $D_{4}$, then one node in the Dynkin diagram is fixed by $\theta$, and a nontrivial eigenspace of $\theta$ on $\tt$ must be killed by the corresponding simple root, hence cannot contain any regular elements either.

(3) is an immediate consequence of (1) and (2).
\end{proof}

\subsubsection{Normal form of an admissible slope}\label{sss:normal}
Given a $\theta$-admissible slope $\nu$, we first write it as $\nu=d_{1}/m_{1}$ in lowest terms, where $m_{1}$ is a $\theta$-regular number. Using the map \eqref{RN to RH}, we get a $\theta$-regular homomorphism $\Pi:\mu_{m}\incl\WW'$ of order $m$ (up to conjugacy). We know from Theorem \ref{th:homog vs reg} that $m=\lcm(m_{1},e)$. Therefore we may write $\nu=d/m$ for some integer $d$. This is called the {\em normal form} of $\nu$.

\subsection{Homogeneous elements and Moy-Prasad filtration}
Let $\nu=d/m>0$ be a $\theta$-admissible slope in the normal form. 

\subsubsection{The torus $\Gnu$}\label{sss:Gnu}
Let $\Gm(\nu)$ be the one-dimensional subtorus defined by the homomorphism
\begin{eqnarray}\label{Gnu action with dil}
\Gnu&\to& \tilG^{\ad}:=(G^{\ad}(F)\rtimes\Grote)\times\Gdil\\
s&\mapsto &(s^{d\rho^{\vee}},s^m, s^{-d}).
\end{eqnarray}
Now $\tilG^{\ad}$ acts on $G(F)$ with $\Gdil$ acting trivially, we have an action of $\Gnu$ on $G(F)$. Explicitly it is given by
\begin{equation}\label{s act fl}
s\cdot_{\nu}g(t^{1/e}):=\Ad(s^{d\rho^{\vee}})(g(s^{m/e}t^{1/e})).
\end{equation}
The action of $\Gnu$ on $G(F)$ induces an action on its Lie algebra $\frg(F)$, and gives a decomposition of $\frg$ into weight spaces (which are $\CC$-vector spaces). For $x\in\frac{1}{m}\ZZ$, we let $\frg(F)_{x}$ be the weight space with weight $mx$ under $\Gnu$. We have
\begin{equation*}
\frg(F)=\widehat{\bigoplus}_{x\in\frac{1}{m}\ZZ}\frg(F)_{x}.
\end{equation*}
where $\widehat{\oplus}$ means $t$-adic completion. To identify these weight spaces, we need a bit of Moy-Prasad theory.

\subsubsection{Moy-Prasad filtration}\label{sss:MP} Let $\frA$ be the apartment for $G(F)$ corresponding to $T(F)=\TT(F_{e})^{\mu_{e}}$. The parahoric $\bG$ gives a special vertex  of $\frA$, which allows us to identify $\frA$ with $\xcoch(\TT)^{\mu_{e}}_{\RR}$. The point $\nu\rho^{\vee}\in\frA$ defines a {\em Moy-Prasad filtration} $\frg(F)_{\nu\rho^{\vee}, \geq x}\subset\frg(F)$ indexed by rational numbers $x\in\frac{1}{m}\ZZ$. The $\calO_{F}$-submodule $\frg(F)_{\nu\rho^{\vee}, \geq x}$ is the $t$-adic completion of the span of affine root spaces corresponding to $\alpha\in\Phi_{\aff}$ such that $\alpha(\nu\rho^{\vee})\geq x$. In particular, $\frg(F)_{\nu\rho^{\vee}, \geq 0}$ is the Lie algebra of the parahoric subgroup $\bP_{\nu}$ corresponding to the facet containing $\nu\rho^{\vee}$. Let $L_{\nu}$ be the Levi factor of $\bP_{\nu}$ that contains $\AA$. The root system of $L_{\nu}$ consists of those affine roots $\alpha$ such that $\alpha(\nu\rho^{\vee})=0$. Each $\frg(F)_{\nu\rho^{\vee}, \geq x}$ is an $L_{\nu}$-module.

Suppose $a\in\frc(F)_{\nu}$ and $\gamma=\kappa(a)$. We write $\gamma$ as a formal power series $\gamma(t^{1/e})$ in $t^{1/e}$ with coefficients in $\gg$. Using \eqref{Kos Gm}, we get
\begin{equation}\label{s act gamma}
\Ad(s^{d\rho^{\vee}})\gamma(s^{m/e}t^{1/e})=s^{d}\gamma(t^{1/e}).
\end{equation}
Therefore, $\gamma\in\frg(F)_{\nu}$. In other words, the Kostant section gives a section
\begin{equation*}
\kappa:\frc(F)_{\nu}\incl\frg(F)_{\nu}.
\end{equation*}

\begin{lemma}\label{l:MP}
\begin{enumerate}
\item For each $x\in\frac{1}{m}\ZZ_{\geq 0}$, we have
\begin{equation}\label{MP grading}
\frg(F)_{\nu\rho^{\vee},\geq x}=\wh{\bigoplus}_{x'\geq x}\frg(F)_{x'}
\end{equation}
\item\label{frg vs gg} Let $\gg=\oplus_{\xi\in\frac{1}{m}\ZZ/\ZZ}\gg_{\xi}$ be the decomposition corresponding to the principal periodic grading $\Psi_{\barnu}$ of slope $\barnu$ in \eqref{Psi action}. Then there is a canonical isomorphism $\frg(F)_{x}\cong\gg_{\barx}$, where $\barx=x \mod m\in\frac{1}{m}\ZZ/\ZZ$.
\item The fixed point subgroup $\tilL_{\nu}:=G(F)^{\Gnu}$ is a reductive group over $\CC$, and it contains $L_{\nu}$ as the neutral component.
\item\label{Lnu} The group $\tilL_{\nu}$ is canonically isomorphic to the centralizer $\GG^{\Psi_{\barnu}}$ of the principal grading $\Psi_{\barnu}$ (see \eqref{Psi action}) in $\GG$. 
\end{enumerate}
\end{lemma}
\begin{proof}
(1) Direct calculation shows that $\Gnu$ acts on the affine root space corresponding to $\alpha\in\Phi_{\aff}\cup\{0\}$ by weight $m\alpha(\nu\rho^{\vee})$. Since each step of the Moy-Prasad filtration is a sum of affine root spaces, the statement follows.

(2) The space $\frg(F)_{x}$ is a sum of affine root spaces. For each $\alpha\in\Phi_{\aff}\cup\{0\}$, we identify the corresponding affine root space with $\gg_{\overline{\alpha}}$, where $\overline\alpha\in\Phi\cup\{0\}$ is the linear part of $\alpha$. These isomorphisms sum up to give $\frg(F)_{x}\isom\gg_{\barnu}$. 

(3) We have $\Lie L_{\nu}=\frg(F)_{0}=\Lie G(F)^{\Gnu}$ by identifying the affine roots. Therefore $L_{\nu}\subset G(F)^{\Gnu}$ as the neutral component. We must also show that $ G(F)^{\Gnu}$ only has finitely many components. In fact, $G(F)^{\Gnu}$ normalizes $\frp_{\nu\rho^{\vee},\geq0}$ by (1), hence it normalizes $\bP_{\nu}$. Let $\tbP_{\nu}$ be the normalizer of $\bP_{\nu}$ in $G(F)$, then $\tbP_{\nu}/\bP_{\nu}$ is can be identified with the stabilizer of the facet $\frF_{\bP_{\nu}}$ under $\tilW$, which is a finite group. Therefore  $G(F)^{\Gnu}/L_{\nu}\subset \tbP_{\nu}/\bP_{\nu}$ is a finite group, and $G(F)^{\Gnu}$ only has finitely many components.

(4) By embedding $\GG$ into a matrix group, it is easy to see that $\GG(F_{e})^{\Gnu}\subset\GG(\CC[t^{1/e},t^{-1/e}])$. Therefore we may view an element $g\in G(F)^{\Gnu}$ as a morphism of varieties $g:\Gm\to\GG$. The coordinate on the source $\Gm$ is $t^{1/e}$. The fact that $g\in G(F)=\GG(F_{e})^{\mu_{e}}$ says that $g$ is $\mu_{e}$-equivariant: with the multiplication action on the source $\Gm$ and the action via $\theta$ on the target. The $\Gnu$-equivariance of $g$ is same as saying that $g$ is $\Gm$-equivariant: $s\in\Gm$ acts on the source $\Gm$ by multiplication by $s^{m/e}$ and acts on the target $\GG$ by $\Ad(s^{-d\rho^{\vee}})$. Giving such a map $g$ is equivalent to knowing the value $g(1)\in\GG$ satisfying two conditions
\begin{itemize}
\item $g(\eta)=\theta(\eta)g(1), \forall\eta\in\mu_{e}$;
\item $g(\zeta^{m/e})=\Ad(\zeta^{-d\rho^{\vee}})g(1), \forall\zeta\in\mu_{m}$.
\end{itemize}
The two conditions are consistent if and only if $\theta(\zeta^{m/e})g(1)=\Ad(\zeta^{-d\rho^{\vee}})g(1)$, or in other words, $g(1)$ is invariant under the action of $\Psi_{\barnu}$ defined in \eqref{Psi action}. Therefore the evaluation at $t^{1/e}=1$ gives an isomorphism of reductive groups $\tilL_{\nu}=G(F)^{\Gnu}=\GG(\CC[t^{1/e},t^{-1/e}])^{\mu_{e},\Gnu}\isom\GG^{\Psi_{\barnu}}$.
\end{proof}

\begin{lemma}\label{l:identify nu} Let $\nu$ and $\nu'$ be $\theta$-admissible slopes with the same denominators in lowest terms. Then there are isomorphisms $\tilL_{\nu}\cong \tilL_{\nu'}$ and $\frg(F)_{\nu}\cong\frg(F)_{\nu'}$ compatible with the action of $\tilL_{\nu}$ and $\tilL_{\nu'}$. One can make these isomorphisms canonical except when $G$ is of type $\leftexp{3}{D}_{4}$ and the order of $\nu$ is divisible by $3$.
\end{lemma}
\begin{proof} Let $\nu=d_{1}/m_{1}$ and $\nu'=d'_{1}/m_{1}$ in lowest terms. It suffices to treat the case $d'_{1}=1$. By Lemma \ref{l:MP}, it suffices to construct compatible isomorphisms between $\GG^{\Psi_{\barnu}}$ and $\GG^{\Psi_{\barnu'}}$, and between their grading pieces $\gg_{\barnu}$ and $\gg_{\barnu'}$. In fact, we will show that in most cases these isomorphisms are equalities as subgroups of $\GG$ and subspaces of $\gg$.

First assume either $e=1$ or $e\nmid m_{1}$. In this case $\GG^{\Psi_{\barnu},\circ}\subset\HH=\GG^{\mu_{e},\circ}$ and $\gg_{\barnu}\subset\Lie\HH$. Working with $\HH$ instead of $\GG$ we may assume $e=1$. We let $m=m_{1}$ and $d=d_{1}$. In this case the action $\Psi_{\barnu}:\mu_{m}\to\Aut(\GG)$ is simply the composition $\Psi_{\barnu'}\circ[d]$ (pre-composite with the $d$-th power map on $\mu_{m}$). Since $d$ is prime to $m$ we get $\GG^{\Psi_{\barnu}}=\GG^{\Psi_{\barnu'}}$ and $\gg_{\barnu}=\gg_{\barnu'}$.

Next assume $e=2$ and $2|m_{1}$. Again we write $m=m_{1}$ and $d=d_{1}$. In this case $d$ must be odd, and again $\Psi_{\barnu}=\Psi_{\barnu'}\circ[d]$. Therefore the same conclusion as above holds.

Finally we treat the case $e=3$ and $3|m_{1}$. Again we write $m=m_{1}$ and $d=d_{1}$. In this case $d$ must be prime to $3$. The above argument allows us to replace $\Psi_{\barnu'}$ by $\Psi':=\Psi_{\barnu'}\circ[d]$. In other words we are comparing the two gradings $\Psi_{\barnu}:\zeta\mapsto\Ad(\zeta^{d\rho^{\vee}})\theta(\zeta^{m/e})$ and $\Psi':\zeta\mapsto\Ad(\zeta^{d\rho^{\vee}})\theta(\zeta^{dm/e})$. When $d\equiv1\mod3$ then $\Psi'=\Psi_{\barnu}$ and the conclusion follows. If $d\equiv2\mod3$, then for a primitive $\zeta\in\mu_{m}$ with $\theta(\zeta^{m/e})=\sigma\in\Aut^{\dagger}(\GG)$, the two actions $\Psi_{\barnu}(\zeta)=\Ad(\zeta^{d\rho^{\vee}})\sigma$ and $\Psi'(\zeta)=\Ad(\zeta^{d\rho^{\vee}})\sigma^{2}$. However, the pinned automorphism group $\Aut^{\dagger}(\GG)$ of $\GG$ (of type $D_{4}$) is isomorphic to $S_{3}$. Choose an involution $\tau\in\Aut^{\dagger}(\GG)\cong S_{3}$. Then conjugation by $\tau$ interchanges $\sigma$ and $\sigma^{2}$, and the action of $\tau$ (as with all pinned automorphisms) commutes with $\Ad(\zeta^{d\rho^{\vee}})$. Therefore the automorphism $\tau$ of $\GG$ identifies $\GG^{\Psi_{\barnu}}$ with $\GG^{\Psi_{\barnu'}}=\GG^{\Psi'}$, and identifies $\gg_{\barnu}$ with $\gg_{\barnu'}$. The proof is complete.
\end{proof}

With the Moy-Prasad grading, we can describe the centralizer $G_{\gamma}$ more explicitly.

\begin{lemma}\label{l:local Neron} Let $G^{\flat}_{\gamma}$ be the N\'eron model of $G_{\gamma}$ that is of finite type over $\calO_{F}$. Then
\begin{enumerate}
\item $G^{\flat}_{\gamma}(\calO_{F})\subset \bP_{\nu}\cdot G(F)^{\Gnu}$, which is contained in the normalizer of $\bP_{\nu}$.
\item Let $\tilS_{a}:=\tilL_{\nu,\gamma}$ be the centralizer of $\gamma$ in $\tilL_{\nu}$. The inclusion $\tilS_{a}\incl G^{\flat}_{\gamma}(\calO_{F})$ identifies $\tilS_{a}$ with the Levi factor of the pro-algebraic group $G^{\flat}_{\gamma}(\calO_{F})$.
\item\label{Tmum} Let $\Pi:\mu_{m}\to\WW'$ be the $\theta$-regular homomorphism corresponding to the $\theta$-regular number $m$ by Corollary \ref{c:order det w}, then $\tilS_{a}\cong\TT^{\Pi(\mu_m)}$. 
\item\label{tilL} Let $S_{a}:=L_{\nu,\gamma}$ be the centralizer of $\gamma$ in $L_{\nu}$. Then $\tilS_{a}/S_{a}$ is canonically a subgroup of $\Omega$. Assume further that $\nu$ is elliptic. Then there is an exact sequence
\begin{equation*}
1\to S_{a}\to \tilS_{a}\to \Omega\to1.
\end{equation*}
\end{enumerate}
\end{lemma}
\begin{proof}
(1) We first show that $\Lie G^{\flat}_{\gamma}(\calO_{F})\subset \frg(F)_{\nu\rho^{\vee}, \geq0}=\Lie\bP_{\nu}$. In other words, the weights of the $\Gnu$-action on $\Lie G^{\flat}_{\gamma}(\calO_{F})$ are non-negative. This can be checked after extending $F$ to $F_{m}$, with the action of $\Gnu$ on $\GG(F_{m})$ given by  $\Gnu\ni s: g(t^{1/m})\mapsto\Ad(s^{d\rho^{\vee}})g(st^{1/m})$. Inside $\frg(F_{m})$, we may conjugate $\gamma$ to $Xt^{\nu}$ ($X\in\tt^{\rs}$), and  $G^{\flat}_{\gamma}(\calO_{F_{m}})$ to $\TT(\calO_{F_{m}})$, for which is statement is obvious.

We may write $G^{\flat}_{\gamma}(\calO_{F})$ canonically as the $G^{\flat, +}_{\gamma}\cdot G^{\flat,\red}_{\gamma}$, where $G^{\flat, +}_{\gamma}$ is its the pro-unipotent radical and $G^{\flat,\red}_{\gamma}$ the Levi factor, which is a diagonalizable group over $\CC$. The above argument shows that the neutral component of $G^{\flat}_{\gamma}(\calO_{F})$ is contained in $\bP_{\nu}$, hence $G^{\flat,+}_{\gamma}\subset\bP_{\nu}$ and $G^{\flat, \red,\circ}_{\gamma}\subset L_{\nu}$. It remains to show $G^{\flat,\red}_{\gamma}\subset G(F)^{\Gnu}$. Now $G^{\flat,\red}_{\gamma}$ is a diagonalizable group on which $\Gnu$ acts,  the action must be trivial because the automorphism group of a diagonalizable group is discrete,  therefore $G^{\flat,\red}_{\gamma}\subset G(F)^{\Gnu}$.

(3) Recall how we assigned the $\theta$-regular homomorphism $\Pi:\mu_{m}\to\WW'$ to $\gamma(1)\in\gg_{\barnu}$ in the proof of Theorem \ref{th:homog vs reg}. Since $\gamma(1)\in\gg^{\rs}$, the centralizer $\TT'=C_{\GG}(\gamma(1))$ is a maximal torus and the principal grading $\Psi:\mu_{m}\to\GG^{\ad}\rtimes\mu_{e}$ normalizes it. Therefore $\Psi$ induces a homomorphism $\Pi:\mu_{m}\to N_{\GG^{\ad}\rtimes\mu_{e}}(\TT')/\TT'\cong\WW'$. By Lemma \ref{l:MP}\eqref{Lnu}, the group $\tilS_{a}=\tilL_{\nu,\gamma}$ is the centralizer in $\GG$ of both the grading $\Psi$ and the element $\gamma(1)$, therefore $\tilS_{a}=C_{\GG}(\gamma(1))^{\Psi}=\TT'^{\Pi(\mu_{m})}\cong\TT^{\Pi(\mu_{m})}$.

(2) We have already shown that $G^{\flat,\red}_{\gamma}\subset G(F)^{\Gnu}=\tilL_{\nu}$, therefore $G^{\flat,\red}_{\gamma}\subset \tilL_{\nu,\gamma}=\tilS_{a}$. Since $\tilS_{a}$ is also diagonalizable in (3), hence reductive, we have $G^{\flat,\red}_{\gamma}=\tilS_{a}$. 

(4) The quotient $\tilS_{a}/S_{a}$ is a subgroup of $\tilL_{\nu}/L_{\nu}=\wt\bP_{\nu}/\bP_{\nu}$, therefore a subgroup of $\Omega$. It remains to show that $\tilS_{a}/S_{a}=\Omega$. Since $\nu$ is elliptic, $\tilS_{a}\cong\TT^{\mu_{m}}$ is finite, therefore it suffices to show that $[\tilS_{a}:S_{a}]=\#\Omega$. Let $\GG^{\sc}\to\GG$ be the simply-connected cover with kernel $Z$. Define $S^{\sc}_{a}$ and $\tilS^{\sc}_{a}$ as the counterpart of $S_{a}$ and $\tilS_{a}$ for $\GG^{\sc}$, then in fact $S^{\sc}_{a}=\tilS^{\sc}_{a}$ since $\Omega$ is trivial for $\GG^{\sc}$. We also have $S^{\sc}_{a}/Z^{\mu_{e}}=S_{a}$. Let $w\in \WW'$ be the image of a generator of $\mu_{m}$ under $\Pi$, then $S^{\sc}_{a}\cong\TT^{\sc,w}$ and $\tilS_{a}\cong\TT^{w}$ by (3). Since $w$ acts elliptically on $\tt$, we have $S^{\sc}_{a}=\#\tilS_{a}=\det(1-w|\xcoch(\TT))$. Therefore $[\tilS_{a}:S_{a}]=[S^{\sc}_{a}:S_{a}]=\#Z^{\mu_{e}}$. Since $\#\Omega=\#Z^{\mu_{e}}$, we have $[\tilS_{a}:S_{a}]=\#\Omega$, hence $\tilS_{a}/S_{a}\cong\Omega$.
\end{proof}

\subsubsection{Cartan subspace and the little Weyl group}\label{sss:Cartan} 
Let $\nu$ be a $\theta$-admissible slope which corresponds to a $\theta$-regular homomorphism $\Pi:\mu_{m}\incl\WW'$. Let $\tt_{\barnu}\subset\tt$ be the eigenspace with eigenvalue $\barnu=\nu\mod\ZZ$ as in \eqref{t eigen}. We call $\tt_{\barnu}$ a {\em Cartan subspace} for the slope $\nu$. The subspace $\tt_{\barnu}\otimes t^{\nu}\subset \gg(F_{m})$ then consists of homogeneous elements of slope $\nu$ and the homomorphism $\Pi$ shows that their image under the invariant quotient map $\gg(F_{m})\to \cc(F_{m})$ in fact lies in $\frc(F)_{\nu}$. We therefore get a map $\tt^{\rs}_{\barnu}\to \frc(F)^{\rs}_{\nu}$, which is an \'etale cover with Galois group $W(\Pi, \barnu)$, the {\em little Weyl group} attached to $(\Pi,\barnu)$. This notion is equivalent to the notion of little Weyl groups attached to gradings on Lie algebras reviewed in \cite[\S6]{GLRY}. The group $\mu_{m}$ acts on $\WW$ via $\Pi$ and the conjugation action of $\WW'$ on $\WW$. The fixed point subgroup $\WW^{\Pi(\mu_{m})}$ acts on $\tt_{\barnu}$. Combining \cite[Lemma 19]{GLRY} and \cite[Theorem 4.7]{Pany} (see also \cite[Proposition 20]{GLRY}, applicable here since $\Pi$ arises from a principal grading according to Theorem \ref{th:homog vs reg}), we conclude that $W(\Pi,\barnu)=\WW^{\Pi(\mu_{m})}$.

\subsubsection{Family of regular centralizers over $\frc(F)^{\rs}_{\nu}$}\label{sss:Pnu} Let $\nu>0$ be a $\theta$-admissible slope so that $\frc(F)^{\rs}_{\nu}\subset\frc(\calO_{F})$. Recall we have the regular centralizer group scheme $J$ over $\frc$ defined in \S\ref{ss:J}. For each $a\in \frc(\calO_{F})$ let $J_{a}$ be the group scheme over $\calO_{F}$ defined by the pullback of $J$ along $a:\Spec\calO_{F}\to\frc$. We would like to define a commutative group scheme  $P_{\nu}$ over $\frc(F)^{\rs}_{\nu}$ whose fiber over $a$ is $P_{a}:=J_{a}(F)/J_{a}(\calO_{F})$. 

We use the notation from \S\ref{sss:Cartan}. We first work over $\tt^{\rs}_{\barnu}$. Every point in $\tt^{\rs}_{\barnu}$ gives a lifting $\tila: \Spec\calO_{F_{m}}\to\tt^{\rs}_{\barnu}$ of $a:\Spec\calO_{F}\to\frc$. The alternative construction of the regular centralizers given in \cite{DG} then shows that each $J_{\tila}$ is canonically isomorphic to the same group scheme $\wt\JJ$ over $\Spec\calO_{F_{m}}$ such that  it is a subgroup scheme of $\TT\otimes\calO_{F_{m}}$ with the same generic fiber. Moreover $\wt\JJ$ carries a descent datum from $\Spec\calO_{F_{m}}$ to $\Spec\calO_{F}$, or simply an action of $\mu_{m}$ that is compatible with its Galois action on $\calO_{F_{m}}$. We then have two group ind-schemes over $\tt^{\rs}_{\barnu}$, one is $\tt^{\rs}_{\barnu}\times\TT(F_{m})^{\Pi(\mu_{m})}$ and the other is its subgroup scheme $\tt^{\rs}_{\barnu}\times\wt\JJ(\calO_{F_{m}})^{\Pi(\mu_{m})}$. The desired group ind-scheme $P_{\nu}$, when pulled back to $\tt^{\rs}_{\barnu}$ should be the quotient of these two. Using the action of $\WW^{\Pi(\mu_m)}$ to descend the two group ind-schemes from $\tt^{\rs}_{\barnu}$ to $\frc(F)^{\rs}_{\nu}$, we get the group scheme over $\frc(F)^{\rs}_{\nu}$
\begin{equation*}
P_{\nu}=\tt^{\rs}_{\barnu}\twtimes{\WW^{\Pi(\mu_m)}}\left(\TT(F_{m})^{\Pi(\mu_m)}/\wt\JJ(\calO_{F_{m}})^{\Pi(\mu_m)}\right).
\end{equation*}

\subsubsection{The group schemes $\tilS$ and $S$}\label{sss:tilS} The results in Lemma \ref{l:local Neron} can be stated in families. The group scheme $J^{\flat}_{\nu}(\calO_{F}):=\tt^{\rs}_{\barnu}\twtimes{\WW^{\Pi(\mu_m)}}\TT(\calO_{F_{m}})^{\Pi(\mu_m)}$ over $\frc(F)^{\rs}_{\nu}$ is the family of $\calO_{F}$-points of the N\'eron models of the centralizers $G^{\flat}_{\gamma}$, for $\gamma=\kappa(a), a\in\frc(F)^{\rs}_{\nu}$. The group scheme $\tilS$ over $\frc(F)^{\rs}_{\nu}$, defined as the family of stabilizers under $\tilL_{\nu}$, is canonically isomorphic to the Levi factor of $J^{\flat}_{\nu}(\calO_{F})$. Therefore, Lemma \ref{l:local Neron}\eqref{Tmum} gives an isomorphism of group schemes over $\frc(F)^{\rs}_{\nu}$
\begin{equation*}
\tilS\cong \tt^{\rs}_{\barnu}\twtimes{\WW^{\Pi(\mu_m)}}\TT^{\Pi(\mu_m)}.
\end{equation*}
By Lemma \ref{l:local Neron}\eqref{tilL}, there is a canonical homomorphism of group schemes $\tilS\to \Omega\times \frc(F)^{\rs}_{\nu}$. Let $S\subset \tilS$ be the kernel of this homomorphism. Then $S$ is also a smooth affine group scheme over  $\frc(F)^{\rs}_{\nu}$ whose fiber over $a$ is the group $S_{a}$ defined in Lemma \ref{l:local Neron}\eqref{tilL}.

%%%% Cherednik algebra %%%%%%%

\section{Graded and rational \Chas}\label{s:H}
In this section, we recall the definition of two versions of Cherednik algebras, and collect some basic facts about their structure and representation theory.

\subsection{The graded \Cha}\label{s:grad alg def}We first recall the definition of $\Hgr=\Hgr(\GG,\theta)$. As a $\QQ$-vector space, 
\begin{equation*}
\Hgr:=\Sym(\xch(\Gdil)_{\QQ}\oplus\fra_{\KM}^{*})\otimes\QQ[\tilW].
\end{equation*}
Let $u$ be the generator of $\xch(\Gdil)$, then we may also write
\begin{equation*}
\Hgr=\QQ[u,\delta]\otimes\Sym(\fra^{*})\otimes\QQ[\L_{\can}]\otimes\QQ[\tilW].
\end{equation*}
The grading is given by
\begin{eqnarray*}
&& \deg(\tilw)=0, \forall\tilw\in\tilW \\
&& \deg(u)=\deg(\xi)=2, \forall\xi\in\fra_{\KM}^{*}
\end{eqnarray*}
The algebra structure of $\Hgr$ is determined by
\begin{enumerate}
\item[(GC-1)] $u$ is central;
\item[(GC-2)] $\QQ[\tilW]$ and $\Sym(\fra_{\KM}^{*})$ are subalgebras;
\item[(GC-3)] For each affine simple reflection $s_i\in\Wa$ (corresponding to the affine simple root $\alpha_{i}$) and $\xi\in\fra_{\KM}^{*}$,
\begin{equation*}
s_i\xi-\leftexp{s_i}{\xi}s_i=\jiao{\xi,\alpha_i^\vee}u.
\end{equation*}
\item[(GC-4)] For any $\omega\in\Omega_{\bI}$ and $\xi\in\fra_{\KM}^{*}$,
\begin{equation*}
\omega\xi=\leftexp{\omega}{\xi}\omega.
\end{equation*}
\end{enumerate}

\subsubsection{Central charge} We know $u$ is central. From the relations (GC-3) and (GC-4), we see that $\delta\in\fra_{\KM}^{*}$ is also central. 

Using the $\tilW$-invariance of $B_{\KM}\in\Sym^{2}(\fra_{\KM})^{\tilW}$, we see that $B_{\KM}$ is also a  central element in $\Hgr$.

Let 
\begin{equation*}
\Hgr_{\nu}=\Hgr/(u+\nu\delta).
\end{equation*}
This is the {\em graded \Cha with central charge $\nu$}. We shall write $\ep$ for the image of $\delta=-u/\nu$ in $\Hgr_{\nu}$.

\subsubsection{Specialize $\ep$ to 1} 
Let $\Hgr_{\nu,\ep=1}$ be the specialization $\Hgr_{\nu}/(\ep-1)$. This is the trigonometric \Cha in the literature. The grading on $\Hgr$ only induces a filtration after specialization: we define $\mathfrak{H}^{\gr, \leq i}_{\nu,\ep=1}\subset\Hgr_{\nu,\ep=1}$ to be the image of the degree $\leq i$ part of $\Hgr_{\nu}$, and call this the {\em cohomological filtration}.

\subsection{The rational \Cha}\label{s:rat alg def}

\subsubsection{Parameters}\label{sss:param} We shall define a function:
\begin{equation*}
c:\Phi/W\to\{1,2,3\}.
\end{equation*}
Here $\Phi$ is the relative root system $\Phi(\GG,\AA)$ on which $W$ acts by permutation. Recall $\PPhi=\Phi(\GG,\TT)$ is the absolute root system and we have a projection $\PPhi\to \Phi$. For $\alpha\in\Phi$, let $c_{\alpha}$ be the cardinality of the preimage of $\alpha$ under the projection $\PPhi\to \Phi$. 

We also give the value of $c$ explicitly in the various cases. If $e=1$, $c$ is the constant function $1$. If $e>1$, then the relative root system $\Phi$ is not necessarily reduced. In any case, we can speak about the longest roots (when $G$ is of type $\SUodd$, there are three root lengths). We have
\begin{equation*}
c_{\alpha}=\begin{cases}1 & \alpha \textup{ is longest;}\\ e & \textup{otherwise.}\end{cases}
\end{equation*}

\subsubsection{Rational \Cha}
The bigraded rational \Cha $\Hrat=\Hrat(\GG,\theta)$ is, as a vector space, a tensor product
\begin{equation*}
\Hrat=\QQ[\delta,u]\otimes\Sym(\fra)\otimes\Sym(\fra^{*})\otimes\QQ[W]\otimes\QQ[\L_{\can}].
\end{equation*}
The bigrading is given by
\begin{eqnarray*}
&&\deg(\delta)=\deg(u)=(2,0);\\
&&\deg(\L_{\can})=(2,2);\\
&&\deg(\eta)=(0,-1)\hspace{1cm}\forall\eta\in\fra;\\
&&\deg(\xi)=(2,1)\hspace{1cm}\forall\xi\in\fra^{*};\\
&&\deg(w)=(0,0)\hspace{1cm}\forall w\in W;
\end{eqnarray*}
The first and second component of the grading are called the {\em cohomological grading} and the {\em perverse grading}. The algebra structure is given by
\begin{enumerate}
\item[(RC-1)] $u$ and $\delta$ are central; $\L_{\can}$ commutes with $\fra^{*}$ and $W$;
\item[(RC-2)] $\Sym(\fra),\Sym(\fra^{*})$ and $\QQ[W]$ are subalgebras;
\item[(RC-3)] For $\eta\in\fra, \xi\in\fra^{*}$ and $w\in W$, we have $w\eta=\leftexp{w}{\eta}w, w\xi=\leftexp{w}{\xi}w$ ;
\item[(RC-4)] For $\eta\in\fra,\xi\in\fra^{*}$, we have
\begin{equation*}
[\eta,\xi]=\jiao{\xi,\eta}\delta+\dfrac{1}{2}\left(\sum_{\alpha\in\Phi}c_{\alpha}\jiao{\xi,\alpha^\vee}\jiao{\alpha,\eta}r_{\alpha}\right)u
\end{equation*}
where $r_\alpha\in W$ is the reflection associated to the root $\alpha$ and $c_{\alpha}$ is the constant defined in \S\ref{sss:param};
\item[(RC-5)] For $\eta\in\fra$, we have $[\eta,\L_{\can}]=-\eta^{*}\in\fra^{*}$,
where $\eta\mapsto\eta^{*}$ is the isomorphism $\fra\isom\fra^{*}$ induced by the Killing form $\BB$.
\end{enumerate}

\subsubsection{The $\sl_{2}$-triple}\label{sss:sl2} Let $\{\xi_{i}\}, \{\eta_{i}\}$ be dual orthonormal bases of $\fra^{*}$ and $\fra$ with respect to $\BB$, so that $\eta_{i}^{*}=\xi_{i}$. Consider the following triple $(\ee,\hh,\ff)$ in $\Hrat$
\begin{eqnarray*}
\ee&=&-\frac{1}{2}\sum_{i}\xi_{i}^{2}=-\frac{1}{2}\BB;\\
\hh&=&\frac{1}{2}\sum_{i}(\xi_{i}\eta_{i}+\eta_{i}\xi_{i});\\
\ff&=&\frac{1}{2}\sum_{i}\eta_{i}^{2}.
\end{eqnarray*}
Direct calculation also shows
\begin{equation}\label{hh commutator}
[\hh,\xi]=\delta \xi, [\hh,\eta]=-\delta\eta \textup{ for all }\xi\in\fra^{*}, \l\in\fra.
\end{equation}
From this one easily see that $\{\ee,\hh,\ff\}$ is an almost $\sl_{2}$-triple:
\begin{equation*}
[\hh,\ee]=2\delta\ee, \, [\hh,\ff]=-2\delta\ff, \, [\ee,\ff]=\delta\hh.
\end{equation*}
In other words, $(\ee/\delta,\hh/\delta,\ff/\delta)$ is an $\sl_{2}$-triple.  

\subsubsection{Central charge}The element $B_{\KM}$ defined in \eqref{Killing KM} can be viewed as an element in $\Hrat$ of bidegree (4,2). The $\tilW$-invariance of $B_{\KM}$ implies that it is a central element of $\Hrat$. Let
\begin{equation*}
\Hrat_{\nu}=\Hrat/(u+\nu\delta, B_{\KM}).
\end{equation*}
This is the {\em graded \Cha with central charge $\nu$}. We again use $\ep$ to denote the image of $\delta=-u/\nu$.

\subsubsection{Specializing $\ep$ to 1} Let $\Hrat_{\nu,\ep=1}$ be the specialization $\Hrat_{\nu}/(\ep-1)$. This is the rational \Cha in the literature. Now the perverse grading induces a grading on $\Hrat_{\nu,\ep=1}$ while the cohomological grading only induces a {\em cohomological filtration}: $\mathfrak{H}^{\rat,\leq n}_{\nu,\ep=1}=\Image(\oplus_{i\leq n, j\in\ZZ}\Hrat(i,j)\to \Hrat_{\nu,\ep=1})$. 

%Consider the following general situation. Let $M^{*}$ be a graded $\eqnu{*}{\pt}=\QQ[\ep]$-module and let $M_{\ep=t}=M\otimes_{\QQ[\ep]}\QQ[\ep]/(\ep-t)$. Note that $M_{\ep=0}$ is graded while $M_{\ep=1}$ is only filtered. In fact, let $F^{\leq i}M$ be the preimage of $M^{\leq i}_{\ep=0}$ under the specialization $M\to M_{\ep=0}$. Then define
%\begin{equation*}
%M^{\leq i}_{\ep=1}=\textup{Im}(F^{\leq i}M\to M_{\ep=1}).
%\end{equation*}
%Applying this remark to $\Hrat_{\nu}$, we get a natural filtration on $\Hrat_{\nu,\ep=1}$.

\subsection{Relation between $\Hgr$ and $\Hrat$}

\begin{prop}\label{p:Hratmod} Let $(M,F_{\leq i}M)$ be an object in a filtered $\QQ$-linear triangulated category $\calC$. Assume there is a graded action of $\Hgr$ on $M$ (i.e., a graded algebra homomorphism $\Hgr\to\End^{\bullet}_{\calC}(M)$) in such a way that
\begin{itemize}
\item $\delta,u$ and $\tilW$ preserve the filtration;
\item $\tilw-1$ sends $F_{\leq i}M$ to $F_{\leq i-1}M$ for all $\tilw\in\xcoch(\TT)_{\mu_{e}}=\ker(\tilW\to W)$;
\item $\xi$ sends $F_{\leq i}M$ to $(F_{\leq i+1}M)[2]$ for all $\xi\in\fra^{*}$;
\item $\Lambda_{\can}$ sends $F_{\leq i}M$ to $(F_{\leq i+2})M[2]$.
\end{itemize} 
Then there is a unique bigraded action of $\Hrat$ on $\Gr^{F}_{*}M$ such that
\begin{itemize}
\item The $\delta,u$ and $\WW$ actions on $\Gr^{F}_{i}M$ are induced from that of $\delta,u$ and $\WW$ viewed as elements in $\Hgr$;
\item For $\eta\in\xcoch(\TT)_{\mu_{e}}/\xcoch(\TT)_{\mu_{e},\tors}\subset\fra\subset\Hrat$, its action $\Gr^{F}_{i}M\to\Gr^{F}_{i-1}M$ is induced from $\tilw-1\in\Hgr$, where $\tilw$ is any lifting of $\eta$ to $\xcoch(\TT)_{\mu_{e}}$;
\item For $\xi\in\fra^{*}\subset\Hrat$, its effect $\Gr^{F}_{i}M\to(\Gr^{F}_{i+1}M)[2]$ is induced from $\xi\in\Hgr$;
\item The effect of $\Lambda_{\can}:\Gr^{F}_{i}M\to(\Gr^{F}_{i+2}M)[2]$ is induced from $\Lambda_{\can}\in\Hgr$;
\end{itemize}
\end{prop}

\begin{proof} We first check that the action of $\xcoch(\TT)_{\mu_{e}}/\xcoch(\TT)_{\mu_{e},\tors}$ on $\Gr^{F}_{*}M$ given above is well-defined, and extends to an action of $\Sym(\fra)$. From the assumptions we see that $\xcoch(\TT)_{\mu_{e}}$ acts on $M$ unipotently, therefore its torsion part acts trivially. From this we see that the action of $\tilw-1$ for $\tilw\in\xcoch(\TT)_{\mu_{e}}$ only depends on its image in $\xcoch(\TT)_{\mu_{e}}/\xcoch(\TT)_{\mu_{e},\tors}\subset\fra$. Moreover, for $\tilw_{1},\tilw_{2}\in \xcoch(\TT)_{\mu_{e}}$, we have $\tilw_{1}\tilw_{2}-1=(\tilw_{1}-1)(\tilw_{2}-1)+(\tilw_{1}-1)+(\tilw_{2}-1)$. Passing to associated graded, the induced map of $\tilw_{1}\tilw_{2}-1$ on $\Gr^{F}_{i}M\to \Gr^{F}_{i-1}M$ is the same as the one induced by $(\tilw_{1}-1)+(\tilw_{2}-1)$. Therefore the map $\xcoch(\TT)_{\mu_{e}}/\xcoch(\TT)_{\mu_{e},\tors}\to\Hom(\Gr^{F}_{i}M, \Gr^{F}_{i-1}M)$ is additive. We may extend this map to $\fra$ by linearity and get an action of $\Sym(\fra)$ on $\Gr^{F}_{*}M$.

With this checked, the action of $\Hrat$ is uniquely determined on the generators.  We need to check that the relations (RC-1)-(RC-5) hold. The relations (RC-1), (RC-2) and the first half of (RC-3) are obvious. 

Let $\tilw=s_{i_1}s_{i_2}\cdots s_{i_m}\omega$ be a reduced word for an element $\tilw\in\tilW$, where $s_{i_j}$ are affine simple reflections and $\omega\in\Omega_{\bI}$. By the relation (GC-3) for $\Hgr$, we have
\begin{eqnarray*}
\tilw\xi-\leftexp{\tilw}{\xi}\tilw&=&(s_{i_1}\cdots s_{i_m}\leftexp{\omega}{\xi}-\leftexp{s_{i_1}\cdots s_{i_m}\omega}{\xi})\omega u\\
&=&\sum_{j=1}^ms_{i_1}\cdots s_{i_{j-1}}\jiao{\alpha_{i_j}^\vee,\leftexp{s_{i_{j+1}}\cdots s_{i_m}\omega}{\xi}}s_{i_{j+1}}\cdots s_{i_m}\omega u\\
&=&\sum_{j=1}^{m}\jiao{\beta_j^\vee,\xi}s_{i_1}\cdots s_{i_{j-1}}s_{i_{j+1}}\cdots s_{i_m}\omega u. 
\end{eqnarray*}
where $\beta_j^\vee=\leftexp{\omega^{-1}s_{i_m}\cdots s_{i_{j+1}}}{\alpha_{i_j}^\vee}\in\Phi_{\aff}^\vee$ is an affine real coroot. 

At this point we can check the second half of (RC-3). For $\tilw=w\in W$, the right side above belongs to $\QQ[W]u$ as a map $M\to M[2]$. Passing to the associated graded we get the second half of (RC-3). 

Next we check (RC-4). For this we assume $\tilw\in\xcoch(\TT)_{\mu_{e}}$ and its image in $\fra$ is $\eta$. Let $\beta_j=\leftexp{\omega^{-1}s_{i_m}\cdots s_{i_{j+1}}}{\alpha_{i_j}}\in\Phi_{\aff}$. Then the reflection $r_{\beta_j}\in W_{\aff}$ has the form
\begin{eqnarray*}
r_{\beta_j}&=&(\omega^{-1}s_{i_m}\cdots s_{i_{j+1}})s_{i_j}(s_{i_{j+1}}\cdots s_{i_m}\omega)\\
&=&(\tilw^{-1}s_{i_1}\cdots s_{i_j})s_{i_j}(s_{i_{j+1}}\cdots s_{i_m}\omega)\\
&=&\tilw^{-1}s_{i_1}\cdots s_{i_{j-1}}s_{i_{j+1}}\cdots s_{i_m}\omega.
\end{eqnarray*}
Moreover, since $s_{i_1}s_{i_2}\cdots s_{i_m}\omega$ is reduced, $\{\beta_1,\cdots,\beta_m\}$ is the set of positive real affine roots in $\Phi_{\aff}$ such that $\leftexp{\tilw}{\beta}<0$. Therefore
\begin{equation}\label{eq:sumbeta}
\tilw\xi-\leftexp{\tilw}{\xi}\tilw=\tilw\sum_{\beta>0,\leftexp{\tilw}{\beta}<0}\jiao{\xi,\beta^\vee}r_{\beta}u.
\end{equation}
Let $\pi:\QQ[\tilW]\to\QQ[W]$ be induced from the projection $\tilW\to W$. Eventually, we only care about the effect of $\tilw\sum_{\beta>0,\leftexp{\tilw}{\beta}<0}\jiao{\beta^\vee,\xi}r_{\beta}$ on $\Gr^{F}_{*}M$, which only depends on its image under $\pi$. For $\xi\in\fra$, we have $\leftexp{\tilw}{\xi}=\xi+\jiao{\xi,\eta}\delta$ by Lemma \ref{l:tr action}. Therefore \eqref{eq:sumbeta} implies
\begin{eqnarray}\label{comm pre}
\pi([\tilw-1,\xi])&=&\pi((\leftexp{\tilw}{\xi}-\xi)\tilw)+\pi(\tilw\xi-\leftexp{\tilw}{\xi}\tilw)\\
\notag&=&\jiao{\xi,\eta}\delta+\sum_{\beta>0,\leftexp{\tilw}{\beta}<0}\jiao{\xi,\barbeta^\vee}r_{\barbeta}u.\end{eqnarray}
Here we have written every affine real root $\beta$ uniquely as $\beta=\barbeta+n\delta$ where $\barbeta\in\Phi$. The possible $n$'s are:
\begin{itemize}
\item if $\barbeta$ is not a longest root in $\Phi$, then $n\in\frac{1}{e}\ZZ$;
\item if $G$ is not of type $\SUodd$ and $\barbeta$ is a long root, then $n\in\ZZ$;
\item if $G$ is of type $\SUodd$ and $\barbeta$ is a longest root (twice a short root), then $n\in \frac{1}{2}+\ZZ$.
\end{itemize}
Such $n$ are called {\em $\barbeta$-admissible}.

Fix $\barbeta\in\Phi$, let
\begin{equation*}
\Theta(\barbeta):=\{\textup{affine real roots }\beta=\barbeta+n\delta\textup{ such that }\beta>0,\leftexp{\tilw}{\beta}<0\}.
\end{equation*}
Thus the summation on the right side of \eqref{comm pre} becomes
\begin{equation}\label{sum bar}
\sum_{\barbeta\in\Phi}\jiao{\xi,\barbeta^{\vee}}\#\Theta(\barbeta)r_{\barbeta}u
\end{equation}
If $\barbeta>0$, then $\Theta(\barbeta)$ is in bijection with $\barbeta$-admissible $n$ such  that $0\leq n<\jiao{\barbeta,\eta}$. If $\jiao{\barbeta,\eta}>0$, there are $c_{\barbeta}\jiao{\barbeta,\eta}$ such admissible $n$'s (note we always have $\jiao{\barbeta,\eta}\in\frac{1}{c_{\barbeta}}\ZZ$), hence $\#\Theta(\barbeta)=c_{\barbeta}\jiao{\barbeta,\eta}$. If $\jiao{\barbeta, \eta}\leq0$ then $\Theta(\barbeta)=\varnothing$. Similarly analysis applied to $\Theta(-\barbeta)$ shows that if $\jiao{\barbeta,\eta}>0$ then $\Theta(-\barbeta)=\varnothing$; if $\jiao{\barbeta,\eta}\leq0$, $\#\Theta(-\barbeta)=c_{-\barbeta}\jiao{-\barbeta,\eta}$. In summary, exactly one of the pair roots $\alpha\in\{\pm\barbeta\}$ contribute to the sum \eqref{sum bar}, and $\Theta(\alpha)$ is always equal to $c_{\alpha}\jiao{\alpha,\eta}$. Hence \eqref{sum bar} becomes
\begin{equation*}
\frac{1}{2}\sum_{\alpha\in\Phi}c_{\alpha}\jiao{\xi,\alpha^{\vee}}\jiao{\alpha,\eta}r_{\alpha}u.
\end{equation*}
Plugging into \eqref{comm pre} we get (RC-4).

%the action of $\lambda$ on the affine root $\alpha+n\delta$ takes the form
%\begin{equation*}
%\leftexp{\lambda}{\alpha+n\delta}=\alpha+(n+\jiao{\alpha,\lambda})\delta. 
%\end{equation*}
%Therefore, the set of real affine roots $\beta>0$ such that $\leftexp{\tilw}{\beta}<0$ can be written as the disjoint union of two sets
%\begin{eqnarray*}
%\{\alpha+n\delta|\alpha\in\Phi^{+},0\leq n\leq-\jiao{\alpha,\lambda}-1\};\\
%\{\alpha+b\delta|\alpha\in\Phi^{-},1\leq n\leq-\jiao{\alpha,\lambda}\}.
%\end{eqnarray*}
%For each pair of roots $\{\pm\alpha\}\subset\Phi$, the contribution of the affine roots $\beta=\pm\alpha+\ZZ\delta$ to the sum on the RHS of (\ref{eq:sumbeta}) has image $-\jiao{\xi,\alpha^\vee}\jiao{\alpha,\lambda}r_{\alpha}u\in\QQ[W]u$ under $\pi$. Therefore we have
%\begin{equation*}
%\pi(\lambda\xi-\leftexp{\lambda}{\xi}\lambda)=-\dfrac{1}{2}\sum_{\alpha\in\Phi}\jiao{\xi,\alpha^\vee}\jiao{\alpha,\lambda}r_{\alpha}u\in\QQ[W]u.
%\end{equation*}
%This shows that the action of $[\lambda-1,\xi]$ on $\Gr^{F}_{*}M$ is obeys the relation (RC-4).

Finally we check (RC-5). Again let $\tilw\in\xcoch(\TT)_{\mu_{e}}$ with image $\eta\in\fra$. For $\xi=\Lambda_{\can}$, we have $\leftexp{\tilw}{\Lambda_{\can}}=\Lambda_{\can}-\eta^{*}-\frac{1}{2}\BB(\eta,\eta)\delta$. Therefore,
\begin{eqnarray*}
[\tilw-1,\Lambda_{\can}]&=&(\tilw\Lambda_{\can}-\leftexp{\tilw}{\Lambda_{\can}}\tilw)+(\leftexp{\tilw}{\Lambda_{\can}}-\Lambda_{\can})\tilw\\
&\in&\QQ[\tilW]u+(-\eta^{*}-\frac{1}{2}\BB(\eta,\eta)\delta)(\tilw-1)-\eta^{*}-\frac{1}{2}\BB(\eta,\eta)\delta.
\end{eqnarray*}
Except for the term $-\eta^{*}$, the other terms on the right side above send $F_{\leq i}M$ to $F_{\leq i}M[2]$. Therefore the induced map $[\tilw-1,\Lambda_{\can}]:\Gr^{F}_{i}M\to\Gr^{F}_{i+1}M[2]$ is the same as $-\eta^{*}$, which verifies (RC-5). 
\end{proof}

\subsection{Algebraic representation theory of rational Cherednik algebras}
%In this section we explain how the geometric constructions from the previous sections are related to the algebraic approach to the representation theory. In particular, we  discuss in what generality the results of the previous section hold in the case of $G$ quasi-split and we derive a formula for the dimension of $\frL_{\nu}(\triv)^{W_\mathbf{P}}$.

The theory of category $\calO$ for rational \Chas was developed in \cite{GGOR}. The category $\calO$ consists of $\Hrat_{\nu,\epsilon=1}$ modules with locally nilpotent action of
of $\Sym(\fra)$. The projective generators of the category are constructed as follows. For any irreducible representation $\tau\in \Irr(W)$, view it as a module over 
$\Sym(\fra)\rtimes W$ on which $\fra$ acts by zero. Define $\frM_\nu(\tau)$ to be the induction $\Ind^{\Hrat_{\nu,\epsilon=1}}_{\Sym(\fra)\rtimes W}\tau$. In particular $\frM_\nu(\triv)$ can be identified with the polynomial ring
$\Sym(\fra^*)$ as  vector spaces. The quotient of $\frM_\nu(\tau)$ by the sum of all proper submodules is   an irreducible representation $\frL_{\nu}(\tau)$, see \cite[Corollary 11.5]{ELec}. The kernel of the quotient map can be identified with the kernel of a natural pairing. Indeed, consider the anti-involution $\frF$ on $\Hrat_{\nu,\epsilon=1}$ determined by 
$$ \frF(x_i)=y_i,\quad \frF(y_i)=x_i,\quad \frF(w)=w^{-1},\quad w\in W.$$
A pairing $J: \frM\times \frM'\to\QQ$ is called {\em contravariant} if is satisfies
\begin{equation}\label{contra}
J_\tau(v h,u)=J_\tau(v,\frF(h)u), \textup{ for all } v\in \frM, u\in \frM' \textup{ and } h\in \Hrat_{\nu,\ep=1}.
\end{equation}

\begin{theorem}[{\cite[Lemma 11.6]{ELec}}] There is a unique bilinear contravariant pairing $J_\tau: \frM_\nu(\tau^*)\times \frM_\nu(\tau)\to \QQ$ that extends the natural pairing between $\tau$ and $\tau^*$. Moreover $\frL_{\nu}(\tau)=\frM_\nu(\tau)/\ker(J_\tau)$.
\end{theorem}
%Indeed, the module $\frM_\nu(\tau)$ contains a copy of $\tau$ that is annihilated by $Sym(\fra)$ and by treating $M_\nu(\tau)$ as left and $M_\nu(\tau^*)$ $\Hrat_{\nu,\epsilon=1}$-module

Let $V[j]$ be an $\hh$-eigenspace corresponding eigenvalue $j$ where  $\hh$ is from the $\mathfrak{sl}_{2}$-triple $\{\ee, \hh, \ff\}$. If $V$ is a finite-dimensional then its graded dimension $\dim_q(V):=\Tr(q^{\hh}|V)$ is a Laurent polynomial in $q$. From the representation theory of $\mathfrak{sl}_{2}$ we see that $\dim_q(V)$ is palindromic: $\dim_q(V)=\dim_q(V)|_{q=1/q}$. In the case $\frL_{\nu}(\tau)$ is a finite dimensional representation, one can compute 
the action of $\hh$ on the lowest weight space and conclude that $\deg_q(\frL_{\nu}(\tau))=\frac{\nu}{2}(\sum_{\alpha\in \Phi} c_\alpha r_\alpha)|_{\tau}-\frac{1}{2}r$, see \cite[\S3.7]{EDAHA}. In particular, if $\frL_{\nu}(\triv)$ is finite-dimensional, then its graded dimension has the highest $q$-degree among all finite-dimensional irreducible representations of $\Hrat_{\nu,\ep=1}$. 

\begin{cor}[{\cite[Proposition 3.37]{EDAHA}}]\label{quot} Suppose $V$ is a finite-dimensional quotient of $\frM_\nu(\tau)$ then in the Grothendieck group  of the category $\calO$ we have a decomposition
$$ [V]=[\frL_{\nu}(\tau)]+\sum_{\sigma, \deg_q(\frL_{\nu}(\sigma))<D} m_{\tau,\sigma} [\frL_{\nu}(\sigma)],$$
where $m_{\tau,\sigma}\geq0$ and $D=\deg_q(\frL_{\nu}(\tau))$.
\end{cor}

For an example of the situation when $m_{\tau,\sigma}\ne 0$ see  \cite[\S6.4]{BEG}.

\subsubsection{Frobenius algebra structure} The module $\frL_{\nu}(\triv)$ is quotient of the polynomial ring $\frM_{\nu}(\triv)\cong\Sym(\fra^{*})$ by a graded ideal, hence has a graded commutative ring structure. Suppose $\frL_{\nu}(\triv)$ is finite dimensional, then $N:=\deg_{q}(\frL_{\nu}(\triv))=\frac{\nu}{2}(\sum_{\alpha\in \Phi} c_\alpha)-\frac{1}{2}r=\frac{r}{2}(\nu h_{\theta}-1)$, and the $\hh$-eigenspace with the largest eigenvalue $N$ is one dimensional. Let $\ell: \frL_{\nu}(\triv)\to\frL_{\nu}(\triv)_{N}$ be the projection onto the top eigenspace of $\hh$. Choosing a basis of $\frL_{\nu}(\triv)_{N}$, we may view $\ell$ as a linear functional on $\frL_{\nu}(\triv)$.  By \cite[Theorem 11.12]{ELec} and \cite[Proposition 1.20]{BEG}, the functional $\ell$ provides a Frobenius algebra structure for this ring $\frL_{\nu}(\triv)$. Conversely, a finite dimensional quotient of $\frM_\nu(\triv)$ is a Frobenius algebra if and only if it is $\frL_{\nu}(\triv)$. In \S\ref{Frobenius} we will provide a geometric interpretation of the Frobenius algebra structure on $\frL_{\nu}(\triv)$.

%For type $B$ the number of equivalence classes of finite dimensional irreducible representations of $\Hrat_{\nu,\epsilon=1}$
%is given by the Etingof formula (see conjecture 6.8 in \cite{EtNumIrr}) which was proven recently by Shan and 
%Vasserot \cite{SV}. One can check that in general the algebra has many finite dimensional irreducible representations
%and we need compare total dimension
%of $\cohog{*}{Sp_\gamma}_{st}$ and the dimension of the representation for checking 
%conjecture~\ref{EquivSurj} since the conjecture~\ref{conj:surj} is equivalent 

The representations $\frL_{\nu}(\triv)$ are called spherical. In the case of Cherednik algebras with equal parameters (i.e., $\theta=1$), the classification of the finite-dimensional irreducible spherical 
representations is completed in \cite{VV}. In particular, it is shown in \cite{VV} that $\frL_{\nu}(\triv)$ is finite-dimensional if
and only if $\nu=d/m$, $(d,m)=1$ and $m$ is a regular elliptic number for $\WW$. See also \cite{EtPoly} for the treatment of the case of unequal parameters.

%%%%%%%%%%%%%%%%%%%

\part{Geometry}

\section{Homogeneous affine Springer fibers}\label{s:Spr}
In this section, we study the geometric and homological properties of homogeneous affine Springer fibers for the quasi-split group $G$.

\subsection{The affine flag variety} In this subsection we review some basic facts about the affine flag varieties. The choice of the Borel $\BB\subset \GG$ gives an Iwahori $\bI\subset G(F)$. A parahoric subgroup $\bP\subset G(F)$ is called standard if it contains $\bI$. Let $\bP$ be the a standard parahoric subgroup of $G(F)$. The affine partial flag variety of $G(F)$ of type $\bP$ is the fppf quotient $\Fl_{\bP}=G(F)/\bP$, where we view $G(F)$ as an ind-scheme over $\CC$ and $\bP$ as a group scheme of infinite type over $\CC$,  see \S\ref{sss:loop group}.  When $\bP=\bI$, $\Fl_{\bI}$ is the affine flag variety, and is simply denoted by $\Fl$.

\subsubsection{Line bundles on $\Fl$}\label{sss:line bundle}
Alternatively, we may write $\Fl$ as $G_{\KM}/\bI_{\KM}$, where $\bI_{\KM}=\Gcen\times\bI\rtimes\Grote$.  There is a surjection $\bI_{\KM}\surj \AA_{\KM}$, realizing $\AA_{\KM}$ as the reductive quotient of $\bI_{\KM}$. For each $\xi\in\xch(\AA_{\KM})$, viewed as a homomorphism $\bI_{\KM}\surj \AA_{\KM}\xrightarrow{\xi}\Gm$, we get a $G_{\KM}$-equivariant (via left action on $\Fl$) line bundle $\calL(\xi)=G_{\KM}\twtimes{\bI_{\KM},\xi}\AA^{1}$ over $\Fl$. This construction defines a homomorphism
\begin{equation}\label{define chern}
\xch(\AA_{\KM})\to\Pic_{G_{\KM}}(\Fl)
\end{equation}
sending $\xi$ to $\calL(\xi)$. 

\begin{remark}\label{r:L delta} When $\xi=\delta\in\fra_{\KM}^{*}$, the $\AA_{\KM}$-equivariant line bundle $\calL(\delta)$ over $\Fl$ is the $e\nth$ power of the pull back of the tautological line bundle via $[\AA_{\KM}\backslash \Fl]\to [\Grote\backslash\pt]$.
\end{remark}

\subsubsection{Determinant line bundles}\label{sss:det} We also have a class of line bundles on $\Fl$ given by the determinant construction. Each $\Fl_{\bP}$ carries a determinant line bundle $\det_{\bP}$ whose fiber at $g\bP$ is relative determinant $\det(\Ad(g)\Lie\bP:\Lie\bP)$ of the $\calO_{F}$-lattices $\Ad(g)\Lie\bP$ and $\Lie\bP$ in $\frg(F)$. The line bundle $\det_{\bP}$
carries a $\bP\rtimes\Grote$-equivariant structure (which acts on $\Fl_{\bP}$ by left translation and loop rotation). In fact, for $h\in\bP\rtimes\Grote$ and $g\bP\in\Fl_{\bP}$, the fiber of $\det_{\bP}$ at $hg\bP$ is $\det(\Ad(hg)\cdot\Lie\bP:\Lie\bP)=\det(\Ad(hg)\Lie\bP:\Ad(h)\Lie\bP)$ (because $\Ad(h)\Lie\bP=\Lie\bP$), the latter can be identified with $\det(\Ad(g)\Lie\bP:\Lie\bP)$ by the adjoint action of $h$. Let $\pi_{\bP}:\Fl\to \Fl_{\bP}$ be the projection. Since all $\bP$ contain $\bI$, the line bundles $\pi^{*}_{\bP}\det_{\bP}$ on $\Fl$ all carry $\bI\rtimes\Grote$-equivariant structures.

The line bundle $\det_{\bG}$ on $\Fl_{\bG}$ is what we used to define the central extension  $G(F)^{\cen}$. By construction, we have a $\bG\rtimes\Grote$-equivariant isomorphism of lines bundles on $\Fl$
\begin{equation*}
\pi_{\bG}^{*}{\det}_{\bG}\cong\calL(\L_{\can}).
\end{equation*}

For a parahoric $\bP$, $\pi_{\bP}^{*}\det_{\bP}\otimes\det^{-1}_{\bI}$ has fiber at $g\bI$ equal to $\det(\Ad(g)(\Lie\bP/\Lie\bI))\otimes\det(\Lie\bP/\Lie\bI))^{-1}$. Let $2\rho_{\bP}\in\fra^{*}_{\KM}$ be the sum of positive roots of $L_{\bP}$, then $\bI\rtimes\Grote$-equivariantly, we have
\begin{equation}\label{det eqn}
\pi_{\bP}^{*}{\det}_{\bP}\cong{\det}_{\bI}\otimes\calL(-2\rho_{\bP})\otimes\calO_{\bI\rtimes\Grote}(2\rho_{\bP}),
\end{equation}
where $\calO_{\bI\times\Grote}(2\rho_{\bP})$ means the trivial line bundle on which $\bI\rtimes\Grote$ acts via the character $2\rho_{\bP}$. Comparing \eqref{det eqn} for $\bP$ and for $\bG$ we get
\begin{equation}\label{comp GP}
\pi_{\bG}^{*}{\det}_{\bG}=\pi_{\bP}^{*}{\det}_{\bP}\otimes\calL(-2\rho_{\bG}+2\rho_{\bP})\otimes\calO_{\bI\rtimes\Grote}(-2\rho_{\bP}+2\rho_{\bG})
\end{equation}

\subsubsection{Curves on $\Fl$}
For each affine simple root $\alpha_{i}$, $i=0,\cdots,r$, we have the corresponding homomorphism $\SL_{2}\to G(F)$ (view $G(F)$ as the loop group) and hence an embedding $\PP^{1}\incl  \Fl$. The cycle class of this $\PP^{1}$ defines an element $C_{i}\in\homog{2}{\Fl,\ZZ_{\ell}}(-1)$. We thus get a homomorphism from the affine coroot lattice
\begin{equation}\label{define curve}
\ZZ\Phi^{\vee}_{\aff}=\Span_{\ZZ}\{\alpha^{\vee}_{0},\cdots,\alpha^{\vee}_{r}\}\to \homog{2}{\Fl,\ZZ}.
\end{equation}

\begin{prop}\label{p:pairing degree} The homomorphisms \eqref{define curve} and \eqref{define chern} preserve the pairing between $\xcoch(\AA_{\KM})$ and $\xch(\AA_{\KM})$ and the degree pairing between line bundles and curve classes up to a sign. In other words, for any $\xi\in\xch(\AA_{\KM})$ and $i=0,1,\cdots,r$, we have
\begin{equation}\label{pairing degree}
\jiao{\xi,\alpha_{i}^{\vee}}=-\deg(\calL(\xi)|C_{i}).
\end{equation} 
\end{prop}
\begin{proof}  For $\xi=\delta$, both sides of \eqref{pairing degree} are zero (since $\calL(\delta)$ is a trivial line bundle). For $\xi\in\xch(\AA)$ and $\jiao{\xi,\alpha^{\vee}_{i}}=0$, the character $\xi:\AA\to\Gm$ extends to a character $L_{\bP_{i}}\to\Gm$, therefore $\calL(\xi)$ is the pullback of a line bundle on $\Fl_{\bP_{i}}$, and hence $\deg(\calL(\xi)|C_{i})=0$ because $C_{i}$ is a fiber of the projection $\Fl\to\Fl_{\bP_{i}}$. If $\xi=\alpha_{i}$, then $\jiao{\xi,\alpha^{\vee}_{i}}=2$ and $\deg(\calL(\alpha_{i})|C_{i})=-2$ by an easy calculation on the flag variety of the $\SL_{2}$ associated with the root $\alpha_{i}$. The above discussion shows that \eqref{pairing degree} holds whenever $\xi\in\xch(\AA)\oplus\xch(\Grote)$. 

It remains to show that \eqref{pairing degree} holds for $\xi=\L_{\can}$. Applying \eqref{comp GP} to $\bP=\bP_{i}$, we have 
\begin{eqnarray}\label{compute degree}
&&\deg(\pi_{\bG}^{*}{\det}_{\bG}|C_{i})=\deg(\pi_{\bP_{i}}^{*}{\det}_{\bP_{i}}|C_{i})+\deg(\calL(\alpha_{i}-2\rho_{\bG})|C_{i})\\
\notag
&=&-\jiao{\alpha_{i}-2\rho_{\bG},\alpha_{i}^{\vee}}=-2+2\jiao{\rho_{\bG},\alpha^{\vee}_{i}}.
\end{eqnarray}
Here we have used that $\deg(\pi_{\bP_{i}}^{*}{\det}_{\bP_{i}}|C_{i})=0$ because $C_{i}$ is a fiber of $\pi_{\bP_{i}}$, and we also use the proven identity \eqref{pairing degree} in the case $\xi=\alpha_{i}-2\rho_{\bG}\in\xch(\AA)\oplus\xch(\Grote)$. When $i\neq0$, $\jiao{\rho_{\bG},\alpha^{\vee}_{i}}=1$ hence the right side of \eqref{compute degree} is zero, which is also equal to $\jiao{\L_{\can},\alpha^{\vee}_{i}}$. For $i=0$, we have $\jiao{\rho_{\bG},\alpha^{\vee}_{0}}=-\jiao{\rho_{\bG},\beta^{\vee}}=-(h^{\vee}_{\theta}-a_{0}^{\vee})/a_{0}^{\vee}=1-h^{\vee}_{\theta}/a_{0}^{\vee}$ by the definition of $a_{0}^{\vee}$ and $h^{\vee}_{\theta}$ (see \eqref{dual label} and \eqref{dual Cox}). Therefore the right side of \eqref{compute degree} is $-2h^{\vee}_{\theta}/a_{0}^{\vee}$. On the other hand, the definition of $\alpha^{\vee}_{0}$ in \eqref{define alpha0} makes sure that $\jiao{\L_{\can},\alpha^{\vee}_{0}}=2h^{\vee}_{\theta}/a_{0}^{\vee}$, therefore \eqref{pairing degree} also holds for $\xi=\L_{\can}$ and $i=0$. The proof is now complete.
\end{proof}

By the construction in \S\ref{sss:line bundle}, each line bundle $\calL(\xi)$ carries a canonical $\AA_{\KM}$-equivariant structure. Let $r(\xi)\in\upH^{2}_{\AA_{\KM}}(\Fl)$ denote the  $\AA_{\KM}$-equivariant Chern class of $\calL(\xi)$. This extends to a graded ring homomorphism
\begin{equation*}
r: \Sym^{\bullet}(\fra^{*}_{\KM})\to \upH^{2\bullet}_{\AA_{\KM}}(\Fl).
\end{equation*}
On the other hand, we have another such ring homomorphism coming from the $\AA_{\KM}$-equivariant cohomology of a point
\begin{equation*}
\ell: \Sym^{\bullet}(\fra^{*}_{\KM})\cong \upH^{2\bullet}_{\AA_{\KM}}(\pt)\to \upH^{2\bullet}_{\AA_{\KM}}(\Fl).
\end{equation*}

The next result will be used in construction of the graded \Cha action on the cohomology of affine Springer fibers in \S\ref{ss:localHgr}.

\begin{lemma}\label{l:B}  We have $r(B_{\KM})=\ell(B_{\KM})$ in $\upH^{4}_{\AA_{\KM}}(\Fl)$.
\end{lemma}
\begin{proof} 
By equivariant formality of $\Fl$, we can check this identity by restricting to the $\AA_{\KM}$-fixed points of $\Fl$, namely the points $\tilw\bI$ for $\tilw\in\tilW$. We have $r(\xi)|_{\tilw\bI}=\ell(\tilw\xi)\in\upH^{2}_{\AA_{\KM}}(\pt)$. Since $B_{\KM}$ is $\tilW$-invariant, we have $r(B_{\KM})|_{\tilw\bI}=\ell(\tilw B_{\KM})=\ell(B_{\KM})\in\upH^{4}_{\AA_{\KM}}(\pt)$. This means $r(B_{\KM})-\ell(B_{\KM})$ restricts to zero at every fixed point, hence $r(B_{\KM})=\ell(B_{\KM})$ in $\upH^{4}_{\AA_{\KM}}(\Fl)$.  
\end{proof}

\subsection{Affine Springer fibers} Kazhdan and Lusztig \cite{KL} defined the notion of the affine Springer fiber in the affine flag variety for split groups. Many basic results on affine Springer fibers can be easily generalized to quasi-split groups $G$. 

\begin{defn}[Kazhdan-Lusztig \cite{KL}] Let  $\gamma\in \frg(F)$ be regular semisimple. The {\em affine Springer fiber} of $\gamma$ of type $\bP$ is the reduced closed subvariety $\Sp_{\bP,\gamma}\subset\Fl_{\bP}$ consisting of cosets $g\bP$ such that $\Ad(g^{-1})\gamma\in\Lie\bP$.  When $\bP=\bI$, we drop $\bI$ from the notion and simply write $\Sp_\gamma$.
\end{defn}

We have a natural projection $\pi_{\bP,\gamma}: \Sp_{\gamma}\to\Sp_{\bP,\gamma}$. The fiber of $\pi_{\bP,\gamma}$ at $g\bP\in \Sp_{\bP,\gamma}$ is the usual Springer fiber of the image of $\Ad(g^{-1})\gamma\in \Lie\bP$ in $\frl_{\bP}=\Lie L_{\bP}$ for the group $L_{\bP}$. In particular, $\pi_{\bP,\gamma}$ is surjective.

When $G$ is split, it is shown in \cite{KL} that $\Sp_{\gamma}$ is a possibly infinite union of projective varieties of bounded dimension. Moreover, $\Sp_{\gamma}$ is of finite type if $\gamma$ is elliptic. The next lemma generalizes this to the quasi-split case and to general parahorics $\bP$. 

\subsubsection{The centralizer of $\gamma$} Let $G_{\gamma}$ be the centralizer of $\gamma$ in $G$. This is a torus over $F$ and $G_{\gamma}(F)$ acts on $\Sp_{\bP,\gamma}$ by its left translation action on $\Fl_{\bP}$. Let $A_{\gamma}\subset G_{\gamma}$ be the maximal subtorus that is split over $F$. Using the uniformizer $t$ of $F$, we have an embedding $\xcoch(A_{\gamma})\incl A_{\gamma}(F)\incl G_{\gamma}(F)$. We denote the image of this embedding by $\L_{\gamma}$. Then $\L_{\gamma}$ is a free abelian group of finite rank (equal to the split rank of $G_{\gamma}$) that acts on $\Sp_{\bP,\gamma}$.

\begin{lemma}\label{l:finite type} Let  $\gamma\in \frg(F)$ be regular semisimple, and let $\bP$ be a standard parahoric subgroup of $G(F)$. 
\begin{enumerate}
\item There exists a projective subscheme $Z\subset \Sp_{\bP,\gamma}$ such that $\Sp_{\bP,\gamma}=\L_{\gamma}\cdot Z$. In particular, $\Sp_{\bP,\gamma}$ is a possibly infinite union of projective schemes of bounded dimension.
\item If $\gamma\in\frg^\rs(F)$ is elliptic, then $\Sp_{\bP,\gamma}$ is a projective scheme.
\end{enumerate}
\end{lemma}
\begin{proof} We note that (2) follows from (1) since when $\gamma$ is elliptic, $A_{\gamma}$ is trivial hence $\L_{\gamma}=0$. 

To prove (1), it suffices to consider the case $\bP=\bI$ since the projection $\pi_{\bP,\gamma}: \Sp_{\gamma}\to\Sp_{\bP,\gamma}$ is surjective. Let $\Sp'_{\gamma}$ be the affine Springer fiber defined using $F_{e}$ in place of $F$. There is a natural $\mu_{e}$-action on $\Sp'_{\gamma}$, and $\Sp_{\gamma}$ is a closed subset of $\Sp'^{\mu_{e}}_{\gamma}$.  We use $G'_{\gamma}, A'_{\gamma}$ and $\L'_{\gamma}$ to denote the counterparts of $G_{\gamma}, A_{\gamma}$ and $\L_{\gamma}$ when the field $F$ is extended to $F_{e}$. Then $\mu_{e}$ naturally acts on $\L'_{\gamma}$, and we have a natural embedding $\L_{\gamma}\subset(\L'_{\gamma})^{\mu_{e}}$ with finite index.

We shall apply Kazhdan and Lusztig's results from \cite{KL} to $\Sp'_{\gamma}$ since $G$ is split over $F_{e}$.  By Kazhdan and Lusztig \cite[Prop 2.1]{KL}, there exists a projective subscheme $Z'\subset \Sp'_{\gamma}$ such that $\Sp'_{\gamma}=\L'_{\gamma}\cdot Z'$. By enlarging $Z'$ we may assume that $Z'$ is stable under $\mu_e$ and is still a projective scheme. 

Let $S\subset \L'_{\gamma}$ be the set of $\l\in \L'_{\gamma}$ such that $\l Z'\cap Z'\neq\varnothing$. Since $Z'$ is of finite type, $S$ is a finite set. We claim that $\Sp'^{\mu_{e}}_{G}\subset (\sigma-\id)^{-1}(S)\cdot Z'$. In fact, let $\l\cdot z$ be a $\mu_{e}$-fixed point of $\Sp'_{\gamma}$, where $\l\in\L'_{\gamma}$ and $z\in Z'$. Fix a generator $\sigma\in\mu_{e}$. Then $\sigma(\l)\cdot\sigma(z)=\l\cdot z$ which implies that $(\sigma(\l)-\l)\cdot \sigma(z)= z$ hence $\sigma(\l)-\l\in S$, i.e., $\l\in (\sigma-\id)^{-1}(S)$.

Finally, since $\L_{\gamma}\subset (\L'_{\gamma})^{\mu_{e}}$ has finite index, the quotient $(\sigma-\id)^{-1}(S)/\L_{\gamma}$ is a finite set. Choose coset representatives $\{\l_{1},\cdots,\l_{m}\}$ of $(\sigma-\id)^{-1}(S)/\L_{\gamma}$. Let $Z=(\cup_{i=1}^{m}\l_{i}\cdot Z')\cap\Sp_{\gamma}$, which is a closed subscheme of $\Fl$ of finite type hence projective. Since $\Sp'^{\mu_{e}}_{\gamma}\subset (\sigma-\id)^{-1}(S)\cdot Z'=\L_{\gamma}\cdot(\cup_{i=1}^{m} \l_{i}\cdot Z')$. we conclude that $\Sp_{\gamma}=\L_{\gamma}\cdot Z$.
\end{proof}

\subsubsection{Symmetry on affine Springer fibers}\label{ss:sym afs} Let $\gamma\in\frg^{\rs}(F)$ with $a=\chi(\gamma)\in\frc(\calO_{F})$. Let $a=\chi(\gamma)$ be viewed as a morphism $\Spec\calO_{F}\to\frc$. We define $J_a:=a^{*}J$ where $J$ is the regular centralizer group scheme defined in Section \ref{ss:J}. There is a canonical isomorphism $j_{\gamma}:J_{a,F}\isom G_{\gamma}$ over $F$. The group scheme $J_{a}$ gives a smooth model of $G_{\gamma}$ over $\calO_{F}$. Let $P_{a}=J_{a}(F)/J_{a}(\calO_{F})$, viewed as a fppf quotient in the category of group ind-schemes over $\CC$. As an ind-scheme over $\CC$, $P_{a}$ is the moduli space of $J_a$-torsors over $\Spec \calO_{F}$ together with a trivialization over $\Spec F$.

\begin{lemma}\label{l:Pa acts}
The action of $G_{\gamma}(F)$ on $\Sp_{\bP,\gamma}$ factors through $P_{a}$.
\end{lemma}
\begin{proof} We first prove the statement for $\bP=\bI$. The proof is the same as the argument in \cite[p.42]{NgoFL}. Let us give the reason on the level of $\CC$-points. Let $j_{\gamma}: J_{a}(F)\isom G_{\gamma}(F)$ be the canonical isomorphism. To show that $J_{a}(\calO_{F})$ acts trivially on $\Sp_{\gamma}$, it suffices to show that for any $g\in G(F)$ such that $\Ad(g^{-1})\gamma\in\Lie\bI$, and any $x\in J_{a}(\calO_{F})$, we have $j_{\gamma}(x)g\bI=g\bI$, or $\Ad(g^{-1})j_{\gamma}(x)\in\bI$. Let $\delta=\Ad(g^{-1})\gamma\in\Lie\bI$. Then by Lemma \ref{l:J Iw}, the canonical isomorphism $j_{\delta}=\Ad(g^{-1})\circ j_{\gamma}:J_{a}(F)\isom G_{\delta}(F)=\Ad(g^{-1})G_{\gamma}(F)$ extends to $J_{a}\to I_{\bI,\delta}$ over $\calO_{F}$. Therefore, $j_{\delta}$ sends $J_{a}(\calO_{F})$ to $G_{\delta}(F)\cap \bI$, i.e., $j_{\delta}(x)\in\bI$, or $\Ad(g^{-1})j_{\gamma}(x)\in\bI$. 

For general $\bP$, the morphism $\Sp_{\gamma}\to\Sp_{\bP,\gamma}$ is surjective. Therefore, the fact that $J_{a}(\calO_{F})$ acts trivially on $\Sp_{\gamma}$ implies that it also acts trivially on $\Sp_{\bP,\gamma}$. 
\end{proof}

We may introduce an open subset $\Sp_{\bG,\gamma}^{\reg}$ consisting of those $g\bG\in\Sp_{\bG,\gamma}$ such that $\Ad(g)^{-1}(\gamma)\in\frg^{\reg}$. The following result can be proved in a similar way as \cite[Lemme 3.3.1]{NgoFL}, using Lemma \ref{l:J}.

\begin{lemma}\label{l:AFS reg} The open subset $\Sp_{\bG,\gamma}^{\reg}$ is a torsor under $P_{a}$. 
\end{lemma}

\subsection{Homogeneous affine Springer fibers}
Let $a\in\frc(F)^{\rs}_{\nu}$ be homogeneous of slope $\nu>0$, and let $\gamma=\kappa(a)\in\frg$. We will study the affine Springer fiber $\Sp_{\bP,\gamma}$. We denote $\Sp_{\bP,\gamma}$ by $\Sp_{\bP,a}$. We call such affine Springer fibers {\em homogeneous}. Applying the dimension formula \eqref{dim AFS general} to this situation, we have

\begin{cor}\label{c:AFS dim} 
\begin{enumerate}
\item The dimension of the homogeneous affine Springer fiber is given by
\begin{equation}\label{dim AFS}
\dim\Sp_{\bP,a}=\frac{1}{2}(\nu\#\PPhi-r+\dim\tt^{\Pi(\mu_{m})})
\end{equation}
where $\Pi:\mu_{m}\incl\WW'$ is the $\theta$-regular homomorphism with a regular eigenvalue $\barnu$ as in Theorem \ref{th:homog vs reg}. Recall $r=\rank\HH$ is the $F$-rank of $G$.

\item If $a$ is moreover elliptic, then
\begin{equation}\label{dim AFS ell}
\dim\Sp_{\bP, a}=(h_{\theta}\nu-1)r/2.
\end{equation}
Here we recall $h_{\theta}$ is the twisted Coxeter number of $(\GG,\theta)$.
\end{enumerate}
\end{cor}
\begin{proof}
(1) By Lemma \ref{l:hom in Fs}\eqref{Xtnu}, we may choose $\gamma'$ in formula \eqref{dim AFS general} to be $Xt^{\nu}$ for some $X\in\tt^{\rs}$. Therefore $\val(\alpha(\gamma'))=\nu$ for any root $\alpha\in\PPhi$. Plugging this into \eqref{dim AFS general} we get \eqref{dim AFS}. Note that the homomorphism $\Pi_{a}$ is $\WW$-conjugate to $\Pi$.

(2) By ellipticity, $\tt^{\Pi(\mu_{m})}=0$. We also have $\#\PPhi=h_{\theta}r$ in the twisted situation.
\end{proof}

\subsubsection{Families of homogeneous affine Springer fibers}\label{sss:family Sp} We may put homogeneous affine Springer fibers of a given slope in families. Recall the Kostant section gives an embedding $\kappa:\frc(F)_{\nu}\incl \frg(F)_{\nu}$, which can also be viewed as a section of the GIT quotient of the $L_{\nu}$-action on $\frg(F)_{\nu}$. Let
\begin{equation*}
\tq_{\bP, \nu}:\tSp_{\bP, \nu}\to \frg(F)^{\rs}_{\nu} 
\end{equation*}
be the family of affine Springer fibers (in $\Fl_{\bP}$) over $\frg(F)^{\rs}_{\nu}$. This morphism is clearly $L_{\nu}$-equivariant. Let
\begin{equation*}
q_{\bP, \nu}:\Sp_{\bP, \nu}\to \frc(F)^{\rs}_{\nu}
\end{equation*}
be the restriction of $\tSp_{\bP,\nu}$ via the Kostant section $\kappa$. For $a\in\frc(F)^{\rs}_{\nu}$, $\Sp_{\bP,a}$ is simply the fiber of $q_{\nu}$ at $a$.

\subsubsection{Symmetry on homogeneous affine Springer fibers}\label{sss:sym homog AFS} Recall the one-dimensional torus $\Gnu\subset\tilG^{\ad}:=(G^{\ad}(F)\rtimes\Grote)\times\Gdil$ from \S\ref{sss:Gnu}. The group $(G^{\ad}(F)\rtimes\Grote)\times\Gdil$ acts on $\Fl_{\bP}$ (where $G^{\ad}(F)$ acts by left translation). By \eqref{s act gamma}, $\Gnu$ fixes $\kappa(a)$, and therefore acts on $\Sp_{\bP,a}$. In \S\ref{sss:Pnu} we have defined a group ind-scheme $P_{\nu}$ over $\frc(F)^{\rs}_{\nu}$ whose fibers are $P_{a}$ that acts on $\Sp_{\bP,a}$ by Lemma \ref{l:Pa acts}. The torus $\Gnu$ also acts on $P_{\nu}$. Therefore we have an action of $P_{\nu}\rtimes\Gnu$ on $\Sp_{\bP,\nu}$.

\subsubsection{Fixed points of the $\Gnu$-action on the affine flag variety} 
Recall from \S\ref{sss:MP} that $\nu$ determines a parahoric subgroup $\bP_{\nu}\subset G(F)$ with Levi factor $L_{\nu}$. Let $W_{\nu}\subset\tilW$ be the Weyl group of $L_{\nu}$ with respect to $\AA$. 
 
\begin{lemma}\label{l:fl fix} The fixed points of the $\Gnu$-action on $\Fl_{\bP}$ given in \eqref{s act fl} are exactly $\bigsqcup_{\tilw}L_{\nu}\tilw\bP/\bP$, where $\tilw$ runs over representatives of all double cosets $ W_{\nu}\backslash\tilW/W_{\bP}$.
\end{lemma}
\begin{proof} The affine Bruhat decomposition expresses $\Fl_{\bP}$ as a disjoint union of $\bP_{\nu}\tilw\bP/\bP$, where $\tilw$ runs over the double coset $W_{\nu}\backslash\tilW/W_{\bP}$. Each $\bP_{\nu}\tilw\bP$ can further be written as $\bP^{+}_{\nu, \tilw}\times L_{\nu}\tilw\bP/\bP$, where $\bP^{+}_{\nu, \tilw}\subset \bP^{+}_{\nu}=\ker(\bP_{\nu}\to L_{\nu})$ is the product of finitely many affine root subgroups. By the discussion in \S\ref{sss:MP}, the torus $\Gnu$ acts on affine root subgroups in $\bP^{+}_{\nu}$ with positive weights and fixes $L_{\nu}$ pointwise. Therefore the $\Gnu$-fixed points contained in $\bP_{\nu}\tilw\bP/\bP$ are exactly $L_{\nu}\tilw\bP/\bP$.
\end{proof}

\begin{remark} In the special case when $\nu\rho^{\vee}$ lies in the interior of an alcove, $\bP_{\nu}$ is an Iwahori subgroup and $L_{\nu}=\AA$. In this case, the $\Gnu$-fixed points on $\Fl_{\bP}$ are discrete.
\end{remark}

\subsection{Hessenberg paving} 
In \cite{GKM}, when $e=1$, it is shown that for a class of elements $\gamma\in\gg^{\rs}(F)$ called {\em equivalued}, the corresponding affine Springer fiber admits a paving by affine space bundles over Hessenberg varieties that are smooth and projective, and in particular they have pure cohomology. Homogeneous elements are equivalued, therefore they admit pavings by affine space bundles over Hessenberg varieties. We shall indicate briefly how the Hessenberg paving extends to the quasi-split case for homogeneous affine Springer fibers.

\subsubsection{Hessenberg varieties}\label{sss:Hess}
Let $L$ be a reductive group over $\CC$ and $V$ be a linear representation of $L$. Let $P\subset L$ be a parabolic subgroup and $V^{+}\subset V$ be a $P$-stable subspace. For $v\in V$, we define the {\em Hessenberg variety} $\Hess_{v}(L/P,V\supset V^{+})$ to be the following subvariety of the partial flag variety $L/P$:
\begin{equation*}
\Hess_{v}(L/P, V\supset V^{+})=\{gP\in L/P: g^{-1}v\in V^{+}\}.
\end{equation*}
Observe that  $\textup{Stab}_{L}(v)$ acts on $\Hess_{v}(L/P,V\supset V^{+})$. As $v$ varies in $V$, the Hessenberg varieties form a projective family
\begin{equation*}
\Hess(L/P, V\supset V^{+})\to V.
\end{equation*}

Let $\bP$ be a standard parahoric. For each $\tilw\in\tilW$, let $\frg(F)^{\tilw}_{\bP, \nu}:=\frg(F)_{\nu}\cap\Ad(\tilw)\Lie\bP$, which is stable under the parabolic $L_{\nu}\cap\Ad(\tilw)\bP$ of $L_{\nu}$. For $\gamma\in\frg(F)_{\nu}$, we may define the Hessenberg variety
\begin{equation*}
\Hess^{\tilw}_{\bP,\gamma}:=\Hess_{\gamma}(L_{\nu}/L_{\nu}\cap\Ad(\tilw)\bP, \frg(F)_{\nu}\supset\frg(F)^{\tilw}_{\bP,\nu})
\end{equation*}
as a closed subvariety of the partial flag variety $\fl^{\tilw}_{\bP,\nu}:=L_{\nu}/L_{\nu}\cap\Ad(\tilw)\bP$ of $L_{\nu}$. It only depends on the class of $\tilw$ in $W_{\nu}\backslash \tilW/ W_{\bP}$. 
As $\gamma$ varies in $\frg(F)_{\nu}$ the Hessenberg varieties form a projective family
\begin{equation}\label{whole Hess}
\tpi^{\tilw}_{\bP,\nu}: \tHess^{\tilw}_{\bP,\nu}\to\frg(F)_{\nu} 
\end{equation}
Similarly, let
\begin{equation*}
\pi^{\tilw}_{\bP, \nu}:\Hess^{\tilw}_{\bP, \nu}\to\frc(F)_{\nu}
\end{equation*}
be the pullback of the map \eqref{whole Hess} via the Kostant section $\kappa$. 

When $\bP=\bI$, we often drop the $\bP$ from subscripts. For $\gamma\in\frg(F)_{\nu}$, $\Hess^{\tilw}_{\bP,\gamma}$ is just the fiber of $\tpi^{\tilw}_{\bP,\nu}$; for $a\in\frc(F)_{\nu}$, we shall denote the fiber of $\pi^{\tilw}_{\nu}$ over $a$ by $\Hess^{\tilw}_{\bP, a}$. In the case $\bP=\bI$, $\Hess^{\tilw}_{a}$ is a closed subvariety of the flag variety $\fl^{\tilw}_{\nu}$ of $L_{\nu}$.

\begin{theorem}[Goresky-Kottwitz-MacPherson]\label{p:paving} Let $\nu>0$ be an admissible slope.
\begin{enumerate}
\item The $\Gnu$-fixed points of $\tSp_{\bP,\nu}$ is the disjoint union
\begin{equation*}
\tSp_{\bP,\nu}^{\Gnu}=\bigsqcup_{\tilw\in W_{\nu}\backslash \tilW/ W_{\bP}}\tHess^{\tilw}_{\bP,\nu}|_{\frg(F)^{\rs}_{\nu}}.
\end{equation*}
\item For $a\in\frc(F)^{\rs}_{\nu}$, the homogeneous affine Springer fiber $\Sp_{\bP,a}$ admits a pavement by intersecting with $\bP_{\nu}$-orbits indexed by $\tilw\in W_{\nu}\backslash \tilW/ W_{\bP}$, and each intersection $(\bP_{\nu}\tilw\bP/\bP)\cap\Sp_{\bP,a}$ is an affine space bundle over $\Hess^{\tilw}_{\bP,\nu, a}$ which contracts to $\Hess^{\tilw}_{\bP,\nu, a}$ under the $\Gnu$-action.

\item The projective morphism $\tpi^{\tilw}_{\bP, \nu}: \tHess^{\tilw}_{\bP,\nu}\to\frg(F)_{\nu}$ is smooth over $\frg(F)^{\rs}_{\nu}$.

\item\label{GKMpurity} The cohomology of $\Sp_{\bP,a}$ ($a\in\frc(F)^{\rs}_{\nu}$) is pure.
\end{enumerate}
\end{theorem}
\begin{proof}
(1) and (2) follows from Lemma \ref{l:fl fix}. (3) By definition, $\tHess^{\tilw}_{\bP, \nu}$ is the zero loci of the family of sections of the vector bundle $\calV:=L_{\nu}\twtimes{L_{\nu}\cap\Ad(\tilw)\bP}\frg(F)_{\nu}/\frg(F)^{\tilw}_{\nu}$ give by the elements in $\frg_{\barnu}\cong\frg(F)_{\nu}$. For $v\in \frg^{\rs}_{\barnu}$, the corresponding section is transversal to the zero section of the vector bundle $\calV$ and therefore $\tpi^{\tilw}_{\nu}$ is proper and smooth over $\frg(F)^{\rs}_{\nu}$. For details see \cite[\S2.5]{GKM}. The smoothness of the Hessenberg varieties in question is proved as in \cite{GKM}, using the fact that $\gamma(1)$ is a ``good vector'' in $\gg_{\barnu}$ under the $L_{\nu}$-action. (4) follows from (1) and (2). 
\end{proof}

\begin{exam} Suppose $\nu\rho^{\vee}$ lies in the interior of an alcove. Then $\Sp_{\bP,\gamma}$ admits a pavement by affine spaces (by intersecting with Schubert cells in $\Fl_{\bP}$). In particular, if the denominator of $\nu$ is equal to the $\theta$-twisted Coxeter number $h_{\theta}$, then $\Sp_{\bP,\gamma}$ admits a pavement by affine spaces.
\end{exam}

Next we draw some consequences from the existence of a Hessenberg paving on homogeneous affine Springer fibers. 

\begin{cor}\label{c:eq loc Sp} Let $\nu>0$ be elliptic and $a\in\frc(F)^{\rs}_{\nu}$. Then the restriction map
\begin{equation}\label{localize Sp}
\eqnu{*}{\Sp_{\bP,a}}\to\bigoplus_{\tilw\in W_{\nu}\backslash \tilW/W_{\bP}}\cohog{*}{\Hess^{\tilw}_{\bP,a}}\otimes\QQ[\ep]
\end{equation}
is an isomorphism after inverting $\ep$ (or equivalently specializing  $\ep$ to 1). Note that only finitely many terms on the right hand side are nonzero, by the ellipticity of $\gamma$.
\end{cor}
\begin{proof}
Since $a$ is elliptic, $\Sp_{\bP,a}$ is of finite type by Lemma \ref{l:finite type}. Since the cohomology of $\Sp_{\bP, a}$ is pure by Theorem \ref{p:paving}(4), it is $\Gnu$-equivariantly formal, hence the localization theorem applies.
\end{proof}

\begin{cor}\label{c:Hess loc sys} Let $\nu>0$ be an admissible slope. 
\begin{enumerate}
\item The complexes $\tpi^{\tilw}_{\bP, \nu,*}\QQ$ and $\pi^{\tilw}_{\bP, \nu,*}\QQ$, when restricted to $\frg(F)^{\rs}_{\nu}$ and $\frc(F)^{\rs}_{\nu}$ respectively, are direct sums of shifted semisimple local systems.
\item Suppose $\nu$ is elliptic. Then the complexes $\tq_{\bP, \nu,*}\QQ$ and $q_{\bP, \nu,*}\QQ$ are direct sums of shifted semisimple local systems.
\end{enumerate}
 \end{cor}
\begin{proof} To save notation, we only prove the case $\bP=\bI$.

(1) By Theorem \ref{p:paving}(3), $\tpi^{\tilw}_{\nu,*}\QQ|_{\frg(F)^{\rs}_{\nu}}$ is a direct sum of shifted local systems. By proper base change, $\pi^{\tilw}_{\nu,*}\QQ$ is the pullback of $\tpi^{\tilw}_{\nu,*}\QQ$ to $\frc(F)_{\nu}$, hence also a direct sum of shifted local systems. The semisimplicity of these local systems follows from the decomposition theorem.

(2) We give the argument for $q_{\nu,*}\QQ$, and the argument for $\tq_{\nu,*}\QQ$ is the same. The embedding $\sqcup_{\tilw}\Hess^{\tilw}_{\nu}\incl\Sp_{\nu}$ induces a morphism $q_{\nu,*}\QQ\to \pi^{\tilw}_{\nu,*}\QQ|_{\frc(F)^{\rs}_{\nu}}$ in $D^{b}_{\Gnu}(\frc(F)^{\rs}_{\nu})$. Note that the $\Gnu$-action on $\frc(F)_{\nu}$ is trivial. Further push-forward to $D^{b}(\frc(F)^{\rs}_{\nu})$ and taking cohomology sheaves we obtain $\iota^{*}: \oplus_{i}\bR^{i}_{\Gnu}q_{\nu,*}\QQ\to \QQ[\ep]\otimes(\oplus_{i}\bR^{i}\pi^{\tilw}_{\nu,*}\QQ|_{\frc(F)^{\rs}_{\nu}})$ between sheaves in $\QQ[\ep]$-modules on $\frc(F)^{\rs}_{\nu}$. By Corollary \ref{c:eq loc Sp}, $\iota^{*}$ becomes an isomorphism after inverting $\ep$. By (1), both $\bR^{i}\pi^{\tilw}_{\nu,*}\QQ|_{\frc(F)^{\rs}_{\nu}}$ are local systems, therefore so are $\bR^{i}_{\Gnu}q_{\nu,*}\QQ$. By the purity of the fibers of $q_{\nu}$ from Theorem \ref{p:paving}(3), $\oplus_{i}\bR^{i}_{\Gnu}q_{\nu,*}\QQ$ is non-canonically isomorphic to $\QQ[\ep]\otimes(\oplus_{i}\bR^{i}q_{\nu,*}\QQ)$, therefore $\bR^{i}\tq_{\nu,*}\QQ$ is also a local system. Semisimplicity of local systems again follows from the decomposition theorem.
\end{proof}

\subsubsection{Symmetries on the cohomology of Hessenberg varieties}\label{sss:pi1} Let $a\in\frc(F)^{\rs}_{\nu}$. We consider two symmetries on $\cohog{i}{\Hess^{\tilw}_{\bP, a}}$.

First, recall that $\tilS_{a}$ (resp. $S_{a}$) is the stabilizer of $\gamma=\kappa(a)$ under $\tilL_{\nu}$ (resp. $L_{\nu}$). Since $S_{a}$ acts on $\Hess^{\tilw}_{\bP, a}$, $\pi_{0}(S_{a})$ acts on $\cohog{i}{\Hess^{\tilw}_{\bP, a}}$.

On the other hand, let 
\begin{equation*}
B_{a}:=\pi_{1}(\frc(F)_{\nu}^{\rs}, a).
\end{equation*}
Here $\pi_{1}$ means the topological fundamental group. The notation suggests that it is a braid group. In fact,  it is the braid group attached to the little Weyl group $\WW^{\Pi(\mu_m)}$ introduced in \S\ref{sss:Cartan}. By Lemma \ref{c:Hess loc sys}(1), $\cohog{i}{\Hess^{\tilw}_{\bP, a}}$ is the stalk of of a local system on $\frc(F)^{\rs}_{\nu}$. Therefore $\cohog{i}{\Hess^{\tilw}_{\bP, a}}$ also carries an action of $B_{a}$.

Now we combine the two actions. By the description of $\tilS$ given in \S\ref{sss:tilS}, the group of connected components $\pi_{0}(\tilS_{a})$ also form a local system of finite abelian groups over $\frc(F)^{\rs}_{\nu}$, which we denote by $\pi_{0}(\tilS/\frc(F)^{\rs}_{\nu})$. In particular, $\pi_{0}(\tilS_{a})$ carries an action of $B_{a}$. We form the semidirect product $\pi_{0}(\tilS_{a})\rtimes B_{a}$ using this action. Again by the description of $\tilS$ in \S\ref{sss:tilS}, we may identify $\tilS_{a}$ with $\TT^{\Pi(\mu_m)}$, and the action of $B_{a}$ on $\pi_{0}(\tilS_{a})$ factors through its quotient $B_{a}\surj \WW^{\Pi(\mu_m)}$ which acts on $\TT^{\Pi(\mu_{m})}$.  Similarly, by the description of $S$ given in \S\ref{sss:tilS},  $\pi_{0}(S/\frc(F)^{\rs}_{\nu})=\ker(\pi_{0}(\tilS/\frc(F)^{\rs}_{\nu})\to \Omega)$ is also a local system over $\frc(F)^{\rs}_{\nu}$, hence its fiber $\pi_{0}(S_{a})$ carries an action of $B_{a}$, and we have a subgroup $\pi_{0}(S_{a})\rtimes B_{a}$ of $\pi_{0}(\tilS_{a})\rtimes B_{a}$.

%Therefore it carries an action of the $\tilL_{\nu}$-equivariant fundamental group $\pi^{\tilL_{\nu}}_{1}(\frg(F)^{\rs}_{\nu},\gamma)$. This fundamental group can be thought of as the fundamental group of the quotient stack $[\frg(F)^{\rs}_{\nu}/\tilL_{\nu}]$. 

%Let $\tilS_{\gamma}$ (resp. $S_{\gamma}$) be the stabilizer of $\gamma$ under $\tilL_{\nu}$ (resp. $L_{\nu}$). 

%The projection $\chi:\frg(F)^{\rs}_{\nu}\to \frc(F)^{\rs}_{\nu}$ gives an exact sequence
%\begin{equation*}
%1\to \tilS_{\gamma}\to\pi^{\tilL_{\nu}}_{1}(\frg(F)^{\rs}_{\nu},\gamma)\to B_{\chi(\gamma)}\to1
%\end{equation*}
%If $\gamma=\kappa(a)$ lies in the Kostant section, we may use the Kostant section to split the above exact sequence, and write
%\begin{equation*}
%\pi^{\tilL_{\nu}}_{1}(\frg(F)^{\rs}_{\nu},\kappa(a))\cong \tilS_{a}\rtimes B_{a}
%\end{equation*}
%where $B_{a}$ acts on $\tilS_{a}\cong\TT^{\Pi(\mu_m)}$ through the quotient $B_{a}\surj \WW^{\Pi(\mu_m)}$ (recall the little Weyl group $\WW^{\Pi(\mu_m)}$ acts on $\tilS_{a}\cong\TT^{\Pi(\mu_{m})}$, see \S\ref{sss:tilS}). Similar remarks apply when $\tilL_{\nu}$ is replaced by $L_{\nu}$ and $\tilS_{\gamma}$ replaced by $S_{\gamma}$.

We summarize the above discussion as
\begin{lemma} Let For $a\in\frc(F)^{\rs}_{\nu}$.
\begin{enumerate}
\item For each $\tilw\in W_{\nu}\backslash \tilW/W_{\bP}$, there is a natural action of $\pi_{0}(S_{a})\rtimes B_{a}$ on the cohomology $\cohog{*}{\Hess^{\tilw}_{\bP, a}}$. 
\item If $\nu>0$ is elliptic, then $\pi_{0}(\tilS_{a})\rtimes B_{a}=\tilS_{a}\rtimes B_{a}$ also naturally acts on $\spcoh{}{\Sp_{\bP, a}}$.
\end{enumerate}
\end{lemma}
\begin{proof}
(1) follows from the discussion in \S\ref{sss:pi1}.

(2) When $\nu$ is elliptic, $\tilS_{a}$ is a finite group hence $\pi_{0}(\tilS_{a})=\tilS_{a}$. Since $\tilS_{a}$ acts on $\Sp_{\bP,a}$, it also acts on $\spcoh{}{\Sp_{\bP, a}}$. On the other hand, by Corollary \ref{c:eq loc Sp}, $\spcoh{}{\Sp_{\bP, a}}$ is the direct sum of $\spcoh{}{\Hess^{\tilw}_{\bP, a}}$ for $\tilw\in W_{\nu}\backslash\tilW/W_{\bP}$. Therefore $B_{a}$ acts on $\spcoh{}{\Sp_{\bP, a}}$ by acting on each direct summand $\spcoh{}{\Hess^{\tilw}_{\bP, a}}$. It is easy to see that they together give an action of   $\tilS_{a}\rtimes B_{a}$ on $\spcoh{}{\Sp_{\bP, a}}$.
\end{proof}

\subsection{Cohomology of homogeneous affine Springer fibers} In this subsection we will give a combinatorial formula for computing part of the cohomology of homogeneous affine Springer fibers. We will use the notation introduced in \S\ref{sss:intro ep} on $\Gnu$-equivariant cohomology with the equivariant parameter $\ep$ specialized to 1.

\begin{theorem}\label{thm:surj}
Let $\nu>0$ be a $\theta$-admissible slope and $a\in\frc(F)_{\nu}^{\rs}$.
\begin{enumerate}
\item Let $\fl^{\tilw}_{\bP,\nu}$ be the partial flag variety $L_{\nu}/L_{\nu}\cap\Ad(\tilw)\bP$ of $L_{\nu}$. The image of the restriction map
\begin{equation}\label{rest fl to Hess}
\cohog{*}{\fl^{\tilw}_{\bP, \nu}}\to \cohog{*}{\Hess^{\tilw}_{\bP, a}}
\end{equation}
is exactly the $\pi_{0}(S_{a})\rtimes B_{a}$-invariant part of the target.
\item Suppose further that $\nu$ is elliptic, then the image of the restriction map
\begin{equation*}
\spcoh{}{\Fl_{\bP}}\to \spcoh{}{\Sp_{\bP, a}}
\end{equation*}
is exactly the $S_{a}\rtimes B_{a}$-invariant part of the target.
\end{enumerate}
\end{theorem}
\begin{proof} Again we only prove the $\bP=\bI$ case to save notation. (2) is a consequence of (1) by Corollary \ref{c:eq loc Sp}. Therefore it suffices to prove (1).

(1) By the decomposition theorem, we may fix a decomposition $\tpi^{\tilw}_{\nu,*}\QQ=\oplus_{i}K_{i}$ where $K_{i}$ is a shifted simple perverse sheaf.  Also note that $\cohog{*}{\fl^{\tilw}_{\nu}}\cong\cohog{*}{\tHess^{\tilw}_{\nu}}\cong\cohog{*}{\frg(F)_{\nu}, \oplus K_{i}}$ since the family $\tHess^{\tilw}_{\nu}$ contracts to its central fiber. We may rewrite \eqref{rest fl to Hess} as
\begin{equation}\label{rewrite res}
\bigoplus_{i}\cohog{*}{\frg(F)_{\nu}, K_{i}}\to \bigoplus_{i} K_{i,\kappa(a)}.
\end{equation}
By Corollary \ref{c:Hess loc sys}, each $K_{i}|_{\frg(F)^{\rs}_{\nu}}$ is a shifted local system. Therefore the restriction map $\cohog{*}{\frg(F)_{\nu}, K_{i}}\to K_{i,\kappa(a)}$ is nonzero if and only if $K_{i}|_{\frg(F)^{\rs}_{\nu}}$ is a shifted constant sheaf, in which case it is an isomorphism.

Restricting to the Kostant section, $K_{i}$ gives rise to a shifted local system $K'_{i}$ on $\frc(F)^{\rs}_{\nu}$ equipped with an action of the group scheme $\pi_{0}(S/\frc(F)^{\rs}_{\nu})$ over $\frc(F)^{\rs}_{\nu}$, whose stalk at $a$ carries an action of $\pi_{0}(S_{a})\rtimes B_{a}$. The decomposition $\cohog{*}{\Hess^{\tilw}_{\nu,a}}\cong\oplus_{i}K'_{i,a}$ respects the $\pi_{0}(S_{a})\rtimes B_{a}$-action. The representation $K'_{i,a}$ is trivial if and only if $K'_{i}$ is a shifted constant sheaf on $\frc(F)^{\rs}_{\nu}$ with the trivial $\pi_{0}(S/\frc(F)^{\rs}_{\nu})$-action, if and only if $K_{i}|_{\frg(F)^{\rs}_{\nu}}$ is a shifted constant sheaf $L_{\nu}$-equivariantly because $[\frc(F)^{\rs}_{\nu}/S]\cong[\frg(F)^{\rs}_{\nu}/L_{\nu}]$. Since $L_{\nu}$ is connected, there is only one $L_{\nu}$-equivariant structure on the constant sheaf on $\frg(F)^{\rs}_{\nu}$, hence $K'_{i,a}$ is trivial as a $\pi_{0}(S_{a})\rtimes B_{a}$-module if and only if $K_{i}|_{\frg(F)^{\rs}_{\nu}}$ is a shifted constant sheaf. Combined with the argument in the previous paragraph, we see that the image of \eqref{rewrite res} can be identified with the sum of those $K'_{i,a}$ which are trivial $\pi_{0}(S_{a})\rtimes B_{a}$-modules, i.e., the $\pi_{0}(S_{a})\rtimes B_{a}$-invariants of $\cohog{*}{\Hess^{\tilw}_{\nu,a}}$.
\end{proof} 
  
\begin{remark}\label{r:Flc} The finite abelian group $\Omega$ acts on $\Fl$ and permutes its connected components simply transitively. Let $\Fl^{\c}$ be the neutral component of  $\Fl$. Then Theorem \ref{thm:surj}(2) implies that when $\nu>0$ is elliptic, there is a natural surjective map
\begin{equation*}
\spcoh{}{\Fl^{\c}_{\bP}}\cong\spcoh{}{\Fl_{\bP}}^{\Omega}\surj\spcoh{}{\Sp_{\bP, a}}^{\pi_0(\tilS_{a})\rtimes B_{a}}.
\end{equation*}
In fact, this follows from the exact sequence in Lemma \ref{l:local Neron}\eqref{tilL} and taking $\Omega$-invariants to the statement of Theorem \ref{thm:surj}(2).
\end{remark} 
  
In Section \ref{s:examples} we shall compute explicitly the action of $S_{a}\rtimes B_{a}$  on $\cohog{*}{\Hess^{\tilw}_{a}}$ in all the cases where $G$ has rank two and $\nu$ elliptic. In particular, we will see that the action of the braid group $B_{a}$ does not necessarily factor through the little Weyl group $\WW^{\Pi(\mu_{m})}$. The examples also show that the part of $\eqnu{*}{\Sp_{a}}^{S_{a}}$ that is not invariant under $B_a$ may contain odd degree classes.

\begin{exam}\label{ex:A2} We consider the case $G=\SL_{3}$ ($e=1$) and $\nu=d/3$ for a positive integer $d$ prime to 3. The regular homomorphism $\mu_{3}\incl\WW=S_{3}$ contains the Coxeter elements in the image, and $L_{\nu}=\TT$. The space $\frg(F)_{\nu}$ is spanned by the affine root spaces of $\{\alpha_{1},\alpha_{2},-\alpha_{1}-\alpha_{2}+d\delta\}$. The open subset $\frg(F)^{\rs}_{\nu}$ consists of those $\gamma$ with nonzero component in each of the three affine root spaces. For $a\in\frc(F)^{\rs}_{\nu}$, the variety $\Hess^{\tilw}_{a}$ is nonempty if and only if $\frg(F)_{\nu}\subset\Ad(\tilw)\Lie\bI$, in which case it is a point. Clearly $S_{a}\rtimes B_{a}$ acts trivially on $\cohog{*}{\Hess^{\tilw}_{a}}$. The condition that $\Hess^{\tilw}_{a}$ is nonempty is equivalent to the condition that the alcove of $\Ad(\tilw)\bI$ is contained in the triangle defined by
\begin{equation}\label{triangle}
\jiao{\alpha_{1},x}\geq0, \jiao{\alpha_{2},x}\geq0, \jiao{\alpha_{1}+\alpha_{2},x}\leq d.
\end{equation}
There are $d^{2}$ alcoves in this area. We conclude that
\begin{equation*}
\dim\cohog{*}{\Sp_{a}}=\dim\cohog{*}{\Sp_{a}}^{S_{a}\rtimes B_{a}}=d^{2}.
\end{equation*}
Below we will define the analog of the triangle \eqref{triangle} in general, and see that the homogeneity of $\dim\cohog{*}{\Sp_{a}}$ in $d$ is a general phenomenon.
\end{exam}

\subsubsection{Clans}\label{sss:clans} In notations of the \S\ref{sss:MP}, the building $\frA$ contains the point $\nu\rho^\vee$. The walls of the apartment $\frA$ that pass through $\nu\rho^{\vee}$ are the reflection hyperplanes of the Weyl group $W_{\nu}\subset \tilde{W}$ of the group $L_{\nu}$ (the Levi factor of the parahoric $\bP_{\nu}$ corresponding to the facet containing $\nu\rho^{\vee}$), and these correspond to the affine roots $\alpha$ such that $\alpha(\nu\rho^\vee)=0$. 

The affine roots that appear in the graded piece $\frg(F)_{\nu}$ are those satisfying the equation $\alpha(\nu\rho^\vee)=\nu$. The reflection hyperplanes of these affine roots are called {\em $\nu$-walls}, and they are oriented with a normal vector pointing to $\nu\rho^\vee$.

In \cite{VV} connected components of $\frA-\bigcup(\nu\textup{-walls})$ are called {\it clans}. A clan is a union of alcoves $\tilde{w}\Delta$ where $\Delta$ is the fundamental alcove. Each alcove $\tilw\Delta$ contributes a Hessenberg variety $\Hess^{\tilw}_{a}=\Hess_{a}(L_{\nu}/L_{\nu}\cap\Ad(\tilw)\bI, \frg(F)_{\nu}\supset\frg(F)^{\tilw}_{\nu})$. This variety only depends on the coset of $\tilw$ in $W_\nu\backslash \tilde{W}$.  Thus we only need to study clans in the dominant chamber of $W_\nu$ which is a cone with the vertex $\nu\rho^\vee$.

\begin{lemma}\label{l:same clan} If the alcoves $\tilw\Delta$ and $\tilw'\Delta$ are in the same clan, then there is canonical isomorphism $\Hess^{\tilw}_{\nu}\cong\Hess^{\tilw'}_{\nu}$ over $\frg(F)_{\nu}$. Moreover, if $\Hess^{\tilw}_{a}$ is nonempty for some $a\in\frc(F)^{\rs}_{\nu}$, then its codimension in the partial flag variety $\fl^{\tilw}_{\nu}$ is equal to the number of $\nu$-walls that separate the alcove $\tilw\Delta$ from $\nu\rho$.
\end{lemma}
\begin{proof}
Indeed, we have identities $L_{\nu}\cap\Ad(\tilw)\bI=L_{\nu}\cap\Ad(\tilw')\bI$ and $\frg(F)^{\tilw}_{\nu}=\frg(F)^{\tilw'}_{\nu}$ by identifying the affine root spaces appearing on both sides of the equalities.

Since $\Hess^{\tilw}_{a}$ is the zero locus of a generic section of the vector bundle $L_{\nu}\twtimes{L_{\nu}\cap\Ad(\tilw)\bI}(\frg(F)_{\nu}/\frg(F)^{\tilw}_{\nu})$, we have $\dim \Hess^{\tilw}_{a}=\dim(L_{\nu}/L_{\nu}\cap\Ad(\tilw)\bI)-\dim(\frg(F)_{\nu}/\frg(F)^{\tilw}_{\nu})$. Now observe that $\frg(F)_{\nu}/\frg(F)^{\tilw}_{\nu}$ is the direct sum of affine roots spaces of those $\alpha$ such that $\alpha(\nu\rho^\vee)=\nu$ and $\alpha(\tilw\Delta^{\circ})<0$ ($\Delta^{\circ}$ is the interior of $\Delta$).
\end{proof}

\begin{remark} The expected dimension of the $\Hess^{\tilw}_{a}$ can be calculated by the above lemma. If the expected dimension is negative then the corresponding Hessenberg 
variety is empty. The converse is not true: there are $\tilw$ such that $\dim(L_{\nu}/L_{\nu}\cap\Ad(\tilw)\bI)\geq\dim(\frg(F)_{\nu}/\frg(F)^{\tilw}_{\nu})$ yet $\Hess^{\tilw}_{a}$ is still empty for all $a\in\frc(F)^{\rs}_{\nu}$.
\end{remark}

\begin{cor} For $\nu>0$ elliptic and $a\in\frc(F)^{\rs}_{\nu}$, if $\Hess^{\tilw}_{a}$ is nonempty, then $\tilw\Delta$ lies in a bounded clan.
\end{cor}
\begin{proof}
Since $\Sp_{a}$ is of finite type when $\gamma$ is elliptic by Lemma \ref{l:finite type}, there can be only finitely many nonzero terms on the right side of \eqref{localize Sp}. If $\tilw\Delta$ lay in an unbounded clan, the same $\Hess^{\tilw}_{a}$ would be appear infinitely many times on the right side of \eqref{localize Sp} by Lemma \ref{l:same clan}, which is a contradiction.
\end{proof}

The following proposition reduces the dimension calculation to the case $\nu=1/m$.

\begin{prop}[see also {\cite[Proposition 3.4.1, Corollary 3.4.2]{VV}}]\label{p:reduce to 1/m} Let $\nu=d_{1}/m_{1}>0$ be an elliptic slope in lowest terms. For $a\in\frc(F)^{\rs}_{\nu}$ and $b\in\frc(F)^{\rs}_{1/m_{1}}$, we have an isomorphism $\tilS_{a}\rtimes B_{a}\cong \tilS_{b}\rtimes B_{b}$, and an isomorphism 
\begin{equation}\label{dr times}
\spcoh{}{\Sp_{a}}\cong \spcoh{}{\Sp_{b}}^{\oplus d^{r}_{1}}
\end{equation}
compatible with the actions of $\tilS_{a}\rtimes B_{a}\cong \tilS_{b}\rtimes B_{b}$. Here $r$ is the $F$-rank of the $G$. In particular, 
\begin{equation}\label{dim Sp dr}
\dim\spcoh{}{\Sp_{a}}^{S_{a}\rtimes B_{a}}=d^r_{1} \dim \spcoh{}{\Sp_{b}}^{S_{b}\rtimes B_{b}}.
\end{equation}
\end{prop}
\begin{proof}
Let $\nu'=1/m_{1}$. Recall from Lemma \ref{l:identify nu} we have an isomorphism of pairs $(\tilL_{\nu}, \frg(F)_{\nu})\cong(\tilL_{\nu'}, \frg(F)_{\nu'})$. This gives an isomorphism $\frc(F)^{\rs}_{\nu}\cong\frc(F)^{\rs}_{\nu'}$ together with an isomorphism between their stabilizer group schemes $\tilS$. Therefore we get an isomorphism $\tilS_{a}\rtimes B_{a}\cong \tilS_{b}\rtimes B_{b}$, well-defined up to conjugation by either $B_{a}$ or $B_{b}$. 

The apartment $\frA$ together with $\nu$-walls is the $d_{1}$-fold scale (with the center at the origin corresponding to the special parahoric $\bG$) of the same apartment $\frA$ with $\nu'$-walls. In particular we have a bijection between the $\nu$-clans and the $\nu'$-clans. 

For $\tilw\Delta$ in a $\nu$-clan and $\tilw'\Delta$ in the corresponding $\nu'$-clan, the isomorphism $\frg(F)_{\nu}\cong\frg(F)_{\nu'}$ in Lemma \ref{l:identify nu} restricts to an isomorphism $\frg(F)^{\tilw}_{\nu}\cong\frg(F)^{\tilw'}_{\nu'}$. Also the isomorphism $L_{\nu}\cong L_{\nu'}$ there induces an isomorphism $\fl^{\tilw}_{\nu}\cong\fl^{\tilw'}_{\nu'}$. Therefore we have a $L_{\nu}\cong L_{\nu'}$-equivariant isomorphism $\Hess^{\tilw}_{\nu}\cong\Hess^{\tilw'}_{\nu'}$ over $\frg(F)_{\nu}\cong\frg(F)_{\nu'}$. This induces an isomorphism of $L_{\nu}\cong L_{\nu'}$-equivariant local systems $\bR^{i}\tpi^{\tilw}_{\nu,*}\QQ|_{\frg(F)^{\rs}_{\nu}}\cong\bR^{i}\tpi^{\tilw'}_{\nu',*}\QQ|_{\frg(F)^{\rs}_{\nu'}}$. Taking stalks at $a$ and $b$ we get a non-canonical isomorphism $\cohog{*}{\Hess^{\tilw}_{a}}\cong\cohog{*}{\Hess^{\tilw'}_{b}}$ equivariant under $S_{a}\rtimes B_{a}\cong S_{b}\rtimes B_{b}$. 

Applying Corollary \ref{c:eq loc Sp} to $\Sp_{a}$ and $\Sp_{b}$, the contribution of Hessenberg varieties to $\spcoh{}{\Sp_{a}}$ in a $\nu$-clan is $d^{r}$ times the contribution of Hessenberg varieties to $\spcoh{}{\Sp_{b}}$ in a $\nu'$-clan because the ratio between sizes of the clans is $d^{r}_{1}$. The isomorphism \eqref{dr times} follows.
\end{proof}

\subsubsection{Cohomology of Hessenberg varieties} When $\nu$ is elliptic,  the equality \eqref{dim Sp dr} reduced the calculation of $\dim\spcoh{}{\Sp_{a}}^{S_{a}\rtimes B_{a}}$ to the case where the slope is $\nu=1/m_{1}$ ($m_{1}$ is a regular elliptic number). Corollary \ref{c:eq loc Sp} reduces the calculation to the calculation of $\dim\cohog{*}{\Hess^{\tilw}_{a}}^{S_{a}\rtimes B_{a}}$ for $\tilw$ in various $1/m_{1}$-clans. We will give a formula for this dimension.

It is well-known that the cohomology ring of the flag variety $\fl^{\tilw}_{\nu}$ of $L_{\nu}$ is 
\begin{equation*}
\calH_{W_{\nu}}=\Sym(\fra^{*})/(\Sym(\fra^{*})^{W_{\nu}}_+)
\end{equation*}
where $(\Sym(\fra^*)^{W_{\nu}}_+)$ denotes the ideal generated by the positive degree $W_{\nu}$-invariants on $\Sym(\fra^{*})$. Let
\begin{equation*}
\lambda^{\tilw}_{\nu}=\prod_{\alpha(\nu\rho^{\vee})=\nu, \tilw^{-1}\alpha<0}\overline{\alpha}\in \Sym(\fra^{*})
\end{equation*}
where $\alpha\in\Phi_{\aff}$ runs over the affine roots satisfying the conditions specified under the product symbol, and $\overline{\alpha}\in\Phi$ denote the finite part of $\alpha$. 

More generally, fix a standard parahoric subgroup $\bP\subset G$ and let $W_{\bP}\subset \Wa$ be the Weyl group of its Levi factor. The group $L_{\nu}\cap\Ad(\tilw)\Lie\bP$ is a parabolic subgroup of $L_{\nu}$ that contains the maximal torus $\AA$. Let $W^{\tilw}_{\bP,\nu}\subset W_{\nu}$ be the Weyl group of the Levi factor of $L_{\nu}\cap\Ad(\tilw)\Lie\bP$. Note that $W^{\tilw}_{\nu,\bI}=\{1\}$ for all $\tilw$. The cohomology ring of the partial flag variety $\fl^{\tilw}_{\bP, \nu}=L_{\nu}/(L_{\nu}\cap\Ad(\tilw)\Lie\bP)$ of $L_{\nu}$ can be identified with $\calH^{W^{\tilw}_{\bP,\nu}}_{W_{\nu}}$. For $\tilw\in\tilW$, we define $\l^{\tilw}_{\bP,\nu}\in\Sym(\fra^{*})$ to be the product $\prod\overline{\alpha}$ where $\alpha$ runs over affine roots of such that $\alpha(\nu\rho^{\vee})=\nu$ and $\tilw^{-1}\alpha\in\Phi^{-}_{\aff}-\Phi(L_{\bP})$ (where $\Phi(L_{\bP})\subset\Phi_{\aff}$ is the root system of $L_{\bP}$ with respect to $\AA$). Such affine roots are permuted by $W^{\tilw}_{\bP,\nu}$, so that $\l^{\tilw}_{\bP,\nu}\in\Sym(\fra^{*})^{W^{\tilw}_{\bP,\nu}}$, and we can talk about the image of $\l^{\tilw}_{\bP,\nu}$ in $\calH^{W^{\tilw}_{\bP,\nu}}_{W_{\nu}}$.

\begin{theorem}\label{thm:dim} Let $\bP$ be a standard parahoric subgroup of $G$, and $\nu>0$ be a $\theta$-admissible slope (not necessarily elliptic), and let $a\in\frc(F)^{\rs}_{\nu}$.
\begin{enumerate}
\item The restriction map $\cohog{*}{\fl^{\tilw}_{\bP,\nu}}\cong\calH^{W^{\tilw}_{\bP,\nu}}_{W_{\nu}}\surj \cohog{*}{\Hess^{\tilw}_{\bP, a}}^{\pi_{0}(S_a)\rtimes B_a}$ induces an isomorphism of $\QQ$-algebras
\begin{equation*}
\cohog{*}{\Hess^{\tilw}_{\bP, a}}^{\pi_{0}(S_a)\rtimes B_a}=\calH^{W^{\tilw}_{\bP,\nu}}_{W_{\nu}}/\Ann(\lambda^{\tilw}_{\bP,\nu})
\end{equation*}
where $\Ann(\l^{\tilw}_{\bP,\nu})\subset\calH^{W^{\tilw}_{\bP,\nu}}_{W_{\nu}}$ is the ideal annihilating $\l^{\tilw}_{\bP,\nu}$.
\item If $\nu$ is elliptic, then there is a $\QQ$-algebra isomorphism
\begin{equation*}
\spcoh{}{\Sp_{\bP,a}}^{S_{a}\rtimes B_{a}}=\bigoplus_{\tilw\in W_{\nu}\backslash\tilW/W_{\bP}}\calH^{W^{\tilw}_{\bP,\nu}}_{W_{\nu}}/\Ann(\lambda^{\tilw}_{\bP,\nu}).
\end{equation*}
Here the sum is over double coset representatives (right multiplication by $W_{\bP}$ does not change the summand but left multiplication by $W_{\nu}$ changes it).
\item If $\nu=d_{1}/m_{1}$ in lowest terms and is elliptic, then
\begin{equation}\label{dim coho Sp}
\dim\spcoh{}{\Sp_{\bP,a}}^{S_{a}\rtimes B_{a}}=d^{r}_{1}\sum_{\tilw\in W_{1/m_{1}}\backslash\tilW/W_{\bP}}\dim(\lambda^{\tilw}_{\bP,1/m_{1}}\cdot\calH^{W^{\tilw}_{\bP,1/m_{1}}}_{W_{1/m_{1}}}).
\end{equation}
Here $\lambda^{\tilw}_{\bP,1/m_{1}}\cdot\calH^{W^{\tilw}_{\bP,1/m_{1}}}_{W_{1/m_{1}}}$ denotes the image of multiplication by $\lambda^{\tilw}_{\bP,1/m_{1}}$ on $\calH^{W^{\tilw}_{\bP,1/m_{1}}}_{W_{1/m_{1}}}$.
\end{enumerate} 
 \end{theorem}
\begin{proof}
(1) Let $\calE=L_{\nu}\twtimes{L_{\nu}\cap\Ad(\tilw)\Lie\bP}(\frg(F)_{\nu}/\frg(F)^{\tilw}_{\bP, \nu})$ be the vector bundle over $Y=\fl^{\tilw}_{\bP,\nu}$. Let $n$ be its rank and $c_{n}(\calE)\in\cohog{2n}{\fl^{\tilw}_{\bP,\nu}}$ be its top Chern class. The Hessenberg variety $Z=\Hess^{\tilw}_{\bP,a}$ is the zero locus of a section of of $\calE$ which is transversal to the zero section. Let $i: Z\incl Y$ be the closed embedding. Then we have an isomorphism $i^{!}\QQ[2n]\cong\QQ$. The composition
\begin{equation*}
\QQ_{Y}\xrightarrow{\alpha} i_{*}\QQ_{Z}\isom i_{*}i^{!}\QQ_{Z}[2n]\xrightarrow{\beta} \QQ_{Y}[2n]
\end{equation*}
is given by the cup product with $c_{n}(\calE)$. Moreover, the natural maps $\alpha$ and $\beta$ are Verdier dual to each other. Taking cohomology, and taking $\pi_0(S_{a})\rtimes B_{a}$-invariants we get
\begin{equation}\label{cup cn}
\cup c_{n}(\calE): \cohog{k}{Y}\xrightarrow{\alpha^{k}}\cohog{k}{Z}^{\pi_0(S_{a})\rtimes B_{a}}\xrightarrow{\beta^{k}}\cohog{k+2n}{Y}.
\end{equation}
where $\beta^{k}$ and $\alpha^{k}$ are adjoints of each other under Poincar\'e duality. Since $\alpha^{k}$ is surjective by Theorem \ref{thm:surj}(1), $\beta^{k}$ must be injective by duality. Since the composition of the maps in \eqref{cup cn} is the cup product with $c_{n}(\calE)$, we conclude that
\begin{equation*}
\cohog{*}{Z}^{\pi_0(S_{a})\rtimes B_{a}}\cong \cohog{*}{Y}/\Ann(c_{n}(\calE))\isom \Image(c_{n}(\calE)\cup\cohog{*}{Y})\subset \cohog{*+2n}{Y}
\end{equation*}

To identify  $c_{n}(\calE)$ with $\l^{\tilw}_{\bP,\nu}$, we first consider the case $\bP=\bI$. In this case $\calE$ is a successive extension of $L_{\nu}$-equivariant line bundles $\calL(\overline{\alpha})$ on $\fl^{\tilw}_{\nu}$ attached to characters $L_{\nu}\cap\Ad(\tilw)\Lie\bI\surj\AA\xrightarrow{\overline{\alpha}}\Gm$ where $\overline{\alpha}$ runs over the weights of $\AA$ on $\frg(F)_{\nu}/\frg(F)^{\tilw}_{\nu}$ (counted with multiplicities). The nonzero weights of $\AA$  on $\frg(F)_{\nu}/\frg(F)^{\tilw}_{\nu}$ are given by the finite parts of affine roots $\alpha$ such that $\alpha(\nu\rho^{\vee})=\nu$ and $\tilw^{-1}\alpha<0$. Therefore $c_{n}(\calE)=\prod c_{1}(\calL(\overline{\alpha}))$ corresponds to the image of $\l^{\tilw}_{\nu}=\prod\overline{\alpha}\in\Sym(\fra^{*})$ in the quotient $\calH_{W_{\nu}}$. 

For general $\bP$, we pullback $\calE$ to a vector bundle $\calE'$ the flag variety $\fl^{\tilw}_{\nu}$ and use the above argument to show that $c_{n}(\calE')$ is the image of $\l^{\tilw}_{\bP,\nu}$ in $\calH_{W_{\nu}}$. Since the pullback map $\cohog{*}{\fl^{\tilw}_{\bP,\nu}}\to\cohog{*}{\fl^{\tilw}_{\nu}}$ is injective, $c_{n}(\calE)$ is also the image of $\l^{\tilw}_{\bP,\nu}$.

%Let Let $s_{\tilw}$ be the inclusion $\Hess^{\tilw}_{\bP,\nu}\incl \fl_\nu$ then by \cite{GKM} adjunction formula:
%$s_{\tilw *} s_{\tilw}^*(x)=x\cup \lambda_{\bP}^{\tilw}$ for any $x\in \cohog{*}{\fl_\nu}$. On the other hand map $s_{\tilw*}$ is an injection because the corresponding homological map
%$D\circ s_{\tilw*}\circ D: \homog{*}{ \Hess^{\tilw}_{\bP,\nu}}^{S_a\rtimes B_a}\to \homog{*}{\fl_\nu}$ is injection since $\homog{*}{\fl_{\nu}}\cong\homog{*}{\tHess^{\tilw}_{\nu}}$ and the action of $S_a\rtimes B_a$ is reductive. 

(2) follows from (1) and Corollary \ref{c:eq loc Sp}.

(3) follows from (2) and Proposition \ref{p:reduce to 1/m}, together with the fact that $\calH^{W^{\tilw}_{\bP,\nu}}_{W_{\nu}}/\Ann(\lambda^{\tilw}_{\bP,\nu})\cong\lambda^{\tilw}_{\bP,1/m_{1}}\cdot\calH^{W^{\tilw}_{\bP,1/m_{1}}}_{W_{1/m_{1}}}$.
 \end{proof}

By Theorem \ref{L(triv)} and Corollary \ref{c:Hgr irr} that we will prove later, the formula \eqref{dim coho Sp} (for $\bP=\bI$, and at least for $G$ split) also gives the dimension of the irreducible spherical representations of the graded and rational \Chas with central charge $\nu$.

%% Homogeneous Hitchin fibers

\section{Homogeneous Hitchin fibers}\label{s:Hit}

In this section, we study the geometric properties of global analogs of homogeneous affine Springer fibers, namely homogeneous Hitchin fibers. The connection between affine Springer fibers and Hitchin fibers was discovered by B.C.Ng\^o \cite{NgoHit}. We will need a generalization of Hitchin moduli stacks, namely we need to consider the moduli stack of Higgs bundles over an orbifold curve (the weighted projective line) with structure group a quasi-split group scheme over the curve.

\subsection{Weighted projective line}\label{ss:wp}
Fix a natural number $m$ such that $e|m$.

Let $\Gm$ act on $\AA^2$ by weights $(m,1)$. Let $X$ be the quotient stack $[(\AA^2-\{(0,0)\})/\Gm]$. This is an orbifold (Deligne-Mumford stack) with coarse moduli space a weighted projective line $\PP(m,1)$. We denote a point in $X$ by its weighted homogeneous coordinates $[\xi,\eta]$, i.e., by exhibiting a preimage of it in $\AA^2-\{(0,0)\}$. The point $0=[0,1]\in X$ does not have nontrivial automorphisms the point $\infty=[1,0]\in X$ has automorphism group $\mu_{m}$. 

Let $X'$ be the similar quotient with weights $(m/e,1)$. We denote the weighted homogeneous coordinates of $X'$ by $[\xi',\eta']$.
Let $\pi:X'\to X$ be the natural morphism $[\xi',\eta']\mapsto[\xi'^e,\eta']$, which is a branched $\mu_e$-cover. The only branch point is $0$, whose preimage consists only one point $0'\in X'$. The preimage of $\infty$ in $X'$ is denoted by $\infty'$, which has automorphism $\mu_{m/e}$.  

The Picard group of $X$ is a free abelian group which we identify with $\frac{1}{m}{\ZZ}$. This defines a degree map
\begin{equation*}
\deg:\Pic(X)\isom\frac{1}{m}{\ZZ}.
\end{equation*}
whose inverse is denoted by $\nu\mapsto\calO_{X}(\nu)$.  For $\deg\calL\geq0$, the global sections of $\calL$ is identified with $\CC[\xi,\eta]_{m\deg\calL}$, homogeneous polynomials in $\xi,\eta$ of total weight $m\deg\calL$. Each line bundle $\calL$ on $X$ also admits a trivialization on $U=X-\{\infty\}$: it is given by the rational section $\eta^{m\deg\calL}$ of $\calL$ which is non-vanishing on $U$. Such a trivialization is unique up to a scalar.

Similarly, the Picard group of $X'$ is identified with $\frac{1}{m/e}\ZZ$ and we have $\deg\pi^{*}\calL=e\deg\calL$.

The one-dimensional torus $\Grot$ acts on $X$ via $\Grot\ni t:[\xi,\eta]\mapsto[t\xi,\eta]$. We will fix a $\Grot$-equivariant structure on each line bundles $\calL$ on $X$ such that the action of $\Grote$ on the stalk $\calL(0)$ is trivial. Concretely, if $\deg\calL>0$, the $\Grot$-action on $\Gamma(X,\calL)\cong \CC[\xi,\eta]_{m\deg\calL}$ is given by $\Grot\ni s: f(\xi,\eta)\mapsto f(s\xi,\eta)$.

For a closed point $x\in X(\CC)-\{\infty\}$ (which is not stacky), we use $\hO_x$ and $\hK_x$ to denote the completed local ring and its field of fractions at $x$. The notation $\hK_{\infty}$ still makes sense, and we often identify with $F$.

The orbifold $X'$ contains an open subscheme $U'=X'-\{\infty'\}\cong \AA^1$. On the other hand $V'=X'-\{0'\}$ is isomorphic to $\tilV'/\mu_{m/e}$ where $\tilV'\cong\AA^1$. Denote the preimage of  $\infty'$ in $\tilV'$ by $\wt{\infty'}$. Similarly, the orbifold $X$ is covered by  $U=X-\{\infty\}$ and $V=X-\{0\}=\tilV'/\mu_{m}$. 

To summarize, we have a diagram
\begin{equation*}
\xymatrix{ & & & \tilV'\ar[d]^{/\mu_{m/e}} & \wt{\infty'}\ar[d]\ar@{_{(}->}[l]\\
0'\ar@{^{(}->}[r]\ar[d] &  U'\ar@{^{(}->}[r]\ar[d] & X'\ar[d]& V'\ar@{_{(}->}[l]\ar[d]^{/\mu_{e}} & \infty'/\mu_{m/e}\ar[d]\ar@{_{(}->}[l]\\
0\ar@{^{(}->}[r] & U\ar@{^{(}->}[r] & X & V\ar@{_{(}->}[l] & \infty/\mu_{m}\ar[l]}
\end{equation*}

The dualizing sheaf $\omega_{X}$ is obtained by gluing the dualizing sheaves $\omega_{U}$ and $\omega_{V}$, the latter being the descent of the dualizing sheaf $\omega_{\tilV'}$. We may identify $\omega_{X}$ with the graded $\CC[\xi,\eta]$-submodule $M$ of $\CC[\xi,\eta]d\xi\oplus \CC[\xi,\eta]d\eta$ given by the kernel of the map $fd\xi+gd\eta\mapsto m\xi f+\eta g$. Therefore $M$ is the free $\CC[\xi,\eta]$-module generated by $\eta d\xi-m\xi d\eta$. Hence $\deg\omega_{X}=-1/m-1$.

\subsection{The moduli of $\cG$-torsors}
\subsubsection{The group scheme $\cG$} We define a group scheme $\cG$ over $X$ by gluing $\cG_{U}$ and $\cG_{V}$. The group scheme $\cG_{U}$ is defined as the fiberwise neutral component of $(\Res^{U'}_{U}(\GG\times U'))^{\mu_{e}}$, where $\mu_{e}$ acts on both $U'$ and on $\GG$ via $\theta$ composed with the inverse of $\Out(\GG)$. The group scheme $\cG_{V}=\GG\twtimes{\mu_e}V'$, where $\mu_{e}$ acts on $\GG$ via $\theta$ and on $\tilV'$ via multiplication on coordinates. Both $\cG_{U}|_{U\cap V}$ and $\cG_{V}|_{U\cap V}$ are canonically isomorphic to the group scheme $(\Res^{U'\cap V'}_{U\cap V}(\GG\times (U'\cap V')))^{\mu_{e}}$, therefore we can glue $\cG_{U}$ and $\cG_{V}$ together to get a group scheme $\cG$ over $X$. Upon choosing an isomorphism $\hO_{0}\cong \calO_{F}$, we may identify $\cG|_{\Spec\hO_{0}}$ with $\bG$, as group schemes over $\calO_{F}$.

Alternatively, we may view $\cG$ as glued from $\GG\times{V}$ and $\bG$ (viewed as a group scheme over $\calO_{F}=\hO_{0}$) along $\Spec\hK_{0}$. This point of view allows us to define a group scheme $\cG_{\bP}$ for any parahoric $\bP\subset G(F)$. In fact, $\cG_{\bP}$ is the group scheme over $X$ defined by gluing $\GG\times V$ and $\bP$ (viewed as a group scheme over $\hO_{0}=\calO_{F}$) along $\Spec\hK_{0}$.

The Lie algebra $\Lie\cG$ is the vector bundle over $X$ obtained by gluing $\frg$ (over $\hO_{0}=\calO_{F}$) and $\gg\twtimes{\mu_{e}}V'$. We may similarly define $\Lie\cG_{\bP}$.

We may also define a global version $\frc_{X}$ of $\frc$. This is obtained by gluing $\frc$ (over $\Spec\hO_{0}=\Spec\calO_{F}$) and $\frc_{V}:=\cc\twtimes{\mu_{e}}\tilV'$ together.

\subsubsection{The moduli of $\cG$-torsors over $X$} A $\cG$-torsor over $X$ is the following data $(\calE_U,\calE_V,\tau)$. Here $\calE_U$ is a $\cG_{U}$-torsor over $U$;  $\calE_V$ is a $\mu_{m}$-equivariant $\GG$-torsor over $\tilV'$ (the $\mu_{m}$-action on $\GG$ is given by the natural surjection $\mu_{m}\surj\mu_e$ followed by $\theta$ and the inversion of $\Out(\GG)$). Note that $\calE_{V}|_{U\cap V}$ descends to a $\cG_{U\cap V}$-torsor $\calE^{\flat}_{U\cap V}$over $U\cap V$. Finally $\tau$ is an isomorphism of $\cG_{U\cap V}$-torsors $\tau:\calE_U|_{U\cap V}\isom\calE^\flat_{U\cap V}$.

In particular, a vector bundle $\calV$ over $X$ is the data $(\calV_U,\calV_V,\tau)$ where $\calV_U$ is a vector bundle over $U$, $\calV_{V}$ is a $\mu_{m}$-equivariant vector bundle on  $\tilV'$ (which descends to $\calV^\flat_{U\cap V}$ over $U\cap V$) and $\tau$ is an isomorphism $\calV_U|_{U\cap V}\isom\calV^\flat_{U\cap V}$.

Alternatively, a $\cG$-torsor is also a triple $(\calE_{0},\calE_{V},\tau)$ where $\calE_{0}$ is a $\bG$-torsor over $\Spec\hO_{0}$, $\calE_{V}$ is a $\mu_{m}$-equivariant $\GG$-torsor over $\tilV'$ and $\tau$ is an isomorphism between the two over $\Spec \hK_{0}$.

For a parahoric $\bP\subset G(F)$, we may define the notion of a $\cG$-torsor over $X$ with $\bP$-level structure at $0$: this simply means a $\cG_{\bP}$-torsor over $X$. Concretely, it is a triple $(\calE_{0},\calE_{V},\tau)$ where $\calE_{0}$ is a $\bP$-torsor over $\Spec \hO_{0}$ and the rest of the data is the same as in the triple for a $\cG$-torsor.

We denote by $\Bun_{\cG}$ the moduli stack of $\cG$-torsors over $X$. For a parahoric $\bP\subset G(F)$, let $\Bun_{\bP}$ denote the moduli stack of $\cG$-torsor over $X$ with $\bP$-level structure at $0$; in particular, $\Bun_{\cG}=\Bun_{\bG}$.

\subsection{The Hitchin moduli stack} Fix a line bundle $\calL$ on $X$. Let $\calL'=\pi^*\calL$.

\subsubsection{$\cG$-Higgs bundles and their moduli} 
We define a {\em $\cG$-Higgs bundle over $X$ valued in $\calL$} to be a pair $(\calE,\varphi)$, where $\calE=(\calE_{U},\calE_{V},\tau)$ is a $\cG$-torsor over $X$ and $\varphi$ is a global section of the vector bundle $\Ad(\calE)\otimes\calL$ over $X$. Here $\Ad(\calE)=\calE\twtimes{\cG}\Lie\cG$ is the adjoint bundle of $\calE$. Let $\MHit$ be the moduli stack of $\cG$-Higgs bundles over $X$ valued in $\calL$. 

For a parahoric subgroup $\bP\subset G(F)$, we may define the notion of a $\cG$-Higgs bundle valued in $\calL$ with $\bP$-level structure at $0$. This is a pair $(\calE,\varphi)$ where $\calE$ is a $\cG_{\bP}$-torsor over $X$, and $\varphi\in\Gamma(X,\Ad_{\bP}(\calE)\otimes\calL)$, where $\Ad_{\bP}(\calE)=\calE\twtimes{\cG_{\bP}}\Lie\cG_{\bP}$ is the adjoint vector bundle for the $\cG_{\bP}$-torsor $\calE$. The moduli stack of $\cG$-Higgs bundles valued in $\calL$ with $\bP$-level structure at $0$ is denoted by $\calM_{\bP}$. In particular, for $\bP=\bG$, we have $\MHit=\calM_{\bG}$. When $\bP=\bI$, we often omit the subscript $\bI$ and write $\calM_{\bI}$ as $\calM$.

We may also replace the triple $(X,\cG,\calL)$ by $(X',\GG\times X',\calL')$, and define the moduli stack $\calM'_{\bP'}$ for $\calL'$-valued  $\GG$-Higgs bundles on $X'$ with $\bP'$-structure at $0'$. Here $\bP'$ can be any parahoric subgroup of $\GG(F_{e})$. Suppose we choose $\bP'$ to be a standard parahoric of $\GG(F_{e})$ that is invariant under the action of $\mu_{e}$ on $\GG(F_{e})$, then it corresponds to a unique standard parahoric $\bP\subset G(F)$. Moreover, $\calM'_{\bP'}$ carries a natural $\mu_{e}$-action. We have a natural morphism $\calM_{\bP}\to\calM'^{\mu_{e}}_{\bP'}$.

%The stack $X'$ contains the open subset $U'=X'-\{\infty\}$ which is isomorphic to $\AA^1$. We can therefore talk about $G$-Higgs bundles with parahoric level structures of type $\bP$ at a point $x'\in U'$, as we did in \cite{GS}. This defines a stack $\calM'_{\bP}$ and a fibration $f'_{\bP}:\calM'_{\bP}\to U'\times\calA'$. When the standard parahoric $\bP$ of the loop grop of $G$ is stable under $\mu_e$, $\calM'_{\bP}$ carries an action of $\mu_e$ and $f'_{\bP}$ is $\mu_e$-equivariant.

\subsubsection{Hitchin base}
Recall we have defined a global space of invariant polynomials $\frc_{X}$ over $X$. The weighted action of $\Gdil$ on $\cc$ induces an action on $\frc_{X}$. We can then twist $\frc_{X}$ by $\calL$ to obtain $\frc_{\calL}:=\rho(\calL)\twtimes{\Gdil}_{X}\frc_{X}$. Here $\rho(-)$ is the total space of the $\Gm$-torsor associated to a line bundle. The Hitchin base $\calA$ for the triple  $(X,\cG,\calL)$ is the scheme of sections of the fibration $\frc_{\calL}\to X$.

We may also understand $\calA$ using the curve $X'$ and global sections of $\calL'$ on it. Let $\calA'$ be the Hitchin base for the triple $(X',\GG,\calL')$. The fundamental invariants $f_{1},\cdots, f_{r}$ give an isomorphism 
\begin{equation}\label{A'}
\calA'=\bigoplus_{i=1}^{r}\cohog{0}{X',\calL'^{d_{i}}}=\bigoplus_{i=1}^{r}\CC[\xi',\eta]_{d_{i}m\deg\calL}.
\end{equation}
The fibration $\cc_{\calL'}\to X'$ is  $\mu_{e}$-equivariant, hence $\mu_{e}$ acts on the space of sections, i.e., $\calA'$.  The Hitchin base $\calA$ can be identified with the $\mu_{e}$-fixed subscheme of $\calA'$, consisting of $\mu_{e}$-equivariant sections $a:X'\to \cc_{\calL'}$. Using the description in \eqref{A'} and \eqref{frc}, we have 
\begin{equation*}
\calA=\bigoplus_{i=1}^{\rr}\xi'^{\vep_{i}}\CC[\xi,\eta]_{m(d_{i}\deg\calL-\vep_{i}/e)}.
\end{equation*}

\subsubsection{The cameral curve}\label{sss:cameral} Let $\Ah\subset\calA$ be the locus where $a:X\to\cc_{\calL}$ generically lies in the regular semisimple locus. 
Fix an element $a\in\Ah$, we can define a branched $\WW'$-cover $\pi_{a}:X'_{a}\to X$ called the {\em cameral curve}. In fact, $a$ determines a $\mu_{e}$-equivariant section $a':X'\to \cc_{\calL}$. Then $X'_{a}$ is defined as the Cartesian product
\begin{equation}\label{cam'}
\xymatrix{X'_{a}\ar[r]\ar[d]^{\pi'_{a}} & \tt_{\calL'}\ar[d]\\
X'\ar[r]^{a'} & \cc_{\calL'}}
\end{equation}
The morphism $\pi_{a}:X'_{a}\to X$ is the composition $\pi\circ\pi'_{a}$. The $\mu_{e}$-equivariance of $a'$ allows us to define a $\WW'$-action on $X'_{a}$ extending the $\WW$-action coming from the Cartesian diagram \eqref{cam'}, and $\pi_{a}$ is a branched $\WW'$-cover.

\subsubsection{The elliptic locus} For $a\in\Ah$, the restriction of the cameral curve to the generic point of $X$ gives a homomorphism $\Pi_{a}: \Gal(\overline{K}/K)\to\WW'$ defined up to conjugacy, where $K$ is the function field of $X$. The point $a$ is called {\em elliptic} if the $\Pi_{a}(\Gal(\overline{K}/K))$-invariants on $\tt$ is zero. The elliptic locus of $\Ah$ forms an open subscheme $\Aa$.

\subsubsection{The Hitchin fibration} For each parahoric $\bP$ of $G(F)$, we can define the Hitchin fibration
\begin{equation*}
f_{\bP}:\calM_{\bP}\to\calA.
\end{equation*}
To define this, we first define the Hitchin fibration $f'_{\bP'}:\calM'_{\bP'}\to\calA'$ where the level $\bP'$ is the $\mu_{e}$-invariant parahoric of $\GG(F_{e})$ corresponding to $\bP$. For $(\calE',\varphi')\in\calM'_{\bP'}$, $f'_{\bP'}$ is the collection of invariants $f_{i}(\varphi)\in\Gamma(X',\calL'^{d_{i}})$ for $i=1,\cdots, \rr$. This morphism is $\mu_{e}$-equivariant. Therefore we can define $f_{\bP}$ as the composition
\begin{equation*}
f_{\bP}:\calM_{\bP}\to\calM'^{\mu_{e}}_{\bP'}\xrightarrow{f'_{\bP'}}\calA'^{\mu_{e}}=\calA.
\end{equation*}
We denote the restriction of $f_{\bP}$ to $\Ah$ by $f_{\bP}$. 

\subsubsection{Picard stack action}\label{sss:Pic} 

For $a\in\calA(\CC)$, we can define the regular centralizer group scheme $J_a$ over $X$. In fact, we define $J_{a,V}$ over $V$ as the $\mu_m$-descent from the regular centralizer group scheme $\tilJ'_a$ over $\tilV'$ (with respect to the section $\tilV'\to V'\to \cc_{\calL'}$). In Section \ref{ss:J}, the local definition of the regular centralizer allows to define $J_{a}$ over $\Spec\hO_{0}$. Finally we glue $J_{a,V}$ and $J_{a,\Spec\hO_{0}}$ together along $\Spec\hK_{0}$. 

Alternatively, we may define a universal version $J_{X}\to \frc_{X}$ as follows. This is glued from $J_{0}$ over $\frc$ (over $\Spec\hO_{0}$) and $\JJ\twtimes{\mu_{e}}V'$ along $\Spec\hK_{0}$, here $\JJ$ is the universal regular centralizer over
$\cc$. Let $J_{\calL}=J_{X}\twtimes{\Gdil}_{X}\rho(\calL)$, then we have $J_{\calL}\to \frc_{\calL}$. For $a\in\calA(\CC)$ viewed as a section $X\to \frc_{\calL}$, $J_{a}$ may also be defined using the Cartesian diagram
\begin{equation}\label{alt Ja}
\xymatrix{J_{a}\ar[r]\ar[d] & J_{\calL}\ar[d]\\ X\ar[r]^{a} & \frc_{\calL}}
\end{equation}

Having defined $J_a$, we define $\calP_a$ to be the moduli stack of $J_a$-torsors over $X$. Here, a $J_{a}$-torsor means a triple $(\calQ_{U},\calQ_{V},\tau)$ where $\calQ_{U}$ is a $J_{a}|_{U}$-torsor over $U$; $\calQ_{V}$ is a $\mu_{m}$-equivariant $\tilJ'_{a}$-torsor over $\tilV'$, giving rise to a $J_{a}|_{U\cap V}$-torsor $\calQ^{\flat}_{U\cap V}$ over $U\cap V$ via descent; $\tau$ is an isomorphism of $J_{a}|_{U\cap V}$-torsors $\calQ_{U}|_{U\cap V}\isom\calQ^{\flat}_{U\cap V}$.

We also have a local analog of $\calP_a$. For $x\in X-\{\infty\}$, we define $P_{a,x}$ to be the moduli space of $J_a$-torsors over $\Spec \hO_{x}$ together with a trivialization over $\Spec \hK_x$. This is a group ind-scheme over $\CC$. We have $P_{a,x}(\CC)=J_a(\hK_x)/J_a(\hO_x)$. 

For $x=0$, let $a_{0}$ be the restriction of $a$ to $\Spec \hO_{0}=\Spec\calO_{F}$ (fixing a trivialization of $\calL$ over $\Spec\hO_{0}$). Then $P_{a,0}$ is the same as the $P_{a_{0}}$ defined in Section \ref{ss:J}.

For $x=\infty$, $P_{a,\infty}$ classifies $\mu_{m}$-equivariant $\tilJ'_{a}$-torsors over the formal neighborhood $\Spec \hO_{\wt{\infty}'}$ of $\wt{\infty}'$ together with a $\mu_{e}$-invariant trivialization over $\Spec \hK_{\wt{\infty}'}$. Identifying $\hO_{\wt{\infty}'}$ with $\calO_{F_{m}}$ (in a $\mu_{m}$-equivariant way), we have $P_{a,\infty}=(J_a(F_m)/J_a(\calO_{F_m}))^{\mu_m}$. Note that it is important that we take invariants after taking quotients.

Many properties of the usual Hitchin moduli stacks generalize to the orbifold setting.
\begin{prop}\label{p:smooth}Suppose $\deg\calL\geq0$. Then
\begin{enumerate} 
\item The moduli stack $\calM^{\h}_{\bP}=\calM_{\bP}|_{\Ah}$ is a smooth Artin stack. Its restriction over $\Aa$ is a Deligne-Mumford stack.
\item The morphism $f^{\textup{ell}}_{\bP}:\calM_\bP|_{\Aa}\to\Aa$ is proper.
\end{enumerate}
\end{prop}
\begin{proof} The give the proofs for the Hitchin moduli stack $\calM=\calM_{\bI}$ without Iwahori level structure. The general case is similar.

(1) A $\CC$-point of $\calM$ is a pair $(\calE,\varphi)$ where $\calE$ is a $\cG_{\bI}$-torsor and $\varphi\in\Gamma(X,\Ad_{\bI}(\calE)\otimes\calL)$. The tangent space of $\calM$ at $(\calE,\varphi)$ is $\cohog{0}{X,\calK}$ where $\calK$ is the two term complex $\Ad_{\bI}(\calE)\xrightarrow{[\varphi,-]}\Ad_{\bI}(\calE)\otimes\calL$ placed in degrees $-1$ and $0$. By a variant of Biswas and Ramanan \cite{BR}, the obstruction to infinitesimal deformation of the Higgs bundle $(\calE,\varphi)$ lies in $\cohog{1}{X,\calK}$. By duality, $\cohog{1}{X,\calK}$ is dual to $\cohog{0}{X,\calK^{\vee}\otimes\omega_{X}}$, where $\calK^{\vee}$ is the two term complex $\Ad_{\bI}(\calE)\otimes\calL^{-1}\xrightarrow{[\varphi,-]}\Ad_{\bI}(\calE)$, placed in degrees $0$ and $1$. Let $\calK'=\calH^{0}(\calK^{\vee}\otimes\calL)$ be the kernel of $\Ad_{\bI}(\calE)\xrightarrow{[\varphi,-]}\Ad_{\bI}(\calE)\otimes\calL$. Then $\cohog{0}{X,\calK^{\vee}\otimes\omega}=\cohog{0}{\calK'\otimes\calL^{-1}\otimes\omega}$. The argument of \cite[4.11.2]{NgoFL} shows that $\calK'$ is a subsheaf of $\Lie J^{\flat}_{a}$, where $a=f(\calE,\varphi)$ (argue over $U$ and $V$ separately).

We can still describe $\Lie J^{\flat}_{a}$ using the cameral curve (see \S\ref{sss:cameral}) as for usual Hitchin fibers. Let $X'^{\flat}_{a}$  be the normalization of the cameral curve $X'_{a}$, and let $\pi^{\flat}_{a}:X'^{\flat}_{a}\to X$ be the projection. Then $\Lie J^{\flat}_{a}=\pi^{\flat}_{a,*}(\calO_{X'^{\flat}_{a}}\otimes\tt)^{\WW'}$. Therefore $\cohog{0}{X,\calK^{\vee}\otimes\omega_{X}}\incl\cohog{0}{X,\Lie J^{\flat}_{a}\otimes\calL^{-1}\otimes\omega_{X}}=\cohog{0}{X'^{\flat}_{a},\tt\otimes\pi^{\flat,*}_{a}(\calL^{-1}\otimes\omega_{X})}^{\WW'}$. Since $\deg(\calL^{-1}\otimes\omega_{X})\leq\deg\omega_{X}<0$, we have $\cohog{0}{X'^{\flat}_{a},\tt\otimes\pi^{\flat,*}_{a}(\calL^{-1}\otimes\omega_{X})}=0$, hence the vanishing of the obstruction.

We have $\Aut(\calE,\varphi)\subset\cohog{0}{X,J^{\flat}_{a}}=\cohog{0}{X'^{\flat}_{a},\TT}^{\WW'}$. Choose a component $X^{\dagger}_{a}$ of $X'_{a}$ whose stabilizer we denote by $\WW^{\dagger}_{a}$. Then $\WW^{\dagger}_{a}$ is the image of $\Pi_{a}:\Gal(\overline{K}/K)\to\WW'$ up to conjugacy. We have $\cohog{0}{X'^{\flat}_{a},\TT}^{\WW'}=\TT^{\WW^{\dagger}_{a}}$. When $a$ is elliptic, $\tt^{\WW^{\dagger}_{a}}=\tt^{\Pi_{a}(\Gal(\overline{K}/K))}=0$, hence $\TT^{\WW^{\dagger}_{a}}$ is finite, and therefore $\Aut(\calE,\varphi)$ is finite (and certainly reduced), and $\calM|_{\Aa}$ is Deligne-Mumford. 

(2) Let $\tX'=\PP^{1}$ and let $p:\tX'\to X$ be the morphism given by $[\xi,\eta]\mapsto[\xi^{m},\eta]$. Any $G$-Higgs bundle on $X$ with $\bP$-level structure at $0$ pullbacks to a $\GG$-Higgs bundle on $\tX'$ with certain $\bQ$-level structure at $\wt{0}'=[0,1]\in\tX'$ (where $\bQ\subset\GG(F)$ is a parahoric stable under $\mu_{e}$ whose $\mu_{e}$-invariants give $\bP$ up to a finite index), equipped with a $\mu_{m}$-equivariant structure. Note that $\mu_{m}$ also acts on $\GG$ and $\bQ$ by pinned automorphisms via its quotient $\mu_{e}$. Let $\wt\calM_{\bQ}$ be the moduli stack of $\GG$-Higgs bundles (twisted by the line bundle $p^{*}\calL$) over $\tX'$ with a $\bQ$-levels structure at $\wt{0}'$. There is an action of $\mu_{m}$ on $\wt\calM_{\bQ}$. The pullback map gives a morphism $\calM_{\bP}\to \wt\calM^{\mu_{m}}_{\bQ}$ compatible with the isomorphism Hitchin bases $\calA\cong\wt\calA'^{\mu_{m}}$ (where $\wt{\calA}'$ is the Hitchin base for $\GG$-Higgs bundles over $\tX'$ with respect to the line bundle $p^{*}\calL$). Restricting the $\mu_{m}$-equivariant structure of a point $(\calE,\varphi)\in\wt\calM_{\bQ}$ to $\wt{0}'$, we get a cohomology class $[\calE]_{0}\in\cohog{1}{\mu_{m}, L_{\bQ}}$ (where $L_{\bQ}$ is the reductive quotient of $\bQ$ on which $\mu_{m}$ acts via the quotient $\mu_{e}$). A local calculation of shows that $(\calE,\varphi)$ comes from a Higgs bundle on $X$ with $\bP$-level structure via pullback if and only if $[\calE]_{0}$ is the trivial class. Therefore $\calM_{\bP}$ can be identified with the fiber of the trivial class under the map $\wt\calM^{\mu_{m}}_{\bQ}\to\cohog{1}{\mu_{m}, L_{\bQ}}$, hence a closed substack of $\wt\calM^{\mu_{m}}_{\bQ}$. For the morphism $\wt\calM^{\mu_{m}}_{\bQ}|_{\Aa}\to\Aa$, one can adapt the argument in \cite[II.4]{Faltings} to show it is proper. Because  $\calM_{\bP}|_{\Aa}$ is a closed substack of $\wt\calM^{\mu_{m}}_{\bQ}|_{\Aa}$, it is also proper over $\Aa$.
\end{proof}

\subsection{Homogeneous points in the Hitchin base} 
The $\Grote$-action on $X'$ induces an action on the parabolic Hitchin stack $\calM'=\calM(X',\GG,\calL')$ as well as the Hitchin base $\calA'$. Recall $\calL$ is equipped with the $\Grot$-equivariant structure, which induces a $\Grote$-equivariant structure on $\calL'$.  For any integer $n$, we may identify $\cohog{0}{X',\calL'^{\otimes n}}$ with $\CC[\xi',\eta]_{nm\deg\calL}$, with $t\in\Grote$ acting as $f(\xi',\eta)\mapsto f(t\xi',\eta)$. 

On the other hand, let $\Gdil$ be the one-dimensional torus which acts on $\calM'$ by dilation of the Higgs fields, and it also acts on $\calA'$ with weights $d_1,\cdots,d_{\rr}$.

Similarly, there is an action of $\Gtwo$ on both $\calM$ and $\calA$ such that the Hitchin fibration $f:\calM\to\calA$ is $\Gtwo$-equivariant. The action of $\Gtwo$ on $\calA=\calA'^{\mu_{e}}$ is the restriction of its action on $\calA'$.

Recall we defined a torus $\hGm(\nu)\subset\hGrot\times\Gdil$ in Section \ref{ss:homo local}, and its image $\Gnu$ in $\Gtwo$ (see \S\ref{sss:sym homog AFS}). 

\begin{defn}
A nonzero element $a\in\calA'$ is called {\em homogeneous of slope $\nu$}, if it is fixed by $\Gnu\subset\Gtwo$. 
\end{defn}

\begin{lemma}\label{l:a global} Let $\nu\in\QQ\cap[0,\deg\calL]$. A point $a\in\calA'^\h$ is homogeneous of slope $\nu$ if and only if it is of the form
\begin{equation*}
\sum_{i=1}^{\rr}c_{i}\xi'^{e\nu d_{i}}\eta^{(\deg\calL-\nu)md_{i}}\in\bigoplus_{i=1}^{\rr} \CC[\xi',\eta]_{d_{i}m\deg\calL}.
\end{equation*}
for some $c_{i}\in \CC$. Note that the $i$-th term is nonzero only if $e\nu d_i\in\ZZ$.

If $\nu<0$ or $\nu>\deg\calL$, then there is no homogeneous element in $\calA'$ of slope $\nu$.
\end{lemma}
\begin{proof}
We only need to note that $\Gnu$ fixes the monomial $\xi'^{j}\eta^{\ell}$ in the $i\nth$ factor of $\calA'$ if and only if $j/d_i=e\nu$.
\end{proof}

Let $\Ah_{\nu}$ (resp. $\calA_{\nu}$) denote the subscheme of homogeneous elements in $\Ah$ (resp. $\calA$) of slope $\nu$. We denote the restriction of the Hitchin fibration $f_{\bP}$ to $\Ah_{\nu}$ by
\begin{equation*}
f_{\bP, \nu}: \calM^{\h}_{\bP, \nu}\to \Ah_{\nu}.
\end{equation*}

Fixing a $\Grot$-equivariant trivialization of $\calL$ over $U$, for any $a\in\calA$ we may view $a|_{\Spec \hK_{0}}$ as an element in $\frc(F)^{\rs}$ (upon identifying $F$ with $\hK_{0}$). This defines a map
\begin{equation}\label{e0}
e_{0}: \calA_{\nu}\to \frc(F)_{\nu}
\end{equation}
sending $\Ah_{\nu}$ to $\frc(F)^{\rs}_{\nu}$. In particular, if $\Ah_{\nu}\neq\varnothing$, then $\nu$ is a $\theta$-admissible slope (see Definition \ref{def:slope}).

\subsubsection{Symmetry on homogeneous Hitchin fibers}\label{sss:sym homog Hitchin} By construction, $J_{a}$ fits into the Cartesian diagram \eqref{alt Ja}. Note that the projection $J_{\calL}\to \frc_{\calL}$ is $\Grote\times\Gdil$-equivariant. Since $a$ is homogeneous of slope $\nu$, the section $X\to \frc_{\calL}$ is $\Gnu$-equivariant, where $\Gnu$ acts on $X$ via the projection $\Gnu\to\Grote\to \Grot$. Therefore $J_{a}$ also admits an action of $\Gnu$ making $J_{a}\to X$ equivariant under $\Gnu$. Therefore $\Gnu$ also acts on $\calP_{a}$. As $a$ varies in $\Ah_{\nu}$, we get a Picard stack $\calP_{\nu}$ over $\Ah_{\nu}$ whose fiber over $a$ is $\calP_{a}$. The group stack $\calP_{\nu}\rtimes\Gnu$ acts on the family $\calM^{\h}_{\bP,\nu}$.

\subsection{The case $\calL=\calO_{X}(\nu)$} 
Fix a positive $\theta$-admissible slope $\nu=d/m$ written in the normal form (see \S\ref{sss:normal}). We use $m$ to define the weighted projective line $X=\PP(m,1)$. From now on, we will concentrate on the special case where $\calL=\calO_X(\nu)$, and we fix a $\Grot$-equivariant trivialization of $\calL_{U}$. Let $\Pi:\mu_{m}\incl\WW'$ be the $\theta$-regular homomorphism of order $m$.

The following lemma follows directly from Lemma \eqref{l:a global}.
\begin{lemma}\label{l:e0} The map $e_{0}$ in \eqref{e0} is an isomorphism. Moreover, with one choice of an isomorphism $\Gnu\cong\Gm$, the action of $\Gnu$ on $\calA$ is contracting to $\calA_{\nu}$.
\end{lemma}

\begin{prop} For $a\in\Ah_{\nu}$ we have
\begin{equation}\label{dim Ma}
\dim\calM_{a}=\dim\calP_{a}=\frac{1}{2}(\nu\#\Phi_{\GG}-r-\dim\tt^{\Pi(\mu_m)}).
\end{equation}
\end{prop}
\begin{proof} By considering the regular locus of $\calM_{\bG,a}$, which is a torsor under $\calP_{a}$, we may reduce to the calculation of $\dim\calP_{a}$.

We first work with the curve $X'$ instead of $X$, over which we consider the constant group $\GG\times X'$ and the regular centralizer group scheme $J'_{a}$ therein. In \cite[4.13.2]{NgoFL}, Ng\^{o} gives a formula for $\Lie J'_{a}$ (as a vector bundle over $X'$)
\begin{equation*}
\Lie J'_{a}\cong\cc^{\vee}_{\calL'}\otimes\calL'. 
\end{equation*}
Here $\calL'=\pi^{*}\calL$. By construction, we have $\Lie J_{a}=(\pi_{*}\Lie J'_{a})^{\mu_{e}}$. On the other hand we have $\dim\calP_{a}=-\chi(X,\Lie J_{a})$, therefore we have
\begin{equation*}
\dim\calP_{a}=-\chi(X,\pi_{*}(\cc^{\vee}_{\calL'}\otimes\calL')^{\mu_{e}}).
\end{equation*}
By projection formula,
\begin{equation*}
\pi_{*}(\cc^{\vee}_{\calL'}\otimes\calL')^{\mu_{e}}=(\cc^{\vee}_{\calL}\otimes\calL\otimes\pi_{*}\calO_{X'})^{\mu_{e}}=\bigoplus_{i=1}^{r}\calO_{X}(-d_{i}\nu+\nu-\vep_{i}/e)
\end{equation*}
Recall that $\vep_{i}\in\{0,1,\cdots,e-1\}$ is the unique number such that $\mu_{e}$ acts on the $i\nth$ fundamental invariant $f_{i}$ via the $\vep_{i}\nth$ power of the tautological character (see \S\ref{ss:inv}). The Riemann-Roch formula for $X$ says that for any line bundle $\calL''$ on $X$ we have
\begin{equation*}
\chi(X,\calL'')=[\deg\calL'']+1.
\end{equation*}
Then
\begin{eqnarray*}
&&\dim\calP_{a}=\sum_{i=1}^{r}-\chi(X,\calO_{X}(-d_{i}\nu+\nu-\frac{\vep_{i}}{e}))\\
&=&\sum_{i=1}^{r}-[-\nu(d_{i}-1)-\frac{\vep_{i}}{e}]-1=\sum_{i=1}^{r}[\nu(d_{i}-1)+\frac{\vep_{i}}{e}-\frac{1}{m}].
\end{eqnarray*}

The integers $\{d_{1}-1,\cdots,d_{\rr}-1\}$ are the exponents of the $\WW$-action on $\tt$. By \cite[Theorem 6.4(v)]{Spr}, the action of $\mu_{m}$ on $\tt$ via $\Pi$ decomposes $\tt$ into $\rr$ lines where $\mu_{m}$ acts through the characters $\zeta\mapsto\zeta^{\nu(d_{i}-1)+\vep_{i}/e}, i=1,\cdots,\rr$. Therefore if we write  $\tt=\oplus_{j\in\ZZ/m\ZZ}\tt_{j}$ according to this action of $\mu_{m}$. Therefore,
\begin{equation}\label{pre dim}
\sum_{i=1}^{r}(\nu(d_{i}-1)+\frac{\vep_{i}}{e})-[\nu(d_{i}-1)+\frac{\vep_{i}}{e}-\frac{1}{m}]=\sum_{j=1}^{m}\frac{j}{m}\dim\tt_{j}.
\end{equation}
Since the Killing form restricts to a perfect pairing between  $\tt_{j}$ and $\tt_{-j}$, we see that
\begin{equation}\label{tj}
\sum_{j=1}^{m}\frac{j}{m}\dim\tt_{j}=\frac{1}{2}\sum_{j\neq0}\dim\tt_{j}+\dim\tt_{0}=\frac{\dim\tt+\dim\tt^{\Pi(\mu_m)}}{2}.  
\end{equation}
On the other hand, applying the same reasoning as above to the $\mu_{e}$-action on $\tt$ via $\theta$ (in place of the $\mu_{m}$-action on $\tt$), we conclude that
\begin{equation}\label{unep}
\sum_{i=1}^{r}\frac{\vep_{i}}{e}=\frac{\dim\tt-\dim\tt^{\theta(\mu_{e})}}{2}.
\end{equation}
Plugging \eqref{tj} and \eqref{unep} into \eqref{pre dim}, we get
\begin{eqnarray*}
&&\dim\calP_{a}=\sum_{i=1}^{r}[\nu(d_{i}-1)+\frac{\vep_{i}}{e}-\frac{1}{m}]\\
&=&\nu\sum_{i=1}^{r}(d_{i}-1)+\frac{\dim\tt-\dim\tt^{\theta(\mu_{e})}}{2}-\frac{\dim\tt+\dim\tt^{\Pi(\mu_m)}}{2}=\frac{1}{2}(\nu\#\Phi_{\GG}-r-\dim\tt^{\Pi(\mu_m)}).
\end{eqnarray*}
\end{proof}

\begin{remark} Comparing \eqref{dim Ma} with \eqref{dim AFS}, we get
\begin{equation*}
\dim\Sp_{a}-\dim\calM_{a}=\dim\tt^{\mu_{m}}.
\end{equation*}
which is consistent with Proposition \ref{p:prod} below: $\calM_{a}$ is a disjoint union of copies of $[\Sp_{a}/\tilS_{a}]$ and $\tilS_{a}\cong\TT^{\mu_{m}}$.
\end{remark}

\subsection{Homogeneous Hitchin fibers and homogeneous affine Springer fibers}\label{ss:prod} In this subsection, we would like to link the family of homogeneous affine Springer fibers $\Sp_{\bP,\nu}$ and the family of homogeneous Hitchin fibers $\calM_{\bP,\nu}$ over $\Ah_{\nu}$, through a formula analogous to the product formula in \cite[Proposition 4.15.1]{NgoFL}.

\subsubsection{Kostant section over $V$} Let $a\in\Ah_{\nu}$. We would like to construct a $\cG_{V}$-Higgs bundle $(\calE_{V},\varphi_{V})$ with invariant $a_{V}$, the restriction of $a$ to $V$.  For this we need to make an extra assumption
\begin{equation}\label{assrho}
\mbox{The cocharacter $d\rho^{\vee}\in \xcoch(\AA^{\ad})$ lifts to $\xcoch(\AA)$.}
\end{equation}

Over $V$, $\calL$ represents a line bundle over $\tilV'$ with a $\mu_{m}$-equivariant structure. Such a $\mu_{m}$-equivariant line bundle is classified up to isomorphism by a class in $\cohog{1}{\mu_{m},\Gm}=\Hom(\mu_{m},\Gm)=\ZZ/m\ZZ$. The class of $\calL|_{V}$ is $d\mod m$, which means there is a $\mu_{m}$-equivariant isomorphism $\calL|_{\tilV'}\isom\tilV'\times\AA^{1}$ with $\mu_{m}$ acting on the $\AA^{1}$-factor via $d\nth$ power. 

The point $a_{V}$ can be viewed as a $\mu_{m}$-equivariant morphism $a_{V}:\tilV'\to\cc$, where $\mu_{m}$ acts on $\tilV'$ as usual, and $\zeta\in\mu_{m}$ acts on $c\in\cc$  by $\zeta:c\mapsto\zeta^{d}\cdot\theta(\zeta^{m/e})(c)$ (here $\cdot$ denotes the weighted action of $\Gm$ on $\cc$ with weights $\{d_{1},\cdots, d_{\rr}\}$). Consider the trivial $\cG_{V}$-torsor $\calE^{\triv}_{V}=\cG_{V}$ over $V$. A Higgs field on $\calE^{\triv}_{V}$ is a map $b:\tilV'\to \gg$ together with a 1-cocycle $\ep: \mu_{m}\to \GG$ (with $\mu_{m}$ acts on $\GG$ via $\zeta\mapsto\theta(\zeta^{m/e})$), such that for any $\zeta\in\mu_{m}$, 
\begin{equation}\label{conditionb}
b(\zeta v)=\zeta^{d}\Ad(\ep(\zeta))\theta(\zeta^{m/e})b(v). 
\end{equation}
We take $b$ to be the composition
\begin{equation*}
b:\tilV'\xrightarrow{\tila'_{V}}\cc\xrightarrow{\kappa}\ss\subset\gg.
\end{equation*}
Under the assumption \eqref{assrho}, let $\l\in\xcoch(\AA)$ lift $d\rho^{\vee}$, and we take $\ep(\zeta)=\zeta^{-\l}$. Then \eqref{conditionb} holds by \eqref{Kos Gm} and the fact that $\kappa$ commutes with pinned automorphisms. We denote the $\cG_{V}$-Higgs bundle corresponding to $b$ by $(\calE^{\triv}_{V},\varphi^{a}_{V})$.

We would like to equip $(\calE^{\triv}_{V},\varphi^{a}_{V})$ with a $\Gnu$-equivariant structure. By \eqref{Kos Gm}, we have
\begin{equation}\label{eq b}
\kappa(\tila'_{V}(s^{m/e}x))=\kappa(s^{d}\cdot \tila'_{V}(x))=s^{d}\Ad(s^{-d\rho^{\vee}})(\kappa(\tila'_{V}(x))),\hspace{1cm} s\in\Gm, x\in\tilV'.
\end{equation}
Fix a lifting $\l\in\xcoch(\AA)$ of $d\rho^{\vee}$, the above formula shows that left translation by $s^{\l}$ on the trivial $\cG_{V}$-torsor extends to an isomorphism of $\cG_{V}$-Higgs bundles $s^{*}(\calE^{\triv}_{V}, \varphi^{a}_{V})\isom(\calE^{\triv}_{V}, s^{d}\varphi^{a}_{V})$, hence giving  $(\calE^{\triv}_{V},\varphi^{a}_{V})$ a $\Gnu$-equivariant structure. %In any case $2d\rho^{\vee}\in\xcoch(\AA)$, hence replacing $\Gnu$ by the torus $\Gnu^{[2]}$, we always have a $\Gnu^{[2]}$-equivariant structure on $(\cG_{V},\varphi^{a}_{V})$. 

The above construction clearly works in families as $a$ moves over the base $\Ah_{\nu}$. We therefore get a family of $\cG_{V}$-Higgs bundles $(\calE^{\triv}_{V}\times \Ah_{\nu}, \varphi_{V})$ over $\Ah_{\nu}$ carrying a $\Gnu$-equivariant structure.

\subsubsection{The local-to-global morphism}\label{sss:lg}  We would like to define a morphism
\begin{equation}\label{lg}
\beta_{\bP,\nu}:\Sp_{\bP, \nu}\to \calM_{\bP, \nu}
\end{equation}
over $\Ah_{\nu}$ respecting the symmetries on the affine Springer fibers and the Hitchin fibers. By Lemma \ref{l:e0}, a point $a\in\Ah_{\nu}$ corresponds to a point in $\frc(F)_{\nu}^{\rs}$ by restricting to the formal neighborhood of $0$, which we still denote by $a$. 
Every point $g\bP\in\Sp_{\bP, a}$ gives a $\bP$-Higgs bundle $(\calE_{0},\varphi_{0})$  over $\Spec \hO_{0}$:  $\calE_{0}=\bP$ is the trivial $\bP$-torsor and 
$\varphi_0=\Ad(g^{-1})(\kappa(a))\in \calE_{0}\twtimes{\bP} \Lie \bP$ and $Ad(g)$ is an isomorphism $(\calE_{0},\varphi_{0})|_{\Spec\hK_{0}}\isom (G, \kappa(a))$. 
Since the restriction of $(\calE^{\triv}_{V},\varphi^{a}_{V})$ on  $\Spec\hK_{0}$ is $(G, \kappa(a))$, $\Ad(g)$
 glues $(\calE_{0},\varphi_{0})$ with $(\calE^{\triv}_{V},\varphi^{a}_{V})$ along $\Spec\hK_{0}$ and we get a morphism $\beta_{\bP,a}:\Sp_{\bP, a}\to \calM_{\bP, a}$. This construction works in families over $\Ah_{\nu}$ and gives the desired morphism $\beta_{\bP,\nu}$.

In \S\ref{sss:sym homog AFS} we have defined an action of the group ind-scheme $P_{\nu}\rtimes\Gnu$ on $\Sp_{\bP,\nu}$ over $\frc(F)^{\rs}_{\nu}$. In \S\ref{sss:sym homog Hitchin} we have defined an action of $\calP_{\nu}\rtimes\Gnu$ on $\calM_{\bP, \nu}$ over $\Ah_{\nu}$. By the construction of $P_{\nu}$ and $\calP_{\nu}$, there is a $\Gnu$-equivariant homomorphism over $\frc(F)^{\rs}_{\nu}\cong\Ah_{\nu}$
\begin{equation}\label{Plg}
P_{\nu}\to\calP_{\nu}
\end{equation}
such that $\beta_{\bP,\nu}$ is equivariant with respect to the $P_{\nu}$-action on $\Sp_{\bP,\nu}$ and the $\calP_{\nu}$-action on $\calM^{\h}_{\bP,\nu}$ via \eqref{Plg}.  Moreover, the morphism $\beta_{\bP,\nu}$ is also equivariant under $\Gnu$ by the construction of the $\Gnu$-equivariant structure on $(\cG_{V}, \varphi_{V})$. In summary, the morphism $\beta_{\bP,\nu}$ is equivariant under the$P_{\nu}\rtimes\Gnu$-action on $\Sp_{\bP,\nu}$ and the $\calP_{\nu}\rtimes\Gnu$-action on $\calM^{\h}_{\bP,\nu}$.

Recall from Section \ref{sss:family Sp} the family of homogeneous affine Springer fibers $\Sp_{\bP, \nu}$ over $\frc(F)^{\rs}_{\nu}$ with the action of the group scheme $\tilS$ over $\frc(F)^{\rs}_{\nu}$. 

\begin{prop}\label{p:prod} Let $\nu>0$ be a $\theta$-admissible slope. Fix a lifting $\l\in\xcoch(\AA)$ of $d\rho^{\vee}$.
\begin{enumerate}
\item There is a surjective morphism $\calM^{\h}_{\bP,\nu}:=\calM_{\bP}|_{\Ah_{\nu}}\to\cohog{1}{\mu_m,\TT}$ whose fibers are homeomorphic to $[\Sp_{\bP,\nu}/\tilS]$.
\item If moreover $\nu$ is elliptic, then we have a homeomorphism over $\frc(F)^{\rs}_{\nu}\cong\Ah_{\nu}$
\begin{equation}\label{ell prod family}
[\Sp_{\bP,\nu}/\tilS]\cong\calM^{\h}_{\bP,\nu}
\end{equation}
which is fiberwise equivariant under the actions of $P_{\nu}\rtimes\Gnu$ and $\calP_{\nu}\rtimes\Gnu$. Moreover, $\calP_{\nu}$  has connected fibers.
\end{enumerate}
\end{prop}
\begin{proof} To save notation, we only treat the case $\bP=\bI$. The general case is the same.

(1)  The local-global morphism $\beta_{\bP,\nu}$ together with its equivariance with respect to $P_{\nu}$ gives a morphism of stacks
\begin{equation}\label{pre prod}
\Sp_{\nu}\twtimes{P_{\nu}}\calP_\nu\isom\calM^{\h}_{\nu}.
\end{equation}
We claim that this is a homeomorphism. To check this, we only need to work fiberwise , therefore we fix $a\in \Ah_{\nu}$. The product formula \cite[2.4.2]{GS} holds in the case where $X$ is a Deligne-Mumford curve. Note that the point $\infty$ does not contribute to the product formula: since $\nu=d/m$, by Lemma \ref{l:a global}, we see that $a_{V}$, viewed a $\mu_{e}$-equivariant map $V'\to \cc$, lies entirely in $\cc^{\rs}$.  The equivariance under $\Gnu$ is already observed in \S\ref{sss:lg}.

It remains to examine how far the morphism \eqref{Plg} is from being an isomorphism.  For $a\in\Ah_{\nu}$, we have the regular centralizer group scheme $J_{a}$ over $X$. Since $a$ is regular semisimple over $V$, $J_{a,V}$ is a torus over $V$. When $a$ varies over $\Ah_{\nu}$, the torus $J_{V}$ over $\Ah_{\nu}\times V$ is the automorphism group scheme of $(\cG_{V}, \varphi)$. 

The morphism \eqref{Plg} over $a$ is obtained by gluing a $J_a$-torsor over $\Spec\hO_0$ with the trivial $J_{a,V}$-torsor over $V$. As in \S\ref{sss:Cartan}, the morphism $V\to [\cc^{\rs}/\mu_{e}]$ can also be lifted to 
\begin{equation*}
\tilV'\to \tt^{\rs}_{\barnu}
\end{equation*}
where $\tt_{\barnu}\subset\tt$ is the Cartan space defined in \S\ref{sss:Cartan}. The regular centralizer $J_{a,V}$ can be calculated using this lifting as in \S\ref{sss:Pnu}, i.e.,
\begin{equation*}
J_{a,V}\cong \tilV'\twtimes{\mu_m}\TT. 
\end{equation*}
where $\mu_{m}$ acts on $\TT$ via the homomorphism $\Pi:\mu_{m}\incl\WW'$ and the action of $\WW'$ on $\TT$. In particular, $J_{a,V}$-torsors over $V$ are the same as $\mu_{m}$-equivariant $\TT$-torsors over $\tilV'$. Any $\TT$-torsor over $\tilV'\cong\AA^1$ is trivial, and the $\mu_m$-equivariant structure on the trivial $\TT$-torsor amounts to a class in $\cohog{1}{\mu_m,\TT}$ (obtained by looking at the $\mu_{m}$ action on the fiber of the torsor at $\wt{\infty'}$). This class gives the homomorphism $\calP_{\nu}\to\cohog{1}{\mu_m,\TT}$. If a $J_{a}$-torsor over $X$ has the trivial class in $\cohog{1}{\mu_{m},\TT}$, then it comes from gluing a $J_a$-torsor over $\Spec\hO_0$ with the trivial $J_{a,V}$-torsor, and hence in the image of $P_{a}$. The kernel of $P_{a}\to \calP_{a}$ is the automorphism group of such a $J_{a}$-torsor, which is $\TT^{\Pi(\mu_m)}=\tilS_{a}$. Therefore we have an exact sequence of stacks
\begin{equation}\label{Pa seq}
1\to \tilS\to P_{\nu}\to\calP_{\nu}\to\cohog{1}{\mu_m,\TT}\to1.
\end{equation}
where the first map is $\tilS_{a}\to G^{\flat}_{\gamma}(\calO_{F})/J_{a}(\calO_{F})\subset P_{a}$ (see Lemma \ref{l:local Neron}). This sequence together with \eqref{pre prod} gives the first statement of the proposition.

(2) When $a$ is elliptic, $\cohog{1}{\mu_m,\TT}=0$ by Tate-complex calculation (or part (2) of the proof of proposition~\ref{p:smooth}). Therefore the second statement follows.

Finally we need to show that $\calP_{a}$ is connected when $a$ is elliptic. In this case, $J_{a}$ is an anisotropic torus over $F$, hence $P_{a}$ is a quotient of $J_{a}(F)=G^{\flat,\red}_{\gamma}(\calO_{F})$ whose reductive quotient is $\tilS_{a}$ by Lemma \ref{l:local Neron}\eqref{tilL}. Therefore $\tilS_{a}\to P_{a,0}$ induces a surjection on component groups, hence $\calP_{a}$ is connected by the exact sequence \eqref{Pa seq}.
\end{proof}

\begin{exam}\label{ex:PGL2} We give an example when $\GG=\PGL_{2}$. We may take $\gamma=\left(\begin{array}{cc} 0 & 1 \\ t^{d} & 0\end{array}\right)$ where $d>0$ is an odd integer, and let $a=\chi(\gamma)$, which is homogeneous of slope $\nu=d/2$. The affine Grassmannian $\Fl_{\bG}=\GG(F)/\GG(\calO_{F})$
as well as the affine Springer fiber $\Sp_{\bG, a}$ have two components (that are permuted by $\Omega_{\bI}\cong\ZZ/2\ZZ$). The regular centralizer group scheme $J_{a}$ is simply the centralizer of $\gamma$ in $\bG=\GG(\calO_{F})$. We have $J_{a}(F)=\left\{\left(\begin{array}{cc} x & y \\ t^{d}y & x\end{array}\right);x,y\in F, x^{2}-t^{d}y^{2}\neq0\right\}/F^{\times}$. The inclusion $\tilS_{a}\cong\TT^{\mu_{2}}\cong\{\pm1\}\to J_{a}(F)$ sends $-1$ to the class of $\left(\begin{array}{cc} 0 & 1 \\ t^{d} & 0\end{array}\right)$ modulo scalar matrices (which is the unique element of order two in $J_{a}(F)$). The left multiplication action of $\tilS_{a}$ on $\Sp_{\bG,a}$ permutes the two components. Therefore $\Sp_{\bG,a}/\tilS_{a}$ can be identified with one component of $\Sp_{\bG,a}$. 

One the other hand, we shall describe the geometric structure of $\calM_{\bG,a}$, where $a=-\xi^{d}\in\Gamma(X,\calO(d/2)^{2})=\calA$. A point of $\calM_{\bG,a}$ is the quotient $M_{a}/\Pic(X)$. Here $M_{a}$ classifies equivalent classes of pairs $(\calV,\varphi)$ where $\calV$ is a vector bundle of rank two on $X$, $\varphi:\calV\to\calV(d/2)$ satisfying $\varphi^{2}=\xi^{d}\id$ as a map $\calV\to\calV(d)$. The action of $\calN\in\Pic(X)$ on $M_{a}$ sends $(\calV,\varphi)$ to $(\calN\otimes\calV,\id_{\calN}\otimes\varphi)$.

Any vector bundle on $X$ is a direct sum of line bundles. By tensoring with $\Pic(X)$, we may assume $\calV=\calO\oplus\calO(n/2)$ for $n\in\ZZ_{\geq 0}$. The map $\varphi:\calO\oplus\calO(n/2)\to \calO(d/2)\oplus\calO((n+d)/2)$ can be written as a matrix $\left(\begin{array}{cc} x & y \\ z & -x\end{array}\right)$, where $x\in \CC[\xi,\eta]_{d}, y\in \CC[\xi,\eta]_{d-n}, z\in \CC[\xi,\eta]_{d+n}$ and $x^{2}+yz=\xi^{d}$. This equation forces $n$ to be odd, for otherwise all $x,y,z$ will be divisible by $\eta$. Also $1\leq n\leq d$. Let $S_{n}$ be the space of such matrices. Once we fix $x$, there is a finite number of choices for $y$ and $z$ up to scalar, hence $\dim S_{n}=1+\dim \CC[\xi,\eta]_{d}=(d+3)/2$. Let $H_{n}=\Aut(\calO\oplus\calO(n/2))/\Gm$, which acts on $S_{n}$ by conjugation. Every matrix in $H_{n}$ can be represented uniquely by $\left(\begin{array}{cc} x & 0 \\ y & 1\end{array}\right)$ with $x\in \CC^{\times}, y\in \CC[\xi,\eta]_{n}$. Therefore $\dim H_{n}=(n+3)/2$. By our discussion above, $\calM_{a}$ admits a stratification $\sqcup_{1\leq n\leq d, n\textup{ odd}} S_{n}/H_{n}$ with $\dim S_{n}/H_{n}=(d-n)/2$. The top dimensional stratum is $S_{1}/H_{1}$, which has dimension $(d-1)/2$. One can check for small values of $d$ that $S_{1}$ is irreducible, therefore $\calM_{a}$ only has one top-dimensional component.
\end{exam}

\begin{remark} If we consider $\GG=\SL_{2}$ instead of $\PGL_{2}$, and take the same $a$ as in Example \ref{ex:PGL2}, the affine Springer fiber $\Sp_{\bG,a}$ is isomorphic to a connected component of its $\PGL_{2}$ counterpart. However, the Hitchin fiber $\calM_{\bG,a}$ in this case is empty! In fact, suppose $(\calV,\varphi)\in\calM_{\bG,a}$, then $\calV\cong\calO(-n/2)\oplus\calO(n/2)$ for some integer $n\geq0$, and $\varphi$ can be written as a matrix $\left(\begin{array}{cc} x & y \\ z & -x\end{array}\right)$, where $x\in \CC[\xi,\eta]_{d}, y\in \CC[\xi,\eta]_{d-2n}, z\in \CC[\xi,\eta]_{d+2n}$ and $x^{2}+yz=\xi^{d}$. Since $d$ is odd and $\xi$ has degree $2$, the degrees of $x^{2}$ and $yz$ in $\xi$ are both strictly smaller than $d$, and $x^{2}+yz=\xi^{d}$ has no solution. This does not contradict Proposition \ref{p:prod} because in this case the assumption \eqref{assrho} fails.  
\end{remark}

\part{Representations}

%  Actions of the graded Cherednik algebra

\section{Geometric modules of the graded \Cha}\label{s:Hgr}
In this section, we construct an action of the graded \Cha with central charge $\nu$ on the $\Gnu$-equivariant cohomology of both homogeneous affine Springer fibers and homogeneous Hitchin fibers.

Since affine Springer fibers and Hitchin fibers are only locally of finite type, we make some remarks on the definition of their (co)homology. Let $f:M\to S$ be a morphism of Deligne-Mumford stacks over $\CC$. Suppose $M$ can be written as a union of locally closed substacks $M=\cup_{i\in I}M_{i}$ (for some filtered set $I$) such that each $f_{i}: M_{i}\to S$ is of finite type. In this case $f_{!}\DD_{M/S}$ is defined to be an ind-object $\varinjlim_{i}f_{i,!}\DD_{M_{i}/S}$ of $D^{b}(S)$. If $M_{i}$ are open in $M$, we may define $f_{!}\QQ$ to be the ind-object $\varinjlim_{i}f_{i,!}\QQ$. When $S=\Spec\CC$ we denote $f_{!}\DD_{M}$ by $\homog{*}{M}$ and denote $f_{!}\QQ$ by $\cohoc{*}{M}$ (when defined), which are usual graded vector spaces.  Similar notation applies to the equivariant situation. For details, see \cite[Appendix A]{GSSph}. Affine Springer fibers and Hitchin fibration both can be covered by finite-type open substacks, and the above notion makes sense. In the following we shall work with compactly supported cohomology, although similar statements for homology also hold with the same proof.

\subsection{The $\Hgr$-action on the cohomology of homogeneous affine Springer fibers}\label{ss:localHgr} Recall the family of homogeneous affine Springer fibers $q_{\nu}: \Sp_{\nu}\to \frc(F)^{\rs}_{\nu}$. We will construct an action of $\Hgr_{\nu}$ on the ind-complex $q_{!}\QQ$. We will construct the actions of $\fra^{*}_{\KM}, s_{i}$ and $\Omega_{\bI}$ separately, and then check these actions satisfy the relations in $\Hgr_{\nu}$.

\begin{cons}\label{gr ep} The action of $\ep$ is via $m$ times the generator of $\xch(\Gnu)$, viewed as an element in $\eqnu{2}{\pt}$.
\end{cons}

\begin{cons}\label{gr Chern} We shall construct the action of $\fra_{\KM}^{*}$. In \S\ref{sss:line bundle} we have assigned for each $\xi\in\xch(\AA_{\KM})$ a line bundle $\calL(\xi)$ on $\Fl$ carrying a canonical $G_{\KM}$-equivariant structure.  By \eqref{s act fl}, the action of $\Gnu$ on $\Fl$ is obtained by restricting the left translation action of $G_{\KM}$ on $\Fl$ via the homomorphism $\Gnu\incl\AA\times\Grote\subset G_{\KM}$ given by the cocharacter $d\rho^{\vee}+m\partial$. Therefore, each $\calL(\xi)$ carries a canonical $\Gnu$-equivariant structure. The action of $\xi$ on $ q_{\nu,!}\QQ$ is given by the cup product with the equivariant Chern class $c^{\Gnu}_{1}(\calL(\xi)):  q_{\nu,!}\QQ\to  q_{\nu,!}\QQ[2](1)$.
\end{cons}

\begin{cons}\label{gr si} We shall construct the action of a simple reflection $s_{i}$ on $ q_{\nu,!}\QQ$. Before doing this we need some more notation. For a standard parahoric $\bP$, let $L_{\bP}$ be its Levi factor and let $\pi_{L_{\bP}}: \bP\to L_{\bP}$ be the projection. Let $\frl_{\bP}=\Lie L_{\bP}$. Let $\pi_{\frl_{\bP}}:\Lie\bP\to\frl_{\bP}$ be the projection map on the level of Lie algebras.  

There is an evaluation map $\ev_{\bP}:\Sp_{\bP,\gamma}\to[\frl_{\bP}/L_{\bP}]$ given by $g\bP\mapsto \pi_{\frl_{\bP}}(\Ad(g^{-1})\gamma)\mod L_{\bP}$. Let $\frb^{\bI}_{\bP}$ (resp. $B^{\bI}_{\bP}$) be the image of $\Lie\bI$ (resp. $\bI$) under the projection $\pi_{\frl_{\bP}}$ (resp. $\pi_{L_{\bP}}$), which is a Borel subalgebra of $\frl_{\bP}$ (resp . Borel subgroup of $L_{\bP}$). We also have an evaluation map $\ev^{\bI}_{\bP}:\Sp_{\gamma}\to[\frb^{\bI}_{\bP}/B^{\bI}_{\bP}]$ given by $g\bI\mapsto\pi_{\frl_{\bP}}(\Ad(g^{-1})\gamma)\mod B^{\bI}_{\bP}$. The two evaluation maps fit into a Cartesian diagram
\begin{equation}\label{ev Sp}
\xymatrix{\Sp_{\nu}\ar[r]^{\ev^{\bI}_{\bP}}\ar[d]^{\pi_{\bP,\nu}} & [\frb^{\bI}_{\bP}/B^{\bI}_{\bP}]\ar[d]^{p_{\frl_{\bP}}}\\
\Sp_{\bP,\nu}\ar[r]^{\ev_{\bP}} & [\frl_{\bP}/L_{\bP}]}
\end{equation}
The morphism $p_{\frl_{\bP}}$ is in fact the Grothendieck simultaneous resolution of $\frl_{\bP}$ (quotient by the adjoint action of $L_{\bP}$). Let $\bP_{i}$ be the standard parahoric whose Levi factor $L_{i}$ has root system $\pm\alpha_{i}$. Applying proper base change to the Cartesian diagram \eqref{ev Sp}, we get an isomorphism in the equivariant derived category $D^{b}_{\Gnu}(\Sp_{\bP_{i},\nu})$:
\begin{equation*}
\pi_{\bP,\nu,*}\QQ\cong \ev_{\bP}^{*}p_{\frl_i,*}\QQ.
\end{equation*}
By the Springer theory for the reductive Lie algebra $\frl_{i}$, the complex $p_{\frl_i,*}\QQ\in D^{b}_{L_i}(\frl_i)$ admits an action of an involution $s_{i}$. Therefore $ q_{\nu,!}\QQ\cong q_{\bP_{i},\nu,!}\ev_{\bP}^{*}\pi_{\bP,\nu,*}\QQ$ also admits an action of $s_{i}$.
\end{cons}

\begin{cons}\label{gr Omega} We construct the action of $\Omega_{\bI}$ on $\Sp_{\nu}$ commuting with the $\Gnu$-action, which then induces an action on $ q_{\nu,!}\QQ$. In fact, for each $\omega\in\Omega_{\bI}=N_{G}(\bI)/\bI$, pick a representative $\dot{\omega}\in G$. Then right multiplication by $\dot{\omega}: g\bI\mapsto g\dot{\omega}\bI$ is an automorphism on $\Fl$ which preserves the family $\Sp_{\nu}$ fiberwise. This automorphism is independent of the choice of $\dot{\omega}$ in the coset of $\omega$, and defines an action of $\Omega_{\bI}$ on $\Sp_{\nu}$.
\end{cons}

\begin{theorem}\label{th:localHgr}
Constructions \eqref{gr ep}--\eqref{gr Omega} define an action of $\Hgr_{\nu}$ on  $ q_{\nu,!}\QQ\in D^{b}_{\Gnu}(\frc(F)^{\rs}_{\nu})$. Moreover, the element $\ep\in\Hgr_{\nu}$ acts as $m$ times the equivariant parameter from $\eqnu{2}{\pt}$; the action of $B_{\KM}$ on $ q_{\nu,!}\QQ$ is through the multiplication by $\nu^{2}\BB(\rho^{\vee},\rho^{\vee})\ep^{2}\in\eqnu{4}{\pt}$.
\end{theorem}
\begin{proof} First, we check (GC-2): the Constructions \eqref{gr si} and \eqref{gr Omega} give an action of $\tilW$ on $q_{\nu,!}\QQ$. This is done in the same way as in \cite{GSSph}, using the diagram \eqref{ev Sp}.

Next, we check (GC-3) for $\xi\in\fra^{*}$. Using the diagram \eqref{ev Sp} for $\bP=\bP_{i}$, we reduce to a calculation for  groups of semisimple rank one, and can be checked in the same way as in \cite{GS}.

Next, we check (GC-3) for $\xi=\L_{\can}$. The idea is similar to that of \cite[\S 6.5]{GS}. Recall from \S\ref{sss:det} that each $\Fl_{\bP}$ carries a determinant line bundle $\det_{\bP}$. In view of \eqref{comp GP}, we define
\begin{equation*}
\L_{\bP}=\L_{\can}+2\rho_{\bG}-2\rho_{\bP}+\jiao{\nu\rho^{\vee}+\partial, 2\rho_{\bP}-2\rho_{\bG}}\ep.
\end{equation*}
Then we have $\calL(\L_{\bP})\cong\pi_{\bP}^{*}\det_{\bP}$ as $\Gnu$-equivariant line bundles.
Since we have already checked that (GC-3) for $\xi\in\fra^{*}\oplus\QQ\ep$, by the linearity of the relation (GC-3), it suffices to check for each affine simple reflection $s_{i}$ and the particular element $\L_{i}:=\L_{\bP_{i}}$ that
\begin{equation}\label{siLi}
s_{i}\L_{i}-\leftexp{s_{i}}{\L_{i}}s_{i}=\jiao{\L_{i},\alpha_{i}^{\vee}}u.
\end{equation} 
By Proposition \ref{p:pairing degree},
\begin{equation*}
\jiao{\L_{i},\alpha_{i}^{\vee}}=-\deg(\calL(\L_{i})|C_{i})=-\deg(\pi_{\bP_{i}}^{*}{\det}_{\bP_{i}}|_{C_{i}})=0,
\end{equation*}
hence the right side of  \eqref{siLi} is zero. This also implies that $\leftexp{s_{i}}{\L_{i}}=\L_{i}$. On the other hand, from the construction of the $s_{i}$ action in \eqref{gr si}, cupping with any class pulled back from $\Sp_{\bP_{i},\nu}$ (and in particular $c_{1}^{\Gnu}(\det_{\bP_{i}})$)  commutes with the action of $s_{i}$ on $ q_{\nu,!}\QQ$. Therefore we have $s_{i}\L_{i}=\L_{i}s_{i}=\leftexp{s_{i}}{\L_{i}}$, and \eqref{siLi} is proved. This finishes the verification for (GC-3).

We check (GC-4). Using the construction of $\calL(\xi)$ in Construction \eqref{gr Chern}, the desired relation follows from the commutation relation between $N_{G_{\KM}}(\bI_{\KM})$ and $\AA_{\KM}$.

Finally, we show that $B_{\KM}\in\Sym^{2}(\fra^{*}_{\KM})^{\tilW}\subset \Hgr_{\nu}$ acts as $\nu^{2}\BB(\rho^{\vee},\rho^{\vee})\ep^{2}\in\eqnu{4}{\pt}$. By Lemma \ref{l:B}, we know that the action of $B_{\KM}$ on $q_{\nu,!}\QQ$ is the same as the cup product with the image of $B_{\KM}$ under the restriction map $\upH^{4}_{\AA_{\KM}}(\pt)\to\upH^{2}_{\Gnu}(\pt)$. By \eqref{s act fl}, the embedding $\Gnu\incl\AA_{\KM}$ is given by the cocharacter $d\rho^{\vee}+m\partial=m(\nu\rho^{\vee}+\partial)$. Therefore, the action of $B_{\KM}$ on $q_{\nu,!}\QQ$ is via $B_{\KM}(\nu\rho^{\vee}+\partial, \nu\rho^{\vee}+\partial)\ep^{2}=\nu^{2}\BB(\rho^{\vee},\rho^{\vee})\ep^{2}\in\eqnu{4}{\pt}$.
\end{proof}

Taking stalks at $a\in \frc(F)^{\rs}_{\nu}$, we get the following corollary.

\begin{cor}\label{c:local Hgr} For $a\in\frc(F)^{\rs}_{\nu}$, there is a graded action of $\Hgr_{\nu}/(B_{\KM}-\nu^{2}\BB(\rho^{\vee},\rho^{\vee})\ep^{2})$ on $\upH^{*}_{c,\Gnu}(\Sp_{a})$ that commutes with the action of $\pi_0(\tilS_{a})\rtimes B_{a}$. Specializing to $\ep=1$, there is an action of $\Hgr_{\nu,\ep=1}$ on $\upH_{c,\ep=1}(\Sp_{a})$, commuting with the action of  $\pi_0(\tilS_{a})\rtimes B_{a}$ and compatible with the cohomological filtrations on $\Hgr_{\nu,\ep=1}$ and $\upH_{c,\ep=1}(\Sp_{a})$.
\end{cor}

\subsection{The polynomial representation of $\Hgr$} In this subsection, we shall show that the equivariant cohomology of the whole affine flag variety $\Fl$ gives a geometric model for the polynomial representation of $\Hgr$. 

The construction of the $\Hgr$-action on the equivariant cohomology of homogeneous affine Springer fibers in \S\ref{ss:localHgr} in fact works line by line for $\Fl$ as well. We thus get

\begin{prop}\label{p:Hgr Fl} Suppose $\nu>0$ is an elliptic slope. There is an action of $\Hgr_{\nu}/(B_{\KM}-\nu^{2}\BB(\rho^{\vee},\rho^{\vee})\ep^{2})$ on  $\eqnu{*}{\Fl}$ such that the restriction map $\eqnu{*}{\Fl}\to \eqnu{*}{\Sp_{a}}=\upH^{*}_{c,\Gnu}(\Sp_{a})$ is $\Hgr_{\nu}$-equivariant for any $a\in\frc(F)^{\rs}_{\nu}$.
\end{prop}

Let $\barHgr_{\nu}=\Hgr_{\nu}/(B_{\KM}-\nu^{2}\BB(\rho^{\vee},\rho^{\vee})\ep^{2})$. The connected components of $\Fl$ is indexed by $\Omega_{\bI}$. Let $\Fl^{\circ}$ be the ``neutral'' component (containing the coset $\bI$). The class $1\in\eqnu{0}{\Fl^{\circ}}$ is invariant under $\Wa$. We therefore get a natural map of $\Hgr_{\nu}$-modules
\begin{equation}\label{pol to fl}
\Ind^{\barHgr_{\nu}}_{\QQ[\ep]\otimes\QQ[\Wa]}(\QQ[\ep])\to\eqnu{*}{\Fl}.
\end{equation}
The induced module on the left side is called the {\em polynomial representation} of $\barHgr_{\nu}$.

\begin{lemma}\label{l:pol rep} The map \eqref{pol to fl} is an isomorphism after inverting $\ep$ (or equivalently, it is an isomorphism after specializing $\ep$ to 1).
\end{lemma}
\begin{proof} We abuse notation to use $\fra$ to mean $\fra_{\QQ}$. Ignoring the $\Hgr_{\nu}$-module structure, the left side of \eqref{pol to fl} is isomorphic to $\QQ[\ep]\otimes\Sym(\fra^{*})\otimes\QQ[\Omega_{\bI}]$. Since $\Omega_{\bI}$ permutes the components of $\Fl$ simply-transitively, it suffices to show that $\QQ[\ep,\ep^{-1}]\otimes\Sym(\fra^{*})\to\eqnu{*}{\Fl^{\circ}}[\ep^{-1}]$ is an isomorphism, where $\Fl^{\circ}$ is the identity component of $\Fl$. This map is induced from the map $\fra^{*}\to\eqnu{2}{\Fl}$ by sending $\xi\in\xch(\AA)$ to the equivariant Chern class $c_{1}^{\Gnu}(\calL(\xi))$. Since we only consider one component of $\Fl$, we may assume $\GG$ is simply-connected and hence $\Fl$ is connected.

We first recall the ring structure of $\upH^{*}_{\AA\times\Grote}(\Fl)$. We have $\upH^{*}_{\AA\times\Grote}(\pt)\cong\QQ[\delta]\otimes\Sym(\fra^{*})$. The equivariant Chern classes of the line bundles $\calL(\xi)$ for $\xi\in\xch(\AA)$ give a map of $\upH^{*}_{\AA\times\Grote}(\pt)$-algebras 
\begin{equation*}
\upH^{*}_{\AA\times\Grote}(\pt)\otimes\Sym(\fra^{*})\cong\QQ[\delta]\otimes\Sym(\fra^{*})\otimes\Sym(\fra^{*})\to\upH^{*}_{\AA\times\Grote}(\Fl).
\end{equation*}
Let $f_{1},\cdots, f_{r}$ be the generators of the invariant ring $\Sym(\fra^{*})^{W}$, then $\ell(f_{i})-r(f_{i})$ are divisible by $\delta$ in $\upH^{*}_{\AA\times\Grote}(\Fl)$. (Here we follow the notation preceding Lemma \ref{l:B}  to distinguish the left and right copies of $\Sym(\fra^{*})$). Write $\ell(f_{i})-r(f_{i})=\delta e_{i}$ for a unique $e_{i}\in\upH^{2d_{i}-2}_{\AA\times\Grote}(\Fl)$. Then we have a $\upH^{*}_{\AA\times\Grote}(\pt)$-algebra isomorphism
\begin{equation*}
\upH^{*}_{\AA\times\Grote}(\Fl)\cong\QQ[\delta]\otimes\Sym(\fra^{*})\otimes\Sym(\fra^{*})\otimes\QQ[e_{1},\cdots,e_{r}]/(\ell(f_{i})-r(f_{i})-\delta e_{i}; i=1,\cdots,r). 
\end{equation*}
If we invert $\delta$, the generators $e_{i}$ become redundant, and we get an isomorphism of $\upH^{*}_{\AA\times\Grote}(\pt)[\delta^{-1}]$-module
\begin{equation}\label{arot isom}
\upH^{*}_{\AA\times\Grote}(\pt)[\delta^{-1}]\otimes\Sym(\fra^{*})\isom\upH^{*}_{\AA\times\Grote}(\Fl)[\delta^{-1}].
\end{equation}
Restricting to the subtorus $\Gnu\subset\AA\times\Grote$, we get an algebra homomorphism $\upH^{*}_{\AA\times\Grote}(\pt)[\delta^{-1}]\to\eqnu{*}{\pt}[\ep^{-1}]$ (because $\Gnu$ has a nontrivial projection to $\Grot$). Base change \eqref{arot isom} using this algebra homomorphism, we get an isomorphism
\begin{equation*}
\eqnu{*}{\pt}[\ep^{-1}]\otimes\Sym(\fra^{*})\isom\eqnu{*}{\Fl}[\ep^{-1}].
\end{equation*}
\end{proof}

\subsection{The global sheaf-theoretic action of $\Hgr$}\label{ss:Hgr}
In \cite[Construction 6.1.4]{GS}, we defined an action of the graded \Cha $\Hgr$ on the derived image of the constant sheaf along the parabolic Hitchin fibration. In this subsection, we will extend this construction to the situation where the curve is the weighted projective line $X$ and we shall keep track of the $\Gtwo$-equivariant structures.

We first describe the actions of the generators $u,\delta,\fra^{*}, \L_{\can}, s_{i}$ and $\Omega_{\bI}$ of $\Hgr$ act on $f_{!}\QQ$, viewed as an ind-object in the equivariant derived category $D^{b}_{\Gtwo}(\Ah)$.

\begin{cons}\label{gr global u delta} The equivariant cohomology $\upH^{*}_{\Gtwo}(\pt)$ acts on every object in the category $D^{b}_{\Gtwo}(\Ah)$. In particular, every element in $\xch(\Gtwo)\subset\upH^{2}_{\Gtwo}(\pt)$ gives a map $f_{!}\QQ\to f_{!}\QQ[2](1)$. This gives the action of $\delta\in\xch(\Grote)$ and $u\in\xch(\Gdil)$. 
\end{cons}

\begin{cons}\label{gr global Chern} The action of $\xi\in\xch(\AA_{\KM})$. Let $\Bun_{\bI}$ be the moduli stack of $\cG_{\bI}$-torsors over $X$. As in \cite{GS}, we may write $\Bun_{\bI}$ as a fppf quotient $\wh{\Bun}/\bI_{\KM}$, where $\wh{\Bun}$ classifies $(\calE,t,\tau,\alpha)$: $\calE$ is a $\cG$-torsor over $X$; $t$ is a coordinate at $0\in X$ which is a constant multiple of the canonical coordinate $\xi/\eta^{m}$; $\tau$ is a trivialization of $\calE|_{\Spec\hO_{0}}$ and $\alpha$ is a trivialization of $\det\bR\Gamma(X,\Ad(\calE))$. The moduli space $\wh{\Bun}$ admits a right action of $G_{\KM}$. Therefore it makes sense to quotient by $\bI_{\KM}$. Hence, for each $\xi\in\xch(\AA_{\KM})$, we have a $\Grote$-equivariant line bundle $\calL(\xi)=\wh{\Bun}\twtimes{\bI_{\KM},\xi}\AA^{1}$ over $\Bun_{\bI}$. Pulling back to $\calM$, we get a $\Gtwo$-equivariant line bundle over $\calM$ which we still denote by $\calL(\xi)$. The action of $\xi$ on $f_{!}\QQ$ is the cup product with the equivariant Chern class of $\calL_{\xi}$:
\begin{equation}\label{cup xi}
f_!(\cup c^{\Gtwo}_1(\calL_{\xi})):f_{!}\QQ\to f_{!}\QQ[2](1).
\end{equation}
\end{cons}

\begin{remark} The line bundle $\calL(\delta)$ constructed above is the trivial line bundle over $\calM$ on which $\Gtwo$ act via the character $\delta$. Therefore, the actions of $\delta$ constructed in Constructions \eqref{gr global u delta} and \eqref{gr global Chern} are the same.
\end{remark}

\begin{remark}
When $\xi=\L_{\can}\in\xch(\AA_{\KM})$, the fiber of the line bundle $\calL(\L_{\can})$ at $(\calE_{\bI},\varphi)\in\calM$ is the determinant line $\det\bR\Gamma(X,\Ad(\calE_{\bG}))$, where $\calE_{\bG}$ is the $\cG$-torsor induced from the $\cG_{\bI}$-torsor $\calE_{\bI}$. Since the cotangent complex of $\Bun_{\cG}$ at $\calE$ is given by the dual of $\bR\Gamma(X,\Ad(\calE))[1]$, we see that $\calL(\L_{\can})$ is the pull back of the canonical bundle $\omega_{\Bun_{\cG}}$ of $\Bun_{\cG}$.
\end{remark}

\begin{cons}\label{gr global si} The construction of the action of $s_{i}$ is based on the evaluation diagram (cf. \cite[Diagram (4.6)]{GS}). The maps $\ev^{\bI}_{\bP}$ and $\ev_{\bP}$ are the evaluation of Higgs fields at $0$.
\begin{equation}\label{evdiag}
\xymatrix{\calM\ar[r]^(.3){\ev^{\bI}_{\bP}}\ar[d]^{\pi^{\bI}_{\bP}} & [\frb^{\bI}_{\bP}/B^{\bI}_{\bP}]\ar[d]^{p_{\frl_{\bP}}}\\
\calM_{\bP}\ar[r]^(.3){\ev_{\bP}} & [\frl_{\bP}/L_{\bP}]}
\end{equation}
Here the evaluation maps uses the $\Grote$-equivariant trivialization of $\calL|_{\Spec\hO_{0}}$. The morphisms $\ev^{\bI}_{\bP}$ and $\ev_{\bP}$ are $\Gtwo$-equivariant (the action of $\Gdil$ on $\frl_{\bP}$ and $\frb^{\bI}_{\bP}$ are via the inverse of the scaling action). The rest of the construction is the same as Construction \eqref{gr si} for affine Springer fibers. The $\Gtwo$-equivariance of the diagram ensures that the $s_{i}$-action takes place in $D^{b}_{\Gtwo}(\Ah)$. 
\end{cons}

\begin{cons}\label{gr global Omega} There is an action of  $\Omega_{\bI}$ on $\calM$. In fact, writing $\Bun_{\bI}$ as $\wh{\Bun}/\bI_{\KM}$, we see that $\Omega_{\bI}=N_{G_{\KM}}(\bI_{\KM})$ acts from the right. This action preserves the formation of $\Gamma(X,\Ad_{\bI}(\calE))$, and hence lifts to an action of $\Omega_{\bI}$ on $\calM$. This action commutes with action of $\Gtwo$ and preserves the Hitchin fibration. Therefore this action induces an action of $\Omega_{\bI}$ on $f_{!}\QQ$ in the category $D^{b}_{\Gtwo}(\Ah)$. 
\end{cons}

\begin{theorem}\label{th:Hgr} 
\begin{enumerate}
\item Constructions \eqref{gr global u delta}--\eqref{gr global Omega} define an action of $\Hgr$ on the ind-complex $f_{!}\QQ$ in $D^{b}_{\Gtwo}(\Ah)$.
\item Restricting the above construction to $\Ah_{\nu}$ defines an action of $\Hgr$ on the ind-complexes $f_{\nu,!}\QQ$ in $D^{b}_{\Gnu}(\Ah_{\nu})$. This action factors through $\Hgr_{\nu}$.
\end{enumerate}
\end{theorem}
\begin{proof} The proof of (1) is almost the same as Theorem \ref{th:localHgr}. The only difference is that the determinant line bundle $\det_{\bP}$ in the proof there should be replaced by its global analogue: the line bundle $\det_{\Bun_{\bP}}$ over $\Bun_{\bP}$ whose fiber over the $\cG_{\bP}$-torsor $\calE$ is $\det\bR\Gamma(X,\Ad_{\bP}(\calE))$. For (2) we only need to check that both $\delta/m$ and $-u/d$ act as the canonical equivariant parameter in $\eqnu{2}{\pt}$. This is true because the embedding $\Gnu\incl\Gtwo$ is given by the cocharacter $(m/e,-d)$.
\end{proof}

\subsection{Local-global compatibility} We have constructed actions of $\Hgr_{\nu}$ on both the equivariant cohomology of homogeneous affine Springer fibers and homogeneous Hitchin fibers in the previous subsections. The two fibers are related by the local-global morphism \eqref{lg}, and the two actions are compatible in the following sense.

\begin{prop}\label{p:lgHgr} Let $\nu>0$ be an admissible slope and assume \eqref{assrho} holds. The morphism of ind-complexes over $\Ah_{\nu}\cong\frc(F)^{\rs}_{\nu}$
\begin{equation}\label{coho lg}
f_{\nu, !}\QQ\to   q_{\nu, !}\QQ
\end{equation}
induced by the local-global morphism $\beta_{\bI}$ in \eqref{lg} is a morphism of $\Hgr_{\nu}$-modules. If $\nu$ is elliptic, the morphism \eqref{coho lg} identifies $f_{\nu,!}\QQ$ with the $S$-invariants of  $q_{\nu, !}\QQ$. 
\end{prop}
\begin{proof} We only need to check that the actions of $\ep,\fra^{*}, \L_{\can}, s_{i}$ and $\Omega_{\bI}$ in both $\eqnu{*}{\calM_{a}}$ and $\eqnu{*}{\Sp_{\gamma}}$ are respected by the map \eqref{coho lg}.

The action of $\ep$ in both cases are via $m$ times the equivariant parameter in $\eqnu{2}{\pt}$. The action of $s_{i}$ is defined by the diagrams \eqref{ev Sp} and \eqref{evdiag}.  The required statement follows from the commutative of the diagram
\begin{equation*}
\xymatrix{\Sp_{\nu}\ar[r]^{\beta_{\bI}}\ar[d]^{\pi_{\bP}} & \calM_{\nu}\ar[r]^(.3){\ev^{\bI}_{\bP}}\ar[d]^{\pi_{\bP,\nu}} & [\frb^{\bI}_{\bP}/B^{\bI}_{\bP}]\ar[d]^{p_{\frl_{\bP}}}\\
\Sp_{\bP, \nu}\ar[r]^{\beta_{\bP}} & \calM_{\bP, \nu}\ar[r]^(.3){\ev_{\bP}} & [\frl_{\bP}/L_{\bP}]}
\end{equation*}
The rest of the argument requires a study of the local-global morphism
\begin{equation}\label{Bun lg}
\beta:\Fl\to\Bun_{\bI}
\end{equation}
which we define now. Recall we can write $\Bun_{\bI}=\wh{\Bun}/\bI_{\KM}$, and $\wh{\Bun}$ admits a right action by $G_{\KM}$. The trivial $\cG_{\bI}$-torsor together with the standard trivialization on $\Spec\hO_{0}$, the standard local coordinate at $0$ and a {\em choice} of a trivialization of $\det\bR\Gamma(X,\Lie\cG)$ gives a point $\pt\incl \wh{\Bun}$. Taking the $G_{\KM}$-orbit gives $G_{\KM}\to \wh{\Bun}$, and, passing to the quotient by $\bI_{\KM}$ we get $\beta:\Fl\to\Bun_{\bI}$. Concretely, $\beta(g\bI)$ is the $\cG$-torsor obtained by gluing the trivial $\bI$-torsor over $\Spec\hO_{0}$ and the trivial $\cG_{V}$-torsor over $V$ via the isomorphism $G(F)=\bI(F)\isom \cG_{V}(F)=G(F)$ given by left multiplication by $g$. From this description, we see that $\beta$ is invariant under the left translation by the automorphism group of the trivial $\cG_{V}$-torsor, namely $\cG(\calO_{V})=\GG(\calO_{\tilV'})^{\mu_{m}}$. Hence $\beta$ descends to a morphism
\begin{equation*}
\overline{\beta}:\cG(\calO_{V})\backslash\Fl\to\Bun_{\bI}
\end{equation*}

The morphism $\beta$ is clearly $\Grote$-equivariant. Since $\AA\subset\cG(\calO_{V})$, $\beta$ descends to $\AA\backslash\Fl\to\Bun_{\bI}$. Since the $\Gnu$-action on $\Fl$ is given by the embedding $d\rho^{\vee}+m\partial\in\AA\times\Grote$, we conclude that $\beta$ is in fact $\Gnu$-equivariant, where $\Gnu$ acts on $\Fl$ via \eqref{s act fl} and acts on $\Bun_{\bI}$ through its projection to $\Grote$. Since the construction of $\beta_{\bI,a}$ in Section \ref{sss:lg} is by gluing with the trivial $\cG$-torsor over $V$, we have a commutative diagram
\begin{equation*}
\xymatrix{\Sp_{\nu}\ar[d]\ar[r]^{\beta_{\bI}} & \calM_{\nu}\ar[d]\\ 
\Fl\ar[r]^{\beta} & \Bun_{\bI}}
\end{equation*}

For $\xi\in\xch(\AA_{\KM})$, the action of $\xi$ in both cases are via the Chern class $\calL(\xi)$ on $\Fl$ and $\calL_{\Bun}(\xi)$ on $\Bun_{\bI}$. By the definition of the morphism $\beta$, we have $\beta^{*}\calL_{\Bun}(\xi)\cong\calL(\xi)$ in a $\Gnu$-equivariant way (note that the definition of a base point $\pt\in\wh{\Bun}$ involves a choice of the determinant line $\det\bR\Gamma(X,\Lie\cG)$, which is unique up to a scalar, therefore the isomorphism $\beta^{*}\calL_{\Bun}(\L_{\can})\cong\calL(\L_{\can})$ depends on this choice, and is canonical only up to a scalar). This implies that \eqref{coho lg} respects the $\xi$-action.

Finally the morphism $\beta_{\bI}$ is equivariant under $\Omega_{\bI}$ by construction. Therefore the $\Omega_{\bI}$-action is also respected by \eqref{coho lg}.
\end{proof}

\begin{remark}\label{r:ellHgr} When $\nu$ is elliptic, by Proposition \ref{p:lgHgr} and Theorem \ref{th:localHgr}, $B_{\KM}$ acts on $ f_{\nu,!}\QQ$ as $\nu^{2}\BB(\rho^{\vee},\rho^{\vee})\ep^{2}$. We expect the same to be true without the assuming $\nu$ is elliptic.
\end{remark}

%%% Rational DAHA

\section{Geometric modules of the rational \Cha}\label{s:Hrat}
In this section, we fix a $\theta$-admissible  {\em elliptic} slope $\nu>0$. We will construct an action of the rational \Cha $\Hrat_{\nu}$ on a modification of the $\Gnu$-equivariant cohomology of homogeneous affine Springer fibers. When $G$ is split, we will show that such a construction gives the irreducible spherical module $\frL_{\nu}(\triv)$ of $\Hrat_{\nu,\ep=1}$ using the geometry of homogeneous Hitchin fibers.

%In this section, we assume that $\GG$ is almost simple and {\em simply-connected}. In fact, allowing non-simply-connected $\GG$ will not produce more modules of the rational Cherednik algebra.  

One notational change is that we will identify $\tilS_{a}$ with $\pi_{0}(\tilS_{a})$, because when $m$ is an elliptic number, $\tilS_{a}$ is itself a finite abelian group.

In \S\ref{ss:pol Hrat}-\S\ref{ss:Sp Hrat}, we work in the generality of quasi-split groups as in \S\ref{sss:gp}. From \S\ref{ss:perv fil} we will assume that $e=1$, i.e., $G$ is split over $F$. 

\subsection{The polynomial representation of $\Hrat$}\label{ss:pol Hrat} %Since $\GG$ is assumed to be simply-connected, the affine flag variety $\Fl$ is connected.  
Let $\Fl^{\c}$ be the neutral component of the affine flag variety of $G$ (which is the same as the affine fiber variety of the simply-connected cover of $G$). Lemma \ref{l:pol rep} gives a canonical algebra isomorphism
\begin{equation}\label{pol inv ep}
\QQ[\ep,\ep^{-1}]\otimes_{\QQ}\Sym(\fra^{*})\isom\eqnu{*}{\Fl^{\c}}[\ep^{-1}]
\end{equation}
given be sending $\xi\in\xch(\AA)$ to the $\Gnu$-equivariant Chern classes of $\calL(\xi)$.

\begin{defn} The {\em Chern filtration} $C_{\leq i}\eqnu{*}{\Fl^{\c}}[\ep^{-1}]$ on $\eqnu{*}{\Fl^{\c}}[\ep^{-1}]$ is the image of polynomials of degree $\leq i$ (in $\fra^{*}$) under the map \eqref{pol inv ep}. The {\em Chern filtration} $C_{\leq i}\eqnu{*}{\Fl^{\c}}$ of $\eqnu{*}{\Fl^{\c}}$ is defined as the saturation of the Chern filtration on $\eqnu{*}{\Fl^{\c}}[\ep^{-1}]$, i.e., $C_{\leq i}\eqnu{*}{\Fl^{\c}}=\eqnu{*}{\Fl^{\c}}\cap C_{\leq i}\eqnu{*}{\Fl^{\c}}[\ep^{-1}]$.  
\end{defn}

\begin{prop}\label{p:Hrat Fl}
\begin{enumerate}
\item There is a bigraded action of $\Hrat_{\nu}$ on $\Gr^{C}_{*}\eqnu{*}{\Fl^{\c}}$. 
\item There is a map of $\Hrat_{\nu}$-modules
\begin{equation}\label{pol to fl st}
\Ind^{\Hrat_{\nu}}_{\QQ[\ep]\otimes\Sym(\fra)\otimes\QQ[W]}(\QQ[\ep])\to\Gr^{C}_{*}\eqnu{*}{\Fl^{\c}}.
\end{equation}
that is an isomorphism after inverting $\ep$.
\item\label{hh eigen Fl} Let $N=\frac{1}{2}r(h_{\theta}\nu-1)$, which is the dimension of $\Sp_{a}$ for any $a\in\frc(F)^{\rs}_{\nu}$ (see Corollary \ref{c:AFS dim}(2)). The operator $\hh$ in the almost $\sl_{2}$-triple defined in \S\ref{sss:sl2} acts by $(i-N)\ep$ on $\Gr^{C}_{i}\eqnu{*}{\Fl^{\c}}$.
\end{enumerate}
\end{prop}
\begin{proof} Since $\Fl^{\c}$ is the affine flag variety of the simply-connected cover of $G$, and that $\Hrat_{\nu}$ is invariant under isogeny of $G$, we may assume that $G$ is simply-connected.
 
(1) We already have an action of $\Hgr_{\nu}/(B_{\KM}-\nu^{2}\BB(\rho^{\vee},\rho^{\vee})\ep^{2})$ on $\eqnu{*}{\Fl}$ from Proposition \ref{p:Hgr Fl}. To define the $\Hrat_{\nu}$-action on $\Gr^{P}_{*}\eqnu{*}{\Fl}$, we only need to check the conditions in Proposition \ref{p:Hratmod} are satisfied.

Using (GC-4) and (GC-5) we see that for $\tilw\in\tilW$ with image $w\in W$ and $\xi\in\fra^{*}$, $\tilw \xi-\leftexp{w}{\xi}\tilw\in\QQ[u,\delta]$. By induction on $n$, we see that for $f(\xi)\in\Sym^{n}(\fra^{*})$, $\tilw f(\xi)- f(\leftexp{w}{\xi})\tilw\in\QQ[u,\delta]\otimes\Sym^{\leq n-1}(\fra^{*})$. In particular, for $\tilw$ in the lattice part of $\tilW$  (so $w=1$), the commutator $[\tilw, f(\xi)]$ lies in $\QQ[u,\delta]\otimes\Sym^{\leq n-1}(\fra^{*})$. Let $1\in \eqnu{0}{\Fl}$ be the unit class. Then $(\tilw-\id) \cdot f(\xi)\cdot1=\tilw f(\xi)\cdot1-f(\xi)\tilw\cdot1\in C_{\leq n-1}\eqnu{*}{\Fl}$. 

An element $\xi\in\fra^{*}$ sends $C_{\leq n}$ to $C_{\leq n+1}$ by definition. Finally the action of $\ep\L_{\can}$, a multiple of $B_{\KM}-B$, sends $C_{\leq n}$ to $C_{\leq n+2}$ (because $\deg B=2$ and $B_{\KM}$ acts as a multiple of $\ep^{2}$). Since the Chern filtration on $\eqnu{*}{\Fl}$ is given by saturating the degree filtration on $\eqnu{*}{\Fl}[\ep^{-1}]$, $\L_{\can}$ also sends $C_{\leq n}$ to $C_{\leq n+2}$. This checks all the condition in Proposition \ref{p:Hratmod} and finishes the proof.

(2) follows from the isomorphism \eqref{pol inv ep}.

(3) We only need to show that $\hh$ acts on $1\in\eqnu{0}{\Fl}\subset\Gr^{C}_{0}\eqnu{*}{\Fl}$ by multiplication by $-N\ep$. Let $\{\xi_{i}\}$ and $\{\eta_{i}\}$ be dual bases of $\fra^{*}$ and $\fra$ respectively. Since $\eta_{i}\cdot 1=0$, we have $\hh\cdot 1=\frac{1}{2}\sum_{i}\xi_{i}\eta_{i}\cdot 1+\eta_{i}\xi_{i}\cdot 1=\frac{1}{2}\sum_{i}[\eta_{i},\xi_{i}]\cdot 1=\frac{1}{2}\ep(\sum_{i}\jiao{\xi_{i},\eta_{i}}-\frac{\nu}{2}\sum_{\alpha\in\Phi}c_{\alpha}\sum_{i}\jiao{\xi_{i},\alpha^{\vee}}\jiao{\alpha,\eta_{i}})\cdot 1=\frac{1}{2}\ep(r-\nu\sum_{\alpha\in\Phi}c_{\alpha})\cdot 1$. Since $c_{\alpha}$ is defined as the cardinality of the preimages of $\PPhi\to\Phi$, the above is further equal to $\frac{1}{2}\ep(r-\nu\#\PPhi)\cdot 1=\frac{1}{2}\ep(r-\nu h_{\theta}r)\cdot 1=-N\ep\cdot 1$.
\end{proof}

\subsection{The $\Hrat$-action on the cohomology of homogeneous affine Springer fibers}\label{ss:Sp Hrat}
Let $a\in\frc(F)^{\rs}_{\nu}$. Recall from Theorem \ref{thm:surj} and Remark \ref{r:Flc} that we have the restriction map
\begin{equation*}
\iota^{*}_{a}: \eqnu{*}{\Fl^{\c}}\cong\eqnu{*}{\Fl}^{\Omega}\to \eqnu{*}{\Sp_{a}}^{\tilS_{a}\rtimes B_{a}}
\end{equation*}
which is surjective after inverting $\ep$.

\begin{defn}\label{def:Ch fil} The {\em Chern filtration} $C_{\leq i}\eqnu{*}{\Sp_{a}}^{\tilS_{a}\rtimes B_{a}}[\ep^{-1}]$ on $\eqnu{*}{\Sp_{a}}^{\tilS_{a}\rtimes B_{a}}[\ep^{-1}]$ is the image of  $C_{\leq i}\eqnu{*}{\Fl^{\c}}[\ep^{-1}]$ under $\iota^{*}_{a}$. The {\em Chern filtration} $C_{\leq i}\eqnu{*}{\Sp_{a}}^{\tilS_{a}\rtimes B_{a}}$ on $\eqnu{*}{\Sp_{a}}^{\tilS_{a}\rtimes B_{a}}$ is the saturation of the Chern filtration on $\eqnu{*}{\Sp_{a}}^{\tilS_{a}\rtimes B_{a}}[\ep^{-1}]$.
\end{defn}

The same argument of Proposition \ref{p:Hrat Fl}(1) gives the following result.
\begin{prop}\label{p:Hrat GrC} Let $a\in\frc(F)^{\rs}_{\nu}$. There is a bigraded action of $\Hrat_{\nu}$ on $\Gr^{C}_{*}\eqnu{*}{\Sp_{a}}^{\tilS_{a}\rtimes B_{a}}$ such that the map $\Gr^{C}_{*}\iota^{*}_{a}:\Gr^{C}_{*}\eqnu{*}{\Fl^{\c}}\to \Gr^{C}_{*}\eqnu{*}{\Sp_{a}}^{\tilS_{a}\rtimes B_{a}}
$ is a map of $\Hrat_{\nu}$-modules (which is surjective after inverting $\ep$).
\end{prop}

The main result of this section is the following theorem.
\begin{theorem}\label{L(triv)} Assume $G$ is split (i.e., $e=1$) and let $a\in\frc(F)^{\rs}_{\nu}$.  
\begin{enumerate}
\item There is a geometrically defined filtration $P_{\leq i}\eqnu{*}{\Sp_{a}}^{\tilS_{a}}$ on $\eqnu{*}{\Sp_{a}}^{\tilS_{a}}$, stable under the $B_{a}$-action and extending the Chern filtration on $\eqnu{*}{\Sp_{a}}^{\tilS_{a}\rtimes B_{a}}$, such that the bigraded $\Hrat_{\nu}$-action on $\Gr^{C}_{*}\eqnu{*}{\Sp_{a}}^{\tilS_{a}\rtimes B_{a}}$ extends to $\Gr^{P}_{*}\eqnu{*}{\Sp_{a}}^{\tilS_{a}}$ and commutes with the $B_{a}$-action.

\item\label{L(triv)irr} Specializing $\ep$ to $1$, the $\Hrat_{\nu,\ep=1}$-module $\Gr^{C}_{*}\spcoh{}{\Sp_{a}}^{\tilS_{a}\rtimes B_{a}}$ is isomorphic to the irreducible finite-dimensional spherical module $\frL_{\nu}(\triv)$ of $\Hrat_{\nu,\ep=1}$. 
\end{enumerate}
\end{theorem}

We have an immediate corollary.
\begin{cor}\label{c:Hgr irr} Assume $G$ is split and simply-connected, and let $a\in\frc(F)^{\rs}_{\nu}$.  Then the $\Hgr_{\nu,\ep=1}$-module $\spcoh{}{\Sp_{a}}^{S_{a}\rtimes B_{a}}$ is also irreducible.
\end{cor}
\begin{proof}
By \cite[Proposition 2.3.1 (b)]{VV} the dimensions of the finite dimensional irreducible spherical modules for $\Hgr_{\nu,\ep=1}$ and $\Hrat_{\nu,\ep=1}$ are equal. Thus by Theorem~\ref{L(triv)}\eqref{L(triv)irr},  
$\spcoh{}{\Sp_{a}}^{S_{a}\rtimes B_{a}}$ has the same dimension as the spherical irreducible $\Hgr_{\nu,\ep=1}$-module. On the other hand $\Hgr_{\nu,\ep=1}$-module $\spcoh{}{\Sp_{a}}^{S_{a}\rtimes B_{a}}$  is a quotient of
the polynomial representation of $\Hgr_{\nu,\ep=1}$. Since the analog of  Corollary~\ref{quot} holds for $\Hgr_{\nu,\ep=1}$ by \cite[Proposition 2.3.1 (a)]{VV}, we conclude that $\spcoh{}{\Sp_{a}}^{S_{a}\rtimes B_{a}}$ is the irreducible spherical module for $\Hgr_{\nu,\ep=1}$.
\end{proof}

The proof of Theorem \ref{L(triv)} occupies \S\ref{ss:perv fil}-\S\ref{Frobenius}. It uses the geometry of Hitchin fibration, especially a variant of Ng\^o's support theorem in an essential way. Currently we are unable to generalize certain technical results needed in this global geometric argument to the quasi-split groups. However, the quasi-split examples in \S\ref{s:examples} suggest the following conjecture.

\begin{conj}\label{conj qsplit} The statements of Theorem \ref{L(triv)} hold for quasi-split groups $G$ (in the generality of \S\ref{sss:gp}).
\end{conj}

\begin{exam}\label{ex:Cox} When $\nu=d/h_{\theta}$ with $(d,h_{\theta})=1$ (recall $h_{\theta}$ is the twisted Coxeter number of $(\GG,\theta)$), we may compute the dimension of $\spcoh{}{\Sp_{a}}$ using Proposition \ref{p:reduce to 1/m}. In the case $\nu=1/h_{\theta}$ there is only one bounded clan and it consists of a single alcove. Therefore for $\nu=d/h_{\theta}$, $\dim\spcoh{}{\Sp_{a}}=d^{r}$ for any $a\in\frc(F)^{\rs}_{\nu}$. It is easy to see directly that $S_{a}\rtimes B_{a}$ acts trivially on $\spcoh{}{\Sp_{a}}$. On the other hand, the irreducible $\Hrat_{1/m, \ep=1}$-module $\frL_{\nu}(\triv)$ also has dimension $d^{r}$, see \cite[Theorem 1.11]{BEG}. Therefore $\Gr^{C}_{*}\spcoh{}{\Sp_{a}}\cong\frL_{\nu}(\triv)$ as $\Hrat_{\nu,\ep=1}$-modules. Moreover, \cite[Proposition 1.20]{BEG} shows that $\frL_{\nu}(\triv)$ is a Frobenius algebra. Therefore, we have checked Conjecture \ref{conj qsplit} in the case the denominator of $\nu$ is the twisted Coxeter number.
\end{exam}

Combining Theorem \ref{L(triv)} with Theorem \ref{thm:dim}(3) we get a dimension formula for $\frL_{\nu}(\triv)$ as in \eqref{intro dim}. More generally, we have
\begin{cor} Let $G$ be split and $\nu=d/m>0$ be an elliptic slope in lowest terms (which is also the normal form since $e=1$). Then for any standard parahoric subgroup $\bP$ of $G$, we have
\begin{equation*}
\dim\frL_{\nu}(\triv)^{W_{\bP}}=d^{r}\sum_{\tilw\in W_{\nu}\backslash \tilW/W_{\bP}}\dim(\lambda^{\tilw}_{\bP,1/m}\cdot\calH^{W^{\tilw}_{\bP,1/m}}_{W_{1/m}}).
\end{equation*}
\end{cor}

\subsection{The perverse filtration}\label{ss:perv fil} For the rest of the section we assume $G$ is split. 

For the construction of the rational \Cha action, we need to consider the version of Hitchin moduli stacks in which we also allow the point of the parahoric reduction to move on $X$. Indeed, this is the version considered in the papers \cite{GS} and \cite{GSLD}. We use $\calM_{\bP,U}$ to denote the moduli stack classifying $(x,\calE,\varphi)$ where $x\in U$, $(\calE,\varphi)$ an $\calL$-valued  $\GG$-Higgs bundle over $X$ with $\bP$-reduction at $x$. For details of the construction, see \cite{GS}.

The Hitchin fibration now reads $f_{\bP,U}:\calM_{\bP,U}\to \calA\times U$, which sends $(x,\calE,\varphi)$ to $(f_{1}(\varphi),\cdots,f_{r}(\varphi),x)$. In the sequel we shall exclusively work over the open subset $\Ah\times U$, so we abuse the notation to denote the restriction of $f_{\bP,U}$ to $\Ah\times U$ again by $f_{\bP,U}$. Let $\fa_{\bP,U}$ be the restriction of $f_{\bP,U}$ to $\Aa\times U$. Again, when $\bP=\bI$, we suppress it from subscripts.

\begin{cons}\label{gr move U} The construction of \cite[Construction 6.1.4]{GS}, as recalled in \S\ref{ss:Hgr}, gives an action of $\Hgr$ on $f_{U,!}\QQ$ as an ind-complex $D^{b}(\Ah\times U)$. We would like to upgrade this into an action in the equivariant derived category $D^{b}_{\Grot\times\Gdil}(\Ah\times U)$. In fact, the constructions given in \cite{GS} gives the action of $s_{i},\Omega_{\bI}$ and $\xch(\TT_{\KM})$ also in the $\Grot\times\Gdil$-setting. Only the action of $u$ needs a little extra care. In \cite{GS} we make $u$ act as the Chern class of $\calL$ (pull-back along $\Ah\times X\to X$). Now we are working with $U$ instead of $X$, so the $\Grot$-equivariant Chern class of $\calL|_{U}$ is trivial (since we have fixed the equivariant structure of $\calL$ so this holds). Instead, $u$ should act as the equivariant parameter of $\Gdil$ as in Construction \ref{gr global u delta}. Also note that $\delta$ acts by the equivariant parameter of $\Grot$, and $\L_{\can}$ acts by the pullback of the equivariant Chern class $c_{1}^{\Grot}(\omega_{\Bun_{\GG}})$.
\end{cons}

The same argument as in Theorem \ref{th:Hgr} shows that
\begin{theorem}\label{th:Hgr U} Construction \eqref{gr move U} gives an action of $\Hgr$ on the ind-complex $\fUQl$ in the equivariant derived category $D^{b}_{\Grot\times\Gdil}(\Ah\times U)$.
\end{theorem}

Let   $\faUQl=\fUQl|_{\Aa\times U}$ . According to the action of $\xcoch(\TT)\subset \tilW\subset\Hgr$, we may decompose $\fUQl$ into generalized eigen-complexes (see \cite[\S 2.2]{GSLD})
\begin{equation*}
\faUQl=\bigoplus_{\kappa\in\wh{\TT}\QQ)}(\faUQl)_{\kappa}
\end{equation*}
where $\kappa$ runs over finite order elements in $\wh{\TT}(\QQ)=\Hom(\xcoch(\TT),\QQ^{\times})$. 

\begin{defn} \label{d:stable}
\begin{enumerate}
\item The {\em stable part} $\fst$ of $\faUQl$ is the direct summand $(\faUQl)_{1}$. Equivalently, it is the maximal direct summand of $\faUQl$ on which $\xcoch(\TT)$ acts unipotently. 

\item For $a\in\Aa$, the {\em stable part} $\eqnust{*}{\calM_{a}}$ of $\eqnu{*}{\calM_{a}}$ is the stalk of $\fst$ at $(a,0)$. This is also the maximal direct summand of $\eqnu{*}{\calM_{a}}$ on which $\xcoch(\TT)$ acts unipotently.
\end{enumerate}
\end{defn}

For the Hitchin moduli stack without Iwahori level structures, the stable part is first defined by Ng\^o \cite{NgoHit} as the geometric incarnation of stable orbital integrals, hence the namesake. Ng\^o's original definition uses the action of $\pi_{0}(\calP_{a})$ instead of $\xcoch(\TT)$, and is in fact equivalent to our definition because of the following lemma.

\begin{lemma}\label{l:alt stable global} The stable part $\eqnust{*}{\calM_{a}}$ is also the direct summand of $\eqnu{*}{\calM_{a}}$ on which $\pi_{0}(\calP_{a})$ acts trivially.
\end{lemma}
\begin{proof} In \cite[Theorem  1.5]{GSSph}, we have shown that the action of $\QQ[\xcoch(\TT)]^{\WW}$ on $\cohog{*}{\calM_{a}}$ factors through the action of $\pi_{0}(\calP_{a})$ via a canonical algebra homomorphism
\begin{equation}\label{cent to pi0}
\QQ[\xcoch(\TT)]^{\WW}\to\QQ[\pi_{0}(\calP_{a})].
\end{equation}
The proof works as well in the $\Gnu$-equivariant situation. On the level of spectra,  \eqref{cent to pi0} gives a natural morphism $\iota:\wh{\pi_{0}(\calP_{a})}\to \wh{\TT}\sslash\WW$ (where $\wh{\pi_{0}(\calP_{a})}$ is the diagonalizable group $\Spec\QQ[\pi_{0}(\calP_{a})]$). Let
\begin{equation*}
\eqnu{*}{\calM_{a}}=\bigoplus_{\kappa\in\wh{\TT}}\eqnu{*}{\calM_{a}}_{\kappa}
\end{equation*}
be the decomposition of $\eqnu{*}{\calM_{a}}$ into generalized eigenspaces under the $\xcoch(\TT)$-action. We also let
\begin{equation*}
\eqnu{*}{\calM_{a}}=\bigoplus_{\psi\in\wh{\pi_{0}(\calP_{a})}}\eqnu{*}{\calM_{a}}_{\psi}
\end{equation*}
be the decomposition according to the action of $\pi_{0}(\calP_{a})$, which is a finite abelian group since $\gamma$ is elliptic. Then the global main result of \cite{GSSph} in our situation says that for any $\WW$-orbit $\Xi\subset\wh{\TT}$ (viewed as a point in $\wh{\TT}\sslash\WW$), we have
\begin{equation}\label{Xi}
\bigoplus_{\kappa\in\Xi}\eqnu{*}{\calM_{a}}_{\kappa}=\eqnu{*}{\calM_{a}}_{\iota^{-1}(\Xi)}.
\end{equation}
It is shown in \cite{NgoFL} that $\iota$ can be lifted to an {\em embedding} of groups $\wh{\pi_{0}(\calP_{a})}\incl \wh{\TT}$. Therefore, the preimage $\iota^{-1}(1)$ is set-theoretically supported at $1\in \wh{\pi_{0}(\calP_{a})}$. Applying \eqref{Xi} to $\Xi=1\in\wh{\TT}\sslash \WW$ we get the desired statement. 
\end{proof}

\begin{cor}\label{c:stable is all} For $a\in\Ah_{\nu}$ and $a$ elliptic, we have $\eqnust{*}{\calM_{a}}=\eqnu{*}{\calM_{a}}$.
\end{cor}
\begin{proof}
By Proposition \ref{p:prod}(2), $\calP_{a}$ is connected. Therefore the equality follows from Lemma \ref{l:alt stable global}.
\end{proof}

\subsubsection{The perverse filtration} The perverse $t$-structure on $D^{b}_{\Grot\times\Gdil}(\Aa\times U)$ gives the truncations $\ptau_{\leq i}\fst$ of the complex $\fst$. We shift the degrees of the truncation by setting
\begin{equation*}
P_{\leq i}\fst:=\ptau_{\leq i+\dim(\calA\times U)}\fst.
\end{equation*}
Let $\Gr^{P}_{i}\fst$ be the associated graded of this filtration. In other words,
\begin{equation*}
\Gr^{P}_{i}\fst=\left(\pH^{i+\dim(\calA\times U)}\fst\right)[-i-\dim(\calA\times U)].
\end{equation*}
Our convention makes sure that $\Gr^{P}_{i}\fst$ is a shifted perverse sheaf in perverse degree $i+\dim(\Aa\times U)$, and that $\Gr^{P}_{i}\fst=0$ unless $0\leq i\leq 2\dim \fa_{U}$. 

The Deligne-Mumford stack $\calM^{\textup{ell}}_{U}$ is smooth and the morphism $f^{\textup{ell}}_{U}$ is proper. The decomposition theorem implies that there is a {\em non-canonical} isomorphism
\begin{equation*}
\bigoplus_{i=0}^{2\dim \fa_{U}}\Gr^{P}_{i}\fst\isom\fst.
\end{equation*}
The above decomposition in fact holds in the equivariant derived category $D^{b}_{\Grot\times\Gdil}(\Aa\times U)$. This is a consequence of the decomposition theorem for proper morphisms between Artin stacks with affine stabilizers, proved by S. Sun in \cite[Theorem 1.2, 3.15]{Sun}.

A key property of $\Gr^{P}_{i}\fst$ is the following support theorem. Recall we have the evaluation map $\calA\times U\to\cc$ (using the trivialization of $\calL$ over $U$). Let $(\calA\times U)^{\rs}$ be the preimage of $\cc^{rs}$.

\begin{theorem}\label{th:supp} The support of each simple constituent of the shifted perverse sheaf $\Gr^{P}_{i}\fst$ is the whole $\Aa\times U$.
\end{theorem}
The proof of this support theorem is the same as \cite[Corollary 2.2.4]{GSLD}, which is based on \cite[Theorem 2.1.1]{GSLD}, provided the codimension estimate $\codim\calA_{\geq\delta}$ holds (see \cite{NgoFL}). We give detailed argument for this codimension estimate in Appendix \ref{s:codim}.

Let $a\in\Aa$ be homogeneous of slope $\nu$. The inclusion $i_{a,0}:\{(a,0)\}\hookrightarrow\Aa\times U$ is $\Gnu$-equivariant, therefore the stalk  $i^{*}_{(a,0)}\fst\in D^{b}_{\Gnu}(\pt)$ calculates the part $\eqnust{*}{\calM_{a}}$ of the $\Gnu$-equivariant cohomology of $\calM_{a}$ on which $\xcoch(T)$ acts unipotently. Moreover, the natural map $P_{\leq i}\fst\to\fst$ induces a map on stalks
\begin{equation}\label{Pstalk}
i^{*}_{(a,0)}P_{\leq i}\fst\to i^{*}_{(a,0)}\fst=\eqnust{*}{\calM_{a}}
\end{equation}
which is also a direct summand since $P_{\leq i}\fst\to\fst$ is.

\subsection{The global sheaf-theoretic action of $\Hrat$}\label{ss:Hrat} A bigraded action of $\Hrat$ on $\Gr^{P}_{*}\fst$ is a bigraded algebra homomorphism
\begin{equation*}
\Hrat\to\bigoplus_{j\in\ZZ}\End^{2j}(\bigoplus_{i\in\ZZ}\Gr^{P}_{i}\fst)(j)
\end{equation*}
where the grading of the RHS is given by assigning 
\begin{equation*}
\Hom(\Gr^{P}_{i_{1}}\fst,\Gr^{P}_{i_{2}}\fst[2j](j))
\end{equation*} 
the bidegree $(j,i_2-i_1)$. We shall construct such an action by first specifying the actions of the generators of $\Hrat$.

\begin{cons}\label{rat u delta}
The action of $u$ and $\delta$ are given by the equivariant parameters of $\Gdil$ and $\Grot$ respectively, viewed as maps $\Gr^{P}_{i}\fst\to \Gr^{P}_{i}\fst[2](1)$.
\end{cons}

\begin{cons}\label{rat W} The action of the {\em finite Weyl group} $\WW$. Since $\tilW$ acts on $\fUQl$, the stable part $\fst$ is therefore stable under the finite Weyl group $\WW\subset\tilW$. Passing to the associated graded pieces under the perverse filtration, we get an action of $\WW$ on each $\Gr^{P}_{i}\fst$.
\end{cons}

\begin{cons}\label{rat lattice} The action of $\xcoch(\TT)$. By \cite[3.2.1]{GSLD}, the action of $\l\in\xcoch(T)$ on $\Gr^{P}_{i}\fst$ is the identity. Note that \cite[3.2.1]{GSLD} is applicable here because of the validity of Theorem \ref{th:supp} in our situation. Therefore $\lambda-\id$ sends $P_{\leq i}\fst$ to $P_{\leq i-1}\fst$. The action of $\l\in\xcoch(T)$ on $\Gr^{P}_{*}\fst$ is given by
\begin{equation*}
\l_{\rat}:=\lambda-\id:\Gr^{P}_{i}\fst\to \Gr^{P}_{i-1}\fst.
\end{equation*}
\end{cons}

\begin{cons}\label{rat Chern} The action of $\xch(\TT)$. Let $\calL(\xi)$ be the tautological line bundle on $\calM_{U}$ indexed by $\xi\in\xch(T)$. As in \cite[3.2.2]{GSLD}, we consider the map
\begin{equation*}
(\cup c_{1}(\calL(\xi)))_{\st}:\fst\hookrightarrow\faUQl\xrightarrow{\cup c_{1}(\calL(\xi))}\faUQl[2](1)\twoheadrightarrow\fst
\end{equation*}
where the first and the last arrow are the natural inclusion and projection of the direct summand $\fst$ of $\faUQl$. By \cite[Lemma 3.2.3]{GSLD}, the map $(\cup c_{1}(\calL_{\xi})_{\st}$ induces the zero map $\Gr^{P}_{i}\fst\to\Gr^{P}_{i+2}\fst[2](1)$. Again, \cite[3.2.1]{GSLD} is applicable here because of the validity of Theorem \ref{th:supp} in our situation. Therefore it sends $P_{\leq i}\fst$ to $P_{\leq i+1}\fst[2](1)$, and hence induces\begin{equation*}
\xi_{\rat}=\Gr^{P}_{i}(\cup c_{1}(\calL(\xi)))_{\st}:\Gr^{P}_{i}\fst\to \Gr^{P}_{i+1}\fst[2](1).
\end{equation*}
This is the action of $\xi\in\xch(\TT)$.
\end{cons}

\begin{cons}\label{rat can} The action of $\Lambda_{\can}$ is given by the cup product with $c_{1}^{\Grot}(\omega_{\Bun_{\GG}})$, viewed as a map $\Gr^{P}_{i}\fst\to \Gr^{P}_{i+2}\fst[2](1)$.
\end{cons}

Applying Proposition \ref{p:Hratmod} to $\fst$ with the perverse filtration, we get
\begin{theorem}\label{th:Hrat} Constructions \eqref{rat u delta}--\eqref{rat can} defines a bigraded action of $\Hrat$ on $\Gr^{P}_{*}\fst$ in the equivariant derived category $D^{b}_{\Grot\times\Gdil}(\Aa\times U)$ (in the sense as in the beginning of this subsection).
\end{theorem}

\subsubsection{The homogeneous case} Now let $\nu>0$ be an elliptic slope, then $\Ah_{\nu}\incl\Aa$.  By Corollary \ref{c:stable is all}, the stalks of $f_{\nu,!}\QQ$ along $\Ah_{\nu}$ are always stable in the sense of Definition \ref{d:stable}. In other words, $f_{\nu,!}\QQ=i_{\nu}^{*}(\fa_{!}\QQ)_{\st}$, where $i_{\nu}:\Ah_{\nu}\incl \Aa$ is the inclusion. We define the perverse filtration on $ f_{\nu,!}\QQ$ by
\begin{equation*}
P_{\leq i}f_{\nu,!}\QQ:=i_{\nu}^{*}P_{\leq i}(f_{!}\QQ)_{\st}.
\end{equation*}
For $a\in\Ah_{\nu}$, we have an induced perverse filtration $P_{\leq i}\eqnu{*}{\calM_{a}}$ on $\eqnu{*}{\calM_{a}}$ by taking the stalk of $P_{\leq i}f_{\nu,!}\QQ$ at $a$. 

\begin{cor}\label{c:Hratnu} Let $\nu>0$ be an elliptic slope. There is a bigraded action of $\Hrat_{\nu}$ on the graded complex $\Gr^{P}_{*} f_{\nu,!}\QQ$ on $\Ah_{\nu}$ (in the sense as in the beginning of this subsection). In particular, for $a\in\Ah_{\nu}$, there is a bigraded action of $\Hrat_{\nu}$ on $\Gr^{P}_{*}\eqnu{*}{\calM_{a}}$.
\end{cor}
\begin{proof} Most parts follow directly from Theorem \ref{th:Hrat} by restricting to $\Ah_{\nu}$. The only thing we need to check is that the degree $(4,2)$-element $B_{\KM}$ gives the zero map $\Gr^{P}_{i}( f_{\nu,!}\QQ)_{\st}\to\Gr^{P}_{i+2}( f_{\nu,!}\QQ)_{\st}$. From Remark \ref{r:ellHgr} we know that $B_{\KM}\in\Hgr_{\nu}$ acts on $ f_{\nu,!}\QQ$ as a multiple of $\ep^{2}$. Since $\ep^{2}$ preserves the perverse filtration, $B_{\KM}$ gives the zero map $\Gr^{P}_{i}( f_{\nu,!}\QQ)_{\st}\to\Gr^{P}_{i+2}( f_{\nu,!}\QQ)_{\st}$.
\end{proof}

\subsection{Proof of Theorem \ref{L(triv)}(1)}\label{ss:Hrat local} First note that if we change $G$ to a group isogenous to it, neither $\Hrat_{\nu}$ nor $(q_{\nu,!}\QQ)^{\tilS}$ change (the latter uses the fact that $\tilS_{a}$ surjects onto $\Omega$, proved in Lemma \ref{l:local Neron}\eqref{tilL}). Therefore, to prove Theorem \ref{L(triv)}, it suffices to prove it for any group isogenous to $G$. Therefore, we may assume that Assumption \eqref{assrho} holds for $G$ (for example, we may take $G$ to be adjoint).

Recall $\nu>0$ is elliptic and $a\in\frc(F)^{\rs}_{\nu}$. As in \S\ref{ss:prod}, we use the line bundle $\calL=\calO_{X}(\nu)$ to define the Hitchin moduli stack $\calM$ and the Hitchin fibration $f:\calM\to \calA$. In this case $\frc(F)^{\rs}_{\nu}=\Ah_{\nu}$ and the Hitchin fiber $\calM_{a}$ is defined. By Proposition \ref{p:prod}(2) (which is applicable since \eqref{assrho} holds), we have a canonical isomorphism between complexes on $\frc(F)^{\rs}_{\nu}=\Ah_{\nu}$
\begin{equation}\label{lg stable family}
f_{\nu,!}\QQ\isom(q_{\nu,!}\QQ)^{\tilS}.
\end{equation}
Taking stalks at $a\in\frc(F)^{\rs}_{\nu}$ we get
\begin{equation}\label{lg stable}
\eqnu{*}{\calM_{a}}=\eqnust{*}{\calM_{a}}\isom\eqnu{*}{\Sp_{a}}^{\tilS_{a}}.
\end{equation}

\begin{defn}\label{def:perv fil} Let $\nu>0$ be an elliptic admissible slope and $a\in\frc(F)^{\rs}_{\nu}$.
\begin{enumerate}
\item The {\em perverse filtration} $P_{\leq i}(q_{\nu,!}\QQ)^{\tilS}$ is the transport of the perverse filtration $P_{\leq i}f_{\nu,!}\QQ$ under the isomorphism \eqref{lg stable family}.
\item The {\em perverse filtration} $P_{\leq i}\eqnu{*}{\Sp_{a}}^{\tilS_{a}}$ is the filtration on $\eqnu{*}{\Sp_{a}}$ given by talking the stalk at $a$ of $P_{\leq i}( q_{\nu,!}\QQ)^{\tilS}$. This is also the transport of $P_{\leq i}\eqnu{*}{\calM_{a}}$ under the isomorphism  \eqref{lg stable}.
\end{enumerate}
\end{defn}

\begin{lemma} For $a\in\frc(F)^{\rs}_{\nu}$ and all $i$, $P_{\leq i}\eqnu{*}{\Sp_{a}}^{\tilS_{a}}$ is stable under the action of the braid group $B_{a}$. 
\end{lemma}
\begin{proof}
By the decomposition theorem, each $P_{\leq i}f_{\nu,!}\QQ$ is a direct summand of $f_{\nu,!}\QQ$, therefore $P_{\leq i}( q_{\nu,!}\QQ)^{\tilS}$ is a direct summand of $ q_{\nu,!}\QQ$. By Corollary \ref{c:Hess loc sys}(2), $ q_{\nu,!}\QQ$ is a direct sum of shifted semisimple local systems on $\frc(F)^{\rs}_{\nu}$, therefore so is $P_{\leq i}( q_{\nu,!}\QQ)^{\tilS}$. Taking stalk at $a$ then gives the action of $B_{a}=\pi_{1}(\frc(F)^{\rs}_{\nu}, a)$ on each degree $P_{\leq i}\eqnu{j}{\Sp_{a}}^{\tilS_{a}}$.
\end{proof}

Corollary \ref{c:Hratnu} and \eqref{lg stable} then imply 
\begin{cor}\label{c:Hrat GrP} Let $\nu>0$ be an elliptic slope and $a\in\frc(F)^{\rs}_{\nu}$. Then there is a bigraded action of $\Hrat_{\nu}$ on $\Gr^{P}_{*}\eqnu{*}{\Sp_{a}}^{\tilS_{a}}$ commuting with the action of $B_{a}$. 
\end{cor}

To finish the proof of Theorem \ref{L(triv)}(1), it remains to show that the perverse filtration coincides with the Chern filtration on $\eqnu{*}{\Sp_{a}}^{\tilS_{a}\rtimes B_{a}}$ and the actions of $\Hrat_{\nu}$ also coincide there.

\begin{prop}\label{p:Ch=perv} Let $a\in\frc(F)^{\rs}_{\nu}$. 
\begin{enumerate}
\item The Chern filtration (Definition \ref{def:Ch fil}) and the perverse filtration (Definition \ref{def:perv fil}) on $\eqnu{*}{\Sp_{a}}^{\tilS_{a}\rtimes B_{a}}$ coincide. 
\item The $\Hrat_{\nu}$-module structures defined on $\Gr^{C}_{*}\eqnu{*}{\Sp_{a}}^{\tilS_{a}\rtimes B_{a}}$ (in Proposition \ref{p:Hrat GrC}) and on $\Gr^{P}_{*}\eqnu{*}{\Sp_{a}}^{\tilS_{a}\rtimes B_{a}}$ (Corollary \ref{c:Hrat GrP}) are the same under the identification of the two filtrations.
\end{enumerate}
\end{prop}
\begin{proof} (1) Let $\rhh=\frac{1}{2}\sum_{i}\xi_{i}(\tilw_{i}-1)+(\tilw_{i}-1)\xi_{i}$, $\tilw_{i}\in\xcoch(\TT^{\sc})$ is a basis (viewed as elements in $\Wa$), and $\{\xi_{i}\}$ is the dual basis for $\fra^{*}$. This is a lifting of $\hh\in\Hrat_{\nu}$ into $\Hgr_{\nu}$. Let $M=\spcoh{}{\Sp_{a}}^{\tilS_{a}\rtimes B_{a}}$. We only need to show that both $C_{\leq i}M$ and $P_{\leq i}M$ are the direct sum of generalized eigenspaces of $\rhh$ with eigenvalues $\leq i-N$, where $N=\dim\Sp_{a}=\frac{1}{2}r(\nu h-1)$.

For $C_{\leq i}M$, this follows from Proposition \ref{p:Hrat Fl}\eqref{hh eigen Fl}. For $P_{\leq i}$, since $\Gr^{P}_{*}M$ is an $\Hrat_{\nu,\ep=1}$-modules, the action of $\hh$ on $\Gr^{P}_{i}M$ is $i-N'$ for some integer $N'$. We can determine $N'$ by computing the action of $\rhh$ on the class $1\in \Gr^{P}_{0}M$. The same calculation in Proposition \ref{p:Hrat Fl}\eqref{hh eigen Fl} shows that $N'=N$, and therefore $P_{\leq i}M$ is also the direct sum of generalized eigenspaces of $\rhh$ with eigenvalues $\leq i-N$. 

(2) follows from the uniqueness part of Proposition \ref{p:Hratmod}.
\end{proof}

\subsection{Frobenius algebra structure and proof of Theorem \ref{L(triv)}(2)}\label{Frobenius}
\begin{defn} A {\em Frobenius $\QQ$-algebra} is a finite-dimensional $\QQ$-algebra $A$ equipped with a perfect pairing $\jiao{\cdot,\cdot}_{A}:A\otimes A\to \QQ$ such that $\jiao{ab,c}_{A}=\jiao{a,bc}$ for all $a,b,c\in A$. The pairing is necessarily given by $\jiao{a,b}_{A}=\ell(ab)$ for some $\QQ$-linear map $\ell:A\to \QQ$. 
\end{defn}

Let $N=\dim\Sp_{\gamma}=\frac{1}{2}r(\nu h-1)$. The finite flat commutative $\QQ[\ep]$-algebra $\eqnu{*}{\Sp_{a}}^{\tilS_{a}}$ is equipped with a $\QQ[\ep]$-linear quotient map $\ell_{\ep}:\eqnu{*}{\Sp_{a}}^{\tilS_{a}}\to\Gr^{P}_{2N}\eqnu{*}{\Sp_{a}}^{\tilS_{a}}$. Now $\Gr^{P}_{2N}\eqnu{*}{\Sp_{a}}^{\tilS_{a}}$ is a free $\QQ[\ep]$-module of rank one on which $B_{a}$ also acts trivially. We fix a generator $[\Sp_{a}]^{\Gnu}\in\Gr^{P}_{2N}\eqnu{*}{\Sp_{a}}^{\tilS_{a}}$, hence identifies $\Gr^{P}_{2N}\eqnu{*}{\Sp_{a}}^{\tilS_{a}}$ with $\QQ[\ep]$.Specializing $\ep=1$, we get a linear map
\begin{equation*}
\ell:\spcoh{}{\Sp_{a}}^{\tilS_{a}}\to \Gr^{P}_{2N}\spcoh{}{\Sp_{a}}^{\tilS_{a}}\cong\QQ.
\end{equation*}

Clearly $\Gr^{C}_{*}\spcoh{}{\Sp_{a}}^{\tilS_{a}\rtimes B_{a}}$ has an algebra structure induced from the cup product. By Proposition  \ref{p:Ch=perv}(1), $\Gr^{P}_{*}\spcoh{}{\Sp_{a}}^{\tilS_{a}\rtimes B_{a}}$ also carries an algebra structure induced by the cup product. Projecting to $\Gr^{P}_{2N}$ gives a linear function on $\Gr^{P}_{*}\spcoh{}{\Sp_{a}}^{\tilS_{a}\rtimes B_{a}}$ that we also denote by $\ell$.

\begin{lemma}\label{l:contra pairing} The pairing $(x,y):=\ell(xy)$ on $\Gr^{P}_{*}\spcoh{}{\Sp_{a}}^{\tilS_{a}\rtimes B_{a}}$ satisfies the following relations:
\begin{equation*}
(\xi x, y)= (x, \xi y), \forall \xi\in\fra^{*}; \  (\eta x, y)=-(x,\eta y), \forall \eta\in\fra \textup{ and } (wx, wy)=(x,y), \forall w\in W.
\end{equation*}
%
 %in the sense of \eqref{contra}. 
\end{lemma}
\begin{proof}
We need to show that three equations defining the contravariance are satisfied.
%\begin{equation*}
%(\xi x, y)= (x, \xi y), \forall \xi\in\fra^{*}; \  (\eta x, y)=-(x,\eta y), \forall \eta\in\fra \textup{ and } (wx, wy)=(x,y), \forall w\in W.
%\end{equation*}
The first equality is immediate since both sides are the degree $2N$ part of  $\xi xy$.

For the second equality, since $\Gr^{P}_{*}\spcoh{}{\Sp_{a}}^{\tilS_{a}\rtimes B_{a}}$ is generated by the images of the Chern polynomials, it suffices to prove it for $x=x_{1}\cdots x_{i}\cdot 1$ and $y=y_{1}\cdots y_{2N-i+1}\cdot 1$, where $x_{j}, y_{j}\in \fra^{*}$ and $1$ means the generator in $\Gr^{P}_{0}\spcoh{}{\Sp_{a}}^{\tilS_{a}\rtimes B_{a}}$. Since $xy\cdot 1=0$ and $\eta\cdot 1=0$ for degree reasons, we have 
\begin{equation*}
\sum_{j=1}^{i}x_{1}\cdots x_{j-1}[\eta, x_{j}]x_{j+1}\cdots x_{i}y+\sum_{j=1}^{2N-i+1}xy_{1}\cdots y_{j-1}[\eta, y_{j}]y_{j+1}\cdots y_{2N-i+1}=0
\end{equation*}
Now the first sum is $(\eta x)y$ and the second is $x(\eta y)$. Therefore  $(\eta x)y+x(\eta y)=0$, and the second equality holds.

The last equality follows from the fact that $W$ acts trivially on the one-dimensional space $\Gr^{P}_{2N}\spcoh{}{\Sp_{a}}^{\tilS_{a}}$. In fact, $\Gr^{P}_{2N}\spcoh{}{\Sp_{a}}^{\tilS_{a}}$ is spanned by $\ee^{N}\cdot 1$, and $W$ commutes with $\ee$.
\end{proof}

\begin{prop}\label{prop:Frob} Let $\nu$ be elliptic and $a\in\frc(F)^{\rs}_{\nu}$.
\begin{enumerate} 
\item The algebra $\spcoh{}{\Sp_{a}}^{\tilS_{a}}$ is a Frobenius algebra under the linear map $\ell$.
\item The algebra $\Gr^{P}_{*}\spcoh{}{\Sp_{a}}^{\tilS_{a}\rtimes B_{a}}$ is a Frobenius algebra under the same linear map.
\end{enumerate} 
\end{prop}
\begin{proof} As in the first paragraph of \S\ref{ss:Hrat local}, we may assume that \eqref{assrho} holds for $G$ and apply the local-global isomorphism in Proposition \ref{p:prod}.

(1) We switch to the global point of view and prove that $\spcoh{}{\calM_{a}}$ is a Frobenius algebra. Let $\DD$ be the dualizing complex of $\calM_{U}$, and $\DD'$ be the dualizing complex of $\Aa\times U$. Let $i:\Ah_{\nu}\times\{0\}\incl \Aa\times U$ be the inclusion of the fixed point locus of the $\Gnu$-action. We have a commutative diagram coming from the functoriality of Verdier duality pairings $h_{0}$ and $h$ (note $\fa_{U}$ is proper)
\begin{equation*}
\xymatrix{i_{*}i^{!} f^{\textup{ell}}_{U,*}\DD\otimes i_{*}i^{*}f^{\textup{ell}}_{U,*}\QQ\ar[r]^{i_{*}h_{0}}\ar@<-3ex>[d] & i_{*}i^{!}\DD'=i_{*}\DD_{\Ah_{\nu}}\ar[d]\\
f^{\textup{ell}}_{U,*}\DD\otimes f^{\textup{ell}}_{U,*}\QQ \ar[r]^{h}\ar@<-3ex>[u] & \DD'}
\end{equation*}
This implies that the composition $i^{!}f^{\textup{ell}}_{U,*}\DD\otimes i^{*}f^{\textup{ell}}_{U,*}\QQ\to i^{!}(f^{\textup{ell}}_{U,*}\DD\otimes f^{\textup{ell}}_{U,*}\QQ)\xrightarrow{i^{!}h} i^{!}\DD'=\DD_{\Ah_{\nu}}$ is the Verdier duality pairing $h_{0}$. We have another commutative diagram
\begin{equation}\label{!* diagram}
\xymatrix{i^{!} f^{\textup{ell}}_{U,*}\DD\otimes i^{*}f^{\textup{ell}}_{U,*}\QQ\ar[r]\ar@<-3ex>[d]\ar@<3ex>@{=}[d] & i^{!}(f^{\textup{ell}}_{U,*}\DD\otimes f^{\textup{ell}}_{U,*}\QQ)\ar[r]^(.7){i^{!}h}\ar[d] & i^{!}\DD'\ar[d]\\
i^{*} f^{\textup{ell}}_{U,*}\DD\otimes i^{*}f^{\textup{ell}}_{U,*}\QQ\ar[r] & i^{*}(f^{\textup{ell}}_{U,*}\DD\otimes f^{\textup{ell}}_{U,*}\QQ)\ar[r]^(.7){i^{*}h} & i^{*}\DD'}
\end{equation}
where the first row is the duality pairing $h_{0}$ and all vertical maps are induced from the natural transformation $i^{!}\to i^{*}$. By Lemma \ref{l:!*} below, inverting $\ep$ (or specializing $\ep=1$) makes all vertical maps isomorphisms. In particular, the second row above is also a Verdier duality pairing after inverting $\ep$. Choosing a fundamental class of $\calM_{U}$ allows us to identify $\DD$ with $\QQ[2\dim \calM_{U}](\dim \calM_{U})$. Then the second row above, up to a shift, factors through the cup product
\begin{equation}\label{cup complex}
f_{\nu, *}\QQ\otimes f_{\nu, *}\QQ\xrightarrow{\cup}f_{\nu, *}\QQ\xrightarrow{\wt\ell}\QQ[-2N](-N)
\end{equation}
for some map $\wt\ell: f_{\nu, *}\QQ\xrightarrow{\wt\ell}\QQ[-2N](-N)$. Since $\QQ[-2N](-N)$ lies in perverse degree $\dim \Ah_{\nu}+2N$ while $P_{\leq j}f_{\nu, *}\QQ:=i^{*}P_{\leq j}\fa_{*}\QQ$ lies in perverse degrees $\leq\dim\Ah_{\nu}+j$, the map $\wt\ell$ has to factor through the projection to $\Gr^{P}_{2N}f_{\nu, *}\QQ$, which is isomorphic to $\QQ[-2N](-N)$ (note that the whole complex $f_{\nu, *}\QQ$ is stable). Therefore, taking stalk at $a\in\Ah_{\nu}$, and specializing to $\ep=1$, the pairing \eqref{cup complex} is exactly the pairing $(u,v)\mapsto \ell(u\cup v)$ we defined on $\spcoh{}{\Sp_{a}}^{S_{a}}$. Since the second row of \eqref{!* diagram} is a duality pairing after setting $\ep=1$, and all the complexes involved there are direct sums of shifted local systems over $\Ah_{\nu}$, taking stalk at $a$ also gives a perfect pairing on $\spcoh{}{\calM_{a}}=\spcoh{}{\Sp_{a}}^{S_{a}}$. This proves (1).

(2) Taking $B_{a}$-invariants on $\spcoh{}{\Sp_{a}}^{\tilS_{a}}$ and restricting the pairing there we still get a perfect pairing since $B_{a}$ acts semisimply on $\spcoh{}{\Sp_{a}}^{\tilS_{a}}$ by Corollary \ref{c:Hess loc sys}, and acts trivially on $\Gr^{P}_{2N}\spcoh{}{\Sp_{a}}^{\tilS_{a}}$. Therefore $V=\spcoh{}{\Sp_{a}}^{\tilS_{a}\rtimes B_{a}}$ is also a Frobenius algebra under $\ell$. The pairing on $V$ gives an isomorphism $\iota: V\to V^{*}$ that sends $P_{\leq i}V$ to $(P_{\leq 2N-1-i}V)^{\bot}$. Verdier self-duality of $\spcoh{}{\Sp_{a}}^{\tilS_{a}}$, when restricted to $B_{a}$, gives the dimension equality $\dim P_{\leq i} V=\dim V-\dim P_{\leq 2N-1-i}=\dim(P_{\leq 2N-1-i}V)^{\bot}$. Therefore $\iota$ restricts to an isomorphism $P_{\leq i}V\isom(P_{\leq 2N-1-i}V)^{\bot}$ for all $i$. Taking associated graded we conclude that the pairing is also perfect on $\Gr^{P}_{*}V$. This proves (2). 
\end{proof}

\begin{lemma}\label{l:!*}
Let $i:\Ah_{\nu}\times\{0\}\incl \Aa\times U$ be the inclusion of the fixed point locus of the $\Gnu$-action. Suppose $F\in D^{b}_{\Gnu}(\Aa\times U)$ then the natural map $i^{!}F\to i^{*}F$, viewed as a map in the $\QQ[\ep]$-linear category $D^{b}_{\Gnu}(\Ah_{\nu})$ (where $\Gnu$ acts trivially), becomes an isomorphism after inverting $\ep$. 
\end{lemma}
\begin{proof}
By Lemma \ref{l:e0}, the action of $\Gnu$ on $\calA\times U$ is contracting to the fixed point locus $\calA_{\nu}\times\{0\}$. The contraction map $c_{\calA}:\calA\times U\to \calA_{\nu}\times\{0\}$ sends $(a,x)$ to $\lim_{\Gnu\ni s\to 0}s\cdot (a,x)$. Let $\calB\times U\subset\Aa\times U$ be the preimage of $\Ah_{\nu}\times\{0\}$ under $c_{\calA}$. Let $c:\calB\times U\to \Ah_{\nu}$ be the restriction of $c_{\calA}$, and still denote the inclusion $\Ah_{\nu}\times\{0\}\incl\calB\times U$ by $i$. Then it suffices to show that $i^{!}(F|_{\calB\times U})\to i^{*}(F|_{\calB\times U})$ becomes an isomorphism after inverting $\ep$. 

Now let $F\in D^{b}_{\Gnu}(\calB\times U)$ and consider the map $i^{!}F\to i^{*}F$. Since $F$ is $\Gnu$-equivariant, we have $i^{!}F\cong c_{!}F$, and the natural map $i^{!}F\to i^{*}F$ can be identified with the composition $i^{!}F\cong c_{!}F\to c_{!}i_{*}i^{*}F\cong i^{*}F$. The cone of the map $c_{!}F\to c_{!}i_{*}i^{*}F$ is $c_{!}j_{!}j^{*}F[1]$ where $j$ is the open inclusion of the complement of $\Ah_{\nu}\times U$. Since $\Gnu$ acts on $\calB\times U-\Ah_{\nu}\times \{0\}$ with finite stabilizers, the complex $(c\circ j)_{!}F'\in D^{b}_{\Gnu}(\Ah_{\nu})$ is $\ep$-torsion for any constructible $\Gnu$-equivariant complex $F'$ on $\calB\times U-\Ah_{\nu}\times \{0\}$. Therefore after inverting $\ep$, $c_{!}F\to c_{!}i_{*}i^{*}F$ becomes an isomorphism, hence so is $i^{!}F\to i^{*}F$.
\end{proof}

\begin{remark} The usual cohomology groups $\cohog{*}{\Sp_{a}}^{\tilS_{a}}$ and $\Gr^{P}_{*}\cohog{*}{\Sp_{a}}^{\tilS_{a}\rtimes B_{a}}$ are not necessarily  Frobenius algebras under $\ell$, as one can see from the example $\GG=\PGL_{2}$ and $\nu=3/2$.
\end{remark}

\begin{proof}[Proof of Theorem \ref{L(triv)}(2)] (Compare \cite[Proposition 1.20]{BEG}) The $\sl_{2}$-triple action on the finite-dimensional vector space $\Gr^{P}_{*}\spcoh{}{\Sp_{a}}^{\tilS_{a}\rtimes B_{a}}$ integrates to an action of $\SL_{2}$. Let $F=\mat{0}{1}{-1}{0}\in\SL_{2}$ and consider the pairing $J:\Gr^{P}_{*}\spcoh{}{\Sp_{a}}^{\tilS_{a}\rtimes B_{a}}\times \Gr^{P}_{*}\spcoh{}{\Sp_{a}}^{\tilS_{a}\rtimes B_{a}}\to\QQ$ defined by $J(a,b):=\ell(a\cdot F(b))$. Let $\frM_{\nu}(\triv)\cong \spcoh{}{\Fl^{\c}}$ be the polynomial representation of $\Hrat_{\nu,\ep=1}$ and let $\wt{J}$ be the pullback of $J$ to $\frM_{\nu}(\triv)$. By Lemma \ref{l:contra pairing}, this pairing is contravariant in the sense of \cite[Section 11.3]{ELec}, hence coincides with the pairing on $\frM_{\nu}(\triv)$ constructed in \cite{ELec} up to a scalar. Since $\wt{J}$ factors through the perfect pairing $J$ on $\Gr^{P}_{*}\spcoh{}{\Sp_{a}}^{\tilS_{a}\rtimes B_{a}}$, we conclude that $\Gr^{P}_{*}\spcoh{}{\Sp_{a}}^{\tilS_{a}\rtimes B_{a}}\cong \frL_{\nu}(\triv)$ by \cite[Lemma 11.6]{ELec}.
\end{proof}

\subsection{Langlands duality and Fourier transform}
In this subsection we show that the cohomology of homogeneous affine Springer fibers for the group $\GG$ and its Langlands dual $\dGG$ are in a certain sense Fourier transform of each other. The results in this subsection will not be used elsewhere in this paper.

%Since $\dGG$ is not simply-connected in general (we assumed $\GG$ to be simply-connected), some of the results from previous sections do not directly apply. However it will be clear how to modify the argument to make it work for non-simply-connected groups. One major modification is to switch from $S_{a}$-invariants to $\tilS_{a}$-invariants. We omit details here and refer them to \cite{GSLD}.

\subsubsection{Fourier transform on $\Hrat$}
Let $\sHrat\subset\Hrat$ be the subalgebra generated by $u,\delta, W,\fra$ and $\fra^{*}$. In other words, we suppress $\L_{\can}$ from the generators. The relations (RC-1) through (RC-4) still hold. In the following we will consider the rational \Chas for both $\GG$ and its Langlands dual $\dGG$. We denote them by $\sHrat_{\GG}$ and $\sHrat_{\dGG}$ respectively. For $\eta\in\fra$, let us use $\eta_{\GG}$ (resp. $\eta_{\dGG}$) to denote the corresponding element in $\sHrat_{\GG}$ (resp. $\sHrat_{\dGG}$). Similarly we define $\xi_{\GG}$ and $\xi_{\dGG}$ for $\xi\in\fra^{*}$. Then the relation (RC-4) for $\sHrat_{\dGG}$ reads
\begin{equation*}
[\xi_{\dGG},\eta_{\dGG}]=\jiao{\xi,\l}\delta+\dfrac{1}{2}\left(\sum_{\alpha\in\Phi}c_{\alpha}\jiao{\xi,\alpha^\vee}\jiao{\alpha,\l}r_{\alpha}\right)u.
\end{equation*}

There is an isomorphism $\iota_{\GG\to\dGG}:\sHrat_{\GG}\isom\sHrat_{\dGG}$ called the {\em Fourier transform}. It is determined by
\begin{equation*}
u\mapsto u, \delta\mapsto\delta, w\mapsto w (w\in W), \eta_{\GG}\mapsto-\eta_{\dGG} (\l\in\fra), \xi_{\GG}\mapsto\xi_{\dGG} (\xi\in\fra).
\end{equation*}
The minus sign put in front of $\eta_{\dGG}$ makes sure that (RC-4) for $\sHrat_{\GG}$ is equivalent to (RC-4) for $\sHrat_{\dGG}$.

For Langlands dual groups $\GG$ and $\dGG$, we may identify their Cartan subalgebras $\tt\cong\tt^{\vee}$ using the Killing form, and hence their invariant quotients $\cc\cong\cc^{\vee}$. We define the parabolic Hitchin moduli stacks $\calM$ and $\calM^{\vee}$ for $\GG$ and $\dGG$using the same curve $X$ and the same line bundle $\calL=\calO_{X}(d)$. Under the identification $\cc\cong\cc^{\vee}$, the Hitchin bases in the two situations are the same, which we still denote by $\calA$. 

Similarly we have the version where the point of Borel reduction varies in $U$. We denote the parabolic Hitchin fibrations for $\GG$ and $\dGG$ by
\begin{equation*}
f_{U}:\calM_{U}\to\calA\times U; \hspace{1cm}  f^{\vee}_{U}:\calM^{\vee}_{U}\to\calA\times U.
\end{equation*}
Let $\fa_{U}$ and $\fda_{U}$ be the restrictions of $\fa$ and $\fda$ over $\Aa\times U$. 

We have a $\Gnu$-equivariant version of the main result of \cite{GSLD}.

\begin{prop}[\cite{GSLD}]\label{p:duality} There is a canonical isomorphism of shifted perverse sheaves in $D^{b}_{\Grot\times\Gdil}(\Aa\times U)$:
\begin{equation*}
D^{i}_{\GG\to\dGG}:\Gr^{P}_{i}\fst\isom\Gr^{P}_{2N-i}\fdst[2N-2i](N-i)
\end{equation*}
which intertwines the action of $\sHrat_{\GG}$ and $\sHrat_{\dGG}$ through the Fourier transform $\iota_{\GG\to\dGG}$. Here $i=0,1,\cdots, 2N$ and $N$ is the relative dimension of $f^{\textup{ell}}_{U}$ and $f^{\vee, \textup{ell}}_{U}$.
\end{prop}
\begin{proof} To save notation, we write \footnote{Note the difference from notations in \cite{GSLD}, where $K^{i}$ and $L^{i}$ are perverse sheaves; here $K_{i}$ and $L_{i}$ are shifted perverse sheaves in perverse degree $i+\dim(\calA\times U)$.}
\begin{equation*}
K_{i}:=\Gr^{P}_{i}\fst; \hspace{1cm} L_{i}:=\Gr^{P}_{i}\fdst.
\end{equation*}
In \cite{GSLD}, where we worked without $\Grot\times\Gdil$-equivariance, we have defined the isomorphism $D^{i}_{\GG\to\dGG}:K_{i}\isom L_{2N-i}(N-i)$ and proved that $D^{i}_{\GG\to\dGG}$ intertwines the action of $\eta_{\GG}$ and $-\eta_{\dGG}$ for $\eta\in\fra$. The proof there extends to the $\Grot\times\Gdil$-equivariant setting as well. It is also easy to verify that $D^{i}_{\GG\to\dGG}$ commutes with the action of $u, \delta$ and $W$. It remains to show that $D^{i}_{\GG\to\dGG}$ intertwines $\xi_{\GG}$ with $\xi_{\dGG}$. Switching the roles of $\GG$ and $\dGG$, it suffices to show that $D^{i}_{\dGG\to\GG}$ intertwines $\eta_{\dGG}$ with $\eta_{\GG}$ for all $i$.

Consider the following diagram
\begin{equation}\label{KLK}
\xymatrix{K_{i}\ar[rr]^(.3){D^{i}_{\GG\to\dGG}}\ar[d]^{\eta_{\GG}} && L_{2N-i}[2N-2i](N-i)\ar[rr]^(.7){D^{2N-i}_{\dGG\to\GG}}\ar[d]^{-\eta_{\dGG}} && K_{i}\ar[d]^{\eta_{\GG}}\\
K_{i-1}\ar[rr]^(.3){D^{i-1}_{\GG\to\dGG}} && L_{2N-i+1}[2N-2i+2](N-i+1)\ar[rr]^(.7){D^{2N-i+1}_{\dGG\to\GG}} && K_{i-1}}
\end{equation}
Tracing through the definition of $D^{i}_{\GG\to\dGG}$ in \cite{GSLD}, we find that the composition
\begin{equation*}
D^{2N-i}_{\dGG\to\GG}\circ D^{i}_{\GG\to\dGG}:K_{i}\isom L_{2N-i}[2N-2i](N-i)\isom K_{i}
\end{equation*} 
is not the identity but the multiplication by $(-1)^{i}$. Therefore the outer square of diagram \eqref{KLK} commutes up to $-1$. Since the left square is commutative, the right square also commutes up to $-1$. This means that $D^{2N-i}_{\dGG\to\GG}$ intertwines $\eta_{\dGG}$ with $\eta_{\GG}$. This finishes the proof.
\end{proof}

\subsubsection{Homogeneous affine Springer fiber for the dual group}\label{sss:dual Sp} Let $\nu>0$ be an elliptic slope for $\GG$ and hence for $\dGG$. Identifying $\frc(F)_{\nu}$ and $\frc^{\vee}(F)_{\nu}$ using a Killing form on $\tt$. This also allows us to identify the braid group $B_{a}$ without its counterpart for $\dGG$. Let $a\in\frc(F)^{\rs}_{\nu}$. Let $\Sp^{\vee}_{a}$ be the homogeneous affine Springer fiber for $\dGG$ associated with  $\kappa^{\vee}(a)$ (where $\kappa^{\vee}:\cc^{\vee}\to\gg^{\vee}$ is the Kostant section for $\gg^{\vee}$).  The counterpart of $\tilS_{a}$ for $\dGG$ is denoted by $\tilS^{\vee}_{a}$. Applying the Proposition \ref{p:duality} to the stalk at $(a,0)\in\Ah_{\nu}\times U$, and using the isomorphism \eqref{lg stable}, we get

\begin{cor}\label{c:duality} Let $N=\dim\Sp_{a}$. There is a canonical isomorphism for $0\leq i \leq 2N, j\geq0$
\begin{equation*}
D^{i,j}_{\GG\to\dGG,a}:\Gr^{P}_{i}\eqnu{j}{\Sp_{a}}^{\tilS_{a}}\isom\Gr^{P}_{2N-i}\eqnu{j-2(N-i)}{\Sp^{\vee}_{a}}^{\tilS^{\vee}_{a}}
\end{equation*}
that is equivariant under $B_{a}$, such that the sum of these isomorphisms (for various $i,j$) intertwines the actions of $\Hrat_{\GG,\nu}$ and $\Hrat_{\dGG,\nu}$ through the Fourier transform $\iota_{\GG\to\dGG}$.
\end{cor}

%% Examples

\section{Examples}\label{s:examples}
%\subsection{Remarks on \cite{VV}}
%Particular case of the conjecture~\ref{conj:surj} is discussed in \cite{VV}. It is shown in \cite{VV} that $\Hess_1$  is zero-dimensional and group $A_{m}^{\circ}=\pi_0(L_\nu)/Z(L_\nu)$ 
%acts transitively  on $\Hess_1$ if $|\Hess_1|=|A_{m}^{\circ}|$ (see Lemma 3.3.5 in \cite{VV}). The equality $|\Hess_1|=|A_{m}^{\circ}|$ is stated as conjecture 3.3.3 in \cite{VV} and
%it is shown to be equivalent to the root theoretic equality $(-1)^|W_\nu| |A^{\circ}_m| eu=\sum_{w\in W_\nu} (-1)^{\ell(w)} w(p)$ where $eu=\prod_{\alpha\in\Phi_L} \bar{\alpha}$,
%$p=\prod_{\alpha\in \Psi_\nu \cap \Phi_{aft}^{-}}$. In the section 4 of \cite{VV} the last root theoretic cases is shown for all exceptional groups except $E_8, m<15$ and $F_4, m<6$.
%The authors of \cite{VV} use the particular case of the conjecture~\ref{conj:surj} to classify the decomposition factors of the push forward $\pi_{1,*}^h(\CC_{\dot{\mathcal{N}^h}})$
%$\pi_{1}^h: \dot{\mathcal{N}}^{h}_1\to\mathcal{N}^h_1$ that restrict nontrivially on the regular semi simple part of $\mathcal{N}^h$, thus providing Jordan-H\"older decomposition of 
%$\cohog{*}{\Sp_\gamma}$ as module over double affine Hecke algebra (see theorem 3.3.6).

For quasi-split groups $G$ of rank two, the \Cha $\Hrat_{\nu}$ is the one attached to a dihedral Weyl group with possibly unequal parameters. We will review the algebraic representation theory of such \Chas in \S\ref{ss:dih}. After that, for each almost simple, simply-connected quasi-split group $G$ over $F$ of rank at most two, we study the cohomology of its homogeneous affine Springer fibers and the relevant Hessenberg varieties of various slopes $\nu=1/m_{1}$, where $m_{1}$ is a $\theta$-regular elliptic number. 

The general slopes $\nu=d_{1}/m_{1}$ can be reduced to the case $\nu=1/m_{1}$ by Proposition \ref{p:reduce to 1/m}. The case $m_{1}=h_{\theta}$, the twisted Coxeter number, is treat for all $G$ in Example \ref{ex:Cox}. Therefore we only look at cases where $m_{1}$ is not the twisted Coxeter number. For each $\nu=1/m_{1}$, we describe the $L_{\nu}$-module $\frg(F)_{\nu}$ in terms of simple linear algebra such as quadratic forms, cubic forms, etc.  Following the discussion in \S\ref{sss:clans}, we will show in pictures how to decompose the apartment $\frA$ into clans, and describe the Hessenberg varieties $\Hess^{\tilw}_{a}$ ($\gamma\in\frc(F)^{\rs}_{\nu}$) corresponding to each bounded clan. These Hessenberg varieties turn out to be familiar objects in classical projective geometry. Finally we compute the dimension of $\cohog{*}{\Sp_{a}}^{S_{a}\rtimes B_{a}}$, and verify that it is consistent with the dimension of  irreducible $\Hrat_{\nu,\ep=1}$-modules known from the algebraic theory.

In the final subsection we give computational and conjectural results on the dimension of $\frL_{\nu}(\triv)$ in general.

\subsection{Algebraic theory}\label{ss:dih} The Weyl groups of rank two quasisplit simple groups are dihedral groups $I_{2}(d)$ of order $2d$ for $d=3, 4$ or $6$.  The representation theory of the rational Cherednik algebras of dihedral type was studied by Chmutova in \cite{Chm}. Below we use the classification of the finite-dimensional representations from \cite{Chm} and for convenience of reader we explain how to match our notations for the Cherednik algebras with 
the notations from \cite{Chm}. 

We shall concentrate on the case $d$ is even because $d=3$ only appears in the case of type $A_{2}$ and $m$ is the Coxeter number. The group $I_{2}(d)$ acts on $\mathbb{R}^2=\mathbb{C}$ and the set $S$ of reflections in the group $I_{2}(d)$ consists of elements $s_j$, $1\le j\le d$, $s_j(z)=\omega^j \bar{z}$, $\omega=\exp(2\pi i /d)$,
$z=x+iy$. The dihedral group $I_2(d)$ acts on $S$ with two orbits: $S_1=\{s_{2j+1}\}, S_2= \{s_{2j}\}$. The algebra  $\Hrat_{k_1,k_2}(I_{2}(d))$ over $\RR$ is generated by the elements $g\in I_{2}(d)$ and $x,y\in \CC$ defined by the relation:
\begin{gather}
wxw^{-1}=w(x),\quad wyw^{-1}=w(y),\quad [x,y]=1-\sum_{i=1}^2 k_i \sum_{s\in S_i} \jiao{\alpha_{s},x}\jiao{y, \alpha^\vee_{s}}s.
\end{gather}
The group $I_{2}(d)$ has four one-dimensional representations $\chi_{\epsilon_1,\epsilon_2}$ defined by $\chi_{\epsilon_1,\epsilon_2}(s)=\epsilon_i$ if $s\in S_i$. Results of Chmutova \cite[Theorem 3.2.3]{Chm}
provide formulae for the dimensions of the simple modules $\frL_{\nu}(\chi_{\epsilon_1,\epsilon_2})$. Below we decompose the cohomology of  the affine Springer fibers into $\Hrat_{k_1,k_2}(I_{2}(d))$-isotypical 
components and compare the dimensions of these components with the results of \cite{Chm}.

\subsection{Type $^{2}\! A_{2}$, $m_{1}=2$} The only regular elliptic number besides the twisted Coxeter number $m_{1}=6$ is $m_{1}=2$. Let $\nu=1/2$. Let $\GG=\SL(V)$ where $V$ is a $3$-dimensional vector space $\CC$ equipped with a quadratic form $q_{0}$. Denote the associated symmetric bilinear form by $(\cdot,\cdot)$. The vector space $V\otimes_{\CC}\CC((t^{1/2}))$ is equipped with a $\CC((t^{1/2}))$-valued Hermitian form $h(x+t^{1/2}y,z+t^{1/2}w)=(x,z)-t(y,w)+t^{1/2}(y,z)-t^{1/2}(x,w)$, for $x,y,z,w\in V\otimes_{\CC}F$. Let $G$ be the special unitary group $\SU(V\otimes_{\CC}\CC((t^{1/2})), h)$.
 
\subsubsection{$L_{\nu}$ and $\frg(F)_{\nu}$} Let $\alpha$ be the simple root of $\HH=\GG^{\theta}$. The real affine roots of $G(F)$ are:
\begin{equation*}
\pm\alpha+\ZZ\delta/2, \quad\pm2\alpha+\delta/2+\ZZ\delta.
\end{equation*}

\begin{figure}
\centering
\includegraphics[trim=0cm 11cm 0cm 11cm, clip=true,height=2cm]{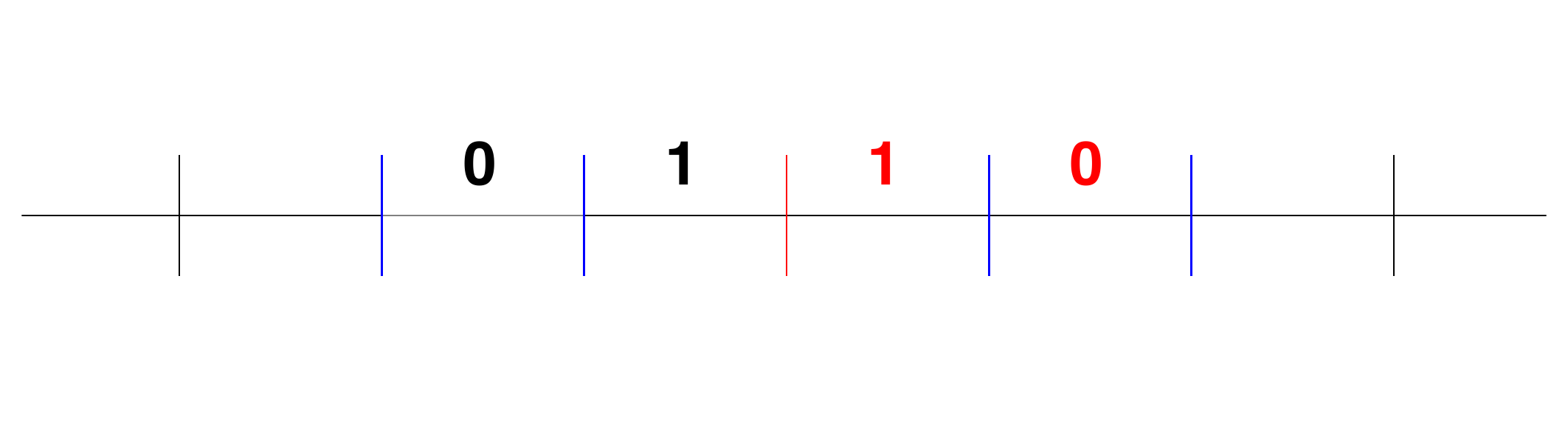}
\caption{The apartment $\frA$ for $^{2}\! A_2$, $\nu=1/2$}
\end{figure}

We have $L_{\nu}\cong\PGL_{2}$ with roots $\pm(\alpha-\delta/2)$. The affine roots appearing in $\frg(F)_{\nu}$ are $\alpha, -\alpha+\delta, 2\alpha-\delta/2, -2\alpha+3\delta/2$ and $\delta/2$. We write $L_{\nu}=\PGL(U)$ for a $2$-dimensional vector space $U$, then $\frg(F)_{\nu}\cong\Sym^{4}(U^{\vee})\otimes\det(U)^{\otimes2}$. To each $\gamma\in\frg(F)^{\rs}_{\nu}$, viewed as a binary quartic form, one can attach a curve $C_{\gamma}$ of genus one as the double cover of $\PP^{1}=\PP(U)$ ramified above the four zeros of $\gamma$. The pair $(L_{\nu}, \frg(F)_{\nu})$ has been used by Bhargava and Shankar to compute the average size of $2$-Selmer groups of elliptic curves over $\QQ$, see \cite{BS2} and \cite{Gross}. 

\subsubsection{The Hessenberg varieties} The picture above shows the apartment  for $^{2}\! A_2$ with the alcoves marked with the expected dimensions of the Hessenberg varieties. We use gray to mark the fundamental alcove and the red numbers
to mark the alcoves that represent the $W_\nu$-orbits.
The $1$-dimensional Hessenberg variety is $\PP(U)\cong\PP^{1}$. The 0-dimensional Hessenberg variety consists of four points which are zeros of the binary quartic form $\gamma$, or equivalently the Weierstrass points of $C_{\gamma}$. One can identify the stabilizer $S_{\gamma}$ with $J(C_{\gamma})[2]$ where $J(C_{\gamma})$ is the Jacobian of $C_{\gamma}$. Then $S_{\gamma}$ acts simply transitively on the 0-dimensional Hessenberg variety. 

\subsubsection{Irreducible modules} From the above discussion we conclude that for $a\in\frc(F)^{\rs}_{\nu}$,
\begin{equation*}
\spcoh{}{\Sp_{a}}^{S_{a}}=\spcoh{}{\Sp_{a}}^{S_{a}\rtimes B_{a}}\cong\cohog{*}{\PP^{1}}\oplus\cohog{*}{\pt}
\end{equation*}
is 3-dimensional.  On the other hand, $\Hrat_{\nu,\ep=1}$ in this case is isomorphic to the rational Cherednik algebra of type $A_{1}$ with parameter $3/2$, which has a unique finite-dimensional irreducible module $\frL_{\nu}(\triv)$, and $\dim \frL_{\nu}(\triv)=3$. Hence $\Gr^{C}_{*}\spcoh{}{\Sp_{a}}^{S_{a}}\cong \frL_{\nu}(\triv)$.

\subsection{Type $C_2$, $m=2$} In type $C_{2}$, the only regular elliptic number besides the Coxeter number $m=4$ is $m=2$. Let $\nu=1/2$. Let $\GG$ be the split group $\Sp(V,\omega)$ over $\CC$, where $(V,\omega)$ is a symplectic vector space over $\CC$ of dimension four. Let $G=\GG\otimes_{\CC}F$.

\subsubsection{$L_{\nu}$ and $\frg(F)_{\nu}$} The real affine roots of $G(F)$ are:
$$ \pm 2\epsilon_i+\ZZ\delta, \quad  \pm \epsilon_1\pm\epsilon_2+\ZZ\delta.$$
The roots of $L_\nu$ are $\pm(\epsilon_1+\epsilon_2-\delta)$. The affine roots appearing in $\frg(F)_{\nu}$ are: $$2\epsilon_1-\delta,\quad\epsilon_1-\epsilon_2,\quad-2\epsilon_2+\delta,\quad 2\epsilon_2,
\quad\epsilon_2-\epsilon_1+\delta,\quad -2\epsilon_1+2\delta.$$

The Levi factor $L_{\nu}$ can be identified with $\GL(U)$ where $V=U\oplus U^{\vee}$ is a decomposition of $V$ into two Lagrangian subspaces. As an $L_\nu$-representation,  $\frg(F)_{\nu}\cong\Sym^2(U)\oplus \Sym^2(U^{\vee})$, i.e., the space of pairs of quadratic forms $q_1$ and $q_2$ on $U$ and $U^{\vee}$ respectively. 

Choose a basis $e_{1}, e_{2}$ for $U$ and let $f_{1}$ and $f_{2}$ be the dual basis for $U^{\vee}$. Up to $L_{\nu}$-conjugation an element $\gamma\in\frg(F)^{\rs}_{\nu}$ takes the form $q_{1}=f_{1}^{2}+f_{2}^{2}$ and $q_{2}=\l_{1}e_{1}^{2}+\l_{2}e_{2}^{2}$ for some $\l_{1},\l_{2}\in \CC^{\times}$, $\l_{1}\neq\l_{2}$. The centralizer group $S_{\gamma}$ of  the regular element $\gamma$ is $\mu_{2}\times\mu_{2}$, and it acts by changing signs of the basis vectors of $U$ and $U^{\vee}$.

\subsubsection{The Hessenberg varieties} The picture with $\nu$-wall is given above and it indicates that there  are only three classes in $W_\nu\backslash \tilW$ such that the corresponding Hessenberg variety is non-empty. 

\begin{figure}
\centering
\includegraphics[height=5cm]{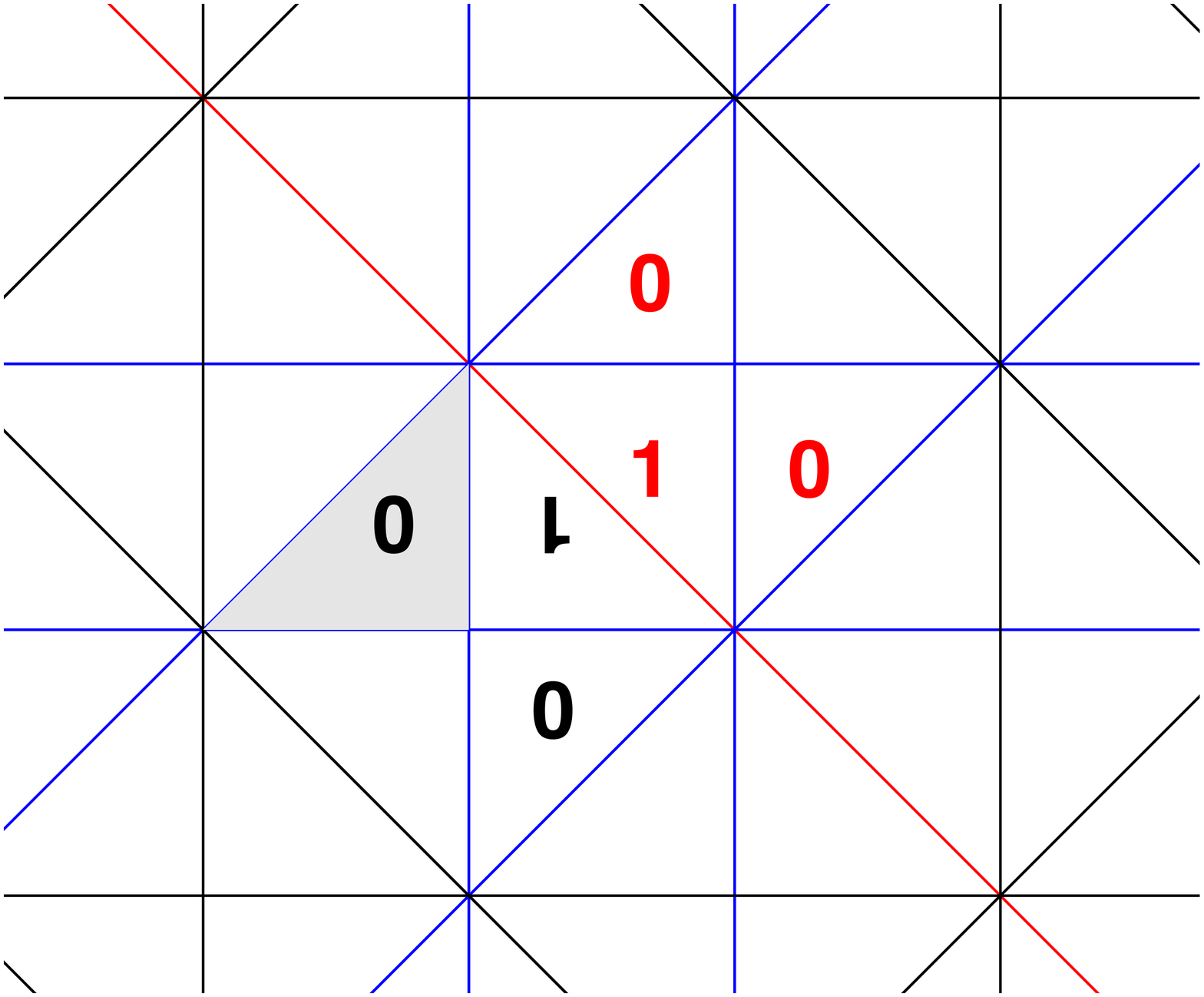}
\caption{The apartment $\frA$ for $C_2$, $\nu=1/2$}
\end{figure}

The one-dimensional Hessenberg variety is $\PP(U)\cong\PP^{1}$, the flag variety of $L_{\nu}$. There are two kinds of zero-dimensional Hessenberg varieties, each consisting of lines $\ell\subset U$ which are isotropic with respect to one of the forms $q_{1}$ and $q_{2}$. In either case the zero-dimensional Hessenberg variety consists of two points that are permuted transitively by $S_{\gamma}$.

\subsubsection{Irreducible modules}  From the above discussion we conclude that for $a\in\frc(F)^{\rs}_{\nu}$,
\begin{equation*}
\spcoh{*}{\Sp_{a}}^{S_{a}}=\spcoh{*}{\Sp_{a}}^{S_{a}\rtimes B_{a}}\cong\cohog{*}{\mathbb{P}^1}\oplus\cohog{*}{\pt}\oplus \cohog{*}{\pt}
\end{equation*}
has dimension $4$.   On the other hand, by \cite[Theorem 3.2.3(vi)]{Chm}, the algebra $\Hrat_{1/2,1/2}(I_{2}(4))$ has a unique finite-dimensional representation $\frL_{\nu}(\triv)$ and it is of dimension $4$. Hence $\Gr^{P}_{*}\spcoh{}{\Sp_{a}}^{S_{a}}\cong \frL_{\nu}(\triv)$.

\subsection{Type ${}^2\! A_3$, $m_{1}=2$}Besides the twisted Coxeter number$m_{1}=6$, the only regular elliptic number in this case is $m_{1}=2$. Let $\nu=1/2$. Let $\GG=\SL(V)$ where $V$ is a $4$-dimensional vector space $\CC$ equipped with a quadratic form $q_{0}$. As in the case of type $^{2}\! A_{2}$, we have the associated symmetric bilinear form $(\cdot,\cdot)$, which extends to a Hermitian form $h$ on $V\otimes_{\CC}\CC((t^{1/2}))$. Let $G$ be the special unitary group $\SU(V\otimes_{\CC}\CC((t^{1/2})), h)$.

\subsubsection{$L_{\nu}$ and $\frg(F)_{\nu}$} The real affine roots of $G(F)$ are:
$$2\epsilon_i+\ZZ\delta,\quad \pm \epsilon_1\pm \epsilon_2+\ZZ\delta/2,\quad i=1,2,$$
and $\nu\rho^\vee=(3/4,1/4)$ under the coordinates $(\ep_{1},\ep_{2})$.

The roots of $L_{\nu}$ are $\pm(\epsilon_1+\epsilon_2-\delta)$ and $\pm(\epsilon_1-\epsilon_2-\delta/2)$. The affine roots in $\frg(F)_{\nu}$ are:
$$\delta/2, \quad\pm (2\epsilon_1-3\delta/2)+\delta/2,\quad \pm(2\epsilon_2-\delta/2)+\delta/2,\quad \pm(\epsilon_1+\epsilon_2-\delta)+\delta/2,\quad\pm(\epsilon_1-\epsilon_2-\delta/2)+\delta/2.$$ 

The group $L_{\nu}$ may be identified with $\SO(V,q_{0})$, and $\frg(F)_{\nu}$ may be identified with the vector space of self-adjoint traceless endomorphisms of $V$ with respect to $(\cdot,\cdot)$ . 

Choose an orthonormal basis $e_{1},\cdots, e_{4}$ for $V$, we may take
the Cartan subspace $\frs$ to consist of diagonal matrices $\textup{diag}(a_1,a_2,a_3,a_4)$ with respect to this basis, with small Weyl group $S_{4}$ acting by permutations of $a_i$. The regular 
semisimple locus $\frs^{\rs}$ consists of the diagonal matrices with distinct eigenvalues. For $\gamma\in\frs^{\rs}$, the group $S_\gamma\subset\mu^{4}_{2}$ consists of diagonal elements of determinant $1$.

\subsubsection{The Hessenberg varieties} The picture contains marking by the expected dimensions of the Hessenberg varieties and the fundamental alcove is shaded grey. There are four alcoves with nonnegative 
expected dimension in the dominant chamber with respect to $W_{\nu}$, and study the relevant Hessenberg varieties below.

\begin{figure}
\centering
\includegraphics[height=5cm]{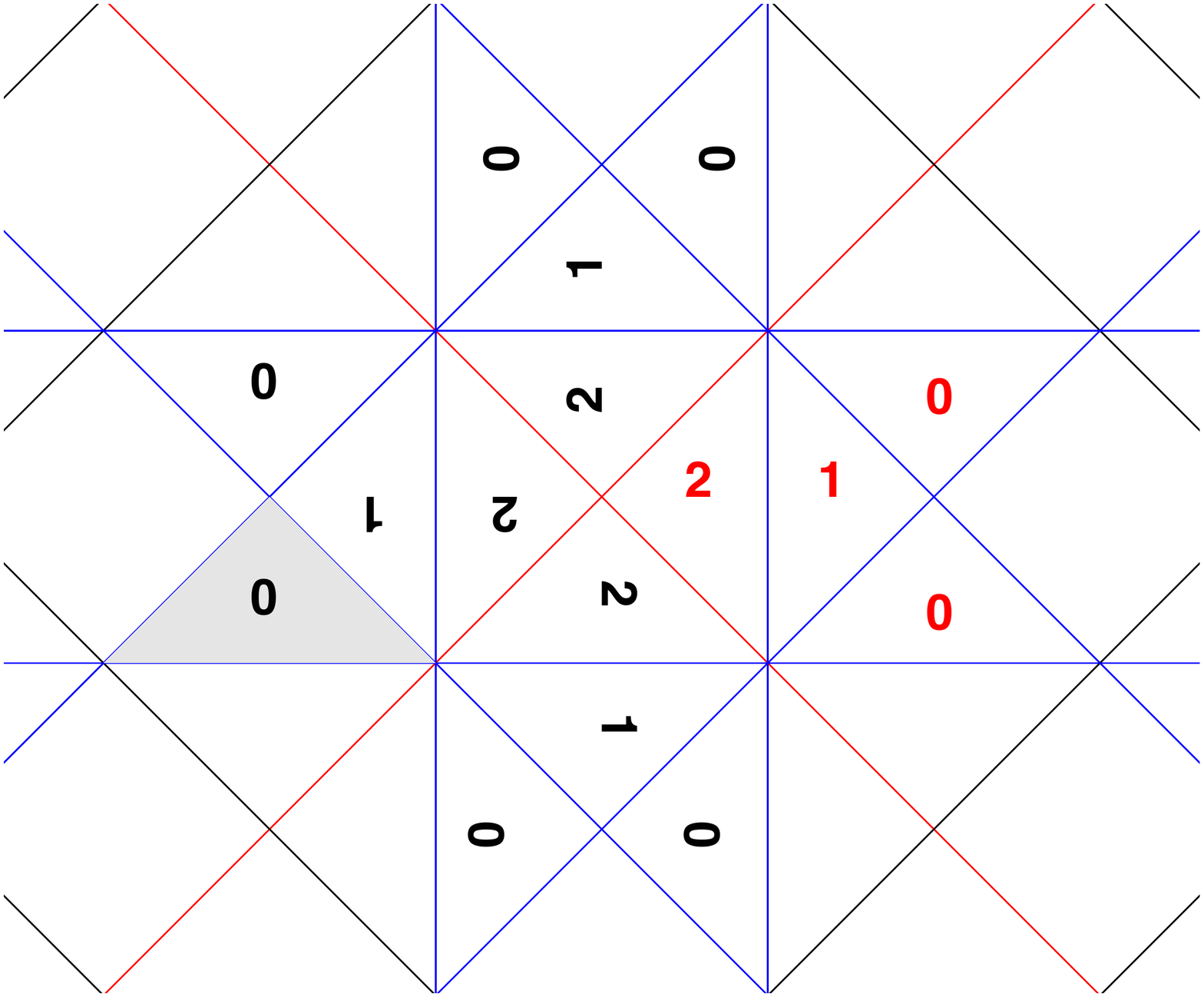}
\caption{The apartment $\frA$ for ${}^2\! A_3$, $\nu=1/2$}
\end{figure}

Thus the two-dimensional Hessenberg variety is the flag variety of $L_{\nu}$, hence isomorphic to $\mathbb{P}^1\times\mathbb{P}^1$.  Indeed, the flag variety of $L_{\nu}$ consisting of flags $0\subset V_1\subset V_3=V_{1}^{\bot}\subset V$, hence is isomorphic to the quadric $Q_{1}\subset \PP(V)$ defined by $(v,v)=0$,  which is isomorphic to $\PP^{1}\times\PP^{1}$. The space of isotropic 2-planes $V_2\subset V$ consists of two one-dimensional families of rulings of the quadric $Q_{1}\cong\mathbb{P}^1\times \mathbb{P}^1$. For each flag $0\subset V_1\subset V_3=V_{1}^{\bot}\subset V$, there are exactly two isotropic 2-planes $V_{2}$ satisfying $V_{1}\subset V_{2}\subset V_{3}$, one from each ruling. We denote the two isotropic 2-planes by $V'_{2}$ and $V''_{2}$.

The one-dimensional Hessenberg variety parametrizes flags $0\subset V_1\subset V_3=V_{1}^{\bot}\subset V$ such that $\gamma(V_{1})\subset V_{3}$. A nonzero isotropic vector $v\in V_{1}$ determines a flag $V_{1}=\jiao{v}\subset\jiao{v}^{\bot}=V_{3}$, and the condition $\gamma(V_{1})\subset V_{3}$ is equivalent to saying that $(v,\gamma v)=0$. Define another quadric $Q_{\gamma}\subset\PP(V)$ to be the locus $(v,\gamma v)=0$,  then the one-dimensional Hessenberg variety is isomorphic to the intersection of two quadrics $\Sigma_{\gamma}:=Q_{1}\cap Q_{\gamma}\subset \PP(V)$, which is a curve of genus one. 

There are two kinds of zero-dimensional Hessenberg varieties each consisting of flags satisfying the condition $\gamma V_1\subset V'_2$ or $\gamma V_1\subset V''_2$. The union of these two zero-dimensional varieties is the intersection of three quadrics $(v,v)=(v,\gamma v)=(\gamma v,\gamma v)=0$, which consists of eight points. However only four of these points belong to each zero-dimensional Hessenberg variety. In fact, let  $\pi_{1}$ and $\pi_{2}$ be the projections $\Sigma_{\gamma}\subset Q_{1}\cong\PP^{1}\times\PP^{1}\to\PP^{1}$ onto either factor $\PP^{1}$. Then the four ramification points of $\pi_{1}$ constitute one Hessenberg variety and the four ramification points of $\pi_{2}$ constitute the other. 

\subsubsection{The genus one curve}
The curve $\Sigma_{\gamma}$ is closely related to another curve $C_{\gamma}$ which we now define. Consider the pencil of quadrics spanned by $Q_{1}$ and $Q_{\gamma}$. Each non-singular member of the pencil contains two rulings, we can then form the moduli space of rulings in this pencil which gives a double cover $C_{\gamma}\to\PP^{1}$ ramified at the points where the quadric is singular (there are 4 such points). The curve $C_{\gamma}$ has genus one. It is equipped with an involution $\sigma$ over $\PP^{1}$. 

There is a map $a:C_{\gamma}\times \Sigma_{\gamma}\to\Sigma_{\gamma}$ defined as follows. A point on $c\in C_{\gamma}$ is a quadric $Q'$ in the pencil together with a ruling on $Q'$. Choosing $p\in \Sigma_{\gamma}$. On the ruling of $Q'$ there is a unique projective  line $\ell$ passing through $p$. The pair $p\in\ell$ corresponds to a line and plane$V_{1}\subset V_{2}\subset V$ both isotropic under $Q'$. Since $V_{2}$ contains exactly two lines that are isotropic for the whole pencil, let $V'_{1}$ be the other line. The map $a$ then sends  $(c,p)$ to the point $V'_{1}\in \Sigma_{\gamma}$. It is easy to see that for fixed $c\in C_{\gamma}$, $a(c,-)$ is an automorphism of $\Sigma_{\gamma}$ (with inverse given by $\sigma(c)$). Therefore $\Sigma_{\gamma}$ is isomorphic to $C_{\gamma}$ but not canonically so. 

The action of $S_{\gamma}$ on $\Sigma_{\gamma}$ factors through the quotient $\overline{S}_{\gamma}:=S_{\gamma}/\Delta(\mu_{2})$ where $\Delta(\mu_{2})$ is the subgroup of scalar matrices in $L_{\nu}$.  We claim that the action of $\overline{S}_{\gamma}$ on $\Sigma_\gamma$ is fixed point free. In fact, the fixed point locus of an involution from $\overline{S}_{\gamma}$ on $\PP(V)$ consists of the
union of two lines $\{x_i=x_j=0\} \cup \{x_{\bar{i}}=x_{\bar{j}}=0\}$, where $\{1,2,3,4\}=\{i,j,\bar{i},\bar{j}\}$.  The condition that $a_{1},\cdots,a_{4}$ are distinct implies that the intersection of $\Sigma_\gamma$ with these lines
is empty. The Jacobian $J(\Sigma_{\gamma})$ of $\Sigma_{\gamma}$ is an elliptic curve and $\Sigma_{\gamma}\cong\Pic^{1}(\Sigma_{\gamma})$ is a torsor under $J(\Sigma_{\gamma})$. Since $\overline{S}_{\gamma}$ acts freely on $\Sigma_\gamma$, it can only acts via translation by $J(\Sigma_{\gamma})[2]$, and we have a natural isomorphism $\overline{S}_{\gamma}\cong J(\Sigma_{\gamma})[2]\cong J(C_{\gamma})[2]$. In particular, $S_{\gamma}$ acts trivially on $\cohog{*}{\Sigma_\gamma}$.

Since the ramification points of both projections $\pi_{1}, \pi_{2}: \Sigma_{\gamma}\to\PP^{1}$ are torsors under $J(\Sigma_{\gamma})[2]$ and $\overline{S}_{\gamma}$ acts on $\Sigma_{\gamma}$ via an isomorphism $\overline{S}_{\gamma}\cong J(\Sigma_{\gamma})[2]$, it acts simply transitively on each of the zero dimensional Hessenberg varieties.

As $a$ varies over $\frc(F)^{\rs}_{\nu}$, we get a family of genus one curves $\pi^{C}:C\to\frc(F)^{\rs}_{\nu}$ whose fiber over $a$ is $C_{a}:=C_{\gamma}$ for $\gamma=\kappa(a)$. We have local systems $\bR^{i}\pi^{C}_{*}\QQ$ on $\frc(F)^{\rs}_{\nu}$ for $i=0,1,2$. Clearly $\bR^{0}\pi^{C}_{*}\QQ$ and $\bR^{2}\pi^{C}_{*}\QQ$ are trivial local systems.  As $a$ varies over $\frc(F)^{\rs}_{\nu}$, the ramification locus of $C_{a}\to \PP^{1}$ runs over all possible configuration of 4 distinct ordered points on $\AA^{1}$, as we have already seen as $\gamma$ runs over $\frs^{\rs}$. Therefore the braid group $B_{a}$ acts on $\cohog{1}{C_{a}}\cong \cohog{1}{\Sigma_{\kappa(a)}}$ via a surjective homomorphism $B_{a}\surj\Gamma(2)=\ker(\SL_2(\ZZ)\to\SL_{2}(\ZZ/2\ZZ))$. In particular, $\bR^{1}\pi^{C}_{*}\QQ$ is an irreducible local system over $\frc(F)^{\rs}_{\nu}$.

\subsubsection{Irreducible modules} From the above discussion we get that for $a\in\frc(F)^{\rs}_{\nu}$,
\begin{equation*}
\spcoh{}{\Sp_{a}}^{S_{a}}\cong\cohog{*}{\mathbb{P}^1\times\mathbb{P}^1}\oplus \cohog{*}{C_{a}}\oplus \cohog{*}{\pt}\oplus\cohog{*}{\pt}.
\end{equation*}
Only the direct summands $\cohog{1}{C_{a}}$ form an irreducible nontrivial local system as $a$ varies, and the other direct summands vary in a constant local system. Therefore
\begin{equation*}
\spcoh{}{\Sp_{a}}^{S_{a}}=\spcoh{}{\Sp_{a}}^{S_{a}\rtimes B_{a}}\oplus \cohog{1}{C_{a}}.
\end{equation*}
We have $\dim\spcoh{}{\Sp_{a}}^{S_{a}\times B_{a}}=8$. By \cite[Theorem 3.2.3]{Chm}, the simple modules $\frL_{\nu}(\triv)$ and $\frL_{\nu}(\chi_{+-})$ are the only finite dimensional irreducible representations of $\Hrat_{1,1/2}(I_{2}(4))$, and they have dimensions $8$ and $1$ respectively. Hence $\Gr^{C}_{*}\spcoh{}{\Sp_{a}}^{S_{a}\times B_{a}}\cong\frL_{\nu}(\triv)$ as $\Hrat_{1,1/2}(I_{2}(4))$-modules. We expect to have an action of $\Hrat_{1,1/2}(I_{2}(4))$ on $\Gr^{P}_{*}\cohog{*}{\Sp_a}^{S_{a}}$ (for some filtration $P_{\leq i}$ on $\spcoh{}{\Sp_{a}}^{S_{a}}$ extending the Chern filtration) commuting with the action of $B_{a}$. If this is true, the multiplicity space of the irreducible local system $\bR^{1}\pi^{C}\QQ$ should be the 1-dimensional module $\frL_{\nu}(\chi_{+-})$, and we should have an isomorphism of $\Hrat_{1,1/2}(I_{2}(4))\times B_{a}$-modules
$$\Gr^{P}_{*}\cohog{*}{\Sp_{a}}^{S_{a}}\stackrel{?}{=}\frL_{\nu}(\triv)\oplus \frL_{\nu}(\chi_{+-})\otimes\cohog{1}{C_{a}}.$$

\subsection{Type ${}^2\! A_4$, $m_{1}=2$} Besides the Coxeter case $m_{1}=10$, the only regular elliptic number in this case is $m_{1}=2$. Let $\nu=1/2$. As in the case of type ${}^2\! A_3$, we fix a five dimensional quadratic space $(V,q_{0})$ over $\CC$, define $\GG=\SL(V)$ and $G=\SU(V\otimes_{\CC}\CC((t^{1/2})),h)$ for the Hermitian form $h$ on $V\otimes_{\CC}\CC((t^{1/2}))$ extending $q_{0}$.

\subsubsection{The pair $L_{\nu}$ and $\frg(F)_{\nu}$} The real affine roots of  $G(F)$ are 
$$\pm \epsilon_i+\delta\ZZ/2,\quad \pm \epsilon_1\pm\epsilon_2+\delta \ZZ/2,\quad \pm 2\epsilon_i+\delta/2+\delta\ZZ,\quad i=1,2$$
and $\nu\rho^\vee=(1,1/2)$. The roots  of $L_{\nu}$ are:
$$\pm(\epsilon_1-\delta),\quad \pm (\epsilon_2-\delta/2),\quad \pm(\epsilon_1-\epsilon_2-\delta/2),\quad \pm(\epsilon_1+\epsilon_2-3\delta/2).$$
The affine roots appearing in $\frg(F)_{\nu}$ are:
\begin{gather*}
\delta/2,\quad \pm(\epsilon_1-\delta)+\delta/2,\quad \pm(\epsilon_2-\delta/2)+\delta/2,\quad \pm(\epsilon_1+\epsilon_2-3\delta/2)+\delta/2,\\
\pm(\epsilon_1-\epsilon_2-\delta/2)+\delta/2, \quad \pm(2\epsilon_1-2\delta)+\delta/2,\quad \pm(2\epsilon_2-\delta)+\delta/2.
\end{gather*}

In the picture the fundamental alcove is shaded gray and alcoves are labeled by the expected dimensions of the corresponding Hessenberg varieties. The $\nu$-walls are blue and the walls of the Weyl group
$W_\nu$ are red.
\begin{figure}
\centering
\includegraphics[height=5cm]{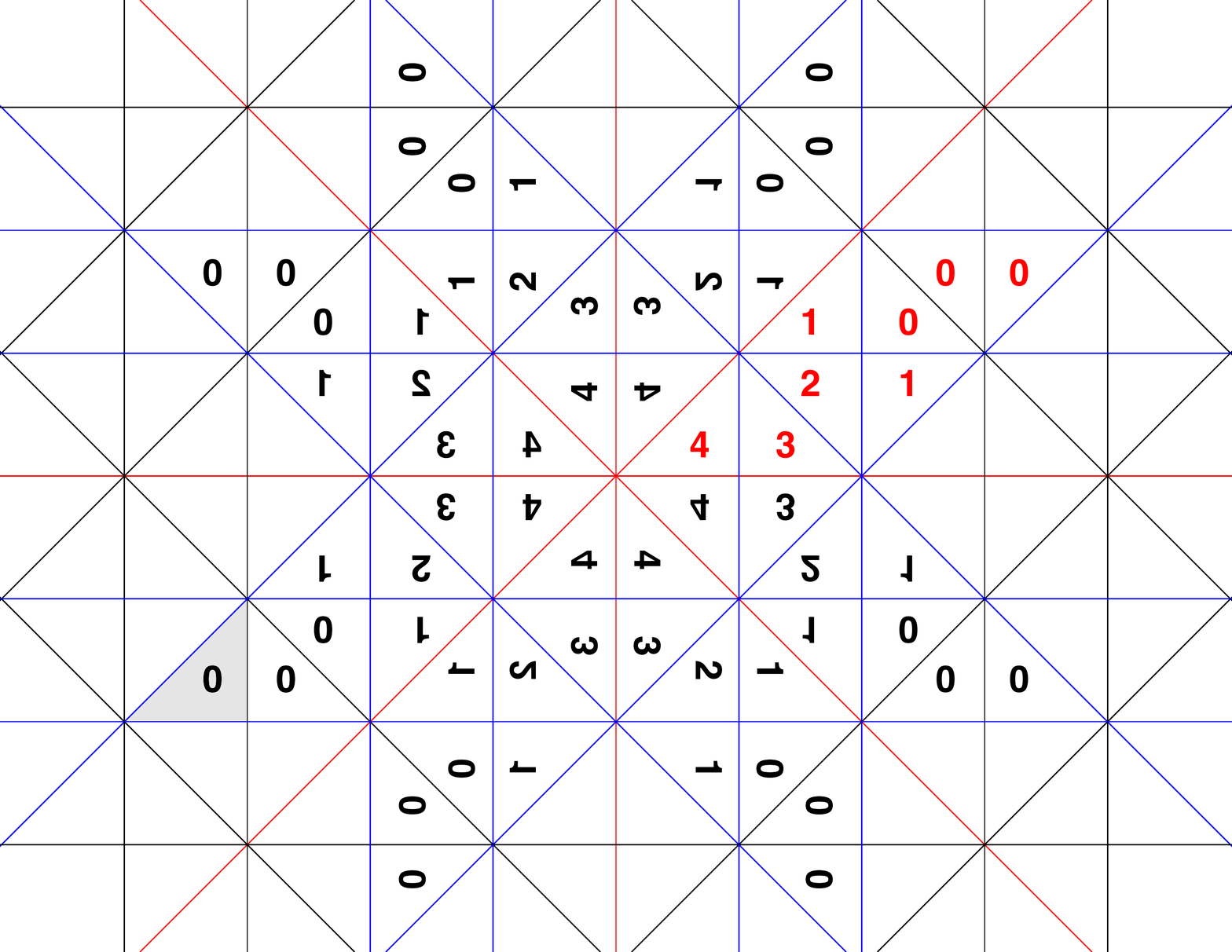}
\caption{The apartment $\frA$ for ${}^2\! A_4$, $\nu=1/2$}
\end{figure}

The group $L_\nu=\SO(V,q_{0})$ and $\frg(F)^{\rs}_{\nu}$ consists of the traceless self-adjoint endomorphisms of $V$. Choose an orthonormal  basis $e_1,\cdots,e_{5}$ of $V$, and a Cartan subspace $\frs\subset\frg(F)^{\rs}_{\nu}$ can be chosen to consist of the diagonal matrices under this basis.  The small Weyl group is $S_5$. The regular semisimple locus $\frs^{rs}$ consists of diagonal matrices with distinct eigenvalues. For $\gamma\in\frs^{\rs}$, the centralizer $S_{\gamma}\subset\mu_{2}^{5}$ consists of diagonal matrices in $L_{\nu}$.

\subsubsection{The Hessenberg varieties}
The four-dimensional Hessenberg variety is the flag variety of $\SO(V,q_{0})$ consisting of flags $0\subset V_1\subset V_2\subset V_3 \subset V_4\subset V$ with $V_1=V_4^{\perp}$, $V_2=V_3^{\perp}$.

The three-dimensional Hessenberg variety $\Hess^3_\gamma$ is defined by the condition $\gamma V_1\subset V_4$. The point $V_{1}\in\PP(V)$ then lies on the intersection of two quadrics $Q_{1}: (v, v)=0$ and $Q_{\gamma}: (v,\gamma v)=0$. The fibers of $\Hess^{3}_{\gamma}\to Q_{1}\cap Q_{\gamma}=:\Sigma_{\gamma}$ are isomorphic to $\PP^{1}$, which parametrizes choices of an isotropic line $V_2/V_1\subset V_1^\perp/V_1$, where $V_{1}^{\perp}/V_{1}$ is a three dimension space with a symmetric bilinear pairing induced from the  restriction of $(\cdot,\cdot)$ to $V_{1}^{\perp}$.

The two-dimensional Hessenberg variety $\Hess^2_\gamma$ is defined by the condition $\gamma V_1\subset V_3$ which is equivalent to the condition $\gamma V_2\subset V_4$. The projection $\eta:\Hess^{2}_{\gamma}\to \Sigma_{\gamma}$ (by remembering $V_{1}$ only) is a 2-to-1 branched cover. The fiber of $\eta$ over $V_1=\langle v_1\rangle$ consists of isotropic lines $\jiao{v_{2}}=V_2/V_1\subset V_1^\perp/V_1$ such that $(\gamma v_2, v_1)=0$. In other words, the fiber of $\eta$ at $\jiao{v_{1}}$ is the intersection of the projective line $(v_{2},\gamma v_{1})=0$ (where $\jiao{v_{2}}\in\PP(V_{1}^{\perp}/V_{1})$) with the conic $(v_{2}, v_{2})=0$ (again $\jiao{v_{2}}\in\PP(V_{1}^{\perp}/V_{1})$).

There two alcoves in the fundamental chamber of $W_\nu$ with expected dimension of the Hessenberg variety equal to $1$. The alcove that has one red side corresponds to the Hessenberg variety $\Hess^{1,line}_\gamma$ parametrizing flags satisfying $\{\gamma V_2\subset V_3\}$ which is equivalent to the vanishing of the restriction of 
the symmetric bilinear forms $(\cdot,\cdot)$ and $(\cdot,\gamma\cdot)$ on $V_2$. That is $\Hess^{1,line}_\gamma$ is a $\mathbb{P}^1$ fibration over the finite set of lines in $\Sigma_\gamma$. There are $16$ lines in $\Sigma_\gamma$.

The alcove that has no red sides corresponds to the Hessenberg variety $\Hess^{1,br}_\gamma$ parametrizing flags satisfying $\gamma V_1\subset V_2$. Therefore $\Hess^{1,br}_\gamma$ is the branch locus $B_\gamma$ of the double cover $\Hess^{2}_{\gamma}\to\Sigma_{\gamma}$. In fact, $B_{\gamma}$ consists of flags such that $V_1=\langle v_1\rangle$ such that the projective line $\PP\Span(v_{1},\gamma v_1)$ is the tangent line to the conic $\overline{Q}_{1}\subset\PP(V_1^\perp/V_1)$ defined by the restriction of $(\cdot,\cdot)$, and that happens if and only if $(\gamma v_1,\gamma v_1)=0$. On the other hand, $B_\gamma$  is an intersection of three quadrics $Q_{1}, Q_{\gamma}$ and $Q_{\gamma^{2}}$, and the adjunction formula implies that the genus of $B_\gamma$ is $5$.

Finally the zero dimensional Hessenberg varieties $\Hess^0_\gamma$ parametrizes flags satisfying $\gamma V_1\subset V_2$ and $\gamma V_2\subset V_3$. It consists of $16$ points corresponding to the 16 lines on the surface $\Sigma_\gamma$.

\subsubsection{The Kummer surface} The surface $\Sigma_\gamma$ is a del Pezzo surface anti-canonically embedded in $\mathbb{P}^4$. Since $B_{\gamma}$ is the intersection of $\Sigma_{\gamma}$ with another quadric, we have $B_\gamma\sim-2K_{\Sigma_\gamma}$ as divisor classes on $\Sigma_{\gamma}$. Therefore, $\Hess^2_\gamma$ is a surface with trivial canonical class and Euler characteristics $2\chi(\Sigma_\gamma)-\chi(B_\gamma)=24$, i.e., $\Hess^{2}_{\gamma}$ is a $K3$ surface. Below we shall see that this K3 surface comes from the well-known construction of Kummer $K3$ surfaces from a torsor of the Jacobian of a genus two curve $C_{\gamma}$. Our presentation is strongly influenced by \cite{BG}.

Let $\tilV=V\oplus\CC$ and introduce two quadrics in $\PP(\tilV)$ given by $\tilQ_{1}:(v,v)=0$ (this is degenerate) and $\tilQ_{\gamma}:(v,\gamma v)+a^{2}=0$ where $(v,a)\in\tilV$. Let $C_{\gamma}\to\PP^{1}$ be the double cover ramified at the singular locus of the pencil of quadrics spanned by $\tilQ_{1}$ and $\tilQ_{\gamma}$ together with $\infty$ (there are 6 singular points). Then $C_{\gamma}$ classifies rulings on this pencil, and is a curve of genus two. The involution $\sigma=\id_{V}\oplus(-1):\tilV\to\tilV$ fixes each member of the pencil, and induces the hyperelliptic involution on $C_{\gamma}$.

Let $F_{\gamma}$ be the Fano variety of projective lines in the base locus $\wt{\Sigma}_{\gamma}:=\tilQ_{1}\cap\tilQ_{\gamma}$. This is a torsor under $J(C_{\gamma})$. The involution $\sigma$ acts on $F_{\gamma}$ compatible with the inversion on $J(C_{\gamma})$. The fixed point locus $F_{\gamma}^{\sigma}$ consists of 16 points which corresponding to the 16 lines in $\Sigma_{\gamma}=\wt{\Sigma}_{\gamma}\cap\PP(V)$. 

The moduli space of isotropic planes $V_{2}\subset V$ under $(\cdot,\cdot)$ is isomorphic to $\PP^{3}$ (a partial flag variety of $\SO(V)$). Let $P_{\gamma}\subset\PP^{3}$ be the subvariety classifying those $V_{2}$ on which $(\cdot,\gamma\cdot)$ is degenerate. This is a singular quartic with 16 ordinary double points corresponding to the 16 projective lines in $\Sigma_{\gamma}$ (i.e., those $V_{2}$ that are also isotropic under $(\cdot,\gamma\cdot)$). We have the following diagram
\begin{equation*}
\xymatrix{\Hess^{2}_{\gamma}\ar[dr]^{f} & & F_{\gamma}\ar[dl]^{g} \\ & P_{\gamma}\cong F_{\gamma}/\sigma}
\end{equation*}
Here $f$ sends a flag $V_{1}\subset V_{2}\subset \cdots\subset V$ to $V_{2}$. The morphism $f$ is birational with 16 exceptional divisors ($\cong\PP^{1}$) over the 16 singular points on $P_{\gamma}$. The map $g$ sends a projective line, corresponding to a plane $\tilV_{2}\subset\tilV$ to the plane $V_{2}:=\textup{Im}(\tilV_{2}\to V)$. It is easy to see that $g$ realizes $P_{\gamma}$ as the GIT quotient $F_{\gamma}/\sigma$. Therefore $\Hess^{2}_{\gamma}$ is the Kummer K3 surface coming from the $J(C_{\gamma})$-torsor $F_{\gamma}$ together with the involution $\sigma$.

The group $S_{\gamma}$ acts on the Hessenberg varieties by the diagonal matrices with entries $\pm 1$ and of determinant $1$. Hence $\Sigma_\gamma/S_{\gamma}=\mathbb{P}^2$ and $\cohog{*}{\Hess^3_\gamma}^{S_{\gamma}}=\cohog{*}{\mathbb{P}^1}\otimes\cohog{*}{\mathbb{P}^2}$. Similarly, $B_\gamma/S_{\gamma}=\mathbb{P}^1$, and hence $\cohog{*}{\Hess^{1,br}_\gamma}^{S_{\gamma}}=\cohog{*}{\mathbb{P}^1}$.
There is a natural identification $S_{\gamma}\cong J(C_{\gamma})[2]$ and the action of $S_{\gamma}$ on $F_{\gamma}$ is via the translation by $J(C_{\gamma})[2]$. Therefore $S_{\gamma}$ permutes the lines on $\Sigma_\gamma$ (i.e., $F^{\sigma}_{\gamma}$) simply transitively, hence it permutes simply transitively the connected components of $\Hess^{1,line}_{\gamma}$ and $\Hess^{0}_{\gamma}$, giving $\cohog{*}{\Hess^{1,line}_{\gamma}}^{S_{\gamma}}\cong\cohog{*}{\PP^{1}}$ and $\cohog{*}{\Hess^{0}_{\gamma}}^{S_{\gamma}}\cong\cohog{*}{\pt}$. For the same reason, $S_{\gamma}$ permutes the 16 exceptional divisors of $f:\Hess^{2}_{\gamma}\to P_{\gamma}$ simply transitively, and its action on $\cohog{*}{P_{\gamma}}\cong\cohog{*}{F_{\gamma}}^{\sigma}$ is trivial (because it extends to an action of the connected group $J(C_{\gamma})$). Therefore $\cohog{*}{\Hess^{2}_{\gamma}}^{S_{\gamma}}\cong\cohog{2}{\PP^{1}}\oplus\cohog{*}{F_{\gamma}}^{\sigma}$.
 Choosing an isomorphism $J(C_{\gamma})\cong F_{\gamma}$, we may identify $\sigma$ with the inversion on $J(C_{\gamma})$, therefore $\cohog{*}{F_{\gamma}}^{\sigma}=\oplus_{i=0,2,4}\wedge^{i}\cohog{1}{C_{\gamma}}$. 

Let $\pi^{C}:C\to \frc(F)_{\nu}^{\rs}$ be the family of genus two curves $C_{a}:=C_{\kappa(a)}$. Using the $S_{5}$-cover $\frs\to\frc(F)_{\nu}$, we may identify $\frc(F)_{\nu}^{\rs}$ with the space of monic polynomials $f(x)$ of degree 5 in $\CC[x]$ with distinct roots, and $\pi^{C}$ the family of curves $y^{2}=f(x)$. By \cite[Theorem 10.1.18.3]{KS} (in fact its obvious characteristic zero analog), the monodromy of $\bR^{1}\pi^{C}_{*}\QQ$ is Zariski dense. The upshot of the above discussion is that $\wedge^{2}\bR^{1}\pi^{C}_{*}\QQ$ can be decomposed as a direct sum of {\em irreducible} local systems $\QQ(-1)\oplus(\wedge^{2}\bR^{1}\pi^{C}_{*}\QQ)_{\prim}$ where $\QQ(-1)$ restricts to the polarization class in $\wedge^{2}\cohog{1}{C_{a}}$ for each $a$ and $(\wedge^{2}\bR^{1}\pi^{C}_{*}\QQ)_{\prim}$ is its complement under the cup product.

\subsubsection{Irreducible modules}
From the above discussion we get for all $a\in\frc(F)^{\rs}_{\nu}$,
\begin{equation*}
\spcoh{}{\Sp_{a}}^{S_{a}}=\cohog{*}{\fl_{\nu}}\oplus\cohog{*}{\PP^{2}}\otimes\cohog{*}{\PP^{1}}\oplus\cohog{2}{\PP^{1}}\oplus(\oplus_{i=0,2,4}\wedge^{i}\cohog{1}{C_{a}})\oplus\cohog{*}{\PP^{1}}\oplus\cohog{*}{\PP^{1}}\oplus\cohog{*}{\pt}^{\oplus 3}.
\end{equation*}
Adding up dimensions we get $\dim \cohog{*}{ \Sp_a}^{S_a}=30$. By the above discussion, only the 5-dimensional spaces $(\wedge^{2}\cohog{1}{C_{a}}))_{\prim}$ form an irreducible local system, and other pieces are invariant under $B_{a}$. We have
\begin{equation*}
\spcoh{}{\Sp_{a}}^{S_{a}}=\spcoh{}{\Sp_{a}}^{S_{a}\rtimes B_{a}}\oplus(\wedge^{2}\cohog{1}{C_{a}}))_{\prim}.
\end{equation*}
hence $\dim\spcoh{}{\Sp_{a}}^{S_{a}\rtimes B_{a}}=30-5=25$.

On the other hand \cite[Theorem 3.2.3]{Chm} implies that the algebra $\Hrat_{1,3/2}(I_{2}(4))$ has only two finite dimensional irreducible representations $\frL_{\nu}(\triv)$ and $\frL_{\nu}(\chi_{-+})$, the first one of dimension $25$ and the 
last one is one-dimensional. Therefore $\Gr^{C}_{*}\spcoh{}{\Sp_{a}}^{S_{a}\rtimes B_{a}}\cong\frL_{\nu}(\triv)$. We expect that $\Hrat_{1,3/2}(I_{2}(4))$ acts on $\Gr^{P}_{*}\spcoh{}{\Sp_{a}}^{S_{a}}$  (for some filtration $P_{\leq i}$ on $\spcoh{}{\Sp_{a}}^{S_{a}}$ extending the Chern filtration), commuting with $B_{a}$, such that we would have an isomorphism of $\Hrat_{1,3/2}(I_{2}(4))\times B_{a}$-modules
\begin{equation*}
\Gr^{P}_{*}\spcoh{}{\Sp_{a}}^{S_{a}}\stackrel{?}{=}\frL_{\nu}(\triv)\oplus \frL_{\nu}(\chi_{-+})\otimes(\wedge^{2}\cohog{1}{C_{a}}))_{\prim}.
\end{equation*}

\subsection{Type $G_2$, $m=3$} For type $G_{2}$, the only regular elliptic numbers besides the Coxeter number $m=6$ are $m=3$ and $m=2$. We first consider the case $m=3$ and $\nu=1/3$. 

\subsubsection{$L_{\nu}$ and $\frg(F)_{\nu}$}\label{sss:G2m3} Let $\alpha_{1}$ (resp. $\alpha_{2}$) be the long (resp. short) simple root of $G_{2}$. Thus real affine roots of $G(F)$ are
$$\pm\alpha_1+\ZZ\delta,\quad \pm \alpha_2+\ZZ\delta,\quad \pm (\alpha_1+\alpha_2)+\ZZ\delta,\quad \pm(\alpha_1+2\alpha_2)+\ZZ\delta,\quad \pm(\alpha_1+3\alpha_2)+\ZZ\delta,\quad \pm(2\alpha_1+3\alpha_2)+\ZZ\delta,$$
and $\rho^\vee=5\alpha^{\vee}_1+3\alpha^{\vee}_2$.

We have $\nu\rho^{\vee}=\frac{5}{3}\alpha^{\vee}_1+\alpha^{\vee}_2$. The roots of $L_{\nu}$ are $\pm(\alpha_{1}+2\alpha_{2}-\delta)$ and the affine roots appearing in $\frg(F)_\nu$ are:
$$\alpha_{1}+3\alpha_{2}-\delta, \quad\alpha_{2}, \quad-\alpha_{1}-\alpha_{2}+\delta, \quad-2\alpha_{1}-3\alpha_{2}+2\delta, \quad\alpha_{1}.$$

The group $L_{\nu}$ is isomorphic to $\GL_{2}$, and $\frg(F)_{\nu}\cong\Sym^{3}(V^{\vee})\otimes\det(V)\oplus\det(V)$, where $V$ is the standard representation of $\GL_{2}$. An element $\gamma\in\frg(F)^{\rs}_{\nu}$ can be written as a pair $(\gamma',\gamma'')$ where $\gamma'$ is a nondegenerate binary cubic form on $V$ and $\gamma''\in\det(V)$ is nonzero. 

\subsubsection{The Hessenberg varieties}
\begin{figure}
\centering
\includegraphics[height=5cm]{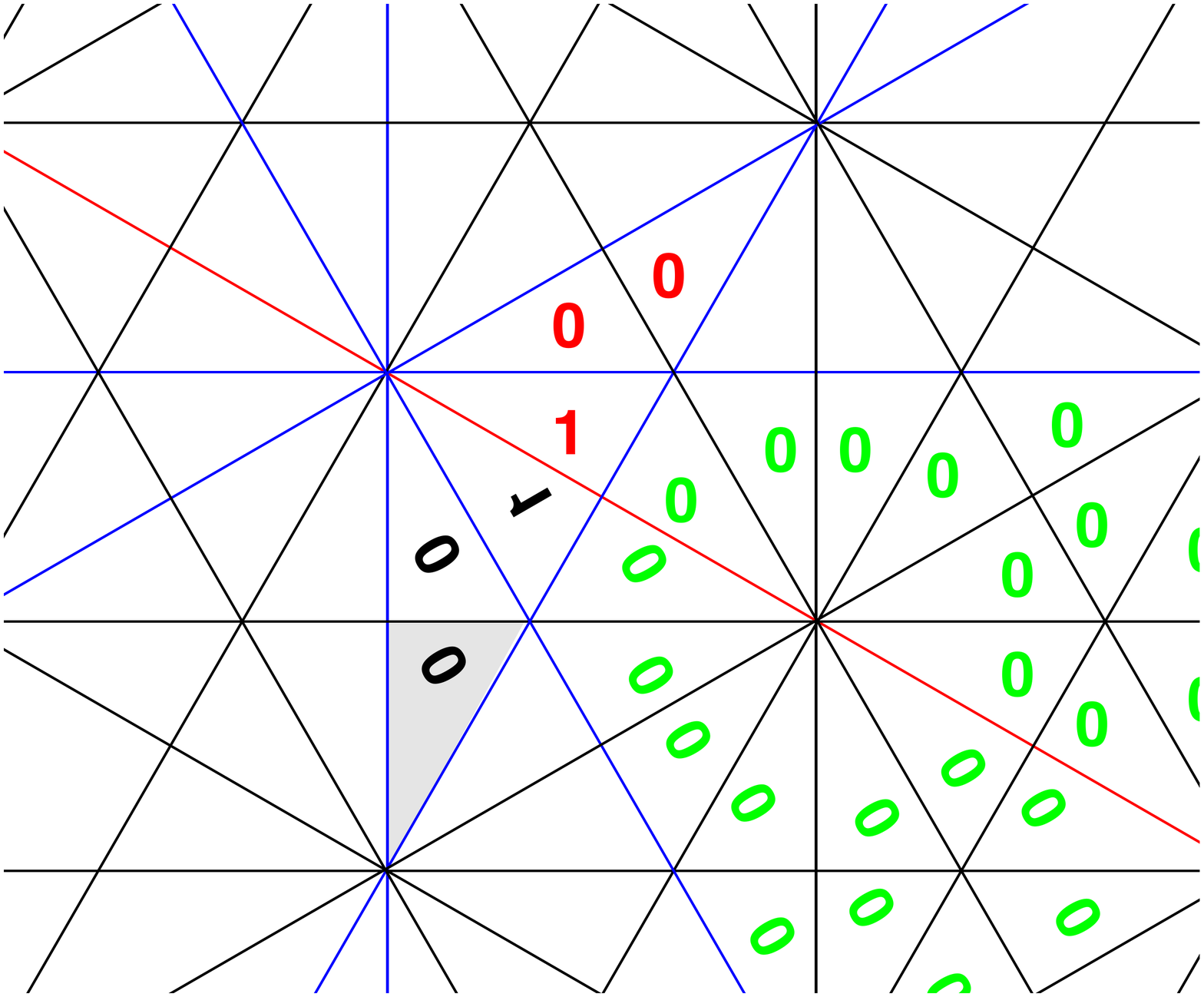}
\caption{The apartment $\frA$ for $G_2$, $\nu=1/3$}
\end{figure}

In the picture above we mark the alcoves with the expected dimensions of the Hessenberg varieties. We use green numbers for the alcoves with empty Hessenberg varieties despite having non-negative expected dimension.
The one-dimensional Hessenberg variety is $\PP^{1}=\PP(V)$. The zero-dimensional Hessenberg varieties consist of three points in $\PP(V)$ corresponding to the three zeros of $\gamma'$, and they are permuted simply transitively by the stabilizer $S_{\gamma}\cong\mu_{3}$. 

\subsubsection{Irreducible modules} From the above discussion we conclude that for $a\in\frc(F)^{\rs}_{\nu}$,
\begin{equation*}
\spcoh{}{\Sp_{a}}^{S_{a}}=\spcoh{}{\Sp_{a}}^{S_{a}\rtimes B_{a}}=\cohog{*}{\PP^{1}}\oplus\cohog{*}{\pt}^{\oplus2}
\end{equation*}
has dimension $4$. On the other hand, by \cite[Theorem 3.2.3(i)]{Chm}, $\Hrat_{1/3,1/3}(I_{2}(6))$ has only one finite dimensional irreducible representation $\frL_{\nu}(\triv)$, and it is of dimension $4$. Hence  $\Gr^{P}_{*}\spcoh{}{\Sp_a}^{S_a}\cong\frL_{\nu}(\triv)$.

\subsection{Type $G_{2}$, $m=2$} Next we consider the case $\nu=1/2$.
We keep the notation from \S\ref{sss:G2m3}. 

\subsubsection{$L_{\nu}$ and $\frg(F)_{\nu}$} We have $\nu\rho^{\vee}=\frac{5}{2}\alpha^{\vee}_1+\frac{3}{2}\alpha^{\vee}_2$. The roots of $L_{\nu}$ are $\pm(\alpha_1+\alpha_2-\delta)$ and $\pm(\alpha_1+3\alpha_2-2\delta)$ and the affine roots appearing in $\frg(F)_\nu$ are:
$$\pm(\alpha_1-\delta/2)+\delta/2,\quad \pm(\alpha_2-\delta/2)+\delta/2,\quad \pm(\alpha_1+2\alpha_2-3\delta/2)+\delta/2,\quad \pm(2\alpha_1+3\alpha_2-5\delta/2)+\delta/2,$$
The group $L_\nu$ is $\SL_2\times \SL_2/\Delta(\mu_{2})$ and $\frg(F)_\nu$ is $V\otimes \Sym^3(W)$ where $V$ and $W$ are the standard representations of the first and the second copy of $\SL_2$ in $L_\nu$.

\subsubsection{The Hessenberg varieties}
\begin{figure}
\centering
\includegraphics[height=5cm]{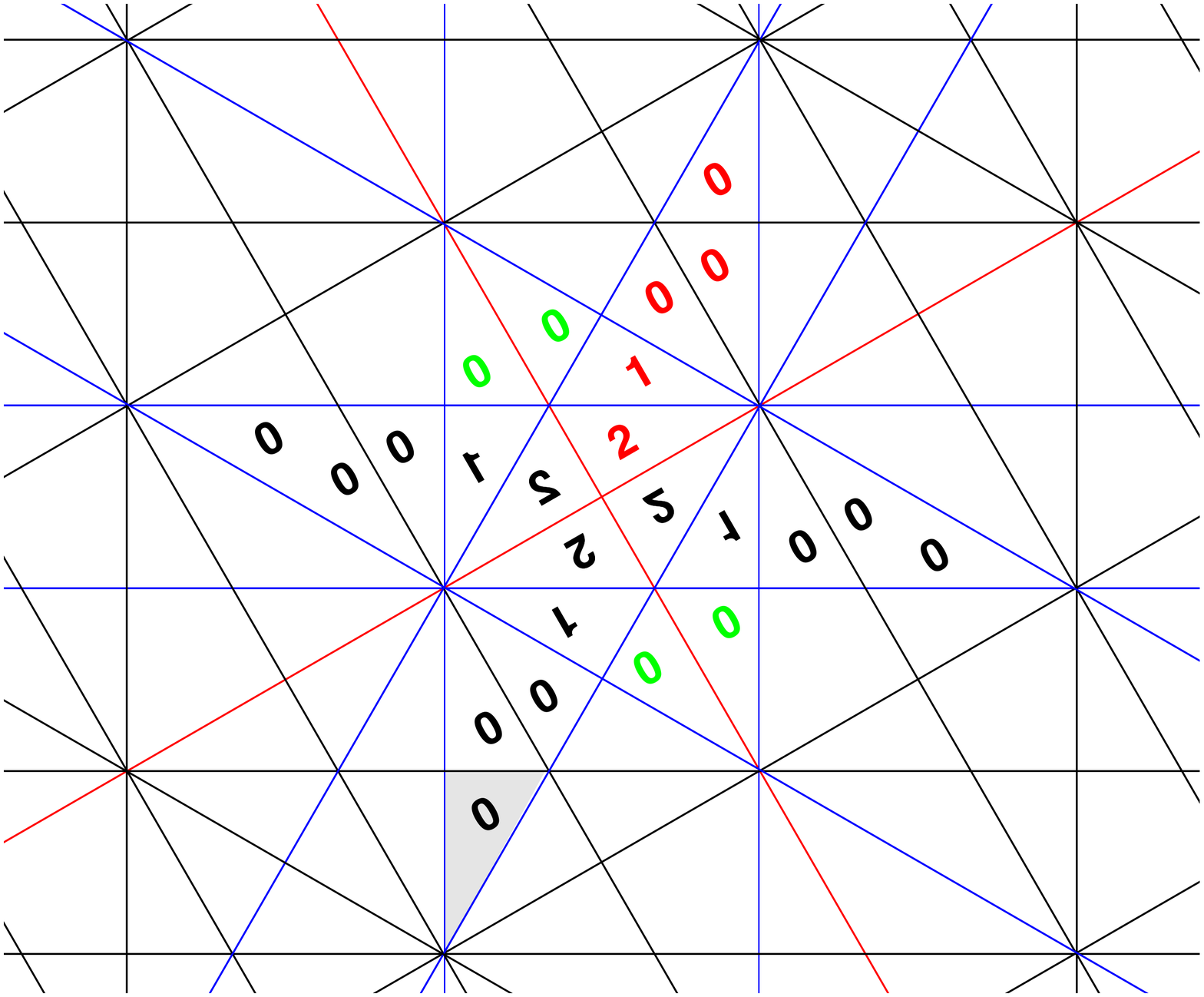}
\caption{The apartment $\frA$ for $G_2$, $\nu=1/2$}
\end{figure}

The two-dimensional Hessenberg variety is $\fl_{\nu}\cong\mathbb{P}^1\times\mathbb{P}^1$. The regular element $\gamma$ defines form $F_\gamma$ on $V\otimes W$ of the homogeneous bidegree
$(1,3)$. The one-dimensional  Hessenberg variety $\Hess^1_\gamma$ consists of pairs $(L,M)$ of one-dimensional subspaces of $V$, $W$ which are isotropic with respect to $F_\gamma$.
Other words the Hessenberg variety is $\mathbb{P}^1$ realized as the smooth $(3,1)$ curve in $\mathbb{P}^1\times\mathbb{P}^1$.

The zero-dimensional Hessenberg varieties corresponding to the alcoves in 
three alcove cluster are the branch points of the 
degree $3$ projection from $\Hess^1_\gamma$ to second $\mathbb{P}^1$ in $\mathbb{P}^1$. The degree $3$ map from $\mathbb{P}^1$ has four branch points and there are permuted by the stabilizer $S_\gamma\cong\mu_{2}^{2}$. Finally, the clan with only one alcove corresponds to the empty Hessenberg variety even though we expect a zero-dimensional variety. We use green numbers to mark these alcoves.

\subsubsection{Irreducible modules} From the above discussion we conclude that for $a\in\frc(F)^{\rs}_{\nu}$,
\begin{equation*}
\spcoh{}{\Sp_a}^{S_a}=\spcoh{}{\Sp_a}^{S_{a}\rtimes B_{a}}\cong\cohog{*}{\mathbb{P}^1\times\mathbb{P}^1}\oplus \cohog{*}{\mathbb{P}^1}\oplus \cohog{*}{\pt}^{\oplus3}
\end{equation*}
has dimension $9$. On the other hand \cite[Theorem 3.2.3(vi)]{Chm}
states that $\Hrat_{1/2,1/2}(I_{2}(6))$ has only one finite dimensional irreducible representation $\frL_{\nu}(\triv)$, and it is of dimension $9$. Hence $\Gr^{P}_{*}\spcoh{}{\Sp_a}^{S_a}\cong\frL_{\nu}(\triv)$.

\subsection{Type ${}^3\! D_4$, $m_{1}=6$} Besides the twisted Coxeter number $m_{1}=12$, the regular elliptic numbers are $m_{1}=6$ and $m_{1}=3$.  We first consider the case $m_{1}=6$ and $\nu=1/6$.

\subsubsection{$L_{\nu}$ and $\frg(F)_{\nu}$}\label{sss:D4m6} The finite part of the affine root system ${}^3\! D_4$ is the $G_2$ root system. The real affine roots of ${}^3\! D_4$ are:
\begin{gather*}
\pm\alpha_1+\ZZ\delta,\quad \pm \alpha_2+\ZZ\delta/3,\quad \pm (\alpha_1+\alpha_2)+\ZZ\delta/3,\quad \pm(\alpha_1+2\alpha_2)+\ZZ\delta/3,\\ 
\pm(\alpha_1+3\alpha_2)+\ZZ\delta,\quad
 \pm(2\alpha_1+3\alpha_2)+\ZZ\delta,
\end{gather*}
and $\rho^\vee=5\alpha^{\vee}_1+3\alpha^{\vee}_2$. 

The roots of $L_{\nu}$ are $\pm(\alpha_1+\alpha_2-\delta/3)$ and the affine roots appearing in $\frg(F)_\nu$ are
$$\alpha_2,\quad -\alpha_2+\delta/3,\quad \alpha_1+2\alpha_2-\delta/3,\quad -(\alpha_1+2\alpha_2)+2\delta/3, \quad -(2\alpha_1+3\alpha_2)+\delta,\quad\alpha_1.$$

The group $L_\nu$ is $\GL_2$ and  $\frg(F)_{\nu}=V\oplus \Sym^3(V^{\vee})\otimes\det(V)$ where $V$ is the standard representation of $\GL_2$. The element $\gamma=(\gamma_1,\gamma_2)\in V\oplus \Sym^3(V^{\vee})\otimes\det(V)$
is regular semisimple if and only if the zeros of the linear form $\gamma_{1}$ and the binary cubic form $\gamma_2$ are distinct points on $\PP^{1}$. In this case, the stabilizer $S_{\gamma}$ is trivial.

\begin{figure}\label{D4_1/6}
\centering
\includegraphics[trim=0cm 0cm 3cm 2cm, clip=true, height=5cm]{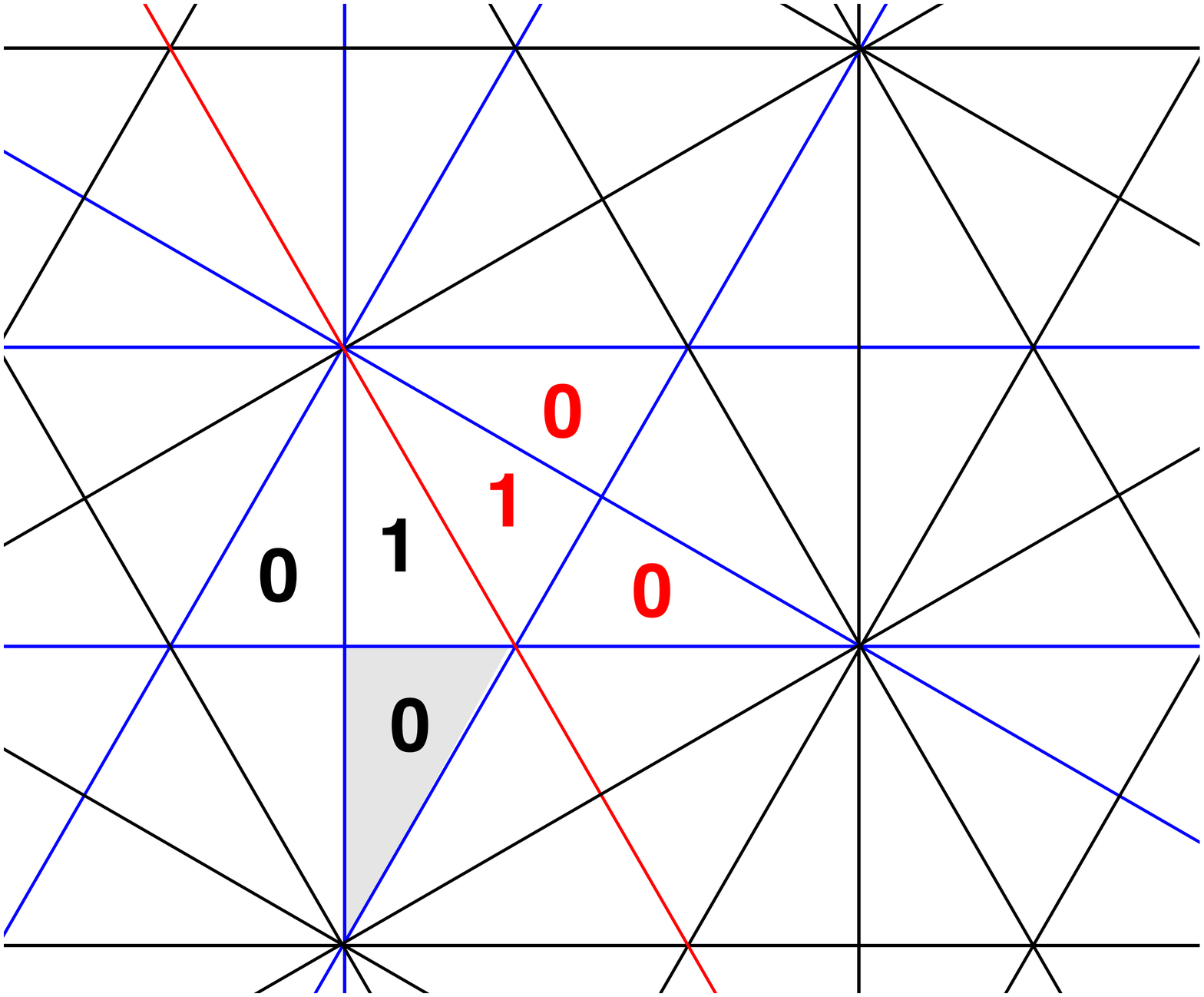}
\caption{The apartment $\frA$ for ${}^3\! D_4$, $\nu=1/6$}
\end{figure}

\subsubsection{The Hessenberg varieties}

The one-dimensional Hessenberg variety is $\PP^1=\PP(V)$. The zero-dimensional Hessenberg 
varieties are of two types.  The first type is the zero locus of $\gamma_1$ (one point) and the second is the zero locus of $\gamma_2$ (three points). Each type occurs only once in $\Sp_\gamma$. 

As $a$ runs over $\frc(F)^{\rs}_{\nu}$ and considering the corresponding Hessenberg varieties for $\gamma=\kappa(a)$, $B_{a}$ permutes the three points of the zero-dimensional Hessenberg variety of second type transitively. 
Thus the local system $\oplus_{i}\bR^{i}q_{\nu,*}\QQ$ is the sum of 
local system of two types: the trivial one and an irreducible local system $M$ of rank two (with monodromy $S_{3}$ acting as the standard 2-dimensional representation on its fibers).

\subsubsection{Irreducible modules} By the above discussion, we have that for $a\in\frc(F)^{\rs}_{\nu}$,
\begin{equation*}
\spcoh{}{\Sp_a}=\spcoh{}{\Sp_a}^{S_{a}}=\spcoh{}{\Sp_a}^{S_{a}\rtimes B_{a}}\oplus M_{a}.
\end{equation*}
where
\begin{equation*}
\spcoh{}{\Sp_a}^{S_{a}\rtimes B_{a}}\cong\cohog{*}{\PP^{1}}\oplus\cohog{*}{\pt}^{\oplus2}
\end{equation*}
has dimension 4. By \cite[Theorem 3.2.3]{Chm}, $\Hrat_{1/2,1/6}(I_{2}(6))$ has only two  finite dimensional irreducible representations: the $4$-dimensional $\frL_{\nu}(\triv)$ and the $1$-dimensional 
$\frL_{\nu}(\chi_{+-})$. Thus $\Gr^{C}_{*}\spcoh{}{\Sp_{a}}^{S_{a}\rtimes B_{a}}\cong\frL_{\nu}(\triv)$. We expect that  $\Hrat_{1/2,1/6}(I_{2}(6))$ acts on $\Gr^{P}_{*}\spcoh{}{\Sp_a}$  (for some filtration $P_{\leq i}$ on $\spcoh{}{\Sp_{a}}^{S_{a}}$ extending the Chern filtration), such that we would have an isomorphism of $\Hrat_{1/2,1/6}(I_{2}(6))\times B_{a}$-modules:
$$  \Gr^{P}_{*}\spcoh{}{\Sp_a}\stackrel{?}{=}\frL_{\nu}(\triv)\oplus \frL_{\nu}(\chi_{+-})\otimes M_{a}.$$

\subsection{Type ${}^3\! D_4$, $m_{1}=3$} Next we consider the case $m_{1}=3$ and let $\nu=1/3$. We keep the notation from \S\ref{sss:D4m6}.

\subsubsection{$L_{\nu}$ and $\frg(F)_{\nu}$}
The roots of $L_{\nu}$ are 
$$\pm(\alpha_2-\delta/3),\quad \pm(\alpha_1+\alpha_2-2\delta/3),\quad \pm(\alpha_1+2\alpha_2-\delta),$$
The affine roots that appear in $\frg(F)_{\nu}$ are 
\begin{gather*}
\delta/3, \quad \alpha_1,\quad \alpha_1+3\alpha_2-\delta,\quad -2\alpha_1-3\alpha_2+2\delta,\quad \pm(\alpha_2-\delta/3)+\delta/3,\\
 \pm(\alpha_1+\alpha_2-2\delta/3)+\delta/3,\quad \pm(\alpha_1+2\alpha_2-\delta)+\delta/3.
\end{gather*}
The group $L_\nu\cong\PGL_3$ and $\frg(F)_{\nu}\cong\Sym^3(V^{\vee})\otimes\det(V)$ where $V$ is the standard representation of $\GL_{3}$. 

\begin{figure}\label{D4_1/3}
\centering
\includegraphics[height=5cm]{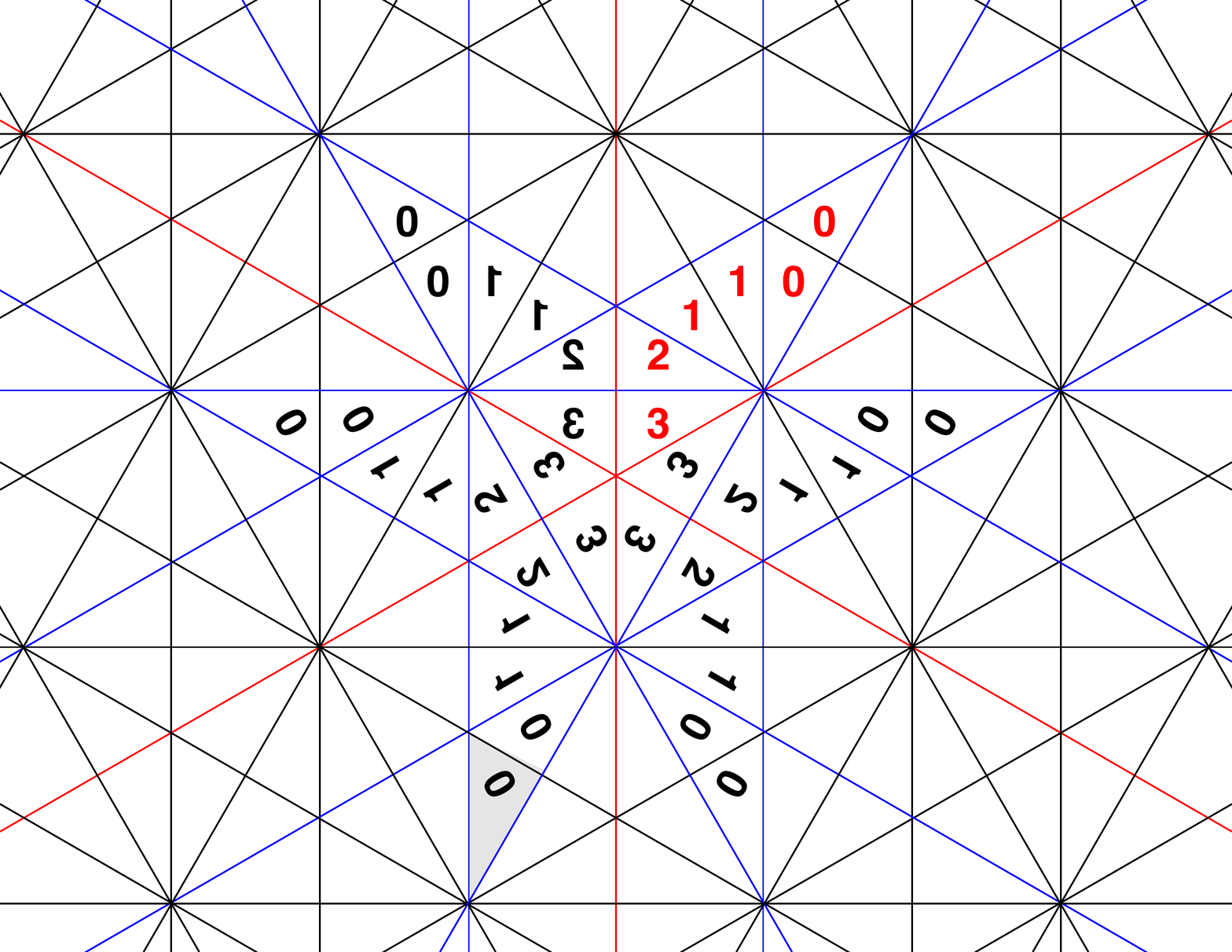}
\caption{The apartment $\frA$ for ${}^3\! D_4$, $\nu=1/3$}
\end{figure}

\subsubsection{The Hessenberg varieties}
Figure 8 presents the picture of the apartment with alcoves marked by the expected dimensions.

The three dimensional Hessenberg variety is the flag variety $\fl_{\nu}$ of $L_{\nu}$, which classifies pairs $p\in\ell\subset\PP(V)$ where $\ell$ is a projective line in $\PP(V)$ and $p$ is a point on $\ell$.

Each $\gamma\in\frg(F)_{\nu}^{\rs}\subset\Sym^3(V^{\vee})\otimes\det(V)$, viewed as a ternary cubic form, defines a smooth planar curve of genus one $C_\gamma\subset\PP(V)$. The two-dimensional Hessenberg variety $\Hess^{\tilw}_{\gamma}$ classifies triples $p\in\ell\subset\PP(V)$ where $p\in C_{\gamma}$. 
%(and no condition on $p$).
% Thus $\ell$ is determined by its tangent point with $C_{\gamma}$.
 We conclude that $\Hess^{\tilw}_{\gamma}$ is a $\PP^{1}$-bundle over $C_{\gamma}$.

The one dimensional Hessenberg varieties $\Hess^{\tilw}_{\gamma}$ classify those $p\in\ell\subset \PP(V)$ where $\ell$ is tangent to $C_{\gamma}$ at $p$. Thus $\Hess^{\tilw}_{\gamma}$ is isomorphic to the curve $C_\gamma$.

The zero dimensional Hessenberg varieties $\Hess^{\tilw}_{\gamma}$ classify those $p\in\ell\subset\PP(V)$ where $p$ is a flex point of $C_{\gamma}$ (there are 9 of them) and $\ell$ is the tangent line to $C_{\gamma}$ at $p$.

The action of $\PGL(V)$ on $\Sym^{3}(V^{\vee})\otimes\det(V)$ is used by Bhargava and Shanker to calculate the average size of the $3$-Selmer groups of elliptic curves over $\QQ$, see \cite{BS3} and also \cite{Gross} for details. In particular, it is known that when $\gamma\in\frg(F)^{\rs}_{\nu}$, the stabilizer group $S_\gamma$ acts on $C_\gamma$ by translation by $3$-torsion elements of the elliptic curve $J(C_\gamma)$. Hence $S_\gamma$ acts transitively on the zero-dimensional Hessenberg varieties and trivially on the cohomology of the Hessenberg varieties of positive dimension. 

\subsubsection{Irreducible modules} From the above discussion we conclude that for $a\in\frc(F)^{\rs}_{\nu}$,
\begin{equation*}
\spcoh{}{\Sp_a}^{S_a}=\cohog{*}{\fl_{\nu}}\oplus \cohog{*}{C_a}\otimes\cohog{*}{\mathbb{P}^1}\oplus \cohog{*}{C_a}^{\oplus2}\oplus \cohog{*}{\pt}^{\oplus2}
\end{equation*}
Here $C_{a}=C_{\kappa(a)}$ form a family of curves $\pi^{C}: C\to \frc(F)^{\rs}_{\nu}$. In particular the local system $\oplus_{i}(\bR^{i}q_{\nu,*}\QQ)^{S}$ over $\frc(F)_\nu^{rs}$ is a sum of irreducible local systems of two types: the trivial one and $\bR^{1}\pi^{C}_{*}\QQ$.
The multiplicity space of the trivial local system is the $16$-dimensional irreducible representation $\frL_{\nu}(\triv)$ of $\Hrat_{1,1/3}(I_{2}(6))$. The multiplicity space of $\bR^{1}\pi^{C}_{*}\QQ$ is $4$-dimensional. On the other hand, by \cite[Theorem 3.2.3]{Chm}, the algebra $\Hrat_{1,1/3}(I_{2}(6))$ has only two irreducible finite-dimensional representation: the $16$-dimensional $\frL_{\nu}(\triv)$ and the $4$-dimensional $\frL_{\nu}(\chi_{+-})$. Thus $\Gr^{C}_{*}\spcoh{}{\Sp_a}^{S_{a}\rtimes B_{a}}\cong\frL_{\nu}(\triv)$. We expect $\Hrat_{1,1/3}(I_{2}(6))$ acts on $\Gr^{P}_{*}\spcoh{}{\Sp_a}^{S_a}$  (for some filtration $P_{\leq i}$ on $\spcoh{}{\Sp_{a}}^{S_{a}}$ extending the Chern filtration), such that we would have an isomorphism of $\Hrat_{1,1/3}(I_{2}(6))\times B_{a}$-modules
$$ \Gr^{P}_{*}\spcoh{}{\Sp_a}^{S_a}\stackrel{?}{=}\frL_{\nu}(\triv)\oplus \frL_{\nu}(\chi_{+-})\otimes \cohog{1}{C_a}.$$

\subsection{Dimensions of $\frL_{\nu}(\triv) $: tables and conjectures}\label{ss:tables}
For a simple root system $R$, let $\frL_\nu(R)$ be the irreducible spherical representation $\frL_\nu(\triv)$ of $\Hrat_{\nu,\ep=1}$ attached to the root system $R$.  Using the formula \eqref{intro dim} and computer we have calculated the dimension of $\frL_{\nu}(R)$ in the following cases.

\begin{table}[H]
\begin{minipage}{0.2\textwidth}
\begin{tabular}{l|l}
          $m$& $\dim \frL_{1/m}(F_{4})$\\
               \hline
             12&1\\  
            8&6  \\
            6&20\\
            4&96\\
            3&256\\
            2&1620\\
\end{tabular}
\end{minipage}\quad\quad
\begin{minipage}{0.2\textwidth}
\begin{tabular}{l|l}
$m$&$\dim \frL_{1/m}(E_{6})$\\
\hline
12&1\\
9&8\\
6&92\\
3&4152
\end{tabular}
\end{minipage}\quad\quad
\begin{minipage}{0.2\textwidth}
\begin{tabular}{l|l}
$m$&$\dim\frL_{1/m}(E_{7})$\\
\hline
18&1\\
14&9\\
6&3894\\
2&?
\end{tabular}
\end{minipage}
\quad
\begin{minipage}{0.2\textwidth}
\begin{tabular}{l|l}
$m$&$\dim\frL_{1/m}(E_{8})$\\
\hline
30&1\\
24&10\\
20&54\\
15&576\\
12&3380\\
10&14769\\
8&62640\\
$\leq 6$&?\\
\end{tabular}
\end{minipage}
\end{table}

Computations also suggest the following conjecture.
\begin{conj} 
\begin{enumerate}
\item $1+\sum_{n>0} \dim \frL_{1/2n}(D_{2n}) x^n=(1-4x)^{-3/2}.$
\item $1+\sum_{n>0} \dim \frL_{1/2n}(C_{2n}) x^n=(1-4x)^{-3/2}(1+\sqrt{1-4x})^2/4.$ 
\end{enumerate}
\end{conj}

\appendix

\section{Dimension of affine Springer fibers for quasi-split groups} We keep the notation from \S\ref{s:gp} and \S\ref{s:Spr}. The dimension of affine Springer fibers for split groups over $F=\CC((t))$ were conjectured by Kazhdan and Lusztig \cite{KL} and proved by Bezrukavnikov \cite{Be}. In this appendix we generalize this formula to quasi-split groups $G$ in the generality considered in \S\ref{sss:gp}.

\begin{aprop}\label{p:AFS dim} The ind-scheme $\Sp_{\bP,\gamma}$ is equidimensional. Its dimension is given by
\begin{equation}\label{dim AFS general}
\dim\Sp_{\bP,\gamma}=\frac{1}{2}(\sum_{\alpha\in\PPhi}\val(\alpha(\gamma'))-r+\dim\tt^{\Pi_{a}(\hZZ(1))}).
\end{equation}
Here $\gamma'\in\tt(F_{\infty})$ is conjugate to $\gamma$, $\Pi_{a}:\hZZ(1)\to\WW'$ is defined in \S\ref{sss:ell} (for $a=\chi(\gamma)$), $\val:F^{\times}_{\infty}\to\QQ$ is normalized so that $\val(t)=1$ and recall that $\PPhi$ is the root system of $\GG$ respect to $\TT$.
\end{aprop}
\begin{proof}
Kazhdan and Lusztig \cite{KL} prove that when $e=1$, $\Sp_{\gamma}$ is equidimensional and has the same dimension as $\Sp_{\bG,\gamma}$. The proof there generalizes to the quasi-split case and other parahorics $\bP$.  It remains to calculate the dimension of $\Sp_{\bG, \gamma}$. According to Lemma \ref{l:AFS reg}, it suffices to calculate the dimension of $P_{a}=J_{a}(F)/J_{a}(\calO_{F})$. Let $J_{a}^{\flat}$ be the finite-type Neron model of $J_{a}$ over $\calO_{F}$. We may take $\gamma=\kappa(a)$, then $\Lie J_{a}^{\flat}$ is an $\calO_{F}$-lattice $\frg^{\flat}_{\gamma}\subset\frg_{\gamma}(F)$ containing $\frg_{\gamma}(\calO_{F})$. We need to calculate $\dim P_{a}$, which is the same as $\dim J_{a}^{\flat}(\calO_{F})/J_{a}(\calO_{F})=[\frg^{\flat}_{\gamma}:\frg_{\gamma}]$.  Here,  for two $\calO_{F}$-lattice $\L_{1}$ and $\L_{2}$ in an $F$-vector space, the notation $[\L_{1}:\L_{2}]$ means $\dim_{\CC}\L_{1}/\L_{1}\cap\L_{2}-\dim_{\CC}\L_{2}/\L_{1}\cap\L_{2}$.

We use the strategy in \cite{Be} by considering the $F$-valued Killing form $(\cdot,\cdot)$ on $\frg_{\gamma}(F)$ (given as the restriction of the $F_{e}$-valued Killing form on $\gg(F_{e})$, which is induced from the Killing form on $\gg$). For an $\calO_{F}$-lattice $L$ in $\frg_{\gamma}(F)$ we denote its dual under the Killing form by $L^{\vee}=\{x\in\frg_{\gamma}(F)|(x,L)\subset\calO_{F}\}$, another $\calO_{F}$-lattice in $\frg_{\gamma}(F)$. Similar notation applies to lattices in $\frg(F)$. We have
\begin{equation*}
\frg_{\gamma}\subset\frg_{\gamma}^{\flat}\subset\frg_{\gamma}^{\flat,\vee}\subset\frg_{\gamma}^{\vee}
\end{equation*}
Therefore
\begin{equation}\label{pre dim Pa}
\dim P_{a}=[\frg_{\gamma}^{\flat}:\frg_{\gamma}]=\frac{[\frg^{\vee}_{\gamma}:\frg_{\gamma}]-[\frg^{\flat,\vee}_{\gamma}:\frg^{\flat}_{\gamma}]}{2}.
\end{equation}

We first calculate $[\frg^{\flat,\vee}_{\gamma}:\frg^{\flat}_{\gamma}]$. Let $m$ be the order of the homomorphism $\Pi_{a}$, then $\frg^{\flat}_{\gamma}$ is conjugate to $\tt(\calO_{F_{m}})^{\mu_{m}}$ inside $\GG(F_{m})$ (where $\mu_{m}$ is identified with the image of $\Pi_{a}:\hZZ(1)\to\WW'$). Calculating the Killing form on $\tt(\calO_{F_{m}})^{\mu_{m}}$, we see that
\begin{equation}\label{flat gamma}
[\frg^{\flat,\vee}_{\gamma}:\frg^{\flat}_{\gamma}]=\dim\tt/\tt^{\mu_{m}}.
\end{equation}

Next we calculate $[\frg^{\vee}_{\gamma}:\frg_{\gamma}]$. Note that $\frp(F):=[\gamma,\frg(F)]$ is the orthogonal complement of $\frg_{\gamma}(F)$ in $\frg(F)$ under the Killing form. We have
\begin{equation}\label{quot1}
\frg^{\vee}_{\gamma}/\frg_{\gamma}\cong\frg^{\vee}/(\frg_{\gamma}\oplus\frp(F)\cap\frg^{\vee}).
\end{equation}
We observe that the quotient
\begin{equation}\label{quot2}
\frg^{\vee}/\left(\frg_{\gamma}(F)\cap\frg^{\vee}\oplus[\gamma,\frg^{\vee}]\right)
\end{equation}
is the cokernel of the endomorphism $[\gamma,-]$ of $\frg^{\vee}/\frg_{\gamma}(F)\cap\frg^{\vee}$, hence
\begin{equation*}
[\frg^{\vee}:\frg_{\gamma}(F)\cap\frg^{\vee}\oplus[\gamma,\frg^{\vee}]]=\val\det([\gamma,-]|\frg(F)/\frg_{\gamma}(F))=\sum_{\alpha\in\PPhi}\val(\alpha(\gamma')).
\end{equation*}
Comparing the quotient in \eqref{quot2} with the right side of \eqref{quot1}, the difference between their lengths is the difference between $[\frg_{\gamma}(F)\cap\frg^{\vee}:\frg_{\gamma}]$ and $[\frp(F)\cap\frg^{\vee}:[\gamma,\frg^{\vee}]]$. Hence
\begin{equation}\label{frg gamma}
[\frg^{\vee}_{\gamma}:\frg_{\gamma}]=\sum_{\alpha\in\PPhi}\val(\alpha(\gamma'))+[\frg_{\gamma}(F)\cap\frg^{\vee}:\frg_{\gamma}]-[\frp(F)\cap\frg^{\vee}:[\gamma,\frg^{\vee}]].
\end{equation}

We need to calculate the two extra terms in the above formula. Write $\gg=\oplus_{i=0}^{e-1}\gg_{i}$ under the action of $\mu_{e}$. Then $\frg=\gg_{0}(\calO_{F})\oplus t^{1/e}\gg_{e-1}(\calO_{F})\oplus\cdots\oplus t^{(e-1)/e}\gg_{1}(\calO_{F})$, $\frg^{\vee}=\gg_{0}(\calO_{F})\oplus t^{-1/e}\gg_{1}(\calO_{F})\oplus\cdots\oplus t^{-(e-1)/e}\gg_{e-1}(\calO_{F})$. We have a filtration
\begin{equation*}
\frg=\frg^{(0)}\subset\frg^{(1)}\subset\frg^{(2)}\subset\cdots\subset\frg^{(e-1)}=\frg^{\vee}
\end{equation*}
such that $\frg^{(i)}/\frg^{(i-1)}=t^{-i/e}\gg_{i}$, $i=1,\cdots,e-1$. Let $\frg^{(i)}_{\gamma}=\frg_{\gamma}(F)\cap\frg^{(i)}$. The map $[\gamma,-]$ preserves each $\frg^{(i)}$ and induces the map $[\gamma_{0},-]:\gg_{i}\to\gg_{i}$ on the associated graded. Here $\gamma_{0}$ is the image of $\gamma\mod t$ in $\hh=\gg^{\mu_{e}}$. We consider the exact sequence $0\to\frg^{(i-1)}\to\frg^{(i)}\to t^{-i/e}\gg_{i}\to0$ and its endomorphism given by $[\gamma,-]$. The snake lemma then gives an exact sequence
\begin{equation}\label{giseq}
0\to\frg^{(i-1)}_{\gamma}\to\frg^{(i)}_{\gamma}\to t^{-i/e}\ker([\gamma_{0},-]|\gg_{i})\to\frg^{(i-1)}/[\gamma,\frg^{(i-1)}]\to\frg^{(i)}/[\gamma,\frg^{(i)}].
\end{equation}

\begin{claim} For $i=0,\cdots, e-1$, the lattice $[\gamma,\frg^{(i)}]$ is saturated in $\frg^{(i)}$.
\end{claim}
\begin{proof}[Proof of Claim]  To show that the image of the map $[\gamma,-]:\frg^{(i)}\to\frg^{(i)}$ is saturated, it suffices to show that the kernel of its reduction modulo $t$ has dimension at most $\rr$ (because it is at least $\rr$, the $F$-dimension of the kernel of $[\gamma,-]$ on $\frg(F)$). The filtration $t\frg^{(i)}\subset t\frg^{(i+1)}\subset\cdots\subset t\frg^{\vee}\subset\frg^{(0)}\subset\cdots\subset\frg^{(i)}$ is stable under $[\gamma-,]$. It induces a filtration on $\frg^{(i)}/t\frg^{(i)}=\frg^{(i)}\otimes_{\calO_{F}}\CC$ with associated graded $\gg=\oplus_{i=0}^{e-1}\gg_{i}$, on which the action $[\gamma,-]$ becomes $[\gamma_{0},-]$. Since $\gamma_{0}$ lies in the Kostant section of $\gg$ by construction, it is regular, and hence $\dim\ker([\gamma_{0},-]|\gg)=\rr$. Since the kernel dimension does not decrease when passing to the associated graded, $\ker([\gamma,-]|\frg^{(i)}\otimes_{\calO_{F}}\CC)\leq\rr$. This proves the claim.
\end{proof}

Since $[\gamma,\frg^{\vee}]$ is saturated in $\frg^{\vee}$, hence $[\frp(F)\cap\frg^{\vee}:[\gamma,\frg^{\vee}]]=0$. Also, by the Claim, the last arrow in \eqref{giseq} is injective, and the first three terms form a short exact sequence. We thus get a filtration of $\frg_{\gamma}(F)\cap\frg^{\vee}/\frg_{\gamma}=\frg^{(e-1)}_{\gamma}/\frg^{(0)}_{\gamma}$ with associated graded equal to $\oplus_{i=1}^{e-1}t^{-i/e}\ker([\gamma_{0},-]|\gg_{i})$. The dimension of the latter is $\dim\ker([\gamma_{0},-]|\oplus_{i=1}^{e-1}\gg_{i})=\dim\ker([\gamma_{0},-]|\gg)-\dim\ker([\gamma_{0},-]|\hh)=\rr-r$, again by the regularity of $\gamma_{0}$. Therefore $[\frg_{\gamma}(F)\cap\frg^{\vee}:\frg_{\gamma}]=\dim\tt/\tt^{\mu_{e}}=\rr-r$. These facts and  \eqref{frg gamma} imply that 
\begin{equation}\label{diff gamma}
[\frg^{\vee}_{\gamma}:\frg_{\gamma}]=\sum_{\alpha\in\PPhi}\val(\alpha(\gamma'))+\rr-r.
\end{equation}
Plugging \eqref{diff gamma} and \eqref{flat gamma} into \eqref{pre dim Pa}, we get the desired dimension formula \eqref{dim AFS general}.
\end{proof}

\section{Codimension estimate on the Hitchin base}\label{s:codim} We work in the following generality. Let $X$ be an irreducible smooth Deligne-Mumford curve over an algebraically closed field $k$, of which the weighted projective line in \S\ref{ss:wp} is an example. Let $\GG$ be a semisimple group. Let $\calL$ be a line bundle on $X$. We define $\MHit$ to be the moduli stack of $\calL$-valued $\GG$-Higgs bundles over $X$, and let $f:\MHit\to \calA$ be the Hitchin fibration. 

Recall from \cite{NgoFL} that we have an upper semi-continuous function $\delta:\calA\to \ZZ_{\geq0}$ which measures how far (in terms of dimensions) $\calP_{a}$ is from an abelian variety. Let $\calA_{\delta}\subset\calA$ be the locally closed subscheme consisting of $a\in\calA$ with $\delta(a)=\delta$.
The goal of this section is to show

\begin{aprop}\label{p:codim} Suppose $\textup{char}(k)=0$ and $\deg\calL>\deg\omega_{X}$ (where $\omega_{X}$ is the canonical bundle of $X$). Then $\codim_{\calA}\calA_{\delta}\geq\delta$.
\end{aprop}
The argument is in fact sketched in \cite{Ngo Decomp}. We give a detailed proof for completeness.

We shall first study the tangent map of the Hitchin fibration  $f:\MHit\to \calA$. Let $(\calE,\varphi)\in\MHit(k)$ with image $a\in\Aa(k)$. Let
\begin{equation*}
T\fHit|_{(\calE,\varphi)}:T_{(\calE,\varphi)}\MHit\to T_{a}\calA
\end{equation*}
be the tangent map of $\fHit$ at $(\calE,\varphi)$. On the other hand, the action of  $\calP_{a}$ on $\MHit_{a}$ induces a map
\begin{equation*}
\act_{(\calE,\varphi)}:\Lie\calP_{a}\to T_{(\calE,\varphi)}\MHit.
\end{equation*}

\begin{lemma}\label{l:coker T} We have
\begin{equation}\label{tangent and ker act}
\dim_{k}\coker(T\fHit|_{(\calE,\varphi)})=\dim\ker(\act_{(\calE,\varphi)}).
\end{equation}
\end{lemma}
\begin{proof}
We have $T_{(\calE,\varphi)}\MHit=\cohog{0}{X,\calK}$ where $\calK=[\Ad(\calE)\xrightarrow{[\varphi,-]}\Ad(\calE)\otimes\calL]$ placed in degrees -1 and 0. The tangent space $T_{a}\calA$ can be identified with $\calA$ itself, hence equal to $\cohog{0}{X,\cc_{\calL}}$. The tangent map is obtained by taking $\cohog{0}{X,-}$ of the map of complexes
\begin{equation}\label{adad}
\xymatrix{\Ad(\calE)\ar[d]\ar[r]^{[\varphi,-]} & \Ad(\calE)\otimes\calL\ar[d]^{\partial^{\varphi}}\\ 0\ar[r] & \cc_{\calL}}
\end{equation} 
Note that $\cc_{\calL}=\oplus_{i=1}^{r}\calL^{d_{i}}$ given by the fundamental invariants $f_{1},\cdots, f_{r}$ on $\gg$. The map $\partial^{\varphi}$ sends $\alpha\in\Ad(\calE)\otimes \calL$ to $\{\frac{d}{dt}f_{i}(\varphi+t\alpha)|_{t=0}\}_{i=1,\cdots, r}$. Since $\deg\calL>\deg\omega_{X}$, we have $\cohog{1}{X,\calK}=0$ by \cite[Th\'eor\`eme 4.14.1]{NgoFL}. Therefore the cokernel of the tangent map $T_{(\calE,\varphi)}\MHit\to T_{a}\calA$ is the same as $\cohog{0}{X,\calC}$ where $\calC$ is the cone of the map \eqref{adad}, which is represented by the three 
term complex (in degrees -2, -1 and 0)
\begin{equation*}
\calC=[\Ad(\calE)\xrightarrow{[\varphi,-]}\Ad(\calE)\otimes\calL\xrightarrow{\partial^{\varphi}}\cc_{\calL}].
\end{equation*}
By calculation at the generic point of $X$, using the fact that $\varphi$ is generically regular semisimple, we see that $\calH^{-1}\calC$ and $\calH^{0}\calC$ are both torsion. A simple spectral sequence argument shows that $\cohog{0}{X,\calC}=\cohog{0}{X,\calH^{0}\calC}$. In other words, we have
\begin{equation}\label{tangent and coker}
\dim_{k}\coker(T_{(\calE,\varphi)}\MHit\to T_{a}\calA)=\leng_{k}\coker(\partial^{\varphi}).
\end{equation}

Recall we have a natural embedding $\Lie J_{a})\incl\ker([\varphi,-])$. Next we would like to relate $\coker(\partial^{\varphi})$ and $\calQ:=\ker([\varphi,-])/\Lie J_{a}$. We have an isomorphism $\Lie (J_{a})\cong\cc_{\calL}^{\vee}\otimes\calL$ (see \cite[Proposition 4.13.2]{NgoFL}). The Killing form on $\gg$ induces a self-duality on $\Ad(\calE)$. Under this duality, the natural map $\Lie J_{a}\to\Ad(\calE)$ is dual to $\partial^{\varphi}:\Ad(\calE)\to\cc_{\calL}\otimes\calL^{-1}$, as can be checked at the generic point of $X$. Therefore $\coker(\partial^{\varphi})\cong\underline{\Ext}^{1}_{\calO_{X}}(\calQ,\calL)$, and in particular,
\begin{equation}\label{coker and Q}
\leng_{k}\calQ=\leng_{k}\coker(\partial^{\varphi}).
\end{equation}
The long exact sequence associated with the exact sequence $0\to \Lie J_{a}\to \ker([\varphi,-])\to \calQ\to0$ looks like
\begin{equation}\label{longQQ}
\cdots\to \cohog{0}{X,\ker([\varphi,-])}\to\cohog{0}{X,\calQ}\to \cohog{1}{X,\Lie J_{a}}\to \cohog{1}{X,\ker([\varphi,-])}\to\cdots.
\end{equation}
Since $a\in\Aa$,  $\Lie\Aut(\calE,\varphi)=\cohog{0}{X,\ker([\varphi,-])}=0$. From this we also know that $\cohog{1}{X,\ker([\varphi,-])}\incl\cohog{0}{X,\calK}=T_{(\calE,\varphi)}\MHit$. Moreover, the composition $\cohog{1}{X,\Lie J_{a}}\to \cohog{1}{X,\ker([\varphi,-])}\incl T_{(\calE,\varphi)}\MHit$ is the map $\act_{(\calE,\varphi)}$. Therefore, by \eqref{longQQ}, $\cohog{0}{X,\calQ}$ is exactly the kernel of the map $\act_{(\calE,\varphi)}$. Combining with \eqref{coker and Q} and \eqref{tangent and coker}, we get \eqref{tangent and ker act}.
\end{proof}

\begin{proof}[Proof of Proposition \ref{p:codim}]. Let $Z\subset\calA_{\delta}$ be an irreducible component. We have an diagonalizable group scheme $R$ over $Z$ whose fiber $R_{a}$ is the affine part of $\calP_{a}$. Fiberwise the fixed points of $R_{a}$ on $\MHit_{a}$ is nonempty. Therefore the fixed point locus of $R$ on $\MHit_{Z}$ is a closed substack $Y$ mapping surjectively to $Z$. For any $(\calE,\varphi)\in Y$ over $a\in Z$, we have $\dim\ker(\act_{(\calE,\varphi)})\geq\dim R_{a}=\delta$, therefore, by Lemma \ref{l:coker T}, we have 
\begin{equation}\label{coker delta}
\dim\coker(T\fHit|_{(\calE,\varphi)})\geq \delta.
\end{equation}
However, since $Y\to Z$ is surjective, its tangent map (at the smooth points of $Y$) must be surjective at some point $(\calE_{0},\varphi_{0})$ (since we are in characteristic zero). At this point the rank of $T\fHit|_{(\calE_{0},\varphi_{0})}$ is at least $\dim Z$, hence we have $\dim\coker(T\fHit|_{(\calE_{0},\varphi_{0})})\leq\codim_{Z}(\calA)$. Combined with \eqref{coker delta} we get $\codim_{Z}(\calA)\geq\delta$.
\end{proof}

\textbf{Acknowledgement} The authors would like to thank Roman Bezrukavnikov, Pavel Etingof, Ivan Losev and Emily Norton for useful discussions. A.O. is  supported by the NSF grant DMS-1001609 and by the Sloan Foundation.
Z.Y. is supported by the NSF grant DMS-1302071 and the Packard Fellowship.

\end{document}